\documentclass[11pt]{article}
\usepackage{geometry}
\usepackage[english]{babel}
\usepackage{amsfonts,amsthm,amsmath,amssymb,amscd,mathrsfs,mathtools}
\usepackage{enumitem}
\usepackage{microtype}
\usepackage{titlesec}
\usepackage[numbers,sort&compress]{natbib}
\usepackage{hyperref}
\hypersetup{hidelinks,pdftitle={On the index map for AF-by-discrete groupoids},pdfauthor={Zheng Kuang},pdfsubject={2020 MSC: Primary 22A22, 20F65; Secondary 37B05, 20J05},pdfkeywords={AF-by-discrete groupoids, topological full groups, index map, groupoid homology, groups of germs, persistent boundary, almost finiteness, AH conjecture}}

\usepackage{mathtools}

\newcommand\addtag{\refstepcounter{equation}\tag{\theequation}}
\makeatletter
\newcommand*{\rom}[1]{\expandafter\@slowromancap\romannumeral #1@}
\makeatother

\titleformat{\subsection}[runin]{\normalfont\bfseries}{\thesubsection}{0.5em}{}[.]
\theoremstyle{plain}
\newtheorem{thm}{Theorem}[section]
\newtheorem{mainthm}{Theorem}

\newtheorem{prop}[thm]{Proposition}
\newtheorem{cor}[thm]{Corollary}
\newtheorem{lemma}[thm]{Lemma}
\theoremstyle{definition}
\newtheorem{defn}[thm]{Definition}
\newtheorem{rmk}[thm]{Remark}
\newtheorem{exmp}[thm]{Example}
\newtheorem{problem}[thm]{Problem}
\newcommand{\G}{\mathfrak{G}}
\newcommand{\T}{\mathfrak{T}_{\mathsf{B}}}
\newcommand{\F}{\mathsf{F}}
\newcommand{\Sym}{\mathsf{S}}
\newcommand{\Alt}{\mathsf{A}}
\newcommand{\Der}{\mathsf{D}}
\newcommand{\Id}{\operatorname{Id}}
\newcommand{\ab}{\mathrm{ab}}
\newcommand{\Dom}{\operatorname{Dom}}
\newcommand{\Ran}{\operatorname{Ran}}
\newcommand{\Tor}{\operatorname{Tor}}
\newcommand{\sg}{\mathsf{s}}
\newcommand{\rg}{\mathsf{r}}
\makeatletter
\newcommand{\keepwithstatement}{\par\nobreak\AddToHookNext{cmd/deferred@thm@head/before}{\@beginparpenalty=10000\relax}}
\makeatother
\numberwithin{equation}{section}
\title{On the index map for AF-by-discrete groupoids}
\author{Zheng Kuang\thanks{School of Mathematics, South China University of Technology, 381 Wushan Rd, Tianhe District, Guangzhou 510641, China; Email: kzkzkzz@scut.edu.cn}}
\date{}
\begin{document}
\maketitle
\begin{abstract}
We give an explicit geometric and algebraic description of the index map for AF-by-discrete groupoids. The index records transport between infinite tiles and abelianized returns in the groups of germs, subject to a defect in the dimension group of the associated Bratteli diagram. We determine the kernel of the index map and construct transposition factorizations of zero-index elements with trivial germs at singleton orbits. In particular, in every minimal case, this kernel equals the subgroup generated by dynamical transpositions. We also characterize approximation by the fixed AF core and prove almost finiteness for minimal systems that are compactly generated. We make higher homology explicit in terms of the groups of germs. We express the secondary parity differential by lifted surface relations and by an extension class constructed from balanced boundary data. Under minimality and comparison, lifting finite relations gives a split AH sequence. Explicit boundary and matrix calculations, together with a nonminimal example with nonzero parity differential, illustrate the theory.
\end{abstract}

\medskip
\noindent\textbf{2020 Mathematics Subject Classification.} Primary 22A22, 20F65; Secondary 37B05, 20J05.

\noindent\textbf{Keywords.} AF-by-discrete groupoids, topological full groups, index map, groupoid homology, groups of germs, persistent boundary, almost finiteness, AH conjecture.

\tableofcontents
\newpage

\section{Introduction}
Let $\G$ be an effective AF-by-discrete groupoid whose unit space is the space of infinite paths $\Omega(\mathsf{B})$, assumed to be a Cantor space, of a Bratteli diagram $\mathsf{B}$, and let $\T\subseteq\G$ be the tail groupoid of $\mathsf{B}$, fixed throughout as the AF core. The subgroupoid $\T$ is open and AF, while $\G\setminus\T$ is discrete. We write $\F(\G)$ for the topological full group and
\begin{equation*}
I:\F(\G)\longrightarrow H_1(\G)
\end{equation*}
for the index map. Here $H_1(\G)$ denotes first groupoid homology, and $H_1(\G,\T)$ denotes the relative homology of the inclusion $\T\subseteq\G$, recalled in Subsection~\ref{subsec:AFbd-homology}. 

Relative to the fixed core $\T$, every $g\in\F(\G)$ has only finitely many non-AF germs. We organize their contribution in terms of the exceptional boundary data: transport and normalized returns, defined geometrically in Subsection~\ref{subsec:boundary-returns}, determine a relative homology class, while the source--range defect in the dimension group associated with $\mathsf{B}$ determines when that class comes from $H_1(\G)$. Under the hypotheses of Subsection~\ref{subsec:hypothesis-scope} that make the boundary quotient finite, these data reduce to finitely many persistent configurations. We use them to determine $\ker I$, identify the secondary parity differential in the AH sequence, and obtain the results on almost finiteness.

Historically, the index map first appeared in the work of Giordano, Putnam, and Skau \cite[Section~5]{GPS99}. For a minimal homeomorphism of a Cantor space, they introduced an integer-valued homomorphism on the topological full group and called it the index map. It records the ``net migration'' of an element over the orbit cut and sends the generating homeomorphism to $1$. They showed, in particular, that this homomorphism is canonical up to normalization, that its kernel still determines the system up to flip conjugacy, and that the normalizer of the topological full group can be described in terms of this kernel \cite[Section~5]{GPS99}. Already in the original setting, the index separates the global direction of motion from a large zero-index subgroup that retains substantial dynamical information.

Matui subsequently placed this construction in groupoid homology: for an essentially principal \'etale groupoid on a Cantor space, a full bisection defines a class in first homology and thus a homomorphism from the topological full group to that homology group \cite[Definition~7.1 and Remark~7.2]{Matui12}. For \emph{almost finite} groupoids in the Hausdorff setting of \cite{Matui12}, he proved that the index map is surjective and that every zero-index element is a product of four elements of finite order. He also related the index to the $K_1$-group of the groupoid $C^*$-algebra \cite[Theorems~7.5 and~7.13 and Corollary~7.15]{Matui12}. For minimal Cantor systems, the kernel of the index, together with the signature with values in the dimension group modulo two introduced in \cite{Matui06}, became a basic tool for describing the normal subgroup structure and the abelianization of the full group. In later work, Matui formulated the AH conjecture \cite[Conjecture~2.9]{Matui16}, in which the index is the map from the abelianization of the full group onto $H_1$, while the remaining contribution comes from parity in $H_0(-;\mathbb{Z}/2\mathbb{Z})$.

The dynamical symmetric and alternating groups $\Sym(\G)$ and $\Alt(\G)$ were introduced by Nekrashevych in \cite[Subsection~3.2]{nnek19}. They are the groupoid analogues of the finite symmetric and alternating groups, defined from multisections. Nekrashevych later incorporated these groups, together with the index map, AF subgroupoids, relative homology, and almost finiteness, into the general structure theory of full groups of \'etale groupoids \cite[Sections~5.4--5.5]{nek22}. In particular, his calculation with one cut in \cite[Section~5.5]{nek22} gives a geometric interpretation of the index as signed transport across the cut, while the almost-finite theory relates the kernel of the index to the subgroup generated by dynamical transpositions. More recently, Li placed the AH sequence in a stable homotopy-theoretic framework and proved Matui's AH conjecture for minimal ample groupoids with comparison \cite[Corollary~6.14]{Li25}. Li and Miller subsequently developed a general discretisation and independent-resolution framework for ample groupoids \cite{LiMiller26}. These developments explain why the index map is central in the subject: it is simultaneously a dynamical measure of transport, a homological invariant, and a map that identifies the quotient by the zero-index subgroup with its image in homology.

The class of \'etale groupoids studied here, that is, AF-by-discrete groupoids, was isolated and named by Nekrashevych in \cite[Subsection~5.2.4]{nek22}. He observes there that most examples in the remaining part of his book have this form, and gives groups generated by bounded automata as a basic family. The corresponding picture for systems of bounded type appeared earlier in the work of Juschenko, Nekrashevych, and de la Salle \cite{JNS16}: homeomorphisms of bounded type on Bratteli path spaces have only finitely many non-AF germs, and the same framework contains both Bratteli--Vershik models of minimal Cantor systems and groups acting on rooted trees by bounded automorphisms. In \cite{nek22}, the AF-by-discrete language is developed in Chapter~5, while Chapter~6 uses the same AF core and exceptional boundary in the study of growth and amenability.

This places AF-by-discrete groupoids at the intersection of two classical sources of topological full groups. On the one hand, a minimal Cantor system admits a Bratteli--Vershik model in which the tail equivalence relation is the AF core and the adic transformation differs from it only at the orbit cut. On the other hand, for a group generated by bounded automata, the AF core is the groupoid of finitary germs and the non-AF behavior is confined to persistent boundary points on the infinite tiles, in the sense made precise in Definition~\ref{def:persistent-boundary-point}. The Grigorchuk group, whose intermediate growth was established in Grigorchuk’s paper~\cite{gri84}, is the paradigmatic example of this family. Hence the same structure appears on both sides: the AF core controls the regular part of the dynamics, while the exceptional boundary records how the AF pieces are connected, returned to, and transported between. In this sense, the exceptional germs do not merely describe the dynamics outside the AF core. They organize the way in which the regular pieces assemble into the global dynamics. We therefore take the exceptional boundary, rather than the AF core alone, as the basic object to be computed.

The boundary data below are relative to the specified AF core $\T$: replacing $\T$ can change $H_1(\G,\T)$, the return and transport coordinates, and the defect map into $H_0(\T)$. By contrast, $H_\bullet(\G)$, the index map $I$, and $\ker I$ are intrinsic to $\G$. Thus the AF core supplies coordinates for intrinsic objects rather than an additional invariant of the groupoid.

The main results of the paper consist of four theorems. The first three concern the index map and the AH sequence, while the fourth concerns almost finiteness. The first gives the geometric form of the index (Theorems~\ref{thm:relative-normal-form} and~\ref{thm:geometric-index}). For each exceptional $\G$-orbit $\mathcal{O}$, choose a basepoint $x_{\mathcal{O}}$, let $H_{\mathcal{O}}=\G_{x_{\mathcal{O}}}^{x_{\mathcal{O}}}$ be its group of germs, and write $H_{\mathcal{O}}^{\ab}$ for its abelianization. Let $L_{\mathcal{O}}$ be the transport lattice of Definition~\ref{def:transport-lattice} for the set of $\T$-orbits in $\mathcal{O}$, that is, the integer vectors of finite support whose coordinates sum to zero.

\begin{mainthm}[Geometric index theorem]
\label{thm:introduction-index}
There is a noncanonical decomposition
\begin{equation*}
H_1(\G,\T)\cong
\bigoplus_{\mathcal{O}}
\left(H_{\mathcal{O}}^{\ab}\oplus L_{\mathcal{O}}\right),
\end{equation*}
where the sum runs over the exceptional $\G$-orbits. Under the injection $\iota:H_1(\G)\hookrightarrow H_1(\G,\T)$, the index of $g\in\F(\G)$ is obtained by summing the finitely many non-AF germs of $g$. Each such germ contributes the class of its normalized return in $H_{\mathcal{O}}^{\ab}$ and the difference between the $\T$-orbits of its range and source. The connecting map to the dimension group $H_0(\T)$ is the corresponding source--range defect. Under the hypotheses of Subsection~\ref{subsec:hypothesis-scope} that make the boundary quotient finite, connectors adapted to a tree express transport as signed net migration across its cuts.
\end{mainthm}
The injectivity used here is special to the AF inclusion: since $H_1(\T)=0$, the long exact sequence of the pair gives the injection $\iota$ in~\eqref{eq:relative}. Hence, under $\iota$, $H_1(\G)$ consists precisely of the relative homology classes with zero defect in the dimension group, with cancellation allowed between different exceptional orbits. The exceptional orbits are independent at the level of relative homology, but become globally coupled through their defect classes in $H_0(\T)$. This is a description of absolute first homology. Surjectivity of the index map $I:\F(\G)\to H_1(\G)$ is a separate assertion, obtained under minimality of the AF core in Corollary~\ref{cor:index-surjective}. The theorem recovers the one-cut formula in \cite[Section~5.5]{nek22} for minimal Cantor systems and extends it to arbitrary finite boundary configurations.

The second principal result determines exactly what can remain in the kernel after the relative index class has vanished (Theorem~\ref{thm:singleton-kernel}).
\begin{mainthm}[Kernel theorem]
\label{thm:introduction-kernel}
Germ evaluation gives an exact sequence
\begin{equation*}
1\longrightarrow\Sym(\G)\longrightarrow\ker I
\longrightarrow\bigoplus_{\substack{x\in\Omega(\mathsf{B})\\\G x=\{x\}}}
[\G_x^x,\G_x^x]\longrightarrow1.
\end{equation*}
The sum is the restricted direct product. Hence $\ker I=\Sym(\G)$ if and only if every group of germs at a singleton $\G$-orbit is abelian. In particular, this equality holds when there are no singleton orbits, and thus when $\G$ is minimal.
\end{mainthm}
Here $\Sym(\G)$ is the dynamical symmetric group, recalled in Subsection~\ref{subsec:full-groups-generation}. The proof gives an explicit transposition factorization from finite matchings inside $\T$ and commutator expressions in the groups of germs, thereby identifying the entire quotient $\ker I/\Sym(\G)$ directly from the AF-by-discrete structure. Li's general kernel reduction \cite[Corollary~6.17]{Li25} assumes minimality and comparison. Maps obtained by abelianizing germs and by recording transport between orbits occur in the finite germ extensions of Belk, Hyde, and Matucci \cite{BHM24}. Our construction retains the defect in the dimension group and uses swaps between two points in place of their localization hypothesis at a single singularity. 

The third principal result identifies the part of the AH sequence that is invisible to the index map. Higher homology is concentrated in the exceptional groups of germs. Under this identification, the secondary parity differential is represented by balanced returns at the boundary. The AF parity homomorphism of the AF core used in this formulation is recalled in Subsection~\ref{subsec:AFbd-homology}. The higher relative reduction is already implicit in \cite[Section~5.4]{nek22}. The point here is to make the exceptional summands and the differential explicit. More precisely, Li's exact sequence in low degrees contains the Atiyah--Hirzebruch differential $d^2_{2,0}$ \cite[Theorem~6.12 and its proof]{Li25}. We denote this secondary differential by
\begin{equation*}
\theta_{\G}:H_2(\G;\mathbb{Z})\longrightarrow H_0(\G;\mathbb{Z}/2\mathbb{Z}).
\end{equation*}
We write $\F(\G)_{\ab}$ for the abelianization of the full group.

\begin{mainthm}[Parity of balanced returns]
\label{thm:introduction-parity}
For every abelian group $A$ with trivial $\G$-action and every $n\geq2$,
\begin{equation*}
H_n(\G;A)\cong
\bigoplus_{\mathcal{O}}H_n(H_{\mathcal{O}};A),
\end{equation*}
where the sum runs over the exceptional $\G$-orbits. If $\mathcal{O}$ contains at least three points and $c\in H_2(H_{\mathcal{O}};\mathbb{Z})$, then $\theta_{\G}(c)$ is the image in $H_0(\G;\mathbb{Z}/2\mathbb{Z})$ of the AF parity of a surface relation representing $c$, after it is balanced and lifted. In the minimal case, every exceptional orbit is infinite, and thus these parity maps obtained from balanced returns determine the whole differential. If finite relations can be lifted at the exceptional orbits, then $\theta_{\G}=0$ and the stable AH sequence splits. Under minimality and comparison, the corresponding sequence for $\F(\G)_{\ab}$ splits as well.
\end{mainthm}
Theorem C combines three results proved in the body of the paper. The higher-homology decomposition is Theorem~\ref{thm:higher-germ-homology}. The formula for the parity of balanced returns is Theorem~\ref{thm:balanced-surface-return-parity}, together with Corollary~\ref{cor:minimal-balanced-parity}. The splitting statement under the lifting property for finite relations is Theorem~\ref{thm:finite-relations-split-AH}.

The lifting property for finite relations is defined in Definition~\ref{def:finite-relation-lifting}. Conceptually, it means that any finitely presented system of relations among germs can be realized by an actual local action on one clopen neighborhood of the basepoint, with all nonbasepoint germs there lying in the AF core. The stable and unstabilized AH sequences are distinguished at the beginning of Section~\ref{sec:parity}. The resulting parity is obtained by pairing a surface class with an extension class from the boundary. More precisely, Section~\ref{sec:boundary-extensions} associates, after choosing three marked points and connector coordinates in an exceptional orbit, a balanced boundary extension class
\begin{equation*}
[\omega_{\mathcal{O}}]\in H^2(B_3(H_{\mathcal{O}});\mathsf{P}_{\mathsf{B}}),
\qquad \mathsf{P}_{\mathsf{B}}=H_0(\T;\mathbb{Z}/2\mathbb{Z}).
\end{equation*}
Here $B_3(H_{\mathcal{O}})$ is the group of triples whose abelianized coordinates sum to zero (Definition~\ref{def:balanced-triple-group}). The class is defined for the chosen marked points and connectors. After applying the natural map $q_2:\mathsf{P}_{\mathsf{B}}\to H_0(\G;\mathbb{Z}/2\mathbb{Z})$, its evaluations on balanced surface classes are intrinsic and equal to $\theta_{\G}$. Under minimality and comparison, the strong AH property (Definition~\ref{def:AH-terms}) is therefore equivalent to the vanishing of these pairings. Vanishing of the entire pushed-forward cohomology class is a sufficient, generally stronger, condition.

The connection with the Kapoudjian obstruction is made explicit following \cite{Cornulier19}. The group $\mathbb{Z}^2$ occurs as a group of germs both in a minimal sparse model with zero differential and in a nonminimal Houghton--Cantor model with nonzero differential. Therefore, $\theta_{\G}$ depends on how return relations are realized topologically at the boundary, beyond the abstract group of germs.

The same exceptional boundary also controls almost finiteness (Theorems~\ref{thm:AF-exhaustion-criterion} and~\ref{thm:minimal-AF-by-discrete-almost-finite}).
\begin{mainthm}[Almost finiteness from the exceptional boundary]
\label{thm:introduction-approximation}
Every increasing elementary exhaustion of $\T$ gives almost-finite approximations to $\G$ if and only if every non-AF arrow has source in an infinite $\T$-orbit. Consequently, every minimal AF-by-discrete groupoid on a Cantor space that is compactly generated is almost finite, even when its AF core is not minimal.
\end{mainthm}
The proof combines uniform averaging over $\T$ with uniformly small neighborhoods of the exceptional sources of a compact bisection. The thin-set mechanism goes back to Phillips \cite{Phillips05}. Proposition~5.5.2 of \cite{nek22} proves almost finiteness for systems of bounded type with infinite orbits. The argument here extends this conclusion to minimal compactly generated AF-by-discrete groupoids. No assumption of bounded type or finite isotropy is required. The fixed AF exhaustion and almost finiteness of the whole groupoid are kept separate: when the original core has finite orbits, the proof passes to a full clopen reduction and lifts elementary approximations back to the original unit space.

The paper is organized as follows. Section~\ref{sec:preliminaries} reviews basic notions that will be used in this paper. Section~\ref{sec:boundary} defines boundary paths and returns and classifies persistent boundary configurations. Section~\ref{sec:germs} proves that the groups of germs are finitely generated under compact generation and distinguishes the length of a boundary word from the finitary depth of a section. It treats both general Bratteli models and rooted-tree models. Section~\ref{sec:index} proves the relative normal form, the formulas for the defect at finite levels, and the geometric index theorem. Section~\ref{sec:kernel} studies approximation by the fixed AF core and proves the almost-finiteness results, while Section~\ref{sec:constructive-kernel} gives a direct cancellation argument at the boundary and proves the complete kernel theorem. Section~\ref{sec:certified} develops periodic boundary certificates and carries out explicit computations from verified boundary data. Sections~\ref{sec:parity} and~\ref{sec:finite-return-relations} identify higher homology, analyze coherent realizations and the lifting property for finite relations, prove splitting results, and construct the nonminimal example with nonzero secondary differential. Section~\ref{sec:surface-parity} gives the formula from a single return and the formula obtained after balancing three surface returns for $\theta_{\G}$ and formulates the remaining problem of strong AH in the minimal case as the existence of a boundary signature. Section~\ref{sec:boundary-extensions} packages the same obstruction as an extension class constructed from balanced boundary data. Section~\ref{sec:applications} records applications to near full groups and minimal Cantor systems and collects the remaining questions.

\section{Preliminaries}
\label{sec:preliminaries}
Let $\mathsf{B}$ be a Bratteli diagram (recalled in Subsection~\ref{subsec:bratteli-core}). Denote by $\Omega(\mathsf{B})$ its space of infinite paths, by $\Omega_n(\mathsf{B})$ its set of paths of length $n$, and by $\Omega^*(\mathsf{B})$ its set of finite paths. In rooted-tree specializations with alphabet $\mathsf{X}$, the boundary is denoted $\mathsf{X}^{\omega}$. Throughout, $\G$ denotes an effective AF-by-discrete groupoid with unit space $\Omega(\mathsf{B})$, and $\T\subseteq\G$ denotes the fixed AF core associated with $\mathsf{B}$. For $x\in\Omega(\mathsf{B})$, we write $\G x$ for its orbit and $\G_x^x$ for its isotropy group. We write $\F(\G)$ for the topological full group, $\Sym(\G)$ and $\Alt(\G)$ for the dynamical symmetric and alternating groups, and $H^{\ab}$ for the abelianization of a group $H$. The relative homology of the inclusion $\T\subseteq\G$ is denoted $H_\bullet(\G,\T)$. When a finite inverse-closed generating family of compact open bisections is chosen, it is denoted $\mathcal{S}$ and its elements are denoted $F$. Identity germs are written $(\Id,x)$ or $(\Id,\xi)$. We write $I:\F(\G)\to H_1(\G)$ for the index map, recalled below.

All topological groupoids are locally compact and \'etale, and the arrow space is allowed to be non-Hausdorff. Auxiliary groupoids with other unit spaces will be specified when they are used. We retain the notation for generators, germs, and tiles from \cite{kua26,kua25,kua26a,kua26b}. Words of transformations act from right to left, while Bratteli paths are read from left to right. Thus, $F_n\cdots F_1$ applies $F_1$ first.

\subsection{Partial homeomorphisms, germs, and groupoids}
\label{subsec:germs-prelim}
A \emph{partial homeomorphism} of $\Omega(\mathsf{B})$ is a homeomorphism $f:\Dom(f)\to\Ran(f)$ between open subsets. An \emph{inverse semigroup of partial homeomorphisms} is a collection closed under composition and inverses. We use compact open domains and ranges when specifying finite generating families. Restricting the maps to smaller open sets gives the same local transformations and the same groupoid of germs. These conventions agree with \cite[Section~3.1]{nek22}.

We use the standard germ construction in \cite[Section~3.1]{nek22}.
\begin{defn}[Germ and groupoid of germs]
\label{def:germ-groupoid}
For $x\in\Dom(f)$, the \emph{germ} of $f$ at $x$, denoted $(f,x)$, is its equivalence class under local agreement: $(f,x)=(h,y)$ means that $x=y$ and that $f$ and $h$ coincide on some open neighborhood of $x$ contained in both domains.

Assume that the domains of the partial maps cover the unit space. Their local identity maps provide a unit germ at every point. The \emph{groupoid of germs} of an inverse semigroup $G_0$ consists of all these germs. Its unit at $x$ is $(\Id,x)$, and its operations are
\begin{equation}
\label{eq:germ-operations-prelim}
(f,h(x))(h,x)=(fh,x),\qquad (f,x)^{-1}=(f^{-1},f(x)).
\end{equation}
For a group action, we write $\operatorname{Germ}(G\curvearrowright\Omega(\mathsf{B}))$ for its groupoid of germs.
\end{defn}
The groupoid of germs differs from the transformation groupoid $G\ltimes\Omega(\mathsf{B})$: the latter retains the group label, whereas the former identifies labels that agree on a neighborhood of the source.

We write $\sg,\rg:\G\to\G^{(0)}$ for the source and range maps, which should be distinguished from those for Bratteli diagrams discussed below, where we use $\mathbf{s},\mathbf{r}$. The product $\alpha\beta$ is defined when $\sg(\alpha)=\rg(\beta)$ and applies $\beta$ first. For a point $x$, we use
\begin{equation*}
\G_x=\sg^{-1}(x),\qquad
\G x=\rg(\G_x),\qquad
\G_x^x=\{\gamma\in\G:\sg(\gamma)=\rg(\gamma)=x\}.
\end{equation*}
For a groupoid of germs, the isotropy group is also called the \emph{group of germs at $x$}. For a group action, it is canonically
$H_x=\G_x^x\cong G_x/G_{(x)}$, where here $G_x=\{g:g(x)=x\}$ is the stabilizer in the acting group and $G_{(x)}=\{g:g\text{ is the identity on a neighborhood of }x\}$. The point $x$ is \emph{regular} if $H_x=\{(\Id,x)\}$, and \emph{singular} otherwise. An arrow $c:x\to y$ identifies $H_x$ with $H_y$ by $h\mapsto chc^{-1}$, and thus regularity and singularity are constant on $\G$-orbits.

The topology on a groupoid of germs has basic sets $U(f,A)=\{(f,x):x\in A\}$, where $A\subseteq\Dom(f)$ is open. A \emph{bisection} is a subset on which both source and range are injective. For an open bisection, these maps identify it homeomorphically with open subsets of the unit space. The set $U(f,A)$ is an open bisection, and is compact when $A$ is compact. For a bisection $U$ and $A\subseteq\sg(U)$, write $U|_A=U\cap\sg^{-1}(A)$. The induced partial homeomorphism is $\theta_U=\rg\circ(\sg|_U)^{-1}:\sg(U)\to\rg(U)$. For a clopen set $A\subseteq\G^{(0)}$, write
$1_A=\{(\Id,x):x\in A\}$ for the corresponding unit bisection. In a chain group, the same symbol $1_A$ denotes the indicator function of $A$. The meaning is determined by context. In the present Cantor setting, every arrow has a compact open bisection neighborhood after shrinking an \'etale chart. The arrow space may be non-Hausdorff even though every compact open bisection is Hausdorff, being homeomorphic to its source in the Cantor space.

We use the standard terminology for principal and effective \'etale groupoids as in \cite[Chapter~3]{nek22}.
\begin{defn}[Principal and effective groupoids]
\label{def:principal-effective}
A groupoid is \emph{principal} if every isotropy group is trivial. It is \emph{effective} if the interior of its isotropy bundle consists precisely of unit arrows.
\end{defn}
In the effective setting, an arrow is determined by the germ of the local homeomorphism induced by any bisection containing it. We thus use germ language also for an abstract effective groupoid in the present setting. The groupoid $\G$ is assumed to be effective unless a statement explicitly dispenses with this assumption.

\subsection{Full groups, compact generation, and comparison}
\label{subsec:full-groups-generation}
Topological full groups of minimal Cantor systems go back to Giordano--Putnam--Skau \cite{GPS99}. We use the multisection formulation and the dynamical symmetric and alternating groups introduced by Nekrashevych in \cite[Subsection~3.2]{nnek19}. See also \cite[Section~5.1]{nek22}.
\begin{defn}[Topological full group and dynamical symmetric and alternating groups]
\label{def:full-symmetric-alternating}
A \emph{full bisection} $U$ has $\sg(U)=\rg(U)=\G^{(0)}$. The \emph{topological full group} $\F(\G)$ consists of the homeomorphisms induced by compact open full bisections. A \emph{multisection of degree $d$} is a family of compact open bisections $(V_{ij})_{1\leq i,j\leq d}$ with pairwise disjoint diagonal sets $V_{ii}\subseteq\G^{(0)}$, source $\sg(V_{ij})=V_{jj}$, range $\rg(V_{ij})=V_{ii}$, and multiplication $V_{ij}V_{jk}=V_{ik}$. Each permutation $\pi\in\operatorname{Sym}(d)$ gives an element of the full group with bisection $\bigcup_j V_{\pi(j),j}$ on these sets and the identity elsewhere. If $V$ has disjoint source and range, its \emph{dynamical transposition} has bisection
\begin{equation}
\label{eq:transposition-prelim}
t_V=V\cup V^{-1}\cup1_{\Omega(\mathsf{B})\setminus(\sg(V)\cup\rg(V))}.
\end{equation}
The \emph{dynamical symmetric group} $\Sym(\G)$ is generated by these transpositions, and the \emph{dynamical alternating group} $\Alt(\G)$ is generated by multisection $3$-cycles.
\end{defn}
Effectiveness makes the representing full bisection $U_g$ unique. For a group action, an element $g\in\F(\G)$ agrees on each piece of a finite clopen partition with an element of the acting group. We write $\Der(\G)=[\F(\G),\F(\G)]$ for the derived subgroup. A finite-order element may have a nontrivial germ at a fixed point and need not belong to $\Sym(\G)$. If $\mathfrak{F}$ is AF, however, then $\Sym(\mathfrak{F})=\F(\mathfrak{F})$: every full bisection is contained in an elementary stage, where it acts by finite permutations of multisection levels and hence is generated by dynamical transpositions. Lemma~\ref{lemma:AF-two-transpositions} gives the stronger factorization into at most two such transpositions. In particular, $\Sym(\T)=\F(\T)$. We write $C_n=\mathbb{Z}/n\mathbb{Z}$ for the cyclic group of order $n$. For any group $H$, write $H^{\ab}=H/[H,H]$. For an abelian group $A$, $\Tor A$ denotes its subgroup of finite-order elements.

We use compact generation in the form of \cite[Section~5]{nnek19} and \cite[Section~3.5]{nek22}.
\begin{defn}[Compact generation]
\label{def:compact-generation}
In the present Cantor \'etale setting, the groupoid $\G$ is \emph{compactly generated} if there is a finite family $\mathcal{S}$ of compact open bisections, closed under inverses, such that every arrow of $\G$ belongs to a finite composable product of elements of $\mathcal{S}$. Let $\operatorname{Bis}_c(\G)$ be the inverse semigroup of compact open bisections under product and inverse. For such a family $\mathcal{S}\subseteq\operatorname{Bis}_c(\G)$, write $\langle\mathcal{S}\rangle_{\mathrm{inv}}$ for the inverse subsemigroup it generates. We say that $\mathcal{S}$ \emph{generates $\G$} if every arrow of $\G$ belongs to a bisection in $\langle\mathcal{S}\rangle_{\mathrm{inv}}$.
\end{defn}
Compact generation concerns generation of the arrows, not finite generation of $\F(\G)$. Closure under products and restrictions is imposed on the generated semigroup and pseudogroup, respectively, rather than on $\mathcal{S}$. Restrictions to compact open sets form a basis for the associated pseudogroup of partial homeomorphisms, obtained by allowing restrictions to open sets and unions of compatible partial maps. In the effective setting, we identify arrows with the germs of these local transformations, as explained above. A finitely generated group action is the special case in which the members of $\mathcal{S}$ are full bisections coming from global homeomorphisms.

For a chosen finite symmetric generating family $\mathcal{S}$, the \emph{$\mathcal{S}$-Cayley graph of $\G$ at $x$} has vertex set $\G_x$ and edges $\gamma\to(F,\rg(\gamma))\gamma$ whenever $F\in\mathcal{S}$ is defined at $\rg(\gamma)$. See \cite[Section~3.5]{nek22}. In our setting of groupoids of germs, this is the \emph{graph of germs at $x$}. The range map projects it to the \emph{$\mathcal{S}$-orbital graph $\Gamma_x$} on the orbit $\G x$, with edges $y\to F(y)$. If the group of germs is trivial, the two graphs are isomorphic. Otherwise, the range projection identifies distinct germ vertices. Throughout, an orbit $\mathcal{O}=\G x$ means the unit-space orbit. We write 
\[\mathcal{T}(\mathcal{O}):=\mathcal{O}/\T=\{\T x:x\in\mathcal{O}\}\addtag\label{orbitquotient}\] 
for the set of $\T$-orbits contained in $\mathcal{O}$. The tile terminology below refers to subgraphs of the orbital graph: the $\T$-orbits in $\mathcal{O}$ are the vertex sets of its infinite tiles. When isotropy is nontrivial, the graph of germs may contain several lifted copies of a given tile. Subsection~\ref{subsec:boundary-quotient} describes this distinction.

For the notion of comparison, we use the groupoid formulation in \cite[Subsection~5.5.2]{nek22}.
\begin{defn}[Reduction, minimality, and comparison]
\label{def:minimality-comparison}
For $A\subseteq\G^{(0)}$, the \emph{reduction} $\G|_A$ consists of arrows with both endpoints in $A$. The set $A$ is an \emph{orbit transversal} if it meets every orbit. The groupoid is \emph{minimal} if every orbit is dense. A Borel probability measure $\mu$ on the unit space is \emph{$\G$-invariant} if $\mu(\sg(V))=\mu(\rg(V))$ for every compact open bisection $V$. Write $M(\G)$ for the set of such measures. We say that $\G$ has \emph{comparison} if, whenever clopen sets $A,B$, with $B\ne\varnothing$, satisfy $\mu(A)<\mu(B)$ for every $\mu\in M(\G)$, there is a compact open bisection with source $A$ and range contained in $B$.
\end{defn}
Minimality and comparison will be imposed when we use the unstabilized AH sequence of \cite[Corollary~6.14]{Li25}.

\subsection{Bratteli diagrams, AF cores, and exceptional arrows}
\label{subsec:bratteli-core}
Bratteli diagrams were introduced by Bratteli in \cite{Bratteli72}. We use the standard dynamical terminology recalled in \cite[Section~2]{Matui06}.

\begin{defn}[Bratteli diagram]
\label{def:Bratteli-diagram}
A \emph{Bratteli diagram} $\mathsf{B}=(V_n,E_n,\mathbf{s},\mathbf{r})$ consists of finite vertex sets $V_n$, finite edge sets $E_n$, and source and range maps $\mathbf{s}:E_n\to V_n$, $\mathbf{r}:E_n\to V_{n+1}$. We assume that every vertex has an outgoing edge and that the source and range maps between successive levels are onto. A path $p=e_1\cdots e_n$ starts in $V_1$ and satisfies $\mathbf{r}(e_i)=\mathbf{s}(e_{i+1})$. The set of paths of length $n$ is denoted $\Omega_n(\mathsf{B})$, and the space of finite paths is denoted $\Omega^*(\mathsf{B})$. An element of $\Omega_n(\mathsf{B})$ ends in $V_{n+1}$. The space of infinite paths is denoted $\Omega(\mathsf{B})$. For $p\in\Omega_n(\mathsf{B})$, the cylinder $[p]\subseteq\Omega(\mathsf{B})$ consists of infinite paths beginning with $p$. The diagram is \emph{simple} if for every level there is a later level to which every vertex of the former is connected to every vertex of the latter.
\end{defn}
The tail groupoid of a simple Bratteli diagram is minimal. We retain the Cantor assumption, since simplicity alone also allows a finite path space. \emph{Telescoping} retains selected levels and replaces each intervening path by one edge while preserving tail equivalence. A \emph{stationary} diagram has repeating incidence data after level identifications. The incidence homomorphism $M_n:\mathbb{Z}^{V_n}\to\mathbb{Z}^{V_{n+1}}$ sends the basis vector $e_v$ to $\sum_{\mathbf{s}(e)=v}e_{\mathbf{r}(e)}$. We use the same integer matrix with other abelian coefficients.

For paths $p,q\in\Omega_n(\mathsf{B})$ with $\mathbf{r}(p)=\mathbf{r}(q)$, let $[q,p]$ be the bisection of the \emph{prefix replacement} $pz\mapsto qz$, leaving the infinite continuation $z$ unchanged. We write $\mathfrak{T}_{\mathsf{B},n}$ for the groupoid of all such arrows and $\T=\bigcup_n\mathfrak{T}_{\mathsf{B},n}$ for the \emph{tail groupoid}. The orbit $\T x$ consists exactly of paths eventually equal to a tail of $x$. Between any two points of this orbit, there is a unique tail arrow. This is the principality used in boundary-return calculations. A compact open bisection contained in $\T$ admits, after a sufficiently deep cylinder refinement, a finite table of these prefix replacements. We call such a table \emph{finitary}. Its depth will be defined precisely in Subsection~\ref{subsec:general-sections}.

We use the AF-groupoid terminology of \cite[Section~5.2]{nek22}.
\begin{defn}[Elementary and AF groupoids]
\label{def:elementary-AF-groupoid}
An \emph{elementary groupoid} on the Cantor space is a compact open principal groupoid admitting a finite decomposition into multisections that cover its unit space. An \emph{AF groupoid} is an increasing union of elementary subgroupoids with the same unit space.
\end{defn}
The elementary stages $\mathfrak{T}_{\mathsf{B},n}$ give the standard exhaustion of $\T$. Conversely, we use a Bratteli presentation for the AF subgroupoid under consideration, as in \cite[Section~5.2]{nek22}.

The terminology \emph{AF-by-discrete} is from \cite[Subsection~5.2.4]{nek22}.
\begin{defn}[AF-by-discrete inclusion, exceptional arrows, and exceptional orbits]
\label{def:AF-by-discrete}
\label{def:exceptional-arrow-orbit}
An inclusion $\T\subseteq\G$ with the same unit space is \emph{AF-by-discrete} if $\T$ is an open AF subgroupoid and $\G\setminus\T$ is discrete in the arrow topology. We call the specified subgroupoid $\T$ the \emph{AF core}. An arrow in $\G\setminus\T$ is \emph{exceptional} or \emph{non-AF}. A $\G$-orbit is \emph{exceptional} if it contains the source of an exceptional arrow.
\end{defn}

The source and range of an arrow lie in the same $\G$-orbit, and thus an exceptional orbit equivalently contains the range of an exceptional arrow. We always consider an AF-by-discrete groupoid together with a fixed AF core unless another core is explicitly introduced. Exceptionality and singularity are different notions: an exceptional arrow can join different tail orbits at regular points. We use ``persistent boundary point'' only after specifying a generating family and a compatible tile presentation in Section~\ref{sec:boundary}.

The discreteness assumption gives the following finite-support and isolation property for compact open bisections.
\keepwithstatement
\begin{lemma}
\label{lemma:finite-exceptional-support}
For every compact open bisection $U\subseteq\G$, the set $U\setminus\T$ is finite. Each of its elements has a compact open neighborhood $V\subseteq U$ such that $V\setminus\T$ consists of that element alone.
\end{lemma}
\begin{proof}
Since $\T$ is open, $U\setminus\T$ is closed in the compact space $U$, and is thus compact. It is a subspace of the discrete complement, and thus it is finite. The source map identifies $U$ with a compact open subspace of the Hausdorff Cantor space $\Omega(\mathsf{B})$. Pairwise disjoint clopen source neighborhoods of the finitely many exceptional arrows give the required restrictions of $U$. 
\end{proof}
The proof applies equally when the arrow space of $\G$ is non-Hausdorff.

\begin{defn}[Isolating bisection]
\label{def:isolating-bisection}
Let $\gamma\in\G\setminus\T$. An \emph{isolating bisection for $\gamma$} is a compact open bisection $V\ni\gamma$ such that $V\setminus\T=\{\gamma\}$.
\end{defn}
A finite family of bisections that generates $\G$ has only finitely many exceptional arrows among its generators. These supply the finite exceptional boundary data and the finite generating set of each exceptional group of germs in Sections~\ref{sec:boundary}--\ref{sec:germs}. They also control the finite tail orbits in the proof of Theorem~\ref{thm:minimal-AF-by-discrete-almost-finite}. Individual bisections still have finite exceptional support without compact generation. The set of exceptions for an entire generating family need not then be finite.

Minimality of $\T$ is stronger than minimality of $\G$. Adding exceptional arrows can join a finite $\T$-orbit to a dense one. Lemmas~\ref{lemma:AF-orbit-closures} and~\ref{lemma:finitely-many-finite-AF-orbits} below give the precise restrictions on this phenomenon. 

\subsection{Compatible tile inflation and the scope of generators}
\label{subsec:tile-inflation}
\label{subsec:hypothesis-scope}
We recall the part of the tile inflation construction needed here. See \cite[Subsection~2.1]{kua26} and \cite{Bon}. A finite level-$n$ tile $\mathcal{T}_{v,n}$, indexed by $v\in V_{n+1}$, has vertex set $\{p\in\Omega_n(\mathsf{B}):\mathbf{r}(p)=v\}$. It is a finite directed labeled graph with inverse edges. A \emph{boundary edge} is a labeled edge based at a vertex on the tile whose opposite endpoint has not yet been identified with a vertex of the same finite tile. It therefore records a defined local transformation that has not yet been realized internally at that level. A vertex incident to such an edge is a \emph{boundary point} of the finite tile. Correspondingly, we call the labeled edges that are not boundary edges \emph{internal}. The labels obey the partial-bijection rule: for each label, there is at most one incoming and one outgoing internal edge with the same label at any vertex.

To inflate from level $n$ to level $n+1$, for each edge $e\in E_{n+1}$ ending at $w$ take a copy of $\mathcal{T}_{\mathbf{s}(e),n}$, with vertices renamed $pe$. Keep the internal labeled edges in each copy. The inflation rule then joins selected compatible boundary edges of these copies, with their inverse edges, to form new internal edges in $\mathcal{T}_{w,n+1}$. These joining instructions are the \emph{connectors at finite levels}. Boundary edges not joined at this stage remain unresolved. This describes both the inclusions of old tiles and the new connections. A boundary point may become internal at a later inflation.

The labels in a tile presentation are most naturally compact open bisections, hence partial homeomorphisms with specified source and range. We fix a finite family $\mathcal{S}\subseteq\operatorname{Bis}_c(\G)$ that is closed under inverses and generates $\G$ in the sense of Subsection~\ref{subsec:full-groups-generation}. The inverse semigroup $\langle\mathcal{S}\rangle_{\mathrm{inv}}$ and its restrictions to compact open sets supply the local transformations seen in the tiles. No global group action is assumed. If the groupoid comes from a finitely generated group action, the full bisections coming from the group generators form one special choice of $\mathcal{S}$.

\begin{defn}[Compatible tile presentation]
\label{def:compatible-tile-presentation}
A tile presentation is \emph{compatible with the AF core $\T$ and the bisection family $\mathcal{S}$} if the following hold. An internal edge labeled $F$ from $p$ to $q$ means $[p]\subseteq\sg(F)$ and $F(pz)=qz$ for every continuation $z$, where $F\in\mathcal{S}$ denotes both the bisection and its induced partial homeomorphism. The inverse edge is internal, and under refinement the corresponding edges from $pe$ to $qe$ remain internal. Conversely, every generator germ $(F,x)\in\T$ is represented by an internal edge at some sufficiently deep level. \end{defn}

Throughout the boundary-geometric part of the paper, we additionally assume that the finite tiles obtained by inflation are connected. The source and range of each generator are clopen. A sufficiently deep cylinder partition resolves their domains. A label that is undefined on an entire cylinder gives no edge there. Only defined labels can create boundary points. A boundary edge records a defined generator whose action on that cylinder has not yet become internal. It need not yield a non-AF arrow at every path in the cylinder.

Write $\mathcal{T}_{v,n}^{\circ}$ for the graph obtained from the finite tile $\mathcal{T}_{v,n}$ by deleting its boundary edges. Thus $\mathcal{T}_{v,n}$ is a \emph{graph with boundary} (see \cite[Definition 2.1]{kua26}), while $\mathcal{T}_{v,n}^{\circ}$ is a graph without boundary. With respect to the internal prefix-replacement bisections, $\mathcal{T}_{v,n}^{\circ}$ is the Cayley graph of the elementary subgroupoid $\mathfrak{T}_{\mathsf{B},n}$ at a point with the corresponding level-$n$ prefix. See \cite[Proposition~2.20]{kua26a}. The boundary edges belong to the finite-level tile-inflation data and are not edges of this elementary Cayley graph. The two graphs have the same vertex set.

For $\xi=e_1e_2\cdots$, the maps $p\mapsto pe_{n+1}$ embed the finite tile containing $\xi_n=e_1\cdots e_n$ into the next one. The inductive-limit graph $\mathcal{T}_\xi$ has vertex set $\T\xi$: a prefix $p$ at level $n$ represents the infinite path $pe_{n+1}e_{n+2}\cdots$. We call this an \emph{infinite tile}. Its vertex set need not be infinite if the chosen AF core has finite orbits. Like a finite tile, an infinite tile is a graph with boundary. Its internal edges are represented by generator edges that become internal at some finite stage, while a compatible generator edge that remains a boundary edge at every deeper stage gives a boundary edge of the infinite tile. Thus, the boundary points on an infinite tile are exactly the persistent boundary points introduced in Definition~\ref{def:persistent-boundary-point}. Deleting the boundary edges leaves the subgraph on $\T\xi$. Whenever cardinality matters, we say explicitly that the $\T$-orbit is infinite or finite. Its internal paths represent tail arrows, and connectedness makes every tail arrow realizable by such a path. Thus the internal edges of an infinite tile form a subgraph of the orbital graph. The non-AF generator arrows are precisely the persistent boundary connections that glue these tile pieces, possibly as loops or as connections between distinct tiles. Their lifts to the graph of germs will be distinguished in Subsection~\ref{subsec:boundary-quotient}.

The geometric results with finite boundary use the finite family $\mathcal{S}$ of bisections just specified. The general homology and kernel theorems depend only on the AF-by-discrete inclusion. Rooted-tree and automaton arguments later impose a genuine group action, while the general Bratteli-cylinder statements use partial bisections as in Subsection~\ref{subsec:general-sections}.

We retain the definition of bounded type from \cite[Definition~3.1]{kua26a}, in the tile setting of \cite{kua26}. The individual quantitative hypotheses will be stated separately when a result needs only some of them.
\begin{defn}[Bounded type]
\label{def:bounded-type}
Write $\partial\mathcal{T}_{v,n}$ for the set of boundary points of the finite tile, without counting boundary labels with multiplicity. The tile inflation is \emph{of bounded type} if there are only finitely many persistent boundary points, in the sense of Definition~\ref{def:persistent-boundary-point}, and
\begin{equation*}
\sup_{v,n}|\partial\mathcal{T}_{v,n}|<\infty,\qquad
\sup_n|V_n|<\infty,\qquad
\sup_n|E_n|<\infty.
\end{equation*}
\end{defn}
Bounded Bratteli rank means $\sup_n|V_n|<\infty$. The cylinder-count estimate in Subsection~\ref{subsec:bounded-defects} uses this rank bound and the uniform bound on the boundaries of finite tiles, but not the bound on $|E_n|$. The boundary-to-volume criterion in Section~\ref{sec:kernel} is independent of bounded rank. AF-by-discreteness itself imposes none of these uniform finite-level bounds.

\subsection{Groupoid homology}
\label{subsec:groupoid-homology}
Homology for {\'e}tale groupoids was introduced by Crainic and Moerdijk in \cite{CrainicMoerdijk00}. We use Matui's compact-open formulation for groupoids on totally disconnected spaces \cite[Section~3]{Matui12}. See also \cite[Subsection~5.4.1]{nek22}. For non-Hausdorff arrow spaces, see \cite[Subsection~2.3.1]{Li25}. When an abelian group $A$ is used as coefficients, we regard it as the constant coefficient system $\mathfrak{G}^{(0)}\times A\to\mathfrak{G}^{(0)}$ with \emph{trivial $\mathfrak{G}$-action}, meaning that $g\cdot(\mathsf{s}(g),a)=(\mathsf{r}(g),a)$ for every $g\in\mathfrak{G}$ and $a\in A$.

\begin{defn}[Groupoid homology]
\label{def:groupoid-homology}
Let $A$ be a discrete abelian group with trivial $\G$-action. For $n\geq1$, let $\G^{(n)}$ be the space of composable $n$-tuples $(\gamma_1,\ldots,\gamma_n)$, and put $\G^{(0)}=\Omega(\mathsf{B})$. The face maps $d_i:\G^{(n)}\to\G^{(n-1)}$ are
\begin{equation*}
 d_i(\gamma_1,\ldots,\gamma_n)=
 \begin{cases}
 (\gamma_2,\ldots,\gamma_n),&i=0,\\
 (\gamma_1,\ldots,\gamma_i\gamma_{i+1},\ldots,\gamma_n),&1\leq i\leq n-1,\\
 (\gamma_1,\ldots,\gamma_{n-1}),&i=n,
 \end{cases}
\end{equation*}
with $d_0(\gamma)=\sg(\gamma)$ and $d_1(\gamma)=\rg(\gamma)$ for $n=1$. Together with the degeneracy maps that insert unit arrows, these spaces form the \emph{nerve} of $\G$. For an \'etale map $\pi:Y\to Z$ and a compactly supported function $f:Y\to A$, put $\pi_*f(z)=\sum_{\pi(y)=z}f(y)$. Let $C_c(\G^{(n)};A)$ be the group generated by $A$-valued functions that are continuous on a compact open Hausdorff subset of $\G^{(n)}$ and vanish outside that subset. The boundary operator is $\partial_n=\sum_{i=0}^n(-1)^i(d_i)_*$ for $n\geq1$, with $\partial_0=0$, and
\begin{equation*}
H_n(\G;A)=\ker\partial_n/\operatorname{im}\partial_{n+1}.
\end{equation*}
\end{defn}
The chain group in Definition~\ref{def:groupoid-homology} is appropriate also to a non-Hausdorff arrow space. In the Hausdorff case, it is the usual group of locally constant compactly supported functions. In particular, $\partial_1(1_V)=1_{\sg(V)}-1_{\rg(V)}$ for a compact open bisection $V$. We write $H_n(\G)=H_n(\G;\mathbb{Z})$ and omit integer coefficients throughout the paper.

For a discrete group $H$ and an abelian group $A$ with trivial $H$-action, the notation $H_n(H;A)$ means the homology of its one-object groupoid with coefficients in $A$. The corresponding \emph{bar complex} has symbols $[h_1|\cdots|h_n]$ and the same alternating face rule. In its normalized version, symbols containing an identity entry are set to zero. In particular, $\partial_2[h|k]=[k]-[hk]+[h]$ and $H_1(H;\mathbb{Z})=H^{\ab}$. Its nerve has geometric realization $BH$, the classifying space used in the surface calculations below.

In degree zero, $H_0(\G)$ is generated by the classes $[1_C]$ of compact open subsets $C\subseteq\G^{(0)}$, subject to additivity under clopen decompositions and the relations $[1_{\sg(U)}]=[1_{\rg(U)}]$ for compact open bisections $U$ \cite[Proposition~5.4.1]{nek22}. If $g\in\F(\G)$ and $U_g$ is its full bisection, then $1_{U_g}$ is a $1$-cycle, since $\sg(U_g)=\rg(U_g)=\Omega(\mathsf{B})$.

\begin{defn}[Index map]
\label{def:index-map}
The \emph{index map} of $\G$ is the homomorphism
\begin{equation}
\label{eq:index-map-prelim}
 I:\F(\G)\longrightarrow H_1(\G),\qquad I(g)=[1_{U_g}],
\end{equation}
where $U_g$ is the full bisection representing $g$.
\end{defn}
See \cite[Definition~7.1]{Matui12}.

\subsection{Homology for AF groupoids and the AF-by-discrete groupoids}
\label{subsec:AFbd-homology}
We now apply groupoid homology to the fixed inclusion of Subsection~\ref{subsec:bratteli-core}. The long exact sequence below uses only that $\T$ is AF, while the explicit relative arrow presentation of Nekrashevych additionally uses the AF-by-discrete hypothesis that $\G\setminus\T$ is discrete.

For an AF groupoid, homology in positive degrees vanishes: $H_n(\T;A)=0$ for every $n\geq1$ and every abelian group $A$ with trivial $\T$-action. See \cite[Theorems~4.10 and~4.11]{Matui12} and \cite[Proposition~5.4.9]{nek22}. For the tail groupoid of $\mathsf{B}$, degree zero is the direct limit
\begin{equation}
\label{eq:AF-H0-direct-limit}
H_0(\T;A)\cong\varinjlim\left(A^{V_n},M_n\right).
\end{equation}
With integer coefficients, this is the \emph{dimension group} of $\mathsf{B}$. Its positive cone consists of classes having a nonnegative vector representative at some level, and its distinguished order unit is $[1_{\Omega(\mathsf{B})}]$. Here $A^{V_n}$ means the direct sum of copies of $A$ indexed by $V_n$, and the same incidence map $M_n$ is used with coefficients in $A$.

\begin{defn}[AF parity homomorphism]
\label{def:AF-parity}
Put $\mathsf{P}_{\mathsf{B}}=H_0(\T;\mathbb{Z}/2\mathbb{Z})$. For $f\in\F(\T)$, choose a sufficiently deep Bratteli level at which $f$ is represented by permutations of equal-length cylinders. At each terminal vertex, record the sign of the corresponding finite permutation in $\mathbb{Z}/2\mathbb{Z}$. The resulting vector defines a class in the direct limit $\mathsf{P}_{\mathsf{B}}$. The homomorphism $\varepsilon_{\mathsf{B}}:\F(\T)\to\mathsf{P}_{\mathsf{B}}$ so obtained is the \emph{AF parity homomorphism}.
\end{defn}
Refinement sends the sign vector through the incidence map modulo two, and thus the class in Definition~\ref{def:AF-parity} is independent of the chosen level. Equivalently, a transposition exchanging two cylinders has parity equal to the class of either cylinder. This is the usual signature of the finite permutation representing the element of $\F(\T)$. See \cite[Section~5.4]{nek22}. It will measure the parity of the element of $\F(\T)$ left after a return relation is lifted in Section~\ref{sec:surface-parity}.

Nekrashevych's homology calculation for AF-by-discrete groupoids \cite[Subsection~5.4.4]{nek22} applies the same chain complex to the pair $\T\subseteq\G$. In degree $n$, the relative chain group is the quotient
$C_c(\G^{(n)};A)/C_c(\T^{(n)};A)$, with the boundary induced from the groupoid complex. We denote the resulting homology by $H_n(\G,\T;A)$. The long exact sequence of the pair, together with the vanishing of the homology in positive degrees of $\T$, gives
$H_n(\G;A)\cong H_n(\G,\T;A)$ for $n\geq2$. With integer coefficients, its part in low degrees is
\begin{equation}
\label{eq:relative}
0\longrightarrow H_1(\G)\xrightarrow{\iota}H_1(\G,\T)
\xrightarrow{\delta}H_0(\T)\longrightarrow H_0(\G)\longrightarrow0.
\end{equation}

The discussion in \cite[Section~5.4]{nek22} gives the following presentation. We derive it in low degrees to record where discreteness enters. Since $\T$ and $\G$ have the same unit space, the relative degree-zero chain group is zero. Restriction to $\G\setminus\T$ identifies the relative degree-one chain group with the free abelian group of finitely supported functions on the exceptional arrows: Lemma~\ref{lemma:finite-exceptional-support} gives finite exceptional support on compact bisections and an isolating bisection for each exceptional arrow, while the kernel of restriction consists of chains supported in $\T$. The same restriction argument in degree two gives the free abelian group on exceptional composable pairs. Choose sufficiently small compact open bisections around the two coordinates of such a pair. Their set of composable pairs contains no other exceptional pair, since one exceptional coordinate can be isolated and composability together with the bisection property then fixes the other coordinate. Thus no additional topological relation survives in degrees one and two. With the convention of Subsection~\ref{subsec:groupoid-homology}, the relative boundary of a composable pair $(\alpha,\beta)$ is $[\beta]-[\alpha\beta]+[\alpha]$, where symbols belonging to $\T$ are zero. Consequently $H_1(\G,\T)$ is generated by arrow symbols $[\gamma]$, with $[\gamma]=0$ for $\gamma\in\T$, and the degree-one relations are exactly
$[\alpha\beta]=[\alpha]+[\beta]$ for composable arrows. The all-degree version of the restriction argument is Lemma~\ref{lem:relative-nerve}. These are precisely the relations used in the orbitwise normal form of Theorem~\ref{thm:relative-normal-form}.

The same calculation identifies the connecting homomorphism in~\eqref{eq:relative}. If $\gamma$ is exceptional and $V$ is an isolating bisection for $\gamma$, that is, $\gamma\in V$ and $V\setminus\T=\{\gamma\}$, then, with Nekrashevych's sign convention,
\begin{equation}
\label{eq:relative-defect-prelim}
\delta[\gamma]=[1_{\sg(V)}]-[1_{\rg(V)}]\in H_0(\T).
\end{equation}
The right-hand side is independent of the chosen isolating bisection. Thus $H_1(\G)$ is the kernel of this connecting map inside the relative group, while $H_0(\G)$ is the quotient of the dimension group by its image. In Subsection~\ref{subsec:defect-map} we call $\delta$ the defect map and resolve it into isotropy and transport coordinates. Sections~\ref{sec:boundary}--\ref{sec:index} express these generators, relations, and the connecting map in persistent-boundary data.

\section{Persistent boundary points and boundary configurations}
\label{sec:boundary}

Throughout this section, we keep fixed the finite symmetric family $\mathcal{S}$ and the compatible connected tile presentation from Subsections~\ref{subsec:full-groups-generation} and~\ref{subsec:tile-inflation}. Put $G_0=\langle\mathcal{S}\rangle_{\mathrm{inv}}$. A word $g=F_n\cdots F_1$, with $F_j\in\mathcal{S}$, is understood on the points for which the successive partial maps are defined, equivalently as the compact open product bisection $F_n\cdots F_1$.

\subsection{Boundary points on finite tiles and persistence}\label{subsec:persistence}
Let $\xi=e_1e_2\ldots\in \Omega(\mathsf{B})$ and write $\xi_n=e_1\ldots e_n$. The compatible finite tiles containing the prefixes $\xi_n$ form the increasing sequence whose inductive limit is the infinite tile $\mathcal{T}_\xi$. Recall from Subsection~\ref{subsec:tile-inflation} that a \emph{boundary point on a finite tile} is a vertex incident to a boundary edge. Such a point may become internal after inflation. Thus a finite-level boundary edge records a defined generator label that is not yet realized internally at that level. It does not, by itself, say that the resulting germ lies outside the AF core.

\begin{defn}[Persistent boundary point]
\label{def:persistent-boundary-point}
We say that $\xi\in\Omega(\mathsf{B})$ is a \emph{boundary point on an infinite tile}, or a \emph{persistent boundary point}, if $\xi_n$ is a boundary point on the corresponding finite tile for every $n$. We denote the set of persistent boundary points by $\partial_\infty(\mathsf{B},\mathcal{S})$.
\end{defn}

\begin{prop}
\label{prop:persistence-criterion}
Let $F\in\mathcal{S}$, let $\xi\in\sg(F)$, and suppose that the $F$-edge at $\xi_n$ is a boundary edge at some level. Then $(F,\xi)\in\T$ if and only if this edge becomes internal at some deeper level. Equivalently, $(F,\xi)\notin\T$ if and only if the compatible $F$-edge remains a boundary edge at every deeper level.
\end{prop}

\begin{proof}
If the edge becomes internal at level $m$, then the restriction of the bisection $F$ to the level-$m$ cylinder containing $\xi$ is a prefix replacement, hence $(F,\xi)\in\T$. Conversely, if $(F,\xi)\in\T$, openness of both $F$ and $\T$ gives a compact open tail bisection containing $(F,\xi)$ and contained in a sufficiently small restriction of $F$. The exhaustion property of the compatible tile presentation then makes the corresponding edge internal at a sufficiently deep level. This is the bisection form of the criterion at finite levels in \cite[Remark~2.22]{kua26a}.
\end{proof}

The \emph{exceptional generator set} of the fixed tile presentation is
\begin{equation}
\label{eq:ES}
\mathcal{E}_{\mathcal{S}}=\{(F,x):F\in\mathcal{S},\ x\in\sg(F),\ (F,x)\notin\T\}.
\end{equation}
Its elements are the \emph{exceptional generator arrows}.

Proposition~\ref{prop:persistence-criterion} identifies the support of this set geometrically. The source and range of every arrow in $\mathcal{E}_{\mathcal{S}}$ are persistent boundary points, and conversely every persistent boundary point is the source or range of an arrow in $\mathcal{E}_{\mathcal{S}}$. For the converse, if every generator edge defined at such a point became internal at some level, finiteness of $\mathcal{S}$ would give one common deeper level at which the point is no longer on the boundary.

\begin{prop}
\label{prop:finite-persistent-boundary}
The set $\mathcal{E}_{\mathcal{S}}$ is finite. Consequently, $\partial_\infty(\mathsf{B},\mathcal{S})$ is finite, only finitely many infinite tiles contain persistent boundary points, and only finitely many $\G$-orbits contain germs outside $\T$.
\end{prop}

\begin{proof}
For $F\in\mathcal{S}$, the bisection $F$ is compact. By Lemma~\ref{lemma:finite-exceptional-support}, $F\setminus\T$ is finite. Here closedness inside the bisection is essential. Taking the union over the finite family $\mathcal{S}$ proves that $\mathcal{E}_{\mathcal{S}}$ is finite. By the preceding support observation, $\partial_\infty(\mathsf{B},\mathcal{S})$ is the union of the sources and ranges of these finitely many arrows, and hence is finite as well. Finally, let an orbit contain an arbitrary arrow of $\G\setminus\T$. Represent that arrow by a composable word in $\mathcal{S}$. At least one factor lies outside $\T$, since $\T$ is a subgroupoid. That factor belongs to $\mathcal{E}_{\mathcal{S}}$. Thus every exceptional orbit meets the finite source-range support of $\mathcal{E}_{\mathcal{S}}$, proving the remaining assertions.
\end{proof}

This finiteness is weaker than bounded type. For the fixed finite family $\mathcal{S}$, AF-by-discreteness gives only finitely many persistent boundary points. Without finite generation, there need not be one finite persistent set for the whole groupoid. Bounded type requires, in particular, a uniform bound on the number of boundary points on every finite tile. A compactly generated AF-by-discrete groupoid may thus have complicated finite-level boundaries despite having finitely many persistent boundary points relative to $\mathcal{S}$.

\begin{cor}
\label{cor:boundaryless-tile}
If an infinite tile contains no persistent boundary point, then it coincides with an orbital graph and every point of the infinite tile is regular. It follows that all nontrivial boundary phenomena occur on the finitely many infinite tiles meeting $\partial_\infty(\mathsf{B},\mathcal{S})$.
\end{cor}

\begin{proof}
Every generator germ with source in the infinite tile lies in $\T$. A composable generator word starting there stays in the tile, and all of its factors lie in $\T$. Since $\mathcal{S}$ generates $\G$, the reductions of $\G$ and $\T$ to this tile coincide. It follows that the tile is an entire orbital graph and, by principality of $\T$, every point of it is regular.
\end{proof}

\subsection{Persistent boundary connections}
Every exceptional generator arrow has persistent boundary endpoints. We now classify such arrows according to what remains after the AF motion inside an infinite tile is ignored. Let $\gamma=(F,\xi)\in\mathcal{E}_{\mathcal{S}}$ and put $\eta=F(\xi)=\rg(\gamma)$. Denote by $\mathcal{T}_\xi$ and $\mathcal{T}_\eta$ the infinite tiles containing $\xi$ and $\eta$. Since $\mathcal{S}$ is symmetric, $(F^{-1},\eta)=\gamma^{-1}$ is again an exceptional generator arrow.

In the terms \emph{intra-tile} and \emph{inter-tile} below, ``tile'' always means the infinite tile, whose vertex set is a $\T$-orbit. Thus this classification takes place in the orbital graph. The graph of germs has a finer decomposition into lifted tiles, discussed after Definition~\ref{def:boundary-quotient}.

\begin{defn}[Persistent boundary connections]
\label{def:persistent-boundary-connections}
An arrow $\gamma=(F,\xi)\in\mathcal{E}_{\mathcal{S}}$ is called a \emph{persistent boundary connection}. Writing $\eta=F(\xi)$, it is
\begin{enumerate}[label=(\roman*)]
\item an \emph{isotropy boundary connection} if $\eta=\xi$;
\item an \emph{intra-tile boundary connection} if $\eta\neq\xi$ and $\mathcal{T}_\eta=\mathcal{T}_\xi$;
\item an \emph{inter-tile boundary connection} if $\mathcal{T}_\eta\neq\mathcal{T}_\xi$.
\end{enumerate}
\end{defn}

An intra-tile connection produces isotropy after we remove the unique AF motion between its endpoints.

\begin{prop}
\label{prop:intratile-implies-isotropy}
Let $\gamma:\xi\to\eta$ be an intra-tile boundary connection. If $\alpha:\xi\to\eta$ is the unique arrow of $\T$ with the same source and range, then $\alpha^{-1}\gamma$ is a nontrivial element of $\G_\xi^\xi$.
\end{prop}

\begin{proof}
The arrow $\alpha$ exists since $\xi$ and $\eta$ are tail equivalent and is unique since $\T$ is principal. If $\alpha^{-1}\gamma=(\Id,\xi)$, then $\gamma=\alpha\in\T$, contradicting $\gamma\in\mathcal{E}_{\mathcal{S}}$.
\end{proof}

\begin{cor}
\label{cor:regular-boundary-intertile}
If $\xi$ is regular, every persistent boundary connection with source $\xi$ is inter-tile.
\end{cor}

The converse is false: a composition of inter-tile connections, joined by tail arrows and closed at its starting point, can have a nontrivial isotropy germ. Such a closed composition is a boundary return in the precise sense of Definition~\ref{def:boundary-path-return} below. Regularity will be tested by all of these closed compositions, not by individual boundary edges.

\subsection{The boundary quotient of an exceptional orbit}\label{subsec:boundary-quotient}
Let $\mathcal{O}$ be an exceptional $\G$-orbit in the sense of Definition~\ref{def:exceptional-arrow-orbit}. We now collapse the AF motion inside each $\T$-orbit of $\mathcal{O}$ to a vertex, while retaining every exceptional generator connection together with its actual boundary endpoints. This is the sense in which the construction below is a quotient: only the tiles are collapsed. The labels, multiplicities, and marked endpoints of the persistent connections are not forgotten. For $\xi\in\mathcal{O}$ and $F\in\mathcal{S}$ with $\xi\in\mathsf{s}(F)$, the corresponding generator edge in the orbital graph has source $\xi$, range $F(\xi)$, and is represented by the germ $(F,\xi)\colon \xi\to F(\xi)$.

\begin{defn}[Boundary quotient]\label{def:boundary-quotient}
 The \emph{boundary quotient} $Q_{\mathcal{S}}(\mathcal{O})$ is the directed multigraph obtained from the orbital graph of $\mathcal{O}$ by collapsing each infinite tile to a single vertex and retaining the boundary generator edges. 
\end{defn}

 Thus $V\bigl(Q_{\mathcal{S}}(\mathcal{O})\bigr)=\mathcal{T}(\mathcal{O})$ as in \eqref{orbitquotient}, and every boundary generator germ $(F,\xi)\colon \xi\to F(\xi)$ determines an edge
\[e(F,\xi)\colon \mathcal{T}_{\xi}\longrightarrow \mathcal{T}_{F(\xi)},\]
marked by its source $\xi$, range $F(\xi)$, and germ $(F,\xi)$. Distinct boundary generator edges are retained as distinct edges even when they have the same endpoints. If $\mathcal{T}_{\xi}=\mathcal{T}_{F(\xi)}$, then $e(F,\xi)$ is a loop. This quotient records the tile geometry of the exceptional orbit by a finite connected graph. If two different labels determine the same germ, they remain distinct labeled edges of the quotient even though they represent the same arrow of $\mathcal{E}_{\mathcal{S}}$ and the same generator in the relative homology presentation.

For the moment, index the $\T$-orbits in $\mathcal{O}$ by a set $J$, writing them as $\mathcal{T}_i$, $i\in J$, and put $B_i=\partial_\infty(\mathsf{B},\mathcal{S})\cap\mathcal{T}_i$. 

\begin{prop}\label{prop:finite-boundary-quotient}
For every exceptional orbit $\mathcal{O}$, the graph $Q_{\mathcal{S}}(\mathcal{O})$ is finite and connected. Every vertex is incident to an edge. In particular, there are finitely many $\T$-orbits in $\mathcal{O}$, and every $B_i$ is nonempty.
\end{prop}

\begin{proof}
Only finitely many exceptional generator arrows occur by Proposition~\ref{prop:finite-persistent-boundary}, and the finite family of labels can only produce finitely many labeled edges, and thus the edge set of $Q_{\mathcal{S}}(\mathcal{O})$ is finite. To prove connectedness, take two $\T$-orbits in $\mathcal{O}$ and choose points in them. A composable word in $\mathcal{S}$ joins these points. Collapse each maximal consecutive block of factors lying in $\T$. The remaining non-AF factors are persistent boundary connections and give a path in the quotient between the two vertices. Hence $Q_{\mathcal{S}}(\mathcal{O})$ is connected. Finally, a vertex with no incident edge would be invariant under every generator modulo $\T$ and hence would itself be a whole $\G$-orbit containing no exceptional arrow, contrary to the choice of $\mathcal{O}$. Thus every vertex is incident to an edge, and thus the finite edge set has only finitely many endpoints. This proves finiteness of the vertex set and also shows $B_i\neq\varnothing$ for every $i$.
\end{proof}

Since $J$ is finite by the last proposition, write $J=\{1,\ldots,k\}$. Denote $b_i=|B_i|$. We call the tuple $(k;b_1,\ldots,b_k)$, considered up to permutation of the infinite tiles, the \emph{boundary profile} of $\mathcal{O}$. It records the numbers of tiles and of persistent boundary points on each tile. It forgets the incidence data of the quotient graph and all relations carried by closed paths, and thus is not a classification invariant up to groupoid isomorphism. Nevertheless, the following four coarse incidence patterns are useful:

\begin{enumerate}
\item $k=1$ and $b_1=1$: one infinite tile with one persistent boundary point;
\item $k=1$ and $b_1>1$: several persistent boundary points on one infinite tile;
\item $k>1$ and $b_i=1$ for every $i$: several infinite tiles, one persistent boundary point on each;
\item $k>1$ and $b_i>1$ for some $i$: inter-tile transport and multiple boundary points on at least one infinite tile occur simultaneously.
\end{enumerate}

In the first case, every persistent boundary connection is an isotropy loop. In the second case, every nonloop boundary connection is intra-tile and produces a nontrivial isotropy germ by Proposition~\ref{prop:intratile-implies-isotropy}. The third case may be regular, as for the binary odometer orbit cut, or singular when a nonloop cycle has nontrivial germ, for example one formed by distinct boundary edges with the same source and range tiles. The fourth case is the general configuration with finite boundary.

Here $k$ counts infinite tiles, equivalently the $\T$-orbits contained in $\mathcal{O}$. The graph of germs has a finer decomposition into lifted tiles. Fix $x\in\mathcal{O}$ and write $H_{\mathcal{O}}=\G_x^x$. The range projection from the source fiber $\G_x$ to $\mathcal{O}$ has an $H_{\mathcal{O}}$-torsor over each point: after choosing one arrow $x\to y$, all arrows $x\to y$ are obtained by right multiplication by $H_{\mathcal{O}}$. Consequently, over each tile, there is one $\T$-component for every element of $H_{\mathcal{O}}$. Principality of $\T$ prevents two such components from merging. We call these components the \emph{lifted infinite tiles}. Each lifted tile retains the boundary data of the corresponding infinite tile, but the completed graph of germs is a graph without boundary: its lifted tiles are glued along their boundary points by the exceptional generator germs. After connectors are chosen, the isotropy gluings are organized by the Cayley graph of the group of germs $H_{\mathcal{O}}$. The orbital graph is obtained by gluing the tiles along the persistent boundary connections. Hence, when $H_{\mathcal{O}}$ is finite, the graph of germs has $k\cdot|H_{\mathcal{O}}|$ lifted infinite tiles, with the analogous cardinal statement in the infinite case.

Since $\mathcal{O}$ is exceptional, the number of lifted tiles is at least two. If $k=1$, any exceptional arrow has endpoints in the same $\T$-orbit. Composing it with the inverse of the unique tail arrow having the same endpoints gives nontrivial isotropy, and thus $H_{\mathcal{O}}$ cannot be trivial. For example, in the regular unique-boundary configuration consisting of one boundary connection between two regular tiles, $k=2$ and $H_{\mathcal{O}}=1$, and the graph of germs has exactly two lifted tiles. At the other extreme, if $\mathcal{O}$ contains one persistent boundary point and $H_{\mathcal{O}}\cong C_2$, then $k=1$ but the graph of germs again has two lifted tiles, now lying over the same tile and joined by the nontrivial boundary germ. Thus a one-tile boundary profile describes the orbital-graph decomposition. It never means that the graph of germs itself consists of one lifted tile.

\subsection{Boundary paths, based returns, and regularity}
\label{subsec:boundary-returns}
Fix an exceptional orbit $\mathcal{O}$ and choose a persistent boundary point $x_0\in\partial_\infty(\mathsf{B},\mathcal{S})\cap\mathcal{O}$ as basepoint. A boundary path alternates two kinds of motion: arbitrary motion inside a tile through $\T$, and crossings of persistent boundary connections from $\mathcal{E}_{\mathcal{S}}$. A tail arrow between two points exists exactly when they belong to the same $\T$-orbit and is then unique, although it may be represented by many internal generator paths in the common tile. 

\begin{defn}[Boundary path and boundary return]
\label{def:boundary-path-return}
Let $x,y\in\mathcal{O}$. A \emph{boundary path from $x$ to $y$} is a finite composable list with product
\begin{equation}
\label{eq:boundary-path-product}
\alpha_m\gamma_m\alpha_{m-1}\cdots\alpha_1\gamma_1\alpha_0,
\end{equation}
where the arrows $\gamma_j:\xi_j\to\eta_j$, for $j=1,\cdots,m$, belong to $\mathcal{E}_{\mathcal{S}}$. The arrows $\alpha_0:x\to\xi_1$, $\alpha_j:\eta_j\to\xi_{j+1}$, for $1\leq j<m$, and $\alpha_m:\eta_m\to y$ belong to $\T$. The order of traversal is $\alpha_0,\gamma_1,\alpha_1,\ldots,\gamma_m,\alpha_m$. Identity tail arrows are allowed. We also allow $m=0$, meaning a single tail arrow $x\to y$ and an empty exceptional part. A \emph{boundary return based at $x_0$} is such a path with $x=y=x_0$. Its evaluated product is its \emph{return germ} in $H_{x_0}=\G_{x_0}^{x_0}$.
\end{defn}

The definition concerns a composable list of actual groupoid arrows, not just its drawing in the orbital graph. In particular, a return closes at the same point $x_0$, not merely in the same tile. The list is the boundary path. Its product is one germ. Two different lists may have the same product, and a path that crosses exceptional edges can still have the identity return germ. The case $m=0$ gives only the identity return, since $\T$ is principal.

For a prescribed list of exceptional arrows, the intermediate tail arrows in~\eqref{eq:boundary-path-product} exist precisely when $\eta_j$ and $\xi_{j+1}$ lie in the same tile. The first and last tail arrows impose the analogous endpoint conditions. When they exist, they are unique. Consequently, a based closed walk in $Q_{\mathcal{S}}(\mathcal{O})$ determines a boundary return after its actual boundary endpoints and the basepoint have been retained. For example, if $\gamma:\xi\to\eta$ is an intra-tile boundary connection and $\alpha:\xi\to\eta$ is the tail arrow, then $\alpha^{-1}\gamma$ is the return at $\xi$ obtained by crossing $\gamma$ and coming back internally. Proposition~\ref{prop:intratile-implies-isotropy} says that this germ is nontrivial. In contrast, crossing any connection and immediately crossing its inverse gives the identity return. Therefore, closedness alone is weaker than singularity. The following proposition shows that evaluating boundary returns detects all isotropy.
\keepwithstatement
\begin{prop}
\label{prop:return-criterion}
Every element of $H_{x_0}$ is the evaluated germ of a boundary return based at $x_0$. In particular, these return germs generate $H_{x_0}$, and $\mathcal{O}$ is regular if and only if every boundary return has germ $(\Id,x_0)$.
\end{prop}
\begin{proof}
Let $h\in H_{x_0}$ and represent it by a composable bisection word $g=F_\ell\cdots F_1$ with $F_j\in\mathcal{S}$. Put $z_0=x_0$, $z_j=F_j(z_{j-1})$, and $\gamma_j=(F_j,z_{j-1})$. Since $h$ is isotropy, $z_\ell=x_0$. Each $\gamma_j$ either lies in $\T$ or belongs to $\mathcal{E}_{\mathcal{S}}$. Decompose the word into maximal consecutive blocks of AF arrows separated by the exceptional steps. Every AF block has a product in $\T$, and since $\T$ is principal that product is exactly the unique tail arrow between its endpoints. Replacing each maximal AF block by this unique arrow produces a boundary return of the form~\eqref{eq:boundary-path-product}, without changing the evaluated product. Hence its return germ is $h$. If no exceptional step occurs, the whole word lies in $\T$ and principality forces $h=(\Id,x_0)$. Finally, the isotropy groups at points of one orbit are conjugate, and thus $H_{x_0}$ is trivial exactly when the entire orbit is regular.
\end{proof}

The choice of a boundary point as basepoint is convenient, but not essential. For $y\in\mathcal{O}$ and any arrow $c:x_0\to y$, conjugation sends the above return germs to the isotropy germs at $y$. A word representing $c$ gives the necessary boundary path to and from the chosen basepoint. A regularity test at one boundary point tests the whole exceptional orbit.

More generally, choose arrows $p_z:x_0\to z$ for the endpoints under consideration.

\begin{defn}[Normalized return]
\label{def:normalized-return}
Let $\gamma:x\to y$ be an arrow in $\mathcal{O}$, and choose connectors $p_x:x_0\to x$ and $p_y:x_0\to y$. The \emph{normalized return} of $\gamma$ with respect to these connectors is $p_y^{-1}\gamma p_x\in H_{x_0}$.
\end{defn}

The normalized return depends on the chosen connectors. In particular, if every connector is replaced simultaneously by $p_z h$ for one fixed $h\in H_{x_0}$, then every normalized return is conjugated by $h^{-1}$, and thus its class in $H_{x_0}^{\ab}$ is unchanged. More general choices of connectors are fixed systematically by the construction of the spanning tree in Section~\ref{sec:germs}. Section~\ref{sec:index} then uses the resulting abelianized normalized returns in the index formula. A boundary return in the present sense is distinct from a \emph{traverse} of a finite tile or a \emph{return word} in the growth argument in \cite{kua26,kua25}. These trajectory-related notions are defined and compared with normalized returns in Subsection~\ref{subsec:word-traverses}.

\subsection{Singularities visible from the generators and their separation types}\label{subsec:singularity-types}\label{subsec:separation-types}
For $\xi\in \Omega(\mathsf{B})$, put $H_\xi=\G_\xi^\xi$. Recall that the point $\xi$ is \emph{regular} if $H_\xi$ is trivial and \emph{singular} otherwise. These properties are intrinsic and constant on $\G$-orbits.

\begin{defn}[Generator-visible singularities]
\label{def:generator-visible-singularities}
Fix the family $\mathcal{S}$ of generating bisections.
\begin{enumerate}[label=(\roman*)]
\item A point $\xi$ is \emph{$\mathcal{S}$-singular} if some $F\in\mathcal{S}$ is defined at $\xi$, satisfies $F(\xi)=\xi$, and has $(F,\xi)\neq(\Id,\xi)$.
\item A \emph{nonloop cycle} is a composable generator path $z_0,z_1,\ldots,z_m=z_0$, with $m\geq2$, such that $z_0,\ldots,z_{m-1}$ are pairwise distinct and the $j$th step is $(F_j,z_{j-1})$ for some $F_j\in\mathcal{S}$. Its \emph{cycle germ} is the product of these generator arrows.
\item A singular point $\xi$ is \emph{germ-defining relative to $\mathcal{S}$} if it is the unique $\mathcal{S}$-singular point in its orbit and every nonloop cycle in that orbit has identity cycle germ.
\item An orbit $\mathcal{O}$ has a \emph{nontrivial nonloop return germ} if one of its nonloop cycles has nonidentity cycle germ.
\end{enumerate}
For an exceptional orbit $\mathcal{O}$, write $d_{\mathcal{S}}(\mathcal{O})$ for the number of its $\mathcal{S}$-singular points.
\end{defn}
A nontrivial isotropy germ occurring in item~(i) is necessarily non-AF since $\T$ is principal.

The notion of $\mathcal{S}$-singularity depends on the chosen labeled generating family and is not an intrinsic orbit invariant. Every $\mathcal{S}$-singular point lies in the finite source-range support of $\mathcal{E}_{\mathcal{S}}$, and thus $d_{\mathcal{S}}(\mathcal{O})$ is finite. If $\mathcal{S}$ consists of the full bisections associated with a group generating set $S$, germ-defining relative to $\mathcal{S}$ is exactly the condition of \cite[Definition~3.4]{kua26a}. Thus germ-defining is a property of the labeled bisection presentation together with the orbit. It is not determined merely by the abstract group $H_\xi$. Generator loops and nonloop cycles give the following five regimes.
\keepwithstatement
\begin{prop}\label{prop:classification-singular-boundary}
Exactly one of the following cases occurs for an exceptional orbit $\mathcal{O}$.
\begin{enumerate}
\item\label{caseo1} $d_{\mathcal{S}}(\mathcal{O})=0$ and every nonloop cycle has trivial germ. Then $\mathcal{O}$ is regular.
\item\label{caseo2} $d_{\mathcal{S}}(\mathcal{O})=0$ and there is a nontrivial nonloop return germ. Then $\mathcal{O}$ is singular, but the isotropy is not visible as a nontrivial generator loop.
\item\label{caseo3} $d_{\mathcal{S}}(\mathcal{O})=1$ and every nonloop cycle has trivial germ. Then the unique $\mathcal{S}$-singular point is germ-defining relative to $\mathcal{S}$.
\item\label{caseo4} $d_{\mathcal{S}}(\mathcal{O})\geq2$ and every nonloop cycle has trivial germ. Then the orbit has several generator-visible singular points and no point is germ-defining relative to $\mathcal{S}$.
\item\label{caseo5} $d_{\mathcal{S}}(\mathcal{O})\geq1$ and there is a nontrivial nonloop return germ. Then generator-visible singularities coexist with nontrivial return germs, and no point is germ-defining relative to $\mathcal{S}$.
\end{enumerate}
\end{prop}

\begin{proof}
The two pieces of data are the finite integer $d_{\mathcal{S}}(\mathcal{O})$ and the existence or nonexistence of a nontrivial nonloop return germ. Separating according to the three possibilities $d_{\mathcal{S}}(\mathcal{O})=0$, $d_{\mathcal{S}}(\mathcal{O})=1$, and $d_{\mathcal{S}}(\mathcal{O})\geq2$, and then according to the second datum, gives the five listed alternatives. When $d_{\mathcal{S}}(\mathcal{O})=1$ or $d_{\mathcal{S}}(\mathcal{O})\geq2$, the alternatives with a nontrivial nonloop return germ are combined in case~(\ref{caseo5}). Hence, the cases are disjoint and exhaustive.

It remains to justify what these combinatorial alternatives say about isotropy. Consider a closed path
\[
z_0,z_1,\ldots,z_m=z_0
\]
in the orbital graph labeled by generators $F_1,\ldots,F_m\in\mathcal{S}$, so that the $j$th edge is represented by the germ $(F_j,z_{j-1})$. If an intermediate vertex is repeated, cut the path at the first repetition and continue inductively. The corresponding isotropy germ is then a product of conjugates of simple closed paths labeled by the generators. Each such simple closed path is either a one-edge isotropy loop or a nonloop cycle in the sense of Definition~\ref{def:generator-visible-singularities}. Thus all isotropy germs are generated by these two types of closed paths.

In case~(\ref{caseo1}), there are no nontrivial one-edge isotropy loops since $d_{\mathcal{S}}(\mathcal{O})=0$, and all nonloop cycle germs are trivial by hypothesis. Every isotropy germ is therefore trivial, and thus the orbit is regular. In case~(\ref{caseo2}), the specified nonloop cycle gives nontrivial isotropy, while no generator loop detects it. In case~(\ref{caseo3}), the unique $\mathcal{S}$-singular point is singular by its nontrivial generator loop, and the absence of nontrivial nonloop return germs is exactly the remaining condition in the definition of germ-defining relative to $\mathcal{S}$. In case~(\ref{caseo4}), there is more than one $\mathcal{S}$-singular point, and thus no point can be germ-defining even though every nonloop cycle has trivial germ. Finally, in case~(\ref{caseo5}), the existence of a nontrivial nonloop return germ violates the defining condition for a germ-defining point, independently of how many generator-visible singular points occur.
\end{proof}

\begin{cor}
\label{cor:single-boundary-germ-defining}
If an exceptional orbit contains exactly one persistent boundary point $\xi$, then $\xi$ is singular and germ-defining relative to $\mathcal{S}$.
\end{cor}

\begin{proof}
Since the orbit is exceptional, it contains an exceptional generator arrow. Its source and range are persistent boundary points, and thus both are equal to $\xi$. Hence, $\xi$ carries a nontrivial generator loop and is $\mathcal{S}$-singular. Any other $\mathcal{S}$-singular point would also be a persistent boundary point, and thus $\xi$ is the unique one. Moreover, every non-AF generator arrow in the orbit is a loop at $\xi$. A nonloop cycle, whose successive vertices are distinct, cannot use such a loop. Therefore, all of its steps lie in $\T$, and its cycle germ is the identity since $\T$ is principal. Definition~\ref{def:generator-visible-singularities} now shows that $\xi$ is germ-defining relative to $\mathcal{S}$.
\end{proof}

\begin{cor}
\label{cor:intratile-not-germ-defining}
If an exceptional orbit contains an intra-tile boundary connection between two distinct persistent boundary points, then it has a nontrivial nonloop return germ. In particular, no point in that orbit is germ-defining relative to $\mathcal{S}$.
\end{cor}

\begin{proof}
Let $\gamma:\xi\to\eta$ be the intra-tile boundary connection. Since the common infinite tile is connected, choose a simple internal $\mathcal{S}$-path from $\eta$ back to $\xi$. Its product is the unique AF arrow $\alpha^{-1}:\eta\to\xi$, where $\alpha:\xi\to\eta$ is the tail arrow. Concatenating this simple internal path with $\gamma$ gives a nonloop cycle, and its cycle germ is $\alpha^{-1}\gamma$, which is nontrivial by Proposition~\ref{prop:intratile-implies-isotropy}.
\end{proof}

There is a second, topological classification of a singular point that is independent of the boundary profile and of the generator-visible classification above. We call a singular point $\xi$ \emph{Hausdorff} if no nontrivial isotropy germ at $\xi$ is accumulated by identity germs, and \emph{non-Hausdorff} otherwise. This terminology concerns separation of germs near the singular point. It does not assert that the whole arrow space of $\G$ is Hausdorff. Following \cite{nekrash18,nek22}, a singular point is \emph{purely non-Hausdorff} if, for every nontrivial $h\in\G_\xi^\xi$ and every local bisection representative $U\ni h$, the interior of the fixed-point set of the induced partial homeomorphism $\theta_U$ accumulates on $\xi$. Conjugation by an arrow preserves these properties. Hence, the Hausdorff, non-Hausdorff but not purely non-Hausdorff, and purely non-Hausdorff types are orbit invariants.

This separation type is not determined by the abstract group $H_\xi$. The same finite group can occur as a Hausdorff group of germs in one action and as a purely non-Hausdorff group of germs in another. For the index calculation, the abstract abelianization enters the relative group. For constructions of localized elements in a prescribed acting subgroup, as in \cite{kua26a,kua26b}, the separation type carries additional information. None of these separation conditions alone guarantees the required localized homeomorphisms. The cancellation in $\F(\G)$ in Section~\ref{sec:constructive-kernel} will instead use the marked transpositions of Lemma~\ref{lemma:marked-transposition}.

\subsection{Nested isolation and persistent boundary data}\label{subsec:persistent-data}
Nested clopen neighborhoods isolate a persistent germ at arbitrarily deep levels and separate the boundary points involved in a return.

\keepwithstatement
\begin{prop}
\label{prop:nested-isolation}
Let $\gamma\in\G\setminus\T$ have source $\xi$, and let $V$ be a compact open bisection containing $\gamma$. Then there exists $n_0$ such that for every $n\geq n_0$ there is a clopen neighborhood $U_n\subseteq[\xi_n]\cap\sg(V)$ with $U_{n+1}\subseteq U_n$, $\xi\in U_n$, and
$(V|_{U_n})\cap(\G\setminus\T)=\{\gamma\}.$ If finitely many distinct persistent boundary points are involved in one return calculation, then the neighborhoods may be chosen pairwise disjoint after increasing the level.
\end{prop}

\begin{proof}
By Lemma~\ref{lemma:finite-exceptional-support}, the compact bisection $V$ has only finitely many exceptional arrows. Choose a clopen source neighborhood of $\xi$ avoiding the sources of all other exceptional arrows of $V$, and intersect it with successively deeper cylinders $[\xi_n]$. After a sufficiently deep level, every restriction $V|_{U_n}$ is therefore an isolating bisection for $\gamma$ in the sense of Definition~\ref{def:isolating-bisection}. Distinct infinite paths lie in disjoint sufficiently deep cylinders, and thus the same construction can be made simultaneously with pairwise disjoint neighborhoods for any finite collection of persistent boundary points.
\end{proof}

Hence, a nontrivial boundary return may be tested repeatedly on nested finite levels. The distinct persistent boundary points involved in a return lie in disjoint sufficiently deep cylinders, while their common prefixes record the levels at which they are still indistinguishable. This topological isolation is used later in the trajectory interpretation of the index.

\smallskip
The boundary profile and the quotient graph do not by themselves determine the groupoid. We will use the following package.

\begin{defn}[Persistent boundary data]
\label{def:persistent-boundary-data}
For an exceptional orbit $\mathcal{O}$, its \emph{persistent boundary data relative to the fixed presentation $(\mathsf{B},\mathcal{S})$} consist of:
\begin{enumerate}
\item the infinite tiles $\mathcal{T}_1,\ldots,\mathcal{T}_k$ and the marked sets $B_i$ of persistent boundary points;
\item the directed labeled multigraph $Q_{\mathcal{S}}(\mathcal{O})$, including the source and range boundary point of every edge;
\item the groups of germs at the boundary points and their Hausdorff, non-Hausdorff, or purely non-Hausdorff type;
\item the relations represented by boundary returns;
\item the compatible restrictions at finite levels of the persistent boundary connections.
\end{enumerate}
\end{defn}

The first two items are geometric. The third and fourth record the isotropy and its relations. The last item records how the same marked boundary configuration persists through deeper finite tiles and is the input for arguments with sections along paths. The package depends on the chosen compatible presentation. The groups of germs are intrinsic to $\G$ up to the basepoint identifications below, whereas the relative homology classes belong to the fixed pair $(\G,\T)$. We will distinguish these objects from their presentation-dependent coordinates.

Under the finite singular germ condition of \cite[Definition~3.5]{kua26a}, every exceptional orbit has one persistent boundary point, that point is germ-defining, and its group of germs is finite. That condition also requires a decomposition into finitary and singular generators, with each singular generator fixing its unique exceptional source and having AF germs everywhere else. Beyond that case, persistent boundary data may exhibit regular transport between distinct infinite tiles, several persistent boundary points on one tile, several tiles in one exceptional orbit, nontrivial return germs, or infinite groups of germs. Section~\ref{sec:germs} takes the boundary data constructed here and turns these return germs into a finite generating system for the isotropy groups, while separating finite from infinite germ behavior.

\section{Boundary returns and finite or infinite groups of germs}
\label{sec:germs}

We retain the finite family $\mathcal{S}$ and the compatible tile presentation from Section~\ref{sec:boundary}. Boundary returns give finite generating sets for the groups of germs. We then distinguish the resulting word length from two lengths attached to a section: the prefix length at which it is taken and its additional finitary depth.

In the rooted-tree setting, a theorem of Juschenko--Nekrashevych--de la Salle gives local finiteness from bounded finitary depth. The converse fails. With a common invariant clopen basis, finiteness of the group of germs is instead equivalent to finiteness of the image of a sufficiently deep restriction homomorphism.

\subsection{Finite generation from the boundary quotient}
Fix an exceptional orbit $\mathcal{O}$ and choose a persistent boundary point $x_0\in\partial_\infty(\mathsf{B},\mathcal{S})\cap\mathcal{O}$ in the sense of Definition~\ref{def:persistent-boundary-point}. Put $H_{\mathcal{O}}=\G_{x_0}^{x_0}$.
If $x\in\mathcal{O}$ and $p_x:x_0\to x$ is any arrow, then
\begin{equation*}
\G_x^x\longrightarrow H_{\mathcal{O}},\qquad h\longmapsto p_x^{-1}hp_x,
\end{equation*}
is an isomorphism. Replacing $p_x$ changes this identification by an inner automorphism of $H_{\mathcal{O}}$. The group itself therefore depends on the chosen basepoint only up to isomorphism. The induced identifications on its abelianization and on group homology with trivial coefficients are independent of the connecting arrow, since inner automorphisms act trivially on these groups. Here trivial coefficients means that $H_{\mathcal{O}}$ acts trivially on the coefficient group. We use the single group $H_{\mathcal{O}}$ to express return germs based at different persistent boundary points in common coordinates.

\begin{defn}[Spanning tree and fundamental boundary cycle]
\label{def:spanning-tree-fundamental-cycle}
A \emph{spanning tree} of $Q_{\mathcal{S}}(\mathcal{O})$ is a spanning tree of the underlying undirected multigraph obtained by deleting loops, identifying each pair of edges inverse to each other as one edge, and forgetting orientations.

Let $e\colon \mathcal{T}_{\xi}\to \mathcal{T}_{F(\xi)}$ be a boundary edge. Suppose first that $e$ is not a loop and that its corresponding undirected edge is not contained in the chosen spanning tree. There is then a unique path in the spanning tree from the range tile $\mathcal{T}_{F(\xi)}$ back to the source tile $\mathcal{T}_{\xi}$. The \emph{fundamental boundary cycle} associated with $e$ is the closed path in the orbital graph obtained by traversing $e$, then following in the required orientations the boundary edges corresponding to this path in the spanning tree, and inserting the unique $\T$-arrow inside each tile whenever the actual range of one boundary edge differs from the actual source of the next. The final $\T$-arrow returns to $\xi$. If $e$ is a loop, its fundamental boundary cycle consists of $e$ followed by the unique $\T$-arrow from its actual range $F(\xi)$ back to its actual source $\xi$.
\end{defn}

Choose one representative point in each tile. Starting with the tile containing $x_0$, use the chosen spanning tree to choose connectors from $x_0$ to these representatives. At each step, use the marked boundary endpoints of the relevant boundary edge and the unique $\T$-arrows inside the tiles. Extend the connectors to all points of each tile by the unique $\T$-arrows. With these choices, the normalized return of every boundary edge whose corresponding undirected edge belongs to the chosen spanning tree is the identity. For every loop and every boundary edge whose corresponding undirected edge is not contained in the spanning tree, the normalized return is the germ represented by its fundamental boundary cycle.
\keepwithstatement
\begin{prop}
\label{prop:germ-fg}
For every exceptional orbit $\mathcal{O}$, the group $H_{\mathcal{O}}$ is finitely generated. Choose a spanning tree of $Q_{\mathcal{S}}(\mathcal{O})$ and the connectors constructed above. For every loop $e$, and for every boundary edge $e\colon \mathcal{T}_{\xi}\to \mathcal{T}_{F(\xi)}$ whose corresponding undirected edge is not contained in the chosen spanning tree, let $h_e=p_{F(\xi)}^{-1}(F,\xi)p_{\xi}\in H_{\mathcal{O}}$ be the normalized return of the corresponding exceptional generator arrow. These elements generate $H_{\mathcal{O}}$. Since two boundary edges represented by $(F,\xi)$ and $(F^{-1},F(\xi))$ have normalized returns that are inverse to each other, it is enough to choose one edge from each pair of edges inverse to each other.
\end{prop}

\begin{proof}
For every exceptional generator arrow $\gamma\colon x\to y$ in $\mathcal{O}$, put $h_{\gamma}=p_y^{-1}\gamma p_x\in H_{\mathcal{O}}$. By the choice of the connectors, $h_{\gamma}=(\Id,x_0)$ whenever the corresponding undirected edge belongs to the chosen spanning tree.

Let $h\in H_{\mathcal{O}}$. By Proposition~\ref{prop:return-criterion}, the germ $h$ is represented by a boundary return based at $x_0$. Write this return as $\alpha_m\gamma_m\alpha_{m-1}\cdots\alpha_1\gamma_1\alpha_0$, where each $\gamma_j\colon \xi_j\to\eta_j$ is an exceptional generator arrow and each $\alpha_j$ belongs to $\T$. The connectors were extended inside each tile by the unique arrows of $\T$. Hence, whenever an arrow of $\T$ joins two consecutive boundary endpoints, the connector at its range is obtained by composing that arrow with the connector at its source. Substituting $\gamma_j=p_{\eta_j}h_{\gamma_j}p_{\xi_j}^{-1}$ into the boundary return makes all connector terms cancel. Therefore, $h=h_{\gamma_m}\cdots h_{\gamma_1}$. Thus $H_{\mathcal{O}}$ is generated by the normalized returns of the exceptional generator arrows.

If the corresponding undirected edge belongs to the chosen spanning tree, its normalized return is the identity. Every remaining boundary edge is either a loop or has a corresponding undirected edge not contained in the spanning tree. There are only finitely many such edges since $Q_{\mathcal{S}}(\mathcal{O})$ is finite. If two boundary edges are represented by $(F,\xi)$ and $(F^{-1},F(\xi))$, their normalized returns are inverse to each other. Hence one edge from each pair of edges inverse to each other suffices.

For every loop and every boundary edge whose corresponding undirected edge is not contained in the spanning tree, the associated fundamental boundary cycle gives a closed path in the orbital graph. The isotropy germ represented by that closed path, conjugated to the basepoint $x_0$ by the chosen connector, is exactly the corresponding element $h_e$.
\end{proof}

\keepwithstatement
\begin{cor}
\label{cor:compact-generation-germs}
Suppose that $\T\subseteq\G$ is AF-by-discrete and that $\G$ is compactly generated. Then there are only finitely many exceptional $\G$-orbits, every exceptional $\G$-orbit contains only finitely many $\T$-orbits, and every group of germs is finitely generated. The groups of germs on nonexceptional orbits are trivial.
\end{cor}
\begin{proof}
Choose a finite symmetric family $\mathcal{S}$ of compact open bisections generating $\G$. By~\eqref{eq:ES} and Lemma~\ref{lemma:finite-exceptional-support}, the exceptional generator set $\mathcal{E}_{\mathcal{S}}=\bigcup_{F\in\mathcal{S}}(F\setminus\T)$ is finite. Every exceptional $\G$-orbit contains the source or range of an arrow in $\mathcal{E}_{\mathcal{S}}$: represent an exceptional arrow by a word in $\mathcal{S}$ and take the first non-AF step. Hence there are only finitely many exceptional $\G$-orbits.

Fix one such orbit. Every $\T$-orbit inside it also meets the source or range of $\mathcal{E}_{\mathcal{S}}$. Indeed, join a point of the chosen $\T$-orbit to an exceptional source by a word in $\mathcal{S}$. If the word lies entirely in $\T$, the exceptional source is in the original $\T$-orbit. Otherwise, the source of the first non-AF step is still in that $\T$-orbit and belongs to the source or range set of $\mathcal{E}_{\mathcal{S}}$. Since this set is finite, only finitely many $\T$-orbits occur.

Choose one representative in each of these $\T$-orbits and connectors from a basepoint to the representatives, and extend the connectors using the unique arrows of $\T$. Normalizing a word representing an isotropy arrow kills every AF generator step and expresses the isotropy arrow as a product of the finitely many normalized returns associated with $\mathcal{E}_{\mathcal{S}}$. Thus every exceptional group of germs is finitely generated. On a nonexceptional orbit every arrow belongs to the principal groupoid $\T$, and thus the isotropy is trivial.
\end{proof}

Let $R_{\mathcal{O}}$ be the finite symmetric generating set of $H_{\mathcal{O}}$ obtained from the normalized returns in Proposition~\ref{prop:germ-fg} by adjoining their inverses. The \emph{boundary word length} $\ell_\partial(h)$ is the least number of letters of $R_{\mathcal{O}}$ whose product is $h$. Prefix length and finitary depth are separate quantities, defined for section representatives below.

\subsection{Sections and finitary depth on Bratteli cylinders}
\label{subsec:general-sections}
For $v\in V_{n+1}$, put
\[
\Omega_v^{(n)}(\mathsf{B})=\{e_{n+1}e_{n+2}\cdots:\mathbf{s}(e_{n+1})=v\}.
\]
Let $p\in\Omega_n(\mathsf{B})$ be such that $\mathbf{r}(p)=v$, the \emph{prefix chart} is
\[
\kappa_p:\Omega_v^{(n)}(\mathsf{B})\longrightarrow[p],\qquad z\longmapsto pz.
\]

\begin{defn}[Section in prefix charts]
\label{def:bratteli-section}
Let $f$ be the partial homeomorphism induced by a bisection of $\G$, and let $p,q\in\Omega_n(\mathsf{B})$. The \emph{section of $f$ in the prefix charts $(p,q)$} is the partial map
\begin{equation}
\label{eq:bratteli-section}
f_{p,q}=\kappa_q^{-1}f\kappa_p,\qquad
\Dom(f_{p,q})=\{z:pz\in\Dom(f),\ f(pz)\in[q]\}.
\end{equation}
\end{defn}
We use only nonempty sections. For a compact open bisection, the source and range of a section are clopen in the corresponding path spaces.

A general section is therefore a partial map from $\Omega_{\mathbf{r}(p)}^{(n)}(\mathsf{B})$ to $\Omega_{\mathbf{r}(q)}^{(n)}(\mathsf{B})$. There may be several target prefixes $q$ meeting $f([p])$. If $u,v$ are paths of length $k$ such that $pu$ and $qv$ are defined, further restriction gives $(f_{p,q})_{u,v}=f_{pu,qv}$. Likewise, $(fh)_{p,r}=f_{q,r}h_{p,q}$ on the domain on which the intermediate image under $h$ has prefix $q$. These are identities of partial maps with the indicated domains. If $f([p])=[q]$ and both cylinders lie in the relevant domains, then the section is a homeomorphism from $\Omega_{\mathbf{r}(p)}^{(n)}(\mathsf{B})$ to $\Omega_{\mathbf{r}(q)}^{(n)}(\mathsf{B})$ and may be written $f|_p$ once $q$ is understood. For an arbitrary partial section, we retain both chart indices.

We extend the convention for finite prefixes of \cite[Subsection~4.1]{JNS16} to the prefix charts above.

\begin{defn}[Finitary depth]
\label{def:finitary-depth}
A compact open section $\sigma=f_{p,q}$ has \emph{finitary depth at most $d$} if its domain and range admit finite partitions into cylinder sets determined by paths of length $d$ such that on each piece
\begin{equation}
\label{eq:partial-finitary-depth}
\sigma(uz)=vz,\qquad |u|=|v|=d,\quad \mathbf{r}(pu)=\mathbf{r}(qv).
\end{equation}
Here $z$ is an arbitrary infinite path beginning at the common terminal vertex. The \emph{finitary depth} $d_{\mathrm{fin}}(\sigma)$ is the least such $d$, or $\infty$ if no such finite table exists.
\end{defn}

Depth is counted from the chosen charts $(p,q)$. If $|p|=|q|=n$, a table of prefix replacements giving depth at most $d$ uses total prefix length $n+d$ for $f$. For a partial section, the depth includes the refinement needed to resolve its clopen source and range. The prefix length $n$ records the location of the section. The additional depth $d$ records the complexity of the local AF action after that location has been removed.

\begin{lemma}
\label{lem:finitary-section-AF}
A compact open section has finite finitary depth if and only if the corresponding restricted bisection of $\G$ is contained in $\T$.
\end{lemma}
\begin{proof}
A table of the form~\eqref{eq:partial-finitary-depth} is a finite union of prefix replacements with unchanged infinite continuation, hence consists of tail arrows and lies in $\T$. Conversely, let the restricted bisection lie in $\T$. Compactness gives a finite cover by elementary prefix-replacement bisections. Refining their source and range cylinder partitions to one common deeper level gives a table of the form~\eqref{eq:partial-finitary-depth}. Thus the section has finite depth.
\end{proof}

The common level of these prefix replacements in the converse can be deeper than the least elementary stage containing the arrows, since the required prefix replacement may itself depend on the continuation. In particular, an AF germ at one point becomes finitary only after restriction to a suitable compact open neighborhood. It need not make an initially chosen larger section finitary.

Fix $g\in\F(\G)$ and a persistent boundary point $\xi$. For every $n$, take $p=\xi_n$ and $q=(g(\xi))_n$. We may then pass to smaller source and range cylinders and take sections $g_{pu,qv}$, including sections whose source cylinders do not contain $\xi$. A \emph{uniform bound on the finitary sections along $\xi$}, for this chart family and this fixed element $g$, is a constant $D$ such that $d_{\mathrm{fin}}(g_{pu,qv})\leq D$ for every one of these sections that is finitary. The bound is independent of the initial prefix length and of the further choices $u,v$. If $(g,\xi)\notin\T$, no section containing that germ is finitary, although finitary sections may occur on cylinders arbitrarily close to $\xi$ that do not contain $\xi$.

For fixed source and target vertices at one level, and fixed depth, only finitely many partial tables~\eqref{eq:partial-finitary-depth} occur. Maps with different source and target charts, however, form only a finite collection of partial bijections. They are not one group unless the domains and ranges are identified and invariant. Uniform finitary depth alone therefore does not make taking sections a group homomorphism in an arbitrary Bratteli or inverse-semigroup model. The exact stabilization criterion later in this section needs a common invariant neighborhood basis.

\subsection{Sections on rooted trees and bounded finitary depth}
\label{subsec:rooted-depth}
Suppose that the action is by automorphisms of a rooted tree with level alphabets $E_1,E_2,\ldots$ of uniformly bounded cardinality. Its boundary is the path space of the Bratteli diagram having one vertex at each level and edge set $E_n$ at level $n$. An automorphism $g$ permutes prefixes at each level. Its section at a prefix $p$ is the unique automorphism of the rooted tree determined by the levels below $p$ satisfying $g(pz)=g(p)(g|_p)(z)$. Thus $g|_p=g_{p,g(p)}$ in Definition~\ref{def:bratteli-section}, and
\begin{equation}
\label{eq:rooted-section-rules}
g|_{pu}=(g|_p)|_u,\qquad (gh)|_p=g|_{h(p)}\,h|_p.
\end{equation}
For a fixed prefix, the section map is not a homomorphism on the whole acting group. It becomes a homomorphism on the subgroup fixing that prefix.

An automorphism $a$ of such a rooted tree is \emph{finitary of depth at most $d$} when $a|_u=1$ for every prefix $u$ of length $d$. This is exactly Definition~\ref{def:finitary-depth} in the rooted-tree charts: after the first $d$ letters, the continuation is left unchanged. The identity has depth zero. For an element of the full group, choose a finite clopen table of elements of the acting group. On sufficiently deep cylinders contained in its pieces, the same section notation is defined by the corresponding local action. Thus we may first take $g|_{\xi_n}$ for a long prefix and then a further section $(g|_{\xi_n})|_u$ that is finitary. Its finitary depth is counted \emph{after} the prefix $\xi_nu$ has been removed. It is neither $n$ nor $|u|$.

For a fixed rooted automorphism $g$, the bounded-depth condition used below is
\begin{equation}
\label{eq:uniform-finitary-sections}
\sup\{d_{\mathrm{fin}}(g|_p):p\text{ a finite prefix and }g|_p\text{ finitary}\}<\infty.
\end{equation}
It quantifies over finitary sections at all prefix lengths. By~\eqref{eq:rooted-section-rules}, the same bound applies to every finitary section of every $g|_{\xi_n}$. Different generators may have different bounds. A finite generating family gives their maximum. No bound over the entire full group is intended.

Recall also the \emph{bounded-activity} condition. Put $\alpha_n(g)=|\{p:|p|=n,\ g|_p\neq1\}|$. The automorphism $g$ is \emph{bounded} if $\sup_n\alpha_n(g)<\infty$. This controls the number of active sections at a level, whereas~\eqref{eq:uniform-finitary-sections} controls how much farther a finitary section can act. For a constant alphabet, $g$ is \emph{finite-state} if only finitely many distinct sections $g|_p$ occur. These are Definitions~4.5, 4.6, and 4.10 of \cite{JNS16}. The next proposition is the local-finiteness part of \cite[Proposition~4.9]{JNS16}. That proposition also derives amenability of the acting group, which is not needed here.

\begin{prop}
\label{prop:jns-depth}
Let $G$ act on a rooted tree with $\sup_i|E_i|<\infty$ and be generated by bounded automorphisms. Assume that, for every generator, the depths of all of its finitary sections are uniformly bounded. Then every group of germs on the boundary is locally finite.
\end{prop}

Together with Corollary~\ref{cor:compact-generation-germs}, this gives the following.

\keepwithstatement
\begin{cor}
\label{cor:bounded-depth-finite-germs}
Let $\G$ be the groupoid of germs of an action, generated by bounded automorphisms, on a rooted tree with uniformly bounded level alphabets. Suppose that $\G$ is AF-by-discrete and compactly generated. If, for every generator, the depths of its finitary sections are uniformly bounded, then all groups of germs are finite.
\end{cor}
\begin{proof}
Proposition~\ref{prop:jns-depth} makes every group of germs locally finite, while Corollary~\ref{cor:compact-generation-germs} makes every group of germs finitely generated. A finitely generated locally finite group is finite.
\end{proof}

Here compact generation supplies finite generation of the groups of germs. The converse of Corollary~\ref{cor:bounded-depth-finite-germs} fails: a finite group of germs can have representatives with unbounded finitary depth on cylinders that do not contain the singular point.

In the examples, the \emph{$q$-ary odometer} on $\{0,\ldots,q-1\}^{\omega}$ adds one with the first letter read as the least significant digit: $a(iw)=(i+1)w$ for $i<q-1$, and $a((q-1)w)=0a(w)$. It sends $(q-1)^\omega$ to $0^\omega$. Its depth-$j$ truncation adds one modulo $q^j$ to the first $j$ digits and leaves the remaining continuation fixed. It therefore has order $q^j$ and finitary depth $j$.

\begin{exmp}[Unbounded depth with finite or infinite groups of germs]
\label{exmp:depth-counterexamples}
Let $\Omega(\mathsf{B})=\{0,1\}^{\omega}$, let $\xi=0^\omega$, and choose integers $N_j$ with $N_{j+1}>N_j+j+2$. Put $U_j=[0^{N_j}1]$. On $U_j$, let $b(0^{N_j}1z)=0^{N_j}1f_j(z)$, where $f_j$ is a finitary automorphism of the rooted tree below $0^{N_j}1$, and let $b$ be the identity outside the union of the $U_j$. Thus $b|_{0^{N_j}1}=f_j$: the section is taken at prefix length $N_j+1$, while its additional finitary depth will be $j$. The map $b$ fixes $\xi$ and has a nontrivial germ there since nontrivial supports occur in every neighborhood of $\xi$.

First, choose $f_j$ to interchange the two children of the single vertex $0^{j-1}$ and act trivially elsewhere. It is an involution of finitary depth $j$ with at most one nontrivial section at each level. Then $b^2=1$, and thus the cyclic group generated by the germ $(b,\xi)$ is $C_2$, although the finitary depths of the sections on the cylinders $U_j$ are unbounded. Alternatively, let $f_j$ be the binary adding machine truncated at depth $j$. It has order $2^j$ and again at most one nontrivial section at each level. For every $m>0$, the automorphisms $f_j^m$ are nontrivial for all sufficiently large $j$. Hence, $b^m$ is nontrivial on every neighborhood of $\xi$. The germ $(b,\xi)$ therefore has infinite order. The spacing condition on the $N_j$ ensures that at each level there is at most one nontrivial section outside the prefix of $\xi$, and thus $b$ is a bounded automorphism. Thus exactly the same unbounded-depth geometry is compatible with a finite group of germs and with an infinite cyclic group of germs.
\end{exmp}

Adjoin the binary odometer $a$. Every equal-length binary prefix replacement is locally a power of $a$, and thus the resulting groupoid contains the binary tail groupoid. The only non-AF generator germs are the carry from $1^\omega$ to $0^\omega$, its inverse, and the germ of $b$ at $0^\omega$ together with its inverse. The exceptional orbital graph consequently contains two infinite tiles. Choosing one edge from each pair of edges inverse to each other, its boundary quotient has one edge between the two tiles and one loop represented by $(b,0^\omega)$. Proposition~\ref{prop:germ-fg} then shows that the whole group of germs at $0^\omega$ is generated by $(b,0^\omega)$. It is respectively $C_2$ or $\mathbb{Z}$. Relative homology records the abelianized return germ, rather than the depths of the finitary sections on the cylinders $U_j$.

We finish the rooted-tree case with a criterion using the groups generated by the sections along a ray. The ray stabilizer fixes each prefix, and thus the section at that prefix defines a homomorphism.

\keepwithstatement
\begin{prop}
\label{prop:tail-stabilization}
Let $G$ act by automorphisms of a rooted tree. Let $\xi=v_1v_2\ldots$ be a boundary ray, and suppose that $H_\xi$ is generated by the germs of $r_1,\ldots,r_d\in G_\xi$. Put $Q=\langle r_1,\ldots,r_d\rangle$. For the length-$n$ prefix $\xi_n$, let $Q_n=\langle r_1|_{\xi_n},\ldots,r_d|_{\xi_n}\rangle$, where $Q_n$ acts on the whole rooted subtree below $\xi_n$. Then $H_\xi$ is finite if and only if $Q_N$ is finite for some $N$. If $H_\xi$ is finite, the restriction maps identify $Q_n$ with $H_\xi$ for all sufficiently large $n$.
\end{prop}

\begin{proof}
Since every element of $Q$ fixes $\xi$, taking the section at $\xi_n$ defines a homomorphism $Q\to Q_n$. Let $N_n$ be its kernel. The section identities imply $N_n\subseteq N_{n+1}$. Moreover, $N_\infty=\bigcup_nN_n$ is exactly the kernel of the germ homomorphism $Q\to H_\xi$: an element has trivial germ at $\xi$ precisely when it is the identity on one sufficiently deep prefix cylinder. Hence $H_\xi=Q/N_\infty$.

If $H_\xi$ is finite, then $N_\infty$ has finite index in the finitely generated group $Q$ and is therefore finitely generated. A finite generating set of $N_\infty$ is contained in one $N_N$, and thus $N_N=N_\infty$. It follows that $Q_N=Q/N_N\cong H_\xi$, and the same equality of kernels holds for every $n\geq N$. Conversely, if $Q_N=Q/N_N$ is finite for some $N$, then its quotient $Q/N_\infty=H_\xi$ is finite.
\end{proof}

The groups $Q_n$ act on the whole rooted subtrees below the corresponding prefixes. Replacing them by finite truncations would make every image finite. The criterion requires some $Q_N$ to be finite. Stabilization of the kernels alone does not imply this, since the stabilized quotient may be infinite. In the finite case of Example~\ref{exmp:depth-counterexamples}, each $\langle b|_{0^n}\rangle$ has order at most two, despite the unbounded depths of the sections $b|_{0^{N_j}1}$. Section~\ref{sec:finite-return-relations} uses invariant cylinders to lift finite relations, without requiring groups obtained by restriction to fixed cylinders to be finite.

\subsection{Stabilization along invariant neighborhoods}\label{subsec:tail-stabilization}
For general local homeomorphisms, restriction to shrinking neighborhoods does not automatically define group homomorphisms. The rooted-tree argument works since the representatives fixing $\xi$ preserve every prefix cylinder around $\xi$. The correct replacement is therefore a common invariant clopen basis. Singularity, the condition $H_x\neq\{(\Id,x)\}$, and the germ-defining condition of the finite-singular-germ setting do not by themselves provide such a basis.

When $H\leq\G_x^x$ is finite, its germs can be represented simultaneously on one invariant clopen neighborhood of $x$. This elementary fact will also be used later to obtain coherent realizations when the relevant groups of germs are finite.

\begin{lemma}
\label{lem:finite-germ-coherent-lift}
Let $H\leq\G_x^x$ be finite. There are a clopen neighborhood $U$ of $x$ and a homomorphism $\rho:H\to\F(\G|_U)$ such that $\rho(h)$ fixes $x$, has germ $h$ at $x$, and is AF at every other point of $U$. Extension by the identity outside $U$ gives a homomorphism $H\to\F(\G)$ with no exceptional sources outside $\{x\}$.
\end{lemma}
\begin{proof}
For every nonidentity $h\in H$, choose an isolating bisection containing $h$, as in Definition~\ref{def:isolating-bisection}. For the identity use the unit bisection at a clopen neighborhood of $x$. Let $f_h$ be the associated local homeomorphisms. Since $H$ is finite and the equalities $f_kf_h=f_{kh}$ hold as germs at $x$, there is one clopen neighborhood $W$ of $x$ on which all $f_h$ are defined, all products that occur are defined, and all these multiplication identities hold.

Put $U=\bigcap_{h\in H}f_h(W)$. This is a clopen neighborhood of $x$ contained in $W$. For every $k\in H$, injectivity and the multiplication relations give $f_k(U)=\bigcap_{h\in H}f_kf_h(W)=\bigcap_{h\in H}f_{kh}(W)=U$. Thus the restrictions $f_h|_U$ form an actual action of $H$ on $U$. The chosen isolating bisections are AF away from $x$, and that property is preserved by restriction and by extension with the identity outside $U$.
\end{proof}

In particular, after choosing any decreasing clopen basis $V_n$ at $x$ inside $U$, the sets $U_n=\bigcap_{h\in H}\rho(h)(V_n)$ form a common decreasing $H$-invariant clopen basis. For arbitrary finitely generated groups of germs, we therefore state the stabilization criterion under the existence of such a basis.

\keepwithstatement
\begin{prop}
\label{prop:invariant-neighborhood-stabilization}
Let $r_1,\ldots,r_d$ be local representatives fixing $x$ whose germs generate $H_x$. Suppose they preserve a common decreasing clopen neighborhood basis $U_1\supseteq U_2\supseteq\cdots$ at $x$, and are defined on $U_1$. Put $Q=\langle r_1|_{U_1},\ldots,r_d|_{U_1}\rangle$ and $Q_n=\{q|_{U_n}:q\in Q\}$. Then $H_x$ is finite if and only if some $Q_N$ is finite. In that case, restriction identifies $Q_n$ with $H_x$ for every sufficiently large $n$.
\end{prop}
\begin{proof}
Invariance makes every restriction $Q\to Q_n$ a group homomorphism. Let $N_n$ be its kernel. The kernels increase, and their union is exactly the kernel of the germ homomorphism $Q\to H_x$, since the $U_n$ form a neighborhood basis. If $H_x$ is finite, this union has finite index in the finitely generated group $Q$ and hence is finitely generated. Its finite generating set is contained in one $N_N$, and thus $N_n=\ker(Q\to H_x)$ for all $n\geq N$. Hence $Q_n\cong H_x$ eventually. Conversely, $H_x$ is a quotient of every $Q_n$, and thus finiteness of one $Q_N$ implies finiteness of $H_x$.
\end{proof}

The invariant basis is essential: without it, restriction need not be a homomorphism, and the pointwise stabilizer of a neighborhood need not be normal in $Q$. Finitary depth from Definition~\ref{def:finitary-depth} supplies no replacement for this invariance. Section~\ref{sec:finite-return-relations} returns to the same invariant-neighborhood issue: for a finitely presented group $L$ and a homomorphism $\phi:L\to H_x$, one chooses a clopen neighborhood $U\ni x$ and a homomorphism $\rho:L\to\mathsf{F}(\mathfrak{G}|_U)$ such that the germ of $\rho(l)$ at $x$ is $\phi(l)$ for every $l\in L$.

An infinite perfect group of germs has $H^{\ab}=0$, while a finite group of germs may have nontrivial torsion in its abelianization. The next section therefore uses the actual abelianizations $H_{\mathcal{O}}^{\ab}$, keeping finite $H_{\mathcal{O}}$, infinite $H_{\mathcal{O}}$ with finite abelianization, and $H_{\mathcal{O}}^{\ab}$ of positive rank separate. Neither prefix length nor finitary depth enters the relative-homology decomposition used in the next section.

\section{A structural description of the index map}
\label{sec:index}

\subsection{Relative homology and the orbit normal form}\label{subsec:relative-normal-form}
By Subsection~\ref{subsec:AFbd-homology}, $H_1(\G,\T)$ is the relative group generated by classes determined by exceptional arrows, and it fits into the exact sequence~\eqref{eq:relative}. We now put coordinates on this group and on the connecting map $\delta$. The source and range neighborhoods in~\eqref{eq:relative-defect-prelim} are essential: their classes in the dimension group are not determined by the incidence data of the boundary quotients $Q_{\mathcal{S}}(\mathcal{O})$ alone.

Fix an exceptional orbit $\mathcal{O}$ in the sense of Definition~\ref{def:exceptional-arrow-orbit}, and write $\mathcal{T}_i$ for the $\T$-orbits contained in it. In a compatible tile presentation, these are the infinite tiles. The normal-form calculation also allows finite $\T$-orbits. Choose $x_i\in\mathcal{T}_i$ and distinguish $x_0$ as basepoint. These are arbitrary orbit representatives and need not be persistent boundary points for a generating set. Put $H_{\mathcal{O}}=\G_{x_0}^{x_0}$. Choose arrows $c_i:x_0\to x_i$, with $c_0=(\Id,x_0)$. These \emph{connectors} are groupoid arrows between chosen representatives, as distinct from the instructions for connectors at finite levels in the tile inflation. For every $x\in\mathcal{T}_i$, let $a_x:x_i\to x$ be the unique tail arrow and put $p_x=a_xc_i$.

Recall from Subsection~\ref{subsec:full-groups-generation} that $\mathcal{T}(\mathcal{O})$ denotes the set of $\T$-orbits contained in $\mathcal{O}$. We write $\mathbb{Z}[\mathcal{T}(\mathcal{O})]$ for the free abelian group on $\mathcal{T}(\mathcal{O})$, equivalently the group of finitely supported integer-valued functions on this set, with basis vectors $e_{\mathcal{T}_i}$.

Let $\operatorname{aug}_{\mathcal{O}}:\mathbb{Z}[\mathcal{T}(\mathcal{O})]\to\mathbb{Z}$ send every basis vector to $1$. Explicitly,
$\operatorname{aug}_{\mathcal{O}}(\sum_i n_i e_{\mathcal{T}_i})=\sum_i n_i$. This is the \emph{augmentation}: it sends a vector to the sum of its coefficients.

\begin{defn}[Transport lattice]
\label{def:transport-lattice}
The \emph{transport lattice} of $\mathcal{O}$ is
\begin{equation}
\label{eq:transport-lattice}
L_{\mathcal{O}}=\ker\left(\operatorname{aug}_{\mathcal{O}}:\mathbb{Z}[\mathcal{T}(\mathcal{O})]\longrightarrow\mathbb{Z}\right).
\end{equation}
\end{defn}

The vectors in $L_{\mathcal{O}}$ have total coefficient zero. If $|\mathcal{T}(\mathcal{O})|=k<\infty$, then $L_{\mathcal{O}}\cong\mathbb{Z}^{k-1}$. The lattice $L_{\mathcal{O}}$ records transport among the tiles, while the $H_{\mathcal{O}}^{\ab}$ coordinate in the normal form below records the abelianized normalized return germ.

Nekrashevych's presentation separates the relative group by exceptional orbits. Within one orbit, endpoints give the transport vector, while closing an arrow with the chosen connectors gives an isotropy germ. The next theorem makes this separation explicit and gives a noncanonical direct-sum decomposition.

\begin{thm}[Relative normal form]
\label{thm:relative-normal-form}
There is a noncanonical isomorphism
\begin{equation}
\label{eq:relative-normal-form}
H_1(\G,\T)\cong\bigoplus_{\mathcal{O}}
\left(H_{\mathcal{O}}^{\ab}\oplus L_{\mathcal{O}}\right),
\end{equation}
where the sum runs over all exceptional $\G$-orbits. If $\mathcal{T}(\mathcal{O})$ is infinite, vectors in $\mathbb{Z}[\mathcal{T}(\mathcal{O})]$ have finite support. Under the hypothesis of generation by finitely many bisections, there are only finitely many exceptional orbits and finitely many $\T$-orbits in each. For an arrow $\gamma:x\to y$ contained in an exceptional orbit, its coordinate is
\begin{equation}
\label{eq:normal-coordinate}
[\gamma]\longmapsto
\left([p_y^{-1}\gamma p_x]_{\ab},\ e_{\mathcal{T}_y}-e_{\mathcal{T}_x}\right).
\end{equation}
\end{thm}

\begin{proof}
We start from the relative presentation in \cite[Section~5.4]{nek22}, in the form recalled in Subsection~\ref{subsec:AFbd-homology}. The relative group is generated by arrow symbols $[\gamma]$, with $[\gamma]=0$ for $\gamma\in\T$ and $[\alpha\beta]=[\alpha]+[\beta]$ whenever $\alpha$ and $\beta$ are composable. These relations never mix different $\G$-orbits, and thus $H_1(\G,\T)$ is the direct sum of its orbitwise contributions. Fix one exceptional orbit $\mathcal{O}$ and denote the corresponding summand by $H_1(\G,\T)_{\mathcal{O}}$.

For an arrow $\gamma:x\to y$ in $\mathcal{O}$, set $h_\gamma=p_y^{-1}\gamma p_x\in H_{\mathcal{O}}$. This is exactly the normalized return of Definition~\ref{def:normalized-return} for the connector system fixed above. Define $\Phi_{\mathcal{O}}([\gamma])=([h_\gamma]_{\ab},e_{\mathcal{T}_y}-e_{\mathcal{T}_x})$. We verify that this map respects the relations in Nekrashevych's presentation. If $\beta:x\to y$ and $\alpha:y\to z$ are composable, then $h_{\alpha\beta}=h_\alpha h_\beta$, while the middle transport terms cancel. Hence $\Phi_{\mathcal{O}}([\alpha\beta])=\Phi_{\mathcal{O}}([\alpha])+\Phi_{\mathcal{O}}([\beta])$. If $\gamma\in\T$, then $x$ and $y$ lie in the same $\T$-orbit $\mathcal{T}_i$. The arrows $\gamma a_x$ and $a_y$ are both tail arrows from $x_i$ to $y$. Principality of $\T$ gives $\gamma a_x=a_y$. Therefore $h_\gamma=(\Id,x_0)$ and the transport vector is also zero. Thus $\Phi_{\mathcal{O}}$ kills the AF relations and descends to a homomorphism
$H_1(\G,\T)_{\mathcal{O}}\to H_{\mathcal{O}}^{\ab}\oplus L_{\mathcal{O}}$.

The inverse of $\Phi_{\mathcal{O}}$ is supplied by the same connectors. For $h\in H_{\mathcal{O}}$, send $[h]_{\ab}$ to the relative class $[h]\in H_1(\G,\T)_{\mathcal{O}}$. This is well defined since the relation $[hk]=[h]+[k]$ makes the isotropy subgroup abelian inside $H_1(\G,\T)_{\mathcal{O}}$. For the transport lattice, use the basis $e_{\mathcal{T}_i}-e_{\mathcal{T}_0}$ and send it to $[c_i]$. Denote the resulting homomorphism by $\Psi_{\mathcal{O}}$.

It remains only to check that the two maps are inverse. If $\gamma:x\to y$ with $x\in\mathcal{T}_i$ and $y\in\mathcal{T}_j$, then
$\gamma=p_yh_\gamma p_x^{-1}$. Since $p_x=a_xc_i$ and $p_y=a_yc_j$ with $a_x,a_y\in\T$, Nekrashevych's relations give $[p_x]=[c_i]$ and $[p_y]=[c_j]$, and hence
\begin{equation*}
[\gamma]=[h_\gamma]+[c_j]-[c_i].
\end{equation*}
This is precisely $\Psi_{\mathcal{O}}\Phi_{\mathcal{O}}([\gamma])$. Conversely, $\Phi_{\mathcal{O}}([h])=([h]_{\ab},0)$ for isotropy at $x_0$, and $\Phi_{\mathcal{O}}([c_i])=(0,e_{\mathcal{T}_i}-e_{\mathcal{T}_0})$, and thus $\Phi_{\mathcal{O}}\Psi_{\mathcal{O}}$ is the identity on the two summands. This proves the orbitwise isomorphism and then the direct sum in~\eqref{eq:relative-normal-form}.

Under compact generation, Corollary~\ref{cor:compact-generation-germs} gives finitely many exceptional $\G$-orbits and finitely many $\T$-orbits in each.
\end{proof}

\begin{rmk}
\label{rmk:normal-form-scope}
The proof uses only the AF-by-discrete inclusion and not the finite family of generating bisections. Without that generating hypothesis, the same formula holds over all exceptional orbits, with $\mathbb{Z}[\mathcal{T}(\mathcal{O})]$ understood as the free abelian group of finitely supported vectors and $L_{\mathcal{O}}$ as its augmentation kernel. The summands and their groups of germs can then be infinitely generated. The finite-generation conclusions and the finite-matrix stabilization below retain the additional hypotheses from Subsection~\ref{subsec:hypothesis-scope}.
\end{rmk}

\begin{prop}
\label{prop:connector-shear}
Keep the base point and the representatives $x_i$ fixed. Write $c_i'=c_i u_i$, where $u_i\in H_{\mathcal{O}}$ and $u_0=(\Id,x_0)$, and define $u_\#:L_{\mathcal{O}}\to H_{\mathcal{O}}^{\ab}$ by $u_\#(e_{\mathcal{T}_i}-e_{\mathcal{T}_0})=[u_i]_{\ab}$. The coordinates of the same relative class in the new decomposition are
\begin{equation}
\label{eq:connector-shear}
(a,l)\longmapsto(a-u_\#(l),l).
\end{equation}
\end{prop}

\begin{proof}
If $\gamma:x\to y$ runs from $\mathcal{T}_i$ to $\mathcal{T}_j$, then $p_x'=p_xu_i$ and $p_y'=p_yu_j$. Its new normalized return is $u_j^{-1}(p_y^{-1}\gamma p_x)u_i$. Abelianizing subtracts $[u_j]_{\ab}-[u_i]_{\ab}=u_\#(e_{\mathcal{T}_j}-e_{\mathcal{T}_i})$, while transport is unchanged. The relative classes determined by exceptional arrows generate the relative group, proving the formula.
\end{proof}

The relative group is intrinsic to the pair $(\G,\T)$, whereas the decomposition into return and transport coordinates depends on the connectors. In contrast, $H_1(\G)$ and the index map belong to $\G$. Replacing the AF core is not a change of generators or a telescoping of the same pair.

\subsection{The defect map}\label{subsec:defect-map}

The connecting map in~\eqref{eq:relative} measures the obstruction to lifting a class in $H_1(\G,\T)$ to $H_1(\G)$.

\begin{defn}[Defect map]
\label{def:defect-map}
The connecting homomorphism
$\delta:H_1(\G,\T)\to H_0(\T)$ in~\eqref{eq:relative} is called the \emph{defect map}, and $\delta(z)$ is the \emph{defect} of a relative class $z\in H_1(\G,\T)$. If $z=[\gamma]$ for an exceptional arrow, we also call $\delta[\gamma]$ the \emph{defect of $\gamma$}.

After the identification of Theorem~\ref{thm:relative-normal-form}, we use the same symbol for the coordinate form
\begin{equation}
\label{eq:defect-map}
\delta:
\bigoplus_{\mathcal{O}}\left(H_{\mathcal{O}}^{\ab}\oplus L_{\mathcal{O}}\right)
\longrightarrow H_0(\T).
\end{equation}
For an exceptional orbit $\mathcal{O}$, the restriction of $\delta$ to $H_{\mathcal{O}}^{\ab}$ is the \emph{isotropy defect} $\lambda_{\mathcal{O}}$, and its restriction to $L_{\mathcal{O}}$ is the \emph{transport defect} $\tau_{\mathcal{O}}$.
\end{defn}

If $V\ni\gamma$ is an isolating bisection (Definition~\ref{def:isolating-bisection}) for an exceptional arrow $\gamma$, Nekrashevych's formula~\eqref{eq:relative-defect-prelim} gives
$\delta[\gamma]=[1_{\sg(V)}]-[1_{\rg(V)}]$. In the coordinates of Definition~\ref{def:defect-map},
$\delta((a_{\mathcal{O}},l_{\mathcal{O}})_{\mathcal{O}})=\sum_{\mathcal{O}}(\lambda_{\mathcal{O}}(a_{\mathcal{O}})+\tau_{\mathcal{O}}(l_{\mathcal{O}}))$.
All summands take values in the same dimension group. Under Proposition~\ref{prop:connector-shear}, $\lambda_{\mathcal{O}}'=\lambda_{\mathcal{O}}$ and $\tau_{\mathcal{O}}'=\tau_{\mathcal{O}}+\lambda_{\mathcal{O}}u_\#$. The defect of the total relative class is unchanged. Equation~\eqref{eq:relative} gives the structural description
\begin{equation}
\label{eq:H1-kernel}
H_1(\G)\cong\ker\delta.
\end{equation}

The map $\lambda_{\mathcal{O}}$ need not vanish on an abelianized isotropy class. A local bisection containing an isotropy germ can have source and range with different classes in the dimension group. Likewise, defects coming from different exceptional orbits can cancel in $H_0(\T)$. This is why $H_1$ is not, in general, the direct sum of the individual kernels.

\begin{prop}
\label{prop:torsion-isotropy-zero-defect}
The homomorphism $\lambda_{\mathcal{O}}$ vanishes on $\Tor H_{\mathcal{O}}^{\ab}$. In particular, if $H_{\mathcal{O}}^{\ab}$ is finite, then $\lambda_{\mathcal{O}}=0$.
\end{prop}

\begin{proof}
Recall that (see \eqref{eq:H0-direct-limit} below)
\[
H_0(\T)=\varinjlim(\mathbb{Z}^{V_n},M_n)
\]
is the dimension group associated with the Bratteli diagram $\mathsf{B}$; see
\cite{HPS92}. In particular, it is torsion-free. Indeed, if
$m[z_N]=0$ in $H_0(\T)$ for some nonzero integer $m$, then, by the
definition of the direct limit, there exists $n\geq N$ such that $mM_{N,n}z_N=0$ in $\mathbb{Z}^{V_n}$. Since $\mathbb{Z}^{V_n}$ is torsion-free, we have
$M_{N,n}z_N=0$, and hence $[z_N]=0$ in $H_0(\T)$. Therefore $\delta$
vanishes on $\Tor H_1(\G,\T)$.
\end{proof}

\subsection{Cylinder coordinates at finite levels}
The cylinder-count construction uses only the Bratteli model of the AF core. The interpretation in terms of paths inside tiles in the following subsection additionally uses the tile presentation compatible with the AF core. We use the level indexing fixed in Subsection~\ref{subsec:bratteli-core}: $\Omega_m(\mathsf{B})$ consists of paths of length $m$ ending in $V_{m+1}$, and $[p]$ denotes the cylinder determined by $p$. Cylinder-count vectors will be indexed by the terminal vertex level. The coordinates indexed by $V_n$ count paths in $\Omega_{n-1}(\mathsf{B})$, and the corresponding finite tiles have level $n-1$. We use the incidence homomorphism
\begin{equation}
\label{eq:incidence-map}
M_n:\mathbb{Z}^{V_n}\longrightarrow\mathbb{Z}^{V_{n+1}},\qquad
M_ne_v=\sum_{\substack{e\in E_n\\\mathbf{s}(e)=v}}e_{\mathbf{r}(e)}.
\end{equation}
With this convention,
\begin{equation}
\label{eq:H0-direct-limit}
H_0(\T)\cong\varinjlim(\mathbb{Z}^{V_n},M_n).
\end{equation}

If $A\subseteq \Omega(\mathsf{B})$ is clopen, then for all sufficiently large $n$, it is a union of cylinders ending in $V_n$. Define its \emph{cylinder-count vector at vertex level $n$} by
\begin{equation}
\label{eq:cylinder-count}
\mathbf{c}_n(A)=\sum_{\substack{p\in\Omega_{n-1}(\mathsf{B}),\ [p]\subseteq A}}e_{\mathbf{r}(p)}\in\mathbb{Z}^{V_n}.
\end{equation}
The vectors satisfy $M_n\mathbf{c}_n(A)=\mathbf{c}_{n+1}(A)$ and represent the class $[1_A]\in H_0(\T)$.

Let $\gamma\in\G\setminus\T$ and choose an isolating bisection $V$ as in Definition~\ref{def:isolating-bisection}. Then
\begin{equation}
\label{eq:isolating-bisection}
V\setminus\{\gamma\}\subseteq\T.
\end{equation}
Put $A=\sg(V)$ and $B=\rg(V)$. For every sufficiently large $n$, both $A$ and $B$ are unions of cylinders ending in $V_n$. Define
\begin{equation}
\label{eq:defect-vector}
\mathbf{d}_n(V)=\mathbf{c}_n(A)-\mathbf{c}_n(B)\in\mathbb{Z}^{V_n}.
\end{equation}

\begin{prop}
\label{prop:finite-level-defect-vector}
Let $V\ni\gamma$ be a compact open bisection with $V\setminus\T=\{\gamma\}$, as above. The vectors in \eqref{eq:defect-vector} satisfy
\begin{equation}
\label{eq:defect-vector-compatibility}
M_n\mathbf{d}_n(V)=\mathbf{d}_{n+1}(V)
\end{equation}
for all sufficiently large $n$, and their direct-limit class is the defect of $\gamma$ relative to the fixed AF core:
\begin{equation}
\label{eq:defect-vector-class}
\delta[\gamma]=[\mathbf{d}_n(V)]\in H_0(\T).
\end{equation}
If $W\ni\gamma$ is another compact open bisection with $W\setminus\T=\{\gamma\}$, then the two compatible sequences represent the same element of the direct limit and agree at every sufficiently deep common level. Equality is not asserted at the initial levels.
\end{prop}

\begin{proof}
Compatibility follows immediately from the refinement identity for cylinder-count vectors. By definition of the connecting map in \eqref{eq:relative},
$\delta[\gamma]=[1_{\sg(V)}]-[1_{\rg(V)}]=[1_A]-[1_B]$, and \eqref{eq:defect-vector-class} follows from \eqref{eq:H0-direct-limit}. If $W$ is another isolating bisection, shrink both bisections to a common compact open neighborhood of $\gamma$. The parts removed from either bisection lie in $\T$, and thus their source-minus-range classes vanish in $H_0(\T)$. It follows that the direct-limit class is independent of the isolating bisection. Equality in an inductive limit of abelian groups means equality after applying a sufficiently long connecting map. Compatibility then gives equality at all later levels.
\end{proof}

In a compatible tile presentation, the contributions of trajectories that use only internal edges cancel in these vectors. The remaining terms come from
boundary exits.

\subsection{Transition vectors at the boundary}
Let $F=F_\ell\cdots F_1\in G_0=\langle\mathcal{S}\rangle_{\mathrm{inv}}$ be a product of bisections $F_1,\ldots,F_\ell\in\mathcal{S}$. Let $A\subseteq\sg(F)$ be clopen and put $V=F|_A$, with $B=\rg(V)=F(A)$. Suppose that $A$ and $B$ are unions of cylinders ending in $V_n$. Let $p\in\Omega_{n-1}(\mathsf{B})$ satisfy $[p]\subseteq A$. If the trajectory from $p$ labeled by $F_1,\ldots,F_\ell$ is defined and remains internal at every step, then $F$ gives a prefix replacement that leaves the infinite continuation unchanged:
\begin{equation}
\label{eq:internal-cylinder-map}
F|_{[p]}:[p]\longrightarrow[q]
\end{equation}
for a unique path $q$ ending in $V_n$ with $\mathbf{r}(q)=\mathbf{r}(p)$. Let $\partial_n^+(F,A)$ be the set of paths $p\in\Omega_{n-1}(\mathsf{B})$ with $[p]\subseteq A$ for which this trajectory meets a boundary edge, and put
\begin{equation}
\label{eq:negative-boundary-exits}
\partial_n^-(F,A)=\partial_n^+(F^{-1},B).
\end{equation}
These two sets record the source and range cylinders whose trajectories meet the boundary. The contributions of trajectories that use only internal edges cancel, giving
the following expression for the defect.
\keepwithstatement
\begin{prop}
\label{thm:boundary-transition-formula}
Let $V=F|_A$ be as above. For every sufficiently large $n$,
\begin{equation}
\label{eq:boundary-transition-vector}
\mathbf{d}_n(V)=
\sum_{p\in\partial_n^+(F,A)}e_{\mathbf{r}(p)}-
\sum_{q\in\partial_n^-(F,A)}e_{\mathbf{r}(q)}.
\end{equation}
Hence, the defect is computed entirely by the trajectories at finite levels that meet the boundary of a finite tile.
\end{prop}

\begin{proof}
Consider the cylinders ending in $V_n$ contained in $A$. If the trajectory from $p$ is internal at every step, then \eqref{eq:internal-cylinder-map} pairs its source cylinder with a range cylinder having the same terminal Bratteli vertex. Hence its contribution $e_{\mathbf{r}(p)}$ to $\mathbf{c}_n(A)$ cancels the contribution $e_{\mathbf{r}(q)}$ to $\mathbf{c}_n(B)$. Reversing a trajectory that uses only internal edges gives a trajectory for $F^{-1}$ that also uses only internal edges. Thus the corresponding source and range cylinders are paired bijectively. After all such pairs are cancelled, the remaining source cylinders are precisely $\partial_n^+(F,A)$ and the remaining range cylinders are precisely $\partial_n^-(F,A)$. This gives \eqref{eq:boundary-transition-vector}.
\end{proof}

\begin{rmk}
\label{rmk:BNS-transition}
For Cantor systems with finitely many minimal components, Bezuglyi--Niu--Sun encode source-range incidence vectors by transition graphs of a $k$-simple Bratteli diagram. See \cite[Theorem~7.5]{BNS21}. Their index is the connecting map of a $C^*$-algebra extension with values in the dimension group of an AF ideal. Both maps record source--range incidence, but belong to different exact sequences. Their bounded subgroup is defined using uniform $\ell^\infty$ bounds \cite[Definition~7.8]{BNS21}. Below we use an $\ell^1$ variant: the two boundedness conditions are equivalent under bounded Bratteli rank, but should not be identified without that hypothesis. Formula~\eqref{eq:boundary-transition-vector} concerns exits at finite levels of the selected boundary return or transport bisection and also allows nontrivial groups of germs.
\end{rmk}

\subsection{Bounded type and bounded classes of defects}\label{subsec:bounded-defects}
Assume the tile presentation compatible with the AF core has uniformly bounded cardinalities of the boundaries of finite tiles. For this subsection, we also assume bounded Bratteli rank: fix constants $C_\partial$ and $r_{\mathsf{B}}$ such that every finite tile has at most $C_\partial$ boundary points and $|V_n|\leq r_{\mathsf{B}}$ for every $n$. The proof uses these two bounds separately. Controlling each tile individually gives no uniform bound on the total number of boundary points over a level.

\begin{prop}
\label{prop:uniform-boundary-defect}
Let $F=F_\ell\cdots F_1$ be a bisection word with $F_j\in\mathcal{S}$. For every clopen restriction $V=F|_A$ and every sufficiently large $n$,
\begin{equation}
\label{eq:uniform-boundary-count}
|\partial_n^+(F,A)|\leq \ell C_\partial r_{\mathsf{B}},\qquad
|\partial_n^-(F,A)|\leq \ell C_\partial r_{\mathsf{B}}.
\end{equation}
Consequently,
\begin{equation}
\label{eq:uniform-defect-norm}
\|\mathbf{d}_n(V)\|_1\leq 2\ell C_\partial r_{\mathsf{B}}.
\end{equation}
\end{prop}

\begin{proof}
Suppose that the trajectory from a vertex $p$ of a level-$(n-1)$ finite tile meets a boundary edge. Let $j$ be the first such step. The first $j-1$ steps are internal, hence define an injective partial permutation inside the level-$(n-1)$ finite tiles. For fixed $j$ and fixed finite tile, the possible starting vertices thus inject into the boundary points of that finite tile, of which there are at most $C_\partial$. There are $\ell$ choices of $j$ and at most $r_{\mathsf{B}}$ finite tiles, proving the first inequality. Apply the same argument to $F^{-1}$ for the second. Equation~\eqref{eq:uniform-defect-norm} follows from Proposition~\ref{thm:boundary-transition-formula}.
\end{proof}

Recall that $H_0(\T)=\varinjlim(\mathbb{Z}^{V_n},M_n)$. The next subgroup is the $\ell^1$ analogue, for the displayed Bratteli presentation, of the bounded $\ell^\infty$ vectors in \cite[Definition~7.8]{BNS21}.

\begin{defn}[$\ell^1$-bounded subgroup of the dimension group]
\label{def:l1-bounded-dimension-subgroup}
The \emph{$\ell^1$-bounded subgroup} associated with the displayed Bratteli presentation is
\begin{equation}
\label{eq:l1-bounded-dimension-subgroup}
H_0(\T)_{\mathrm{bd},1}=
\{z=[z_N]\in H_0(\T):\sup_{n\geq N}\|M_{N,n}z_N\|_1<\infty\},
\end{equation}
where $M_{N,n}=M_{n-1}\cdots M_N$.
\end{defn}

The condition is independent of the sufficiently late representative of $z$, but it is relative to the displayed Bratteli presentation. If $\mathsf{B}^{\prime}$ is a telescoping, the canonical identification of the dimension groups gives $H_0(\T)_{\mathrm{bd},1}\subseteq H_0(\mathfrak{T}_{\mathsf{B}^{\prime}})_{\mathrm{bd},1}$, since telescoping retains only a subsequence of the representatives. The reverse inclusion can fail.

\begin{exmp}[Bounded representatives can appear after telescoping]
\label{exmp:bounded-telescoping}
Consider a rank-two Bratteli diagram whose successive incidence maps, after its root level, occur in pairs $A_j,B_j$, where $N_j\geq3$, $N_j\to\infty$, and
\begin{equation}
A_j=\begin{pmatrix}N_j&0\\0&N_j-1\end{pmatrix},\qquad
B_j=\begin{pmatrix}1&1\\N_j-2&N_j-1\end{pmatrix}.
\end{equation}
Both maps are injective and $B_jA_j$ has strictly positive entries. Choosing a root connected to both vertices gives a simple Bratteli diagram with Cantor path space. For $z=(1,-1)^{\mathsf{t}}$ we have $A_jz=(N_j,1-N_j)^{\mathsf{t}}$ and $B_jA_jz=z$. Its nonzero direct-limit class has representatives of norm $2N_j-1$ at the intermediate levels, hence is not in $H_0(\T)_{\mathrm{bd},1}$. Every other representative agrees with this sequence eventually. Telescoping after each pair retains only the vector $z$, and thus the same class belongs to $H_0(\mathfrak{T}_{\mathsf{B}^{\prime}})_{\mathrm{bd},1}$. Boundedness on all displayed levels is thus presentation dependent.
\end{exmp}

\begin{cor}
\label{cor:bounded-type-defect}
Under the bounds on the cardinalities of boundaries of finite tiles and on the Bratteli rank in Proposition~\ref{prop:uniform-boundary-defect},
\begin{equation}
\label{eq:defect-in-l1-bounded-subgroup}
\delta(H_1(\G,\T))\subseteq H_0(\T)_{\mathrm{bd},1}.
\end{equation}
\end{cor}

\begin{proof}
By Proposition~\ref{prop:germ-fg} and Theorem~\ref{thm:relative-normal-form}, the relative group is generated by the classes of finitely many boundary returns and transport arrows. Represent each of them by a product of bisections in $\mathcal{S}$ and restrict that product to an isolating neighborhood of the chosen arrow. Proposition~\ref{prop:uniform-boundary-defect} gives a uniform $\ell^1$ bound for the compatible defect vectors of each generator. Since $H_0(\T)_{\mathrm{bd},1}$ is a subgroup, the conclusion follows.
\end{proof}

A \emph{normalized state} on the dimension group is a positive homomorphism $\rho:H_0(\T)\to\mathbb{R}$ with $\rho([1_{\Omega(\mathsf{B})}])=1$. It corresponds to a $\T$-invariant probability measure through $\rho([1_A])=\mu(A)$ for clopen $A$. The \emph{infinitesimal subgroup} is standard in the ordered-dimension-group approach to minimal Cantor dynamics. See \cite{HPS92,Putnam10}. We write $\operatorname{Inf}(H_0(\T))$ for the infinitesimal subgroup, which is defined to be the intersection of the kernels of all normalized states. If $\T$ is minimal, every invariant probability measure is atomless. Hence an $\ell^1$-bounded compatible vector has zero value under every state: if $z=[z_n]\in H_0(\T)_{\mathrm{bd},1}$ and $\|z_n\|_1\leq C$, then
\begin{equation}
\label{eq:bounded-state-estimate}
|\rho(z)|\leq C\max_{v\in V_n}\rho(e_v),
\end{equation}
where $\rho(e_v)$ is the common measure of the cylinders ending in $V_n$ terminating at $v$. The maximum tends to zero as $n\to\infty$. We obtain the following combinatorial refinement of Proposition~\ref{prop:defect-infinitesimal} below.

\begin{cor}
\label{cor:l1-bounded-subgroup-infinitesimal}
If $\T$ is minimal, then
$H_0(\T)_{\mathrm{bd},1}\subseteq\operatorname{Inf}(H_0(\T)).$
In particular, in the case of bounded type, \eqref{eq:defect-in-l1-bounded-subgroup} gives a proof at finite levels that every defect is infinitesimal.
\end{cor}

\subsection{Defect matrices and stabilization}
We now impose the hypothesis of generation by finitely many bisections from Subsection~\ref{subsec:hypothesis-scope}, or, more generally, assume directly that $H_1(\G,\T)$ is finitely generated. Its vectors at finite levels can then be assembled into a matrix that represents the whole map $\delta$. 
By Theorem~\ref{thm:relative-normal-form}, $\Tor H_1(\G,\T)=\bigoplus_{\mathcal{O}}\Tor H_{\mathcal{O}}^{\ab}$. Choose a splitting
\begin{equation}
\label{eq:relative-free-splitting}
H_1(\G,\T)\cong \Tor H_1(\G,\T)\oplus\mathbb{Z}^m
\end{equation}
and a basis $z_1,\ldots,z_m$ adapted to the individual isotropy and transport summands. With this choice, each $z_j$ can be represented by an isotropy return or by one of the connectors in Theorem~\ref{thm:relative-normal-form}. For a general change of free basis, a column is instead the corresponding integral linear combination of these defect vectors. Since the target $H_0(\T)$ is torsion-free, $\delta$ vanishes on $\Tor H_1(\G,\T)$.

Choose isolating bisections for $z_1,\ldots,z_m$. After increasing the level once, let $\mathbf{d}_{j,n}\in\mathbb{Z}^{V_n}$ be their compatible defect vectors and put
\begin{equation}
\label{eq:defect-matrix}
D_n=\begin{bmatrix}\mathbf{d}_{1,n}&\cdots&\mathbf{d}_{m,n}\end{bmatrix}:
\mathbb{Z}^m\longrightarrow\mathbb{Z}^{V_n}.
\end{equation}
Then
\begin{equation}
\label{eq:defect-matrix-compatibility}
M_nD_n=D_{n+1}.
\end{equation}

The next theorem computes $H_1(\G)$ from the relative torsion and the eventual kernel of these matrices.
\keepwithstatement
\begin{thm}
\label{thm:finite-level-defect}
Under the splitting \eqref{eq:relative-free-splitting}, let $\overline\delta:\mathbb{Z}^m\to H_0(\T)$ be the restriction of $\delta$ to the free summand. For every sufficiently large $n$ and every $a\in\mathbb{Z}^m$,
\begin{equation}
\label{eq:defect-matrix-computes}
\overline\delta(a)=[D_na]\in\varinjlim(\mathbb{Z}^{V_n},M_n).
\end{equation}
Moreover, the subgroups
\begin{equation}
\label{eq:matrix-kernels}
K_n=\ker(D_n:\mathbb{Z}^m\to\mathbb{Z}^{V_n})
\end{equation}
form an increasing sequence and eventually stabilize. If $n$ is sufficiently large, then
\begin{equation}
\label{eq:H1-stable-matrix}
H_1(\G)\cong \Tor H_1(\G,\T)\oplus K_n.
\end{equation}
In particular, after the free relative generators and their isolating bisections have been chosen, the first homology is determined by their cylinder-count defect matrix at one sufficiently deep level.
\end{thm}

\begin{proof}
Choose $N$ sufficiently large so that the defect vectors $\mathbf{d}_{j,n}$ are
defined and satisfy \eqref{eq:defect-vector-compatibility} for every $j$ and
every $n\geq N$. Let $a=(a_1,\ldots,a_m)^{\mathsf{t}}\in\mathbb{Z}^m$. By
Proposition~\ref{prop:finite-level-defect-vector}, we have
$\delta(z_j)=[\mathbf{d}_{j,n}]$ for every $n\geq N$. Since
$\overline\delta$ is the restriction of $\delta$ to the free summand in
\eqref{eq:relative-free-splitting}, it follows that
\[
\overline\delta(a)
=\sum_{j=1}^m a_j[\mathbf{d}_{j,n}]
=\left[\sum_{j=1}^m a_j\mathbf{d}_{j,n}\right]
=[D_na].
\]
This proves \eqref{eq:defect-matrix-computes}.

We next identify $\ker\overline\delta$. By
\eqref{eq:defect-matrix-compatibility}, if $a\in K_n$, then
$D_{n+1}a=M_nD_na=0$, and hence $K_n\subseteq K_{n+1}$. Thus the subgroups
$K_n$ form an increasing sequence.

For $n\geq N$, the class $[D_na]$ is zero in
$\varinjlim(\mathbb{Z}^{V_n},M_n)$ if and only if there exists $q\geq n$
such that
\[
M_{q-1}\cdots M_nD_na=0
\]
in $\mathbb{Z}^{V_q}$. Repeated use of
\eqref{eq:defect-matrix-compatibility} gives
$M_{q-1}\cdots M_nD_n=D_q$, so the preceding condition is equivalent to
$D_qa=0$, that is, to $a\in K_q$. Consequently,
\begin{equation}
\label{eq:kernel-union}
\ker\overline\delta=\bigcup_{n\geq N}K_n.
\end{equation}

Since $\mathbb{Z}^m$ is Noetherian as a $\mathbb{Z}$-module, every increasing
sequence of its subgroups stabilizes. Hence there exists $n_0\geq N$ such that
$K_n=K_{n_0}$ for every $n\geq n_0$. Together with
\eqref{eq:kernel-union}, this gives
$\ker\overline\delta=K_n$ for every $n\geq n_0$.

Finally, under the splitting \eqref{eq:relative-free-splitting}, every element
of $H_1(\G,\T)$ has a unique expression $t+a$ with
$t\in\Tor H_1(\G,\T)$ and $a\in\mathbb{Z}^m$. Since $H_0(\T)$ is
torsion-free, $\delta$ vanishes on $\Tor H_1(\G,\T)$. Therefore
$\delta(t+a)=\overline\delta(a)$, and hence, for every $n\geq n_0$,
\[
\ker\delta=\Tor H_1(\G,\T)\oplus K_n.
\]
Equation~\eqref{eq:H1-kernel} now gives
\[
H_1(\G)\cong\Tor H_1(\G,\T)\oplus K_n,
\]
which is \eqref{eq:H1-stable-matrix}.
\end{proof}

\begin{rmk}
\label{rmk:stabilization-not-effective}
Theorem~\ref{thm:finite-level-defect} proves that the kernels stabilize but gives no effective bound on the stabilization level. Equality of consecutive kernels over any finite interval is insufficient to certify final stabilization. For instance, the positive matrix $P=\left(\begin{smallmatrix}2&1\\1&2\end{smallmatrix}\right)$ preserves $z=(1,-1)^{\mathsf{t}}$, while $Q=\left(\begin{smallmatrix}1&1\\1&1\end{smallmatrix}\right)$ sends it to zero. A compatible one-column system can apply $P$ for any prescribed number of levels before applying $Q$. Its kernel changes only at that late step. This example concerns only compatible matrices. It does not assert that they are realized by a boundary bisection. Effective computation requires sufficient germ relations, explicit clopen representatives, and a separate certificate for the stabilization level.
\end{rmk}

\begin{cor}
\label{cor:lambda-tau-computable}
For each exceptional orbit, the maps $\lambda_{\mathcal{O}}$ and $\tau_{\mathcal{O}}$ are computed by the cylinder-count procedure of Proposition~\ref{prop:finite-level-defect-vector}. Choose generators of the free part of $H_{\mathcal{O}}^{\ab}$ and isolating bisections for representing isotropy arrows. Their columns give $\lambda_{\mathcal{O}}$. Choose the basis $e_{\mathcal{T}_i}-e_{\mathcal{T}_0}$ of $L_{\mathcal{O}}$ and isolating bisections for the connectors $c_i$. Their columns give $\tau_{\mathcal{O}}$. Torsion generators have zero direct-limit defect by Proposition~\ref{prop:torsion-isotropy-zero-defect}. For a torsion generator, the defect vector of a chosen representative may be nonzero at the initial level but becomes zero after a sufficiently long incidence map. For finitely many torsion generators, a common such level can be chosen.
\end{cor}

\begin{cor}
\label{cor:finite-singular-H1}
Assume the finite-singular-germ condition of \cite{kua26a}. Let $\Xi=\{\xi_1,\ldots,\xi_r\}$ be representatives of the exceptional orbits. Then
\begin{equation}
\label{eq:finite-singular-H1}
H_1(\G)\cong\bigoplus_{i=1}^r H_{\xi_i}^{\ab}.
\end{equation}
\end{cor}

\begin{proof}
Under the cited finite-singular-germ condition, every non-AF generator arrow fixes its designated persistent boundary point, and all other generator arrows are AF. It follows that no generator arrow joins distinct infinite tiles, each exceptional orbit contains only one infinite tile, and all transport lattices vanish. The graph of germs over an exceptional orbit can nevertheless have several lifted tiles, indexed by its group of germs. The corresponding isotropy is recorded in $H_{\mathcal{O}}^{\ab}$ rather than in $L_{\mathcal{O}}$. The group $H_1(\G,\T)$ is thus the direct sum of the finite groups $H_{\xi_i}^{\ab}$. Since $H_0(\T)$ is the dimension group associated with $\mathsf{B}$ and is torsion-free, the homomorphism $\delta:H_1(\G,\T)\to H_0(\T)$ is zero. Equation~\eqref{eq:relative} gives the result.
\end{proof}

\begin{cor}
\label{cor:torsion-rank}
Under the finite-generation hypothesis used above, put $k_{\mathcal{O}}=|\mathcal{T}(\mathcal{O})|$. Then
\begin{equation}
\label{eq:torsion-H1}
\Tor H_1(\G)\cong\bigoplus_{\mathcal{O}}\Tor H_{\mathcal{O}}^{\ab}.
\end{equation}
Moreover,
\begin{equation}
\label{eq:rank-H1}
\operatorname{rank}H_1(\G)=
\sum_{\mathcal{O}}\operatorname{rank}H_{\mathcal{O}}^{\ab}
+\sum_{\mathcal{O}}(k_{\mathcal{O}}-1)
-\operatorname{rank}\delta(H_1(\G,\T)).
\end{equation}
\end{cor}

\begin{proof}
Each $H_{\mathcal{O}}^{\ab}$ is a direct summand of $H_1(\G,\T)$ by Theorem~\ref{thm:relative-normal-form}, and is thus finitely generated under the present hypothesis. The target $H_0(\T)$ is torsion-free, and thus all torsion in $H_1(\G,\T)$ lies in $\ker\delta$, and no additional torsion can occur in a subgroup of the finitely generated abelian group $H_1(\G,\T)$. The rank formula follows from \eqref{eq:relative-normal-form} and \eqref{eq:H1-kernel}.
\end{proof}

\subsection{Minimal AF cores and infinitesimal defects}
Suppose now that $\T$ is minimal. We use the infinitesimal subgroup recalled in Subsection~\ref{subsec:bounded-defects}, but do not impose the uniform boundary and vertex-rank bounds of that subsection.

\begin{prop}
\label{prop:defect-infinitesimal}
Every $\T$-invariant probability measure on $\Omega(\mathsf{B})$ is $\G$-invariant. Consequently,
\begin{equation}
\label{eq:defect-infinitesimal}
\delta(H_1(\G,\T))\subseteq\operatorname{Inf}(H_0(\T)).
\end{equation}
In particular, if $H_0(\T)$ has no nonzero infinitesimals, then $\delta=0$ and
\begin{equation}
\label{eq:no-infinitesimals}
H_1(\G)\cong\bigoplus_{\mathcal{O}}
\left(H_{\mathcal{O}}^{\ab}\oplus L_{\mathcal{O}}\right).
\end{equation}
\end{prop}

\begin{proof}
A $\T$-invariant probability measure is atomless since $\T$ is minimal on the Cantor space. Let $U$ be a compact open bisection of $\G$. The set $U\setminus\T$ is finite by Lemma~\ref{lemma:finite-exceptional-support}. Removing the finitely many source and range points of these exceptional arrows preserves the measure. The source of the remaining part is an open subspace of the second-countable zero-dimensional space $\Omega(\mathsf{B})$. A countable disjoint clopen partition on which the remaining bisection lies in $\T$ decomposes that part into compact open bisections contained in $\T$. Countable additivity and $\T$-invariance give equal measures to its source and range. Hence the measure is $\G$-invariant. The source-minus-range defect of every relative generator is thus annihilated by all states on $H_0(\T)$, proving \eqref{eq:defect-infinitesimal}. The last statement is immediate.
\end{proof}

Equality under invariant measures is weaker than equality in the dimension group when infinitesimals are present. In examples with no infinitesimals, however, the entire first homology is read directly from the abelianized groups of germs and the numbers $k_{\mathcal{O}}$ of infinite tiles in the exceptional orbits.

\subsection{The boundary formula for the index map}

\begin{defn}[Exceptional support of a full-group element]
\label{def:exceptional-support}
For $g\in\F(\G)$, its \emph{exceptional support relative to $\T$} is
\begin{equation}
\label{eq:sigma-g}
\Sigma(g)=\{x\in \Omega(\mathsf{B}):(g,x)\notin\T\}.
\end{equation}
We also call $\Sigma(g)$ the exceptional set of $g$ when the AF core is fixed.
\end{defn}

The set $\Sigma(g)$ is finite by Lemma~\ref{lemma:finite-exceptional-support}. The index is determined by the return and transport contributions of these exceptional germs.
\keepwithstatement
\begin{thm}
\label{thm:index-formula}
Under the injection $H_1(\G)\hookrightarrow H_1(\G,\T)$ in \eqref{eq:relative},
\begin{equation}
\label{eq:index}
\iota I(g)=\sum_{x\in\Sigma(g)}[(g,x)].
\end{equation}
After choosing the connectors used in Theorem~\ref{thm:relative-normal-form}, the contribution of $(g,x):x\to g(x)$ is
\begin{equation}
\label{eq:index-coordinate}
\left([p_{g(x)}^{-1}(g,x)p_x]_{\ab},\ e_{\mathcal{T}_{g(x)}}-e_{\mathcal{T}_x}\right).
\end{equation}
Hence the index is obtained by summing the abelianized boundary returns and the transport vectors of the finitely many exceptional arrows of $g$.
\end{thm}

\begin{proof}
Let $U_g$ be the full bisection representing $g$. By Definition~\ref{def:index-map}, $I(g)=[1_{U_g}]$ in $H_1(\G)$. The injection $\iota$ in~\eqref{eq:relative} is induced by the quotient map from the chain complex of $\G$ to the relative chain complex of the pair
$(\G,\T)$. Thus $\iota I(g)$ is represented by the image of $1_{U_g}$ in the relative degree-one chain group.

By the relative chain description in Subsection~\ref{subsec:AFbd-homology}, this image is obtained by restricting $1_{U_g}$ to $U_g\setminus\T$. Since $U_g$ is the bisection of $g$, we have
\[
U_g\setminus\T=\{(g,x):x\in\Sigma(g)\}.
\]
The set $\Sigma(g)$ is finite by Lemma~\ref{lemma:finite-exceptional-support}. Moreover, distinct points of $\Sigma(g)$ give distinct arrows because their sources are distinct. Therefore, the restriction of $1_{U_g}$ to $\G\setminus\T$ is the finitely supported sum of the corresponding exceptional arrow symbols. Passing to relative homology gives
\[
\iota I(g)=\sum_{x\in\Sigma(g)}[(g,x)],
\]
which proves~\eqref{eq:index}.

Now fix the connectors used in Theorem~\ref{thm:relative-normal-form}. For $x\in\Sigma(g)$, let $\mathcal{T}_x$ and $\mathcal{T}_{g(x)}$ denote the $\T$-orbits containing
$x$ and $g(x)$, respectively. The relative normal form sends the class of the arrow $(g,x):x\to g(x)$ to
\[
\left(
 [p_{g(x)}^{-1}(g,x)p_x]_{\ab},
 e_{\mathcal{T}_{g(x)}}-e_{\mathcal{T}_x}
\right).
\]
This is exactly~\eqref{eq:index-coordinate}. Since the decomposition of Theorem~\ref{thm:relative-normal-form} is a homomorphism, summing these coordinates over $x\in\Sigma(g)$ gives the return and transport coordinates of $\iota I(g)$.

For completeness, the resulting relative class has zero defect. Indeed, $\partial_1(1_{U_g})=1_{\sg(U_g)}-1_{\rg(U_g)}=0$, since
$\sg(U_g)=\rg(U_g)=\Omega(\mathsf{B})$. Hence $1_{U_g}$ is an absolute $1$-cycle, and exactness of~\eqref{eq:relative} gives $\delta(\iota I(g))=0$, as required.
\end{proof}

The relative class $\iota(I(g))$ represented by the formula is independent of the chosen connectors. Its return coordinates transform as in Proposition~\ref{prop:connector-shear}, while its transport coordinates are unchanged. Isotropy and intra-tile returns contribute to the germ coordinates, while genuine motion between different infinite tiles contributes to the transport coordinates. Regular boundary connections can thus contribute nonzero transport to the index of an element of the full group even though their endpoint groups of germs are trivial.

\subsection{Geometric transport coordinates and boundary cuts}
\label{subsec:geometric-index}
For the geometric statements in this subsection, assume the hypotheses of Subsection~\ref{subsec:hypothesis-scope} under which the boundary quotient of Definition~\ref{def:boundary-quotient} is finite. We now express transport coordinates by signed counts of crossings across cuts in the boundary quotient. This is the analogue, for a finite boundary, of the picture with one border in \cite[Section~5.5]{nek22}.

Fix an exceptional orbit $\mathcal{O}$ and its boundary quotient $Q_{\mathcal{S}}(\mathcal{O})$. Choose a spanning tree $\mathscr{R}_{\mathcal{O}}$ in the underlying undirected graph after deleting loops, root it at the base infinite tile $\mathcal{T}_0$, and orient every tree edge away from the root. If $e$ is oriented from $\mathcal{T}_i$ to $\mathcal{T}_j$, put
\begin{equation}
\label{eq:tree-boundary}
\partial_{\mathcal{O}}e=e_{\mathcal{T}_j}-e_{\mathcal{T}_i}.
\end{equation}
Extending linearly gives a homomorphism
\begin{equation}
\label{eq:tree-boundary-map}
\partial_{\mathcal{O}}:C_1(\mathscr{R}_{\mathcal{O}};\mathbb{Z})\longrightarrow L_{\mathcal{O}}.
\end{equation}

\begin{lemma}
\label{lemma:tree-transport}
The map \eqref{eq:tree-boundary-map} is an isomorphism. For two infinite tiles $\mathcal{T}_i,\mathcal{T}_j$, its inverse sends $e_{\mathcal{T}_j}-e_{\mathcal{T}_i}$ to the signed edge chain of the unique tree path from $\mathcal{T}_i$ to $\mathcal{T}_j$.
\end{lemma}
\begin{proof}
The tree has $k_{\mathcal{O}}-1$ edges, and thus its first chain group and $L_{\mathcal{O}}$ are free abelian groups of the same rank. The boundary of the tree path from the root to $\mathcal{T}_i$ is $e_{\mathcal{T}_i}-e_{\mathcal{T}_0}$, and these elements form a basis of $L_{\mathcal{O}}$. Hence \eqref{eq:tree-boundary-map} is an isomorphism, and the path formula follows by subtraction.
\end{proof}

For an oriented tree edge $e$, let $B_e$ be the component not containing the root after $e$ is removed and let $A_e$ be its complement. Define
\begin{equation}
\label{eq:cut-functional}
\kappa_e(l)=\sum_{\mathcal{T}_i\in B_e} l_i,
\qquad l=\sum_i l_i e_{\mathcal{T}_i}\in L_{\mathcal{O}}.
\end{equation}
The value $\kappa_e(e_{\mathcal{T}_j}-e_{\mathcal{T}_i})$ is $1$ when the ordered pair $(\mathcal{T}_i,\mathcal{T}_j)$ crosses the cut $A_e\mid B_e$ in the direction of $e$, is $-1$ when it crosses in the opposite direction, and is $0$ otherwise. Equivalently, $\kappa_e(l)$ is the coefficient of $e$ in $\partial_{\mathcal{O}}^{-1}(l)$.

The connectors in the relative normal form of
Theorem~\ref{thm:relative-normal-form} can be chosen so that every boundary edge selected for the tree has identity normalized return. For a tree edge $e$ represented by a persistent boundary arrow $\eta_e:\xi_e\to\zeta_e$ from infinite tile $\mathcal{T}_i$ to infinite tile $\mathcal{T}_j$, let $a_{\xi_e}:x_i\to\xi_e$ and $a_{\zeta_e}:x_j\to\zeta_e$ be the unique AF arrows and put
\begin{equation}
\label{eq:tree-connector-arrow}
d_e=a_{\zeta_e}^{-1}\eta_e a_{\xi_e}:x_i\longrightarrow x_j.
\end{equation}
Starting with $c_0=(\Id,x_0)$, define $c_j=d_ec_i$ recursively along the rooted tree.

\keepwithstatement
\begin{prop}
\label{prop:tree-adapted-connectors}
With the connectors above, every exceptional generator arrow corresponding to an edge of the spanning tree has identity normalized return. Under the identification of Lemma~\ref{lemma:tree-transport}, its transport coordinate is that oriented edge. For a boundary edge outside the spanning tree, the transport coordinate is the signed chain along the tree path from its source tile to its range tile. Its return coordinate is the abelianized germ of the fundamental boundary cycle, conjugated to $x_0$ by the chosen connector.
\end{prop}
\begin{proof}
For the chosen arrow $\eta_e:\xi_e\to\zeta_e$, we have $p_{\xi_e}=a_{\xi_e}c_i$ and $p_{\zeta_e}=a_{\zeta_e}c_j$. Hence
\begin{equation*}
p_{\zeta_e}^{-1}\eta_ep_{\xi_e}
=c_j^{-1}a_{\zeta_e}^{-1}\eta_ea_{\xi_e}c_i
=c_j^{-1}d_ec_i=(\Id,x_0).
\end{equation*}
The second assertion follows from the relative normal form. For a boundary edge outside the spanning tree, Definition~\ref{def:spanning-tree-fundamental-cycle} closes the corresponding arrow by the tree path and the required arrows inside the tiles. Conjugating this isotropy germ to $x_0$ gives its normalized return.
\end{proof}

\subsection{The geometric index theorem}
For $x\in\Sigma(g)\cap\mathcal{O}$, let $\mathcal{T}_x$ and $\mathcal{T}_{g(x)}$ be the source and range infinite tiles. With connectors adapted to the tree, define the normalized return contribution
\begin{equation}
\label{eq:return-flux}
\operatorname{Ret}_{\mathcal{O}}(g)
=\sum_{x\in\Sigma(g)\cap\mathcal{O}}
[p_{g(x)}^{-1}(g,x)p_x]_{\ab}
\in H_{\mathcal{O}}^{\ab}
\end{equation}
and the chain representing the transport coordinate
\begin{equation}
\label{eq:transport-current}
J_{\mathcal{O}}(g)
=\sum_{x\in\Sigma(g)\cap\mathcal{O}}
\partial_{\mathcal{O}}^{-1}
(e_{\mathcal{T}_{g(x)}}-e_{\mathcal{T}_x})
\in C_1(\mathscr{R}_{\mathcal{O}};\mathbb{Z}).
\end{equation}
For an oriented tree edge $e$, its coefficient in $J_{\mathcal{O}}(g)$ is
\begin{align}
\langle e,J_{\mathcal{O}}(g)\rangle
={}&|\{x\in\Sigma(g)\cap\mathcal{O}:
\mathcal{T}_x\in A_e,\ \mathcal{T}_{g(x)}\in B_e\}| \notag\\
&-|\{x\in\Sigma(g)\cap\mathcal{O}:
\mathcal{T}_x\in B_e,\ \mathcal{T}_{g(x)}\in A_e\}|.
\label{eq:edge-flux-count}
\end{align}

\begin{thm}[Geometric index theorem]
\label{thm:geometric-index}
For every $g\in\F(\G)$, the image of its index in the relative normal form, written with the connectors above and the transport coordinates of Lemma~\ref{lemma:tree-transport}, is
\begin{equation}
\label{eq:geometric-index}
\iota I(g)=
\bigoplus_{\mathcal{O}}
\left(\operatorname{Ret}_{\mathcal{O}}(g),
J_{\mathcal{O}}(g)\right).
\end{equation}
Hence, the transport part of the index is the signed net migration across the cuts of the rooted tree chosen in the boundary quotient, and the remaining part is the sum of the normalized returns around the persistent boundary, taken after abelianization.
\end{thm}
\begin{proof}
By Theorem~\ref{thm:index-formula},
\[
\iota I(g)=\sum_{x\in\Sigma(g)}[(g,x)]
\]
in $H_1(\G,\T)$. Since $\Sigma(g)$ is finite, we may group these terms according to the exceptional $\G$-orbits that meet $\Sigma(g)$.

Fix such an orbit $\mathcal{O}$. Under the relative normal form of
Theorem~\ref{thm:relative-normal-form}, the class of an exceptional arrow
$(g,x)$ with $x\in\mathcal{O}$ has coordinates
\[
\left(
[p_{g(x)}^{-1}(g,x)p_x]_{\ab},
e_{\mathcal{T}_{g(x)}}-e_{\mathcal{T}_x}
\right).
\]
Therefore the return coordinate of the contribution from $\mathcal{O}$ is
\[
\sum_{x\in\Sigma(g)\cap\mathcal{O}}
[p_{g(x)}^{-1}(g,x)p_x]_{\ab}=\operatorname{Ret}_{\mathcal{O}}(g)
\]
by~\eqref{eq:return-flux}.

The transport coordinate is
\[
\sum_{x\in\Sigma(g)\cap\mathcal{O}}
\left(e_{\mathcal{T}_{g(x)}}e_{\mathcal{T}_x}\right)\in L_{\mathcal{O}}.
\]
By Lemma~\ref{lemma:tree-transport}, the inverse of
$\partial_{\mathcal{O}}$ sends this vector to
\[
\sum_{x\in\Sigma(g)\cap\mathcal{O}}
\partial_{\mathcal{O}}^{-1}
\left(e_{\mathcal{T}_{g(x)}}-e_{\mathcal{T}_x}\right)
=
J_{\mathcal{O}}(g).
\]
Thus the contribution of $\mathcal{O}$ to $\iota I(g)$ is
$\bigl(\operatorname{Ret}_{\mathcal{O}}(g),J_{\mathcal{O}}(g)\bigr)$.
Summing over the exceptional orbits gives~\eqref{eq:geometric-index}.

It remains to identify the coefficients of the transport chain. Let $e$ be
an oriented edge of $\mathscr{R}_{\mathcal{O}}$, and let
$A_e\mid B_e$ be the cut obtained by removing $e$, with $B_e$ the component
not containing the root. By~\eqref{eq:cut-functional}, the coefficient of
$e$ in
$\partial_{\mathcal{O}}^{-1}
(e_{\mathcal{T}_{g(x)}}-e_{\mathcal{T}_x})$
is $1$ when $\mathcal{T}_x\in A_e$ and
$\mathcal{T}_{g(x)}\in B_e$, is $-1$ when
$\mathcal{T}_x\in B_e$ and $\mathcal{T}_{g(x)}\in A_e$, and is $0$
otherwise. Summing over $x\in\Sigma(g)\cap\mathcal{O}$ gives exactly
\eqref{eq:edge-flux-count}. Hence the transport coordinate records the
signed net migration across every tree cut.

Finally, Proposition~\ref{prop:tree-adapted-connectors} shows that the
chosen tree edges have identity normalized return. Therefore, with these
connectors, the complementary coordinate records precisely the sum of the
normalized returns contributed by the exceptional arrows, after
abelianization. This proves the stated geometric interpretation.
\end{proof}

The tree and connectors supply coordinates for the intrinsic class $I(g)$. Changing the spanning tree changes the basis of $L_{\mathcal{O}}$, while changing the connectors transforms the return coordinate as in Proposition~\ref{prop:connector-shear}. Once transport vanishes, the return coordinate is independent of the connector system.

\begin{defn}[Boundary current]
\label{def:boundary-current}
With a rooted spanning tree and connectors chosen as in Proposition~\ref{prop:tree-adapted-connectors} on every exceptional orbit, the element
$\bigoplus_{\mathcal{O}}(\operatorname{Ret}_{\mathcal{O}}(g),J_{\mathcal{O}}(g))$
in~\eqref{eq:geometric-index} is called the \emph{boundary current} of $g$. Its two coordinates are the normalized-return contribution and the transport coordinate represented by $J_{\mathcal{O}}(g)$.
\end{defn}

\begin{cor}
\label{cor:zero-boundary-flux}
An element $g\in\F(\G)$ has zero index if and only if, for every exceptional orbit $\mathcal{O}$,
\begin{equation}
\label{eq:zero-flux-criterion}
\langle e,J_{\mathcal{O}}(g)\rangle=0
\quad\text{for each }e\in E(\mathscr{R}_{\mathcal{O}}),
\qquad
\operatorname{Ret}_{\mathcal{O}}(g)=0.
\end{equation}
The vanishing condition is independent of the chosen rooted tree and connectors.
\end{cor}
\begin{proof}
The injection $\iota$ and the direct-sum decomposition in Theorem~\ref{thm:relative-normal-form} reduce the assertion to vanishing of both coordinates. These coefficients are all zero exactly when $J_{\mathcal{O}}(g)=0$, since the oriented tree edges form a basis. Tree independence follows since this is equivalent to vanishing of the intrinsic transport vector. When that vector is zero, Proposition~\ref{prop:connector-shear} leaves the return coordinate unchanged.
\end{proof}

By Definition~\ref{def:boundary-current}, the right-hand side of \eqref{eq:geometric-index} is the boundary current of $g$. Every boundary current arising from an element of the full group has zero defect, since
\begin{equation}
\label{eq:boundary-current-zero-defect}
\delta(\iota I(g))=0\quad\text{in }H_0(\T).
\end{equation}
This equality holds in the dimension group. It does not require the ordinary graph boundary of $J_{\mathcal{O}}(g)$ to vanish at every vertex. For example, the generator of a minimal Cantor $\mathbb{Z}$-system has one unit of transport across its orbit cut.

\subsection{Word trajectories, traverses, and normalized returns}\label{subsec:word-traverses}
We now read the geometric index formula from a word trajectory. The notion of \emph{trajectory} used here comes from our earlier growth paper \cite[Definitions~3.10 and~4.1]{kua26}. A \emph{traverse} of a finite tile is a finite a path whose endpoints are boundary points of that tile and whose intermediate vertices are not boundary points. The endpoints may coincide. We call the trajectory of a traverse without endpoints an \emph{internal path} on a finite tile. A \emph{return word} in the growth argument records a boundary-entry step at a persistent boundary point $\xi$, an internal excursion in its infinite tile returning to that same $\xi$, and a boundary-exit step there. The intervening edges are internal to the tile. The trajectory of the whole word need not end where it started. This is a word with a specified occurrence, not an isotropy germ defined only by its product.

By contrast, a based boundary return in Definition~\ref{def:boundary-path-return} is a composable path closed at its chosen basepoint. The normalized return $p_y^{-1}\gamma p_x$ in Theorem~\ref{thm:relative-normal-form} is its evaluated isotropy germ after an arrow is closed by the chosen connectors. Such a path can involve several traverses and boundary crossings and may contain several return-word occurrences. 

Let $F=F_m\cdots F_1$ be a product of bisections $F_j\in\mathcal{S}$ and let $x\in\sg(F)$. Let $\gamma=(F,x)\in F$ be the germ with source $x$. Put $z_0=x$, $z_j=F_j(z_{j-1})$, and $\gamma_j=(F_j,z_{j-1})$. These trajectory vertices are independent of the basepoint used for normalized returns. The sequence
\begin{equation}
\label{eq:word-trajectory}
z_0,z_1,\ldots,z_m
\end{equation}
is the concrete trajectory of $F$ from $x$. Let
\begin{equation}
\label{eq:exceptional-skeleton-set}
E(F,x)=\{j: \gamma_j\notin\T\}.
\end{equation}
The ordered list $(\gamma_j)_{j\in E(F,x)}$ records the exceptional steps of the trajectory. The following decomposition separates these homological contributions from the intervening finite-level traverses.
\keepwithstatement
\begin{prop}
\label{prop:traverse-index-decomposition}
For the germ $\gamma=(F,x)\in F$,
\begin{equation}
\label{eq:exceptional-skeleton-class}
[\gamma]=\sum_{j\in E(F,x)}[\gamma_j]
\quad\text{in }H_1(\G,\T).
\end{equation}
For every sufficiently deep finite level $n$, each nonempty segment between consecutive exceptional steps is an internal path in one level-$n$ finite tile with boundary endpoints. Cutting at each intermediate boundary point gives a sequence of level-$n$ traverses. The initial segment before the first exceptional step and the final segment after the last may have an endpoint in the interior, so they need not be traverses. Refining the level can change the decomposition into traverses, but it does not change \eqref{eq:exceptional-skeleton-class} or the boundary current in Theorem~\ref{thm:geometric-index}.
\end{prop}
\begin{proof}
The defining relation in $H_1(\G,\T)$ gives $[\gamma]=\sum_{j=1}^m[\gamma_j]$. Every step with $\gamma_j\in\T$ vanishes, proving \eqref{eq:exceptional-skeleton-class}.

There are only finitely many steps of this trajectory belonging to $\T$. By the condition that tail arrows eventually become internal edges in Subsection~\ref{subsec:tile-inflation}, all of them are represented by internal $\mathcal{S}$-edges at one sufficiently deep level and at every subsequent level. It follows that each nonempty segment between consecutive exceptional steps is an internal path in one finite tile. The range of the preceding exceptional arrow and the source of the following one are persistent boundary points, hence their level-$n$ truncations are boundary points. Cutting the path at all intermediate boundary visits gives the stated traverses. The initial and final segments can start or end at the arbitrary endpoints of $\gamma$, so no boundary condition is asserted at those endpoints. At deeper levels, internal edges remain internal and boundary points may become internal, and thus consecutive traverses may merge. All such pieces still lie in $\T$ and vanish in the relative group. Therefore, the ordered exceptional steps and the resulting relative class are unchanged.
\end{proof}

Proposition~\ref{prop:traverse-index-decomposition} relates the index to the traverse method used in growth estimates. Traverses describe the internal segments of a trajectory between exceptional steps. Those internal arrows vanish in relative homology. The index records the signed transport of the exceptional steps across the boundary cuts and their normalized return germs after abelianization.

\begin{exmp}[The formula with one border]
\label{exmp:one-border-index}
Let $\tau$ be a minimal homeomorphism of a Cantor set $X$, put $\G=\operatorname{Germ}(\langle\tau\rangle\curvearrowright X)$, and let $\T$ be the AF groupoid obtained by cutting the orbit at $x_0$, as in \cite[Section~5.5]{nek22}. The exceptional orbit contains exactly the two $\T$-orbits
\begin{equation*}
\mathcal{O}_-=\{\tau^k(x_0):k\leq0\},
\qquad
\mathcal{O}_+=\{\tau^k(x_0):k>0\}.
\end{equation*}
The groups of germs are trivial and the boundary tree has one edge, oriented from $\mathcal{O}_-$ to $\mathcal{O}_+$. Hence Theorem~\ref{thm:geometric-index} becomes
\begin{align}
I(g)={}&|\{x\in\Sigma(g)\cap\mathcal{O}_-:g(x)\in\mathcal{O}_+\}|\notag\\
&-|\{x\in\Sigma(g)\cap\mathcal{O}_+:g(x)\in\mathcal{O}_-\}|.
\label{eq:one-border-flux}
\end{align}
Hence, $I(g)$ is the net migration across the border between $x_0$ and $\tau(x_0)$. In particular, $I(\tau)=1$, after choosing the displayed orientation. This is exactly the geometric interpretation described in \cite[Section~5.5]{nek22}.
\end{exmp}

\begin{exmp}[Pure return and mixed boundary motion]
\label{exmp:pure-return-mixed}
If $k_{\mathcal{O}}=1$, so that the orbital graph has only one infinite tile, the chosen spanning tree $\mathscr{R}_{\mathcal{O}}$ has no edges and the geometric formula reduces to the sum of the normalized returns after abelianization. For a mixed example, suppose two infinite tiles are joined by a chosen tree edge $\eta:\xi\to\zeta$ and by a second persistent edge $q:\xi'\to\zeta'$ with the same orientation between the tiles. Let $a:\xi\to\xi'$ and $b:\zeta'\to\zeta$ be the tail arrows. With connectors adapted to the tree, $[\eta]=(0,e)$ and $[q]=([h]_{\ab},e)$ for a return germ $h$. The closed path $\eta^{-1}bqa$ thus has coordinate $([h]_{\ab},0)$. This elementary model separates net transport from the return left after the same transport has been cancelled.
\end{exmp}

\subsection{Surjectivity under minimality}
When $\G$ is minimal and has comparison, the AH exact sequence of \cite[Corollary~6.14]{Li25} implies that the index map $I:\F(\G)\to H_1(\G)$ is surjective. In the present setting, these hypotheses follow from minimality of the AF core.

The distinction between the defect condition and index surjectivity is worth emphasizing. The equation $\delta(z)=0$ says only that the relative class $z\in H_1(\G,\T)$ comes from $H_1(\G)$. Surjectivity asks for more: the class must be realizable by a single compact open full bisection, rather than merely by an integral $1$-cycle. Zero defect does not by itself arrange the required local bisections with pairwise disjoint sources and ranges on the unit space. Under minimality, comparison gives such a realization by Li's theorem.

\keepwithstatement
\begin{lemma}
\label{lemma:AF-comparison}
If $\T$ is minimal and clopen subsets $A,B\subseteq \Omega(\mathsf{B})$ satisfy $\mu(A)<\mu(B)$ for every $\T$-invariant probability measure, there is a compact open bisection of $\T$ with source $A$ and range contained in $B$.
\end{lemma}
This is the AF case of the comparison argument in \cite[Subsection~5.5.2]{nek22}, using the matching lemma for elementary groupoids \cite[Lemma~5.5.3(1)]{nek22}. We repeat the proof here.
\begin{proof}
We use the usual argument with finite classes. Let $\mu_{n,x}$ be the uniform probability measure on the finite class $\mathfrak{T}_{\mathsf{B},n}x$. Every weak limit of measures $\mu_{n_j,x_j}$ with $n_j\to\infty$ is $\T$-invariant. Indeed, every fixed compact open bisection of $\T$ is contained in one stage, and at all later stages it bijects its source and range inside each finite class. If arbitrarily deep classes had at least as many points in $A$ as in $B$, a weak limit would contradict the strict inequality for invariant measures. It follows that at a sufficiently deep stage every class has fewer points in $A$ than in $B$. Refine the finitely many elementary multisections so that membership in $A$ and $B$ is constant on their levels. Matching the $A$-levels injectively into the $B$-levels produces the asserted bisection.
\end{proof}

\begin{cor}
\label{cor:index-surjective}
If $\T$ is minimal, then $\G$ is minimal, has comparison, and the index map $I:\F(\G)\to H_1(\G)$ is surjective.
\end{cor}

\begin{proof}
Minimality of $\T$ implies minimality of $\G$. By Proposition~\ref{prop:defect-infinitesimal}, the $\T$-invariant and $\G$-invariant probability measures coincide. Lemma~\ref{lemma:AF-comparison} gives comparison for the AF core. Hence, if compact open sets $U,V\subseteq \Omega(\mathsf{B})$ satisfy $\mu(U)<\mu(V)$ for every $\G$-invariant probability measure, the same inequalities hold for every $\T$-invariant measure, and a comparison bisection may be chosen inside $\T$. Hence, $\G$ has comparison. Since $\Omega(\mathsf{B})$ is a Cantor space, \cite[Corollary~6.14]{Li25} applies and gives surjectivity of the index map.
\end{proof}

\section{Almost finiteness and the kernel of the index map}
\label{sec:kernel}
The last section analyzed the index map from both algebraic and geometric perspectives. In particular, the surjectivity of the index map is governed by the defect map in the relative homology sequence: a relative class can be realized by an absolute homology class precisely when its defect vanishes. This section studies almost finiteness from the geometric side. The boundary data appearing in the index map also controls the construction of finite approximations, and the vanishing of the corresponding obstruction allows the exceptional boundary to be incorporated into these approximations.

We first characterize approximation of $\G$ by an elementary exhaustion of its fixed AF core. We then prove almost finiteness under minimality and compact generation, allowing elementary approximations outside the original core. Finite AF orbits distinguish these two problems. The section also gives a nonminimal obstruction, a quantitative finite-tile estimate, and the known kernel consequences.

We use the standard groupoid formulation of almost finiteness introduced by Matui in \cite{Matui12}. 

\begin{defn}\label{def:almost-finite}
The groupoid $\G$ is \emph{almost finite} if for every compact set $C\subseteq\G$ and every $\varepsilon>0$ there is an elementary subgroupoid $\mathfrak{K}\subseteq\G$ with $\mathfrak{K}^{(0)}=\Omega(\mathsf{B})$ such that
\begin{equation}
\label{eq:almost-finite-condition}
\frac{|C\mathfrak{K}_x\setminus\mathfrak{K}_x|}{|\mathfrak{K}_x|}<\varepsilon
\end{equation}
for every $x\in \Omega(\mathsf{B})$. Here $\mathfrak{K}_x=\{\kappa\in\mathfrak{K}:\sg(\kappa)=x\}$ is the source fiber at $x$, and
$C\mathfrak{K}_x=\{\gamma\kappa:\gamma\in C,\ \kappa\in\mathfrak{K}_x,\ \sg(\gamma)=\rg(\kappa)\}$.    
\end{defn}

 Since $\mathfrak{K}$ is principal, the range map identifies $\mathfrak{K}_x$ with the finite $\mathfrak{K}$-orbit of $x$. The numerator nevertheless counts arrows in $C\mathfrak{K}_x\setminus\mathfrak{K}_x$, not merely their distinct range points. This distinction matters when $\G$ has nontrivial isotropy. We use \cite[Definition~5.5.1]{nek22}, which treats totally disconnected groupoids with compact unit space. See also the original Hausdorff formulation \cite[Definition~6.2]{Matui12}.

\subsection{Uniform averages in the AF core}
We first work with an arbitrary increasing elementary exhaustion $\mathfrak{K}_1\subseteq\mathfrak{K}_2\subseteq\cdots$ of $\T$, with $\mathfrak{K}_n^{(0)}=\Omega(\mathsf{B})$. No finite family of generating bisections, hypothesis of bounded type, or minimality assumption is used in this subsection. Write $\mathfrak{K}_n(x)=\rg((\mathfrak{K}_n)_x)$ for its finite orbit and $M(\T)$ for the space of invariant Borel probability measures. For $f\in C(\Omega(\mathsf{B}),\mathbb{R})$, put
\begin{equation}
\label{eq:AF-averages}
P_nf(x)=\frac{1}{|\mathfrak{K}_n(x)|}\sum_{y\in\mathfrak{K}_n(x)}f(y),
\qquad q_n(f)=\max_{x\in \Omega(\mathsf{B})}P_nf(x).
\end{equation}
On each clopen level of a finite $\mathfrak{K}_n$-tower, the finite set $\mathfrak{K}_n(x)$ varies by the corresponding tower identifications. Hence $P_nf$ is continuous. \begin{lemma}
\label{lemma:uniform-AF-averaging}
The set $M(\T)$ is nonempty and compact. For every $f\in C(\Omega(\mathsf{B}),\mathbb{R})$,
\begin{equation}
\label{eq:AF-averages-limit}
q_n(f)\ \downarrow\ \max_{\mu\in M(\T)}\int_{\Omega(\mathsf{B})} f\,d\mu.
\end{equation}
\end{lemma}
\begin{proof}
Every $\mathfrak{K}_{n+1}$-orbit is a disjoint union of $\mathfrak{K}_n$-orbits. Its average is therefore a convex combination of the averages over those smaller orbits, and thus $q_{n+1}(f)\leq q_n(f)$.

For $x\in\Omega(\mathsf{B})$, let $\mu_{n,x}$ be the uniform probability measure on $\mathfrak{K}_n(x)$. Choose arbitrary $x_n$ and a weakly convergent subsequence of $\mu_{n,x_n}$. Such a subsequence exists since the space of probability measures on the compact unit space is weakly compact. Its limit is $\T$-invariant. Indeed, if $V$ is a compact open $\T$-bisection, then $V\subseteq\mathfrak{K}_m$ for some $m$. For $n\geq m$, the source and range of $V$ meet every $\mathfrak{K}_n$-orbit in equally many points. Hence $\mu_{n,x}(\sg(V))=\mu_{n,x}(\rg(V))$, and the equality passes to a weak limit since these sets are clopen. The same argument applied to clopen restrictions of $V$ gives invariance. Thus $M(\T)$ is nonempty. It is compact since invariance is a closed condition in the compact space of probability measures.

For $\mu\in M(\T)$, invariance on each finite tower gives $\int P_nf\,d\mu=\int f\,d\mu$, and hence $q_n(f)\geq\int f\,d\mu$. Conversely, choose $x_n$ attaining $q_n(f)$ and take a convergent subsequence of $\mu_{n,x_n}$. Its limit $\mu$ lies in $M(\T)$, while $\int f\,d\mu_{n,x_n}=q_n(f)$. Since $(q_n(f))$ is decreasing, this subsequence has the same limit as the full sequence. Therefore $\lim_nq_n(f)=\int f\,d\mu$, which gives \eqref{eq:AF-averages-limit}.
\end{proof}

\begin{lemma}
\label{lemma:exceptional-neighborhoods}
Let $E\subset \Omega(\mathsf{B})$ be finite, and suppose that every point of $E$ has an infinite $\T$-orbit. For every $\eta>0$, there is a clopen neighborhood $A$ of $E$ and an integer $N$ such that
\begin{equation}
\label{eq:uniform-small-AF-density}
\frac{|A\cap\mathfrak{K}_n(x)|}{|\mathfrak{K}_n(x)|}<\eta
\qquad(x\in \Omega(\mathsf{B}),\ n\geq N).
\end{equation}
Moreover, for any positive integer $d$, $A$ can be chosen as the common source of $d$ compact open $\T$-bisections with pairwise disjoint ranges. This gives the non-strict bound $1/d$ in \eqref{eq:uniform-small-AF-density} once these bisections belong to $\mathfrak{K}_n$.
\end{lemma}
\begin{proof}
An invariant probability measure assigns equal masses to points in one $\T$-orbit. Indeed, if $\alpha:e\to e'$ is a $\T$-arrow, choose a compact open $\T$-bisection $V\ni\alpha$ and a decreasing clopen neighborhood basis $B_j\subseteq\sg(V)$ at $e$. Invariance gives $\mu(B_j)=\mu(\theta_V(B_j))$, and continuity from above gives $\mu(\{e\})=\mu(\{e'\})$. Since the orbit of every point of $E$ is infinite, every $\mu\in M(\T)$ is therefore nonatomic at $E$. Choose decreasing clopen neighborhoods $A_j$ with intersection $E$. The continuous functions $\mu\mapsto\mu(A_j)$ on the compact space $M(\T)$ decrease pointwise to zero. By Dini's theorem the convergence is uniform, and thus some $A_j$ has $\sup_{\mu\in M(\T)}\mu(A_j)<\eta$. Lemma~\ref{lemma:uniform-AF-averaging}, applied to $1_{A_j}$, proves \eqref{eq:uniform-small-AF-density}.

For the additional assertion, start with a clopen neighborhood $A_0$ satisfying the density bound above. For each $e\in E$, choose $d$ arrows in $\T$ with source $e$, so that all their range points, for all $e$, are distinct. This is possible by choosing finitely many points successively in infinite orbits. Extend the arrows to compact open bisections, choose pairwise disjoint neighborhoods of their range points, and shrink their sources to a common clopen neighborhood $A_e\subseteq A_0$ of $e$. The sets $A_e$ may also be chosen pairwise disjoint. Taking the union of the $k$th bisections over $e\in E$ gives $d$ bisections with common source $A=\bigcup_eA_e\subseteq A_0$ and disjoint ranges. Thus $A$ retains the strict density bound already obtained. In every sufficiently deep $\mathfrak{K}_n$-orbit, the images of $A\cap\mathfrak{K}_n(x)$ under these $d$ bisections are disjoint subsets of that same orbit. It follows that $d|A\cap\mathfrak{K}_n(x)|\leq|\mathfrak{K}_n(x)|$.
\end{proof}

The construction of bisections with a common source and disjoint ranges follows Phillips's argument. Compare \cite[Definition~2.1 and Lemma~2.5]{Phillips05}. Small exceptional neighborhoods and localization of the remaining compact part inside an AF subgroupoid already occur in \cite[Lemmas~2.5 and~2.7]{Phillips05}. Applied to a discrete complement, these principles give the approximation criterion for source fibers below. Phillips's almost AF condition and almost finiteness are different definitions. The estimate below is what establishes the latter here.

\subsection{An exact criterion for the fixed AF exhaustion}
For a compact subset $C\subseteq\G$, put
\begin{equation}
\label{eq:AF-escape-ratio}
\varepsilon_n(C)=\sup_{x\in \Omega(\mathsf{B})}
\frac{|C(\mathfrak{K}_n)_x\setminus(\mathfrak{K}_n)_x|}{|(\mathfrak{K}_n)_x|}.
\end{equation}
This quantity measures approximation by the fixed AF exhaustion. Other elementary subgroupoids may still satisfy Definition~\ref{def:almost-finite}.

\begin{thm}[AF-core approximation criterion]
\label{thm:AF-exhaustion-criterion}
Let $\T\subseteq\G$ be an AF-by-discrete inclusion on a Cantor unit space, and let $\mathfrak{K}_1\subseteq\mathfrak{K}_2\subseteq\cdots$ be an increasing elementary exhaustion of $\T$. The following conditions are equivalent:
\begin{enumerate}[label=\textup{(\roman*)}]
\item For every compact $C\subseteq\G$, $\varepsilon_n(C)\to0$.
\item The source of every arrow in $\G\setminus\T$ has an infinite $\T$-orbit.
\item Every finite $\T$-orbit is a $\G$-orbit, and the reduction of $\G$ to it is principal.
\end{enumerate}
When these conditions hold, every increasing elementary exhaustion of $\T$ supplies almost-finite approximations of $\G$.
\end{thm}
Neither effectiveness nor a Hausdorff arrow space is needed for this statement.
\begin{proof}
Assume (ii), and fix a compact set $C\subseteq\G$ and $\varepsilon>0$. Cover $C$ by compact open bisections $U_1,\ldots,U_q$. For each $i$, the set $U_i\setminus\T$ is a compact subset of the discrete complement, hence finite. Therefore
$E=\bigcup_i\sg(U_i\setminus\T)$ is finite. By (ii), every point of $E$ has an infinite $\T$-orbit. Choose $\eta<\varepsilon/q$ and, by Lemma~\ref{lemma:exceptional-neighborhoods}, a clopen neighborhood $A$ of $E$ whose density in every sufficiently deep $\mathfrak{K}_n$-orbit is smaller than $\eta$.

Outside $A$ there are no exceptional arrows in any $U_i$: the restriction $U_i|_{\Omega(\mathsf{B})\setminus A}$ is a compact subset of $\T$. Since $\T=\bigcup_n\mathfrak{K}_n$ and the exhaustion is increasing and open, all these finitely many compact restrictions lie in one stage $\mathfrak{K}_N$. Thus, for $n\geq N$, a product $uk$ with $u\in U_i$ and $k\in(\mathfrak{K}_n)_x$ can leave $(\mathfrak{K}_n)_x$ only when $\rg(k)\in A$. The bisection property gives at most one product for each such $k$, and principality of $\mathfrak{K}_n$ identifies the arrows in $(\mathfrak{K}_n)_x$ with the points of $\mathfrak{K}_n(x)$. Hence
\begin{equation}
\label{eq:finite-exception-escape-bound}
|C(\mathfrak{K}_n)_x\setminus(\mathfrak{K}_n)_x|
\leq\sum_{i=1}^q|A\cap\mathfrak{K}_n(x)|
=q|A\cap\mathfrak{K}_n(x)|.
\end{equation}
After increasing $N$ to include the density bound, division by $|(\mathfrak{K}_n)_x|=|\mathfrak{K}_n(x)|$ gives $\varepsilon_n(C)\leq q\eta<\varepsilon$. Distinct arrows in $C(\mathfrak{K}_n)_x\setminus(\mathfrak{K}_n)_x$ are counted separately even when their range points coincide. This proves (i).

Now suppose that (ii) fails. Let $\gamma\notin\T$ have source $x$ in a finite $\T$-orbit of cardinality $m$, and choose a compact open bisection $U$ containing $\gamma$. The unit arrow at $x$ lies in every $\mathfrak{K}_n$, whereas $\gamma$ lies in none of them. Therefore
\begin{equation}
\label{eq:finite-core-orbit-obstruction}
\varepsilon_n(U)\geq\frac1{|(\mathfrak{K}_n)_x|}\geq\frac1m
\end{equation}
for every $n$, and thus (i) fails. This proves (i)$\Leftrightarrow$(ii).

It remains to compare (ii) and (iii). Let $Y$ be a finite $\T$-orbit. If (ii) holds, every arrow of $\G$ with source in $Y$ already lies in $\T$, and thus it also has range in $Y$. Hence, $Y$ is a $\G$-orbit and $\G|_Y=\T|_Y$ is principal. Conversely, suppose (iii) holds and let $\gamma:x\to y$ have $x$ in a finite $\T$-orbit $Y$. Then $y\in Y$. Since $\T|_Y$ is transitive and principal, it contains a unique arrow from $x$ to $y$. Principality of $\G|_Y$ forces $\gamma$ to be that arrow, and thus $\gamma\in\T$. Thus (ii) holds. The proof of (ii)$\Rightarrow$(i) used only that the chosen sequence is an increasing elementary exhaustion of $\T$, which proves the final assertion for every such exhaustion.
\end{proof}

\begin{cor}
\label{cor:infinite-AF-orbits-almost-finite}
Every AF-by-discrete groupoid with infinite $\T$-orbits is almost finite. In particular, this holds when the AF core is minimal on a Cantor space.
\end{cor}
No hypothesis of bounded type, finite generation, bounded rank, or finite isotropy is required.

Here the criterion applies to every AF exhaustion, since no exceptional arrow can be based on a finite AF orbit.

\begin{rmk}
\label{rmk:thin-generalization}
The sufficiency argument also works without discreteness if, for each compact open bisection $U$, the compact set $\sg(U\setminus\T)$ has measure zero for every $\mu\in M(\T)$. Shrinking clopen neighborhoods and Lemma~\ref{lemma:uniform-AF-averaging} again give the density bound. In the discrete case, this measure-zero condition is equivalent to all exceptional sources having infinite $\T$-orbits: a finite $\T$-orbit carries its uniform invariant probability measure. Thus the estimate uses the density of neighborhoods of exceptional sources in the finite AF orbits, not a uniform bound on the number of boundary points on finite tiles.
\end{rmk}

\subsection{Minimal groupoids with a possibly nonminimal AF core}
The preceding result assumes infinite orbits for the chosen core. We now remove that requirement when $\G$ is minimal and compactly generated. Minimality first restricts the AF orbit closures to the following two possibilities.
\keepwithstatement
\begin{lemma}
\label{lemma:AF-orbit-closures}
Let $\T\subseteq\G$ be an AF-by-discrete inclusion, and suppose that $\G$ is minimal. Every proper closed $\T$-invariant subset of $\Omega(\mathsf{B})$ is finite. Consequently, every $\T$-orbit is either finite or dense in $\Omega(\mathsf{B})$.
\end{lemma}
\begin{proof}
Let $Y$ be a closed $\T$-invariant subset and let $Y'$ be its set of accumulation points. We claim that $Y'$ is $\G$-invariant. Given $x\in Y'$ and an arrow $\gamma$ with source $x$, choose a compact open bisection $U$ containing $\gamma$. There are distinct points $x_j\in Y\cap\sg(U)$ tending to $x$. The AF-by-discrete hypothesis is used here: $U\setminus\T$ is finite, and thus after discarding finitely many terms the points $x_j$ avoid $\sg(U\setminus\T)$. The corresponding arrows of $U$ therefore lie in $\T$, and $\T$-invariance of $Y$ implies that their range points lie in $Y$. These range points are distinct and converge to $\rg(\gamma)$, and thus $\rg(\gamma)\in Y'$. Applying the same argument to inverse arrows proves that $Y'$ is $\G$-invariant.

If $Y$ is infinite, compactness gives $Y'\ne\varnothing$, and minimality of $\G$ gives $Y'=\Omega(\mathsf{B})$. Hence $Y=\Omega(\mathsf{B})$. The closure of an infinite $\T$-orbit is a closed $\T$-invariant set and cannot be finite, and thus it is the whole unit space.
\end{proof}

\begin{lemma}
\label{lemma:finitely-many-finite-AF-orbits}
Let $\T\subseteq\G$ be an AF-by-discrete inclusion, and suppose that $\G$ is minimal and compactly generated. Then the union $Z$ of the finite $\T$-orbits is finite. There exists a nonempty clopen set $A\subseteq \Omega(\mathsf{B})\setminus Z$ such that $\T|_A$ is a minimal AF groupoid and $A$ is an orbit transversal (Definition~\ref{def:minimality-comparison}) for $\G$.
\end{lemma}
\begin{proof}
Choose a finite generating family $\mathcal{S}$ of compact open bisections that is closed under inverses. Its exceptional source set $E=\bigcup_{F\in\mathcal{S}}\sg(F\setminus\T)$ is finite by Lemma~\ref{lemma:finite-exceptional-support}. Every finite $\T$-orbit $Y$ meets $E$. Indeed, if $Y\cap E=\varnothing$, then every generator arrow whose source lies in $Y$ belongs to $\T$ and therefore has range in $Y$. Since $\mathcal{S}$ generates $\G$ and is closed under inverses, no $\G$-arrow can leave $Y$. Thus $Y$ would be a finite $\G$-orbit, contradicting minimality on the Cantor unit space. Distinct finite $\T$-orbits are disjoint, and thus there are at most $|E|$ of them, and their union $Z$ is finite.

Choose a nonempty clopen set $A$ disjoint from $Z$. Every $\T$-orbit meeting $A$ is infinite and thus dense by Lemma~\ref{lemma:AF-orbit-closures}. Its intersection with $A$ is dense in $A$, and thus $\T|_A$ is minimal. It is AF, being exhausted by the elementary reductions $\mathfrak{K}_n|_A$. Finally, every nonempty open subset of a minimal groupoid's unit space is an orbit transversal.
\end{proof}

\begin{lemma}
\label{lemma:full-clopen-AF-lift}
Let $A\subset \Omega(\mathsf{B})$ be clopen and an orbit transversal for $\G$. If $\G|_A$ is almost finite, then $\G$ is almost finite.
\end{lemma}
\begin{proof}
Since $A$ is an orbit transversal, every point of the unit space is the source of an arrow whose range lies in $A$. Compact open bisections around such arrows give an open cover of the compact unit space. After taking a finite subcover, refining the sources to a clopen partition, and restricting the bisections, we obtain
$\Omega(\mathsf{B})=P_0\sqcup\cdots\sqcup P_{r-1}$, with $P_0=A$, and compact open bisections $V_i$ such that $\sg(V_i)=P_i$ and $\rg(V_i)\subseteq A$. Take $V_0=1_A$. For $x\in P_i$, write $v_x$ for the unique arrow of $V_i$ with source $x$, and put $q(x)=\rg(v_x)$. The map $q:\Omega(\mathsf{B})\to A$ is a local homeomorphism and every fiber has between one and $r$ points.

For $0\leq i,j<r$, let
$\G_{ij}=\{\gamma\in\G:\sg(\gamma)\in P_j,\ \rg(\gamma)\in P_i\}$.
This is a clopen subset of the arrow space. The map
\begin{equation*}
N_{ij}(\gamma)=v_{\rg(\gamma)}\gamma v_{\sg(\gamma)}^{-1},
\end{equation*}
is a homeomorphism from $\G_{ij}$ onto the arrows of $\G|_A$ with source in $\rg(V_j)$ and range in $\rg(V_i)$. Its inverse sends $\eta:q(x)\to q(y)$ to $v_y^{-1}\eta v_x$. In particular, for compact $C\subseteq\G$ the set
\begin{equation*}
C_A=\bigcup_{i,j}N_{ij}(C\cap\G_{ij})
      =\bigcup_{i,j}V_iCV_j^{-1}
\end{equation*}
is compact even when the arrow space of $\G$ is non-Hausdorff.

Choose an elementary subgroupoid $\mathfrak{L}\subseteq\G|_A$ satisfying the almost-finite estimate for $C_A$ with tolerance $\varepsilon/r$, and put
\begin{equation}
\label{eq:lifted-elementary-groupoid}
\mathfrak{K}=\bigcup_{i,j}V_i^{-1}\mathfrak{L} V_j.
\end{equation}
Equivalently, $\mathfrak{K}\cap\G_{ij}=N_{ij}^{-1}(\mathfrak{L})$. The maps $N_{ij}$ therefore show that $\mathfrak{K}$ is compact and open, is closed under composition and inversion, contains all units, and is principal. To see elementariness explicitly, take a finite multisection decomposition of $\mathfrak{L}$ and refine the base of each tower so that, on every tower level, membership in each clopen set $\rg(V_i)$ is constant. If $D_u$ is a refined level and $i$ is one of the indices with $D_u\subseteq\rg(V_i)$, then $P_i\cap q^{-1}(D_u)$ is a lifted level. The bisections $V_i^{-1}L_{uv}V_j$ between these lifted levels form a multisection. The finitely many refined towers cover $\mathfrak{K}$. Thus $\mathfrak{K}$ is elementary. Its orbits are precisely the finite sets obtained by taking the $q$-fibers over the finite $\mathfrak{L}$-orbits.

More precisely, principality of $\mathfrak{L}$ gives
$|\mathfrak{K}_x|=\sum_{z\in\mathfrak{L}(q(x))}|q^{-1}(z)|$:
for each $z$ in the $\mathfrak{L}$-orbit of $q(x)$ there is one arrow of $\mathfrak{L}$ from $q(x)$ to $z$, and its lifts with source $x$ are indexed by the points of $q^{-1}(z)$. Since every fiber of $q$ has between one and $r$ points,
$|\mathfrak{L}_{q(x)}|\leq|\mathfrak{K}_x|\leq r|\mathfrak{L}_{q(x)}|$.
If $\gamma\in C\mathfrak{K}_x\setminus\mathfrak{K}_x$, write its range as $y$. Normalizing the factorization of $\gamma$ through an arrow of $\mathfrak{K}_x$ gives
$v_y\gamma v_x^{-1}\in C_A\mathfrak{L}_{q(x)}\setminus\mathfrak{L}_{q(x)}$. Conversely, for fixed source $x$, a normalized arrow with range $z\in A$ has at most $|q^{-1}(z)|\leq r$ lifts, one for each possible range point of the lift. Consequently,
\begin{equation}
\label{eq:clopen-lift-escape}
\frac{|C\mathfrak{K}_x\setminus\mathfrak{K}_x|}{|\mathfrak{K}_x|}
\leq r\frac{|C_A\mathfrak{L}_{q(x)}\setminus\mathfrak{L}_{q(x)}|}{|\mathfrak{L}_{q(x)}|}
<\varepsilon.
\end{equation}
\end{proof}

\begin{thm}[Minimal almost-finiteness theorem]
\label{thm:minimal-AF-by-discrete-almost-finite}
Every compactly generated minimal AF-by-discrete groupoid with Cantor unit space is almost finite.
\end{thm}
The original AF core may be nonminimal, and its finite orbits are allowed.
\begin{proof}
Choose $A$ as in Lemma~\ref{lemma:finitely-many-finite-AF-orbits}. The inclusion $\T|_A\subseteq\G|_A$ is an open AF inclusion with discrete complement. The AF core $\T|_A$ is minimal, and thus Corollary~\ref{cor:infinite-AF-orbits-almost-finite} makes $\G|_A$ almost finite. Lemma~\ref{lemma:full-clopen-AF-lift} recovers the elementary approximations on the whole $\Omega(\mathsf{B})$.
\end{proof}

The original exhaustion of $\T$ need not approximate $\G$: an exceptional arrow based at a finite $\T$-orbit gives the uniform obstruction \eqref{eq:finite-core-orbit-obstruction}. Compact generation is used only to make the union of finite $\T$-orbits finite. The theorem also holds whenever a full clopen set meeting only infinite $\T$-orbits is independently available.

\begin{cor}
\label{cor:minimal-core-replacement}
Under the hypotheses of Theorem~\ref{thm:minimal-AF-by-discrete-almost-finite}, $\G$ admits an open minimal AF subgroupoid $\mathfrak{T}'$ with discrete complement.
\end{cor}
The replacement core may fail to contain the original $\T$.
\begin{proof}
Use the full clopen set $A$ and the bisections $V_i$ constructed above, and put $\mathfrak{L}_n=\mathfrak{K}_n|_A$. The increasing elementary subgroupoids
$\mathfrak{H}_n=\bigcup_{i,j}V_i^{-1}\mathfrak{L}_nV_j$
from \eqref{eq:lifted-elementary-groupoid} have union
$\mathfrak{T}'=\bigcup_{i,j}V_i^{-1}(\T|_A)V_j$.
Their normalization shows that the $\mathfrak{T}'$-orbit of $x$ is $q^{-1}((\T|_A)(q(x)))$. Since $q$ is a local homeomorphism and $\T|_A$ is minimal, these orbits are dense.

On the open set of arrows with source in $P_j$ and range in $P_i$, normalization by $V_i$ and $V_j^{-1}$ is a homeomorphism onto the arrows with source in $\rg(V_j)$ and range in $\rg(V_i)$. It takes the complement of $\mathfrak{T}'$ to the complement of $\T|_A$. The latter is discrete, and thus $\G\setminus\mathfrak{T}'$ is discrete. 
\end{proof}

Changing the AF core is not merely changing boundary generators or connectors. The relative group and defect in the dimension group belong to the chosen pair $(\G,\T)$ and can change under this replacement. The groupoid homology and index remain those of $\G$. The original core is retained in the preceding boundary calculations. The replacement gives a minimal AF core that can be used to obtain almost finiteness and comparison.

\subsection{An unrestricted counterexample and its index kernel}
A finite $\G$-orbit with nontrivial isotropy obstructs almost finiteness. Such orbits, however, are allowed by AF-by-discreteness.

\keepwithstatement
\begin{prop}
\label{prop:finite-orbit-isotropy-obstruction}
A groupoid with a finite orbit containing a nontrivial isotropy arrow is not almost finite.
\end{prop}
\begin{proof}
Let $\gamma\ne(\Id,x)$ be an isotropy arrow at a point in an orbit of cardinality $m<\infty$, and choose a compact open bisection $U$ containing it. For every elementary subgroupoid $\mathfrak{K}$ with unit space $\Omega(\mathsf{B})$, principality gives $\gamma\notin\mathfrak{K}$ and $|\mathfrak{K}_x|\leq m$. Since $(\Id,x)\in\mathfrak{K}$, we have $\gamma\in U\mathfrak{K}_x\setminus\mathfrak{K}_x$, and thus the ratio in \eqref{eq:almost-finite-condition} is at least $1/m$. This excludes all elementary approximations, not just the stages of the chosen AF core.
\end{proof}

\begin{exmp}[One fixed point and a discrete exceptional arrow]
\label{exmp:AF-by-discrete-not-almost-finite}
Let $X=\{0,1\}^{\omega}$ and $\xi=0^\omega$. On the clopen annuli $U_j=0^j1\{0,1\}^{\omega}$, $j\geq0$, define
\begin{equation}
\label{eq:fixed-point-involution}
b(0^j1aw)=0^j1(1-a)w\quad(a\in\{0,1\}),
\qquad b(\xi)=\xi.
\end{equation}
This is a continuous involution. Continuity at $\xi$ follows since each annulus is preserved. Its only fixed point is $\xi$. The germ $(b,\xi)$ is nontrivial since every neighborhood of $\xi$ meets annuli on which $b$ moves points. Hence the set of fixed points of the nontrivial group element has empty interior, and thus the transformation groupoid $\G=\langle b\rangle\ltimes X$ is Hausdorff and effective. No nontrivial group element has the identity germ at any point, and hence the transformation groupoid agrees with the groupoid of germs. Define
\begin{equation}
\label{eq:counterexample-core}
\mathfrak{T}=X\ \cup\ \{(b,x):x\ne\xi\},
\qquad
\mathfrak{K}_n=X\ \cup\ \{(b,x):x\in\textstyle\bigcup_{j<n}U_j\}.
\end{equation}
Here $X$ denotes the unit arrows. The subgroupoids $\mathfrak{K}_n$ are compact open and principal, with classes of size one or two. Hence, every $\mathfrak{T}$-orbit has cardinality at most two. Their union is the open AF subgroupoid $\mathfrak{T}$, whose complement is the singleton $\{(b,\xi)\}$.

We realize $\mathfrak{T}$ as the tail groupoid of a Bratteli diagram $\mathsf{B}$ on the same Cantor space. Write a point of $U_j$ as $0^j1aw$, with $a\in\{0,1\}$ and $w\in\{0,1\}^{\omega}$. Take an infinite spine, attach two exit edges with the same source and range at level $j+1$, and, after the exit, attach a binary tree whose distinct branches never merge. The spine path is $\xi$, while the two exit edges encode $a$ and the subsequent branch encodes $w$. Two paths that leave the spine are tail equivalent exactly when they have the same $j$ and $w$ coordinates, which is precisely the orbit relation of $\mathfrak{T}$. The basic tail bisections correspond to restrictions of $b$ on these annuli and to unit bisections. The identification also respects the groupoid topology. We henceforth identify $X$ with $\Omega(\mathsf{B})$ and $\mathfrak{T}$ with $\T$.

Nevertheless, $\G$ is not almost finite: its singleton orbit $\{\xi\}$ has isotropy $C_2$, and thus Proposition~\ref{prop:finite-orbit-isotropy-obstruction} applies. In fact, the ratio for the full bisection $U_b$ equals one at $\xi$ for every elementary subgroupoid.
\end{exmp}

This example shows that failure of almost finiteness can coexist with $\ker I=\Sym(\G)$.
\keepwithstatement
\begin{prop}
\label{prop:counterexample-kernel}
For Example~\ref{exmp:AF-by-discrete-not-almost-finite}, $H_1(\G)\cong C_2$, the index is the germ homomorphism at $\xi$, and $\ker I=\Sym(\G)$.
\end{prop}
\begin{proof}
There is only one exceptional arrow, $\gamma=(b,\xi)$, and its square is the unit. Nekrashevych's relative arrow presentation therefore gives $H_1(\G,\T)\cong C_2$, generated by $[\gamma]$. Since the defect map is a homomorphism, $2\delta[\gamma]=0$. The dimension group $H_0(\T)$ associated with $\mathsf{B}$ is torsion-free, and thus $\delta[\gamma]=0$. As $H_1(\T)=0$, the relative exact sequence identifies $H_1(\G)$ with this $C_2$. The geometric index formula then identifies $I(g)$ with the germ of $g$ at $\xi$.

If $I(g)=0$, then $g$ is the identity on some neighborhood of $\xi$. The set $B$ on which $g$ acts as $b$ is then a clopen $b$-invariant subset of a finite union of the annuli $U_j$. To see invariance, on each two-point orbit $\{x,bx\}$ a bijection locally choosing between the identity and $b$ either fixes both points or exchanges them. Let $D$ be the union of the halves $0^j10\{0,1\}^{\omega}$ of those finitely many annuli. The restriction of the bisection of $b$ to $B\cap D$ has disjoint source and range, whose union is $B$. Its associated dynamical transposition is precisely $g$. Conversely, every dynamical transposition is the identity near $\xi$: it cannot move $\xi$ out of its singleton orbit, and its moved set is clopen. Hence, $\Sym(\G)\subseteq\ker I$ and equality follows.
\end{proof}

Thus AF-by-discreteness alone gives neither almost finiteness nor an obstruction to $\ker I=\Sym(\G)$ when almost finiteness fails. The singleton orbit makes the example nonminimal.

\subsection{A quantitative boundary F\o lner criterion}
The preceding qualitative theorems make a boundary-to-volume hypothesis unnecessary when the AF core is minimal, or when $\G$ is minimal and compactly generated. The next criterion has a different purpose: it shows directly that the \emph{original finite-level AF stages} form an almost-finite approximation and gives an explicit escape bound in terms of the tile geometry. For systems of bounded type with infinite orbits, the resulting almost-finiteness statement is already \cite[Proposition~5.5.2]{nek22}.

For every level $n$, let $\mathcal{T}_{v,n}$ range over the finite tiles and put
\begin{equation}
\label{eq:tile-boundary-volume}
\beta_n=\max_v|\partial\mathcal{T}_{v,n}|,
\qquad
m_n=\min_v|V(\mathcal{T}_{v,n})|.
\end{equation}
Here $\partial\mathcal{T}_{v,n}$ is the set of boundary points. Let $\mathfrak{T}_{\mathsf{B},n}$ be the level-$n$ elementary subgroupoid of the AF core. Since $\T$ is open in $\G$, each $\mathfrak{T}_{\mathsf{B},n}$ is a compact open elementary subgroupoid of $\G$ with unit space $\Omega(\mathsf{B})$. If the level-$n$ prefix of $x$ belongs to $\mathcal{T}_{v,n}$, then the Cayley graph of $\mathfrak{T}_{\mathsf{B},n}$ at $x$ is $\mathcal{T}_{v,n}^{\circ}$. Deleting the boundary edges does not change the vertex set, and principality identifies the source fiber with that vertex set. Hence, $|(\mathfrak{T}_{\mathsf{B},n})_x|=|V(\mathcal{T}_{v,n})|$. The F\o lner estimate below uses both structures on this common vertex set: $\mathcal{T}_{v,n}^{\circ}$ records the elementary AF arrows, while the boundary edges of $\mathcal{T}_{v,n}$ record where products of the generating bisections can leave that elementary subgroupoid.

\begin{lemma}
\label{lemma:uniform-tile-growth}
Suppose that $\mathsf{B}$ is simple. Then
\begin{equation}
\label{eq:min-tile-growth}
m_n\longrightarrow\infty.
\end{equation}
\end{lemma}

\begin{proof}
Let $h_n(v)$ be the number of finite paths ending at the vertex corresponding to the finite tile $\mathcal{T}_{v,n}$. Since $\Omega(\mathsf{B})$ is infinite, $|\Omega_n(\mathsf{B})|=\sum_v h_n(v)$ tends to infinity. Fix $M>0$ and choose a level $k$ with $|\Omega_k(\mathsf{B})|>M$. Simplicity gives $q>k$ such that every vertex in $V_{q+1}$ receives a path from every vertex in $V_{k+1}$. Consequently, every level-$q$ finite tile contains, for each level-$k$ prefix, at least one descendant, and thus
$|V(\mathcal{T}_{v,q})|\geq |\Omega_k(\mathsf{B})|>M$
for every $v$. At subsequent levels, every vertex has an incoming edge, and thus the minimum cardinality of the finite tiles cannot decrease. Since $M$ was arbitrary, \eqref{eq:min-tile-growth} follows.
\end{proof}

\begin{thm}
\label{thm:boundary-folner}
Assume the tile presentation compatible with the AF core from Subsection~\ref{subsec:hypothesis-scope}, with $\G$ generated by a fixed finite family $\mathcal{S}\subseteq\operatorname{Bis}_c(\G)$ that is closed under inverses, and suppose that
\begin{equation}
\label{eq:boundary-folner-ratio}
\frac{\beta_n}{m_n}\longrightarrow0.
\end{equation}
Then $\G$ is almost finite in the sense of \eqref{eq:almost-finite-condition}. More precisely, the AF subgroupoids $\mathfrak{T}_{\mathsf{B},n}$ form an almost-finite approximation.
\end{thm}

\begin{proof}
Let $C\subseteq\G$ be compact. Every arrow of $\G$ belongs to a product of bisections in $\mathcal{S}$. These products form an open cover of $\G$, and compactness gives words $w_1,\ldots,w_q$ such that
$C\subseteq w_1\cup\cdots\cup w_q.$
Write $\ell_i=|w_i|$ and $L=\ell_1+\cdots+\ell_q$.

Choose $n$ sufficiently large that the level-$n$ cylinders resolve the domains and ranges of all $F\in\mathcal{S}$. Fix a point $x$, and let $\mathcal{T}_{v,n}$ be the finite tile containing its level-$n$ prefix. Identify the vertex set $\rg((\mathfrak{T}_{\mathsf{B},n})_x)$ with $V(\mathcal{T}_{v,n})$. Equivalently, $\mathcal{T}_{v,n}^{\circ}$ is the corresponding elementary Cayley graph. Since $\mathfrak{T}_{\mathsf{B},n}$ is principal, every $y$ in this orbit determines a unique arrow $k_y\in\mathfrak{T}_{\mathsf{B},n}$ from $x$ to $y$.

Consider first one word $F=F_\ell\cdots F_1$, with $F_j\in\mathcal{S}$, and a starting point $y$ in this elementary orbit on which $F$ is defined. Suppose that $(F,y)k_y\notin\mathfrak{T}_{\mathsf{B},n}$. Follow the trajectory of $y$ under the successive letters and let $j$ be its first step that is not represented by an internal edge of the finite tile. Such a step must occur, since a word represented entirely by internal prefix replacements gives an arrow of $\mathfrak{T}_{\mathsf{B},n}$. The first $j-1$ steps therefore remain in $\mathcal{T}_{v,n}^{\circ}$. Since the level-$n$ cylinders resolve the domain of $F_j$, the next defined label is recorded in the finite tile, and since it is not internal it is a boundary edge. Thus the point reached immediately before the $j$th step is a boundary vertex.

For a fixed $j$, the preceding partial word $F_{j-1}\cdots F_1$ is injective on its domain. Its values remain in the same elementary orbit, and within that orbit the level-$n$ prefix determines the point. Consequently, distinct starting points whose trajectories first meet the boundary at step $j$ give distinct boundary vertices. There are at most $\beta_n$ such starting points for each $j$, and therefore at most $\ell\beta_n$ in total. For each starting point there is at most one arrow in the product bisection. We count these arrows separately even when they have the same range. Hence
\begin{equation}
\label{eq:one-word-folner-bound}
|F(\mathfrak{T}_{\mathsf{B},n})_x\setminus(\mathfrak{T}_{\mathsf{B},n})_x|\leq \ell\beta_n.
\end{equation}
Taking the union over the chosen word bisections gives
$|C(\mathfrak{T}_{\mathsf{B},n})_x\setminus(\mathfrak{T}_{\mathsf{B},n})_x| \leq L\beta_n.$
On the other hand, $|(\mathfrak{T}_{\mathsf{B},n})_x|\geq m_n$. Hence,
\begin{equation}
\frac{|C(\mathfrak{T}_{\mathsf{B},n})_x\setminus(\mathfrak{T}_{\mathsf{B},n})_x|}{|(\mathfrak{T}_{\mathsf{B},n})_x|}
\leq L\frac{\beta_n}{m_n},
\end{equation}
uniformly in $x$. By \eqref{eq:boundary-folner-ratio}, the right-hand side is smaller than $\varepsilon$ for all sufficiently large $n$. Taking $\mathfrak{K}=\mathfrak{T}_{\mathsf{B},n}$ proves almost finiteness.
\end{proof}

\begin{cor}
\label{cor:bounded-type-almost-finite}
Assume the finite compatible tile presentation of Subsection~\ref{subsec:hypothesis-scope}. Suppose that $\mathsf{B}$ is simple, and $|\partial\mathcal{T}_{v,n}|\leq C_\partial$ for all $v,n$. Then the standard elementary stages $\mathfrak{T}_{\mathsf{B},n}$ form an almost-finite approximation of $\G$. In particular, $\G$ is almost finite. This applies, in particular, to tile inflations of bounded type.
\end{cor}
This recovers the corresponding conclusion of \cite[Proposition~5.5.2]{nek22} for systems of bounded type with infinite orbits.

\begin{proof}
The uniform boundary bound gives $\beta_n\leq C_\partial$ for all $n$, while Lemma~\ref{lemma:uniform-tile-growth} gives $m_n\to\infty$. Hence \eqref{eq:boundary-folner-ratio} holds and Theorem~\ref{thm:boundary-folner} applies. Corollary~\ref{cor:infinite-AF-orbits-almost-finite} also applies since the AF core is minimal. The present argument gives the explicit finite-tile escape estimate.
\end{proof}

\begin{rmk}
\label{rmk:folner-beyond-bounded-type}
The proof uses only the condition on the ratio of boundary to volume \eqref{eq:boundary-folner-ratio}. Uniform boundedness of $|V_n|$ and $|E_n|$, finiteness of the groups of germs, and the localization hypotheses from \cite{kua26a} are not needed. The approximation criterion also applies when the boundaries of finite tiles are unbounded but satisfy the displayed ratio estimate. The kernel consequences obtained below from almost finiteness use minimality and comparison. Section~\ref{sec:constructive-kernel} gives an independent kernel description without these assumptions. Failure of~\eqref{eq:boundary-folner-ratio} implies neither failure of almost finiteness nor failure of approximation by the chosen elementary stages of $\T$: the largest boundary and the smallest finite tile can occur in different finite tiles.
\end{rmk}

\subsection{Known kernel consequences with hypotheses}
The remaining statements in this subsection are consequences of established full-group theory once almost finiteness has been verified. The principal almost-finite kernel theorem is \cite[Theorem~5.5.9]{nek22}. Since our AF-by-discrete groupoids may have nontrivial groups of germs, we instead use two results valid in the effective setting: minimal almost finite groupoids have comparison, and their derived and alternating full groups coincide. See \cite[Theorems~5.5.4 and~5.5.5]{nek22}, and the compact-open bisection formulation in \cite[Theorems~4.8 and~4.9]{nekrash2025}. Li's results in the present Cantor setting under minimality and comparison then supply index surjectivity and the kernel reduction \cite[Corollaries~6.14 and~6.17]{Li25}. 

\keepwithstatement
\begin{thm}
\label{thm:kernel-folner}
Suppose that $\G$ is effective and that either its AF core is minimal on $\Omega(\mathsf{B})$ or $\G$ is minimal and compactly generated. Then
\begin{equation}
\label{eq:kernel-symmetric}
\ker I=\Sym(\G),
\end{equation}
the index map is surjective, and
\begin{equation}
\label{eq:derived-alternating}
\Der(\G)=\Alt(\G).
\end{equation}
Consequently, there is an exact sequence
\begin{equation}
\label{eq:index-short-exact}
1\longrightarrow\Sym(\G)\longrightarrow\F(\G)
\xrightarrow{\ I\ }H_1(\G)\longrightarrow0.
\end{equation}
More generally, the same conclusions hold whenever $\G$ is minimal and almost finite in this Cantor setting, whether almost finiteness follows from Theorem~\ref{thm:AF-exhaustion-criterion} or from another elementary approximation.
\end{thm}
\begin{proof}
In the first case, minimality of the AF core implies minimality of $\G$, and Corollary~\ref{cor:infinite-AF-orbits-almost-finite} gives almost finiteness. In the second case, Theorem~\ref{thm:minimal-AF-by-discrete-almost-finite} gives almost finiteness directly. The cited results of Nekrashevych then give comparison and, using effectiveness, $\Der(\G)=\Alt(\G)$.

Now apply Li's low-degree exact sequence \cite[Corollary~6.14]{Li25}. Its final map is the homomorphism induced by the index and is onto, and thus the index itself is surjective. Li's kernel reduction \cite[Corollary~6.17]{Li25} says that $\ker I$ is generated by $\Sym(\G)$ and $\Der(\G)$ and that $\ker I=\Sym(\G)$ is equivalent to $\Der(\G)=\Alt(\G)$. This proves the theorem. 
\end{proof}
Section~\ref{sec:constructive-kernel} gives a direct transposition factorization and extends the kernel equality to every AF-by-discrete groupoid without singleton $\G$-orbits.

\begin{cor}
\label{cor:full-quotient-matrix}
Assume either of the hypotheses in Theorem~\ref{thm:kernel-folner}. With the hypothesis finite generation, the splitting, and defect matrices from Theorem~\ref{thm:finite-level-defect}, for every sufficiently large $n$, there is an isomorphism
\begin{equation}
\label{eq:full-quotient-matrix}
\F(\G)/\Sym(\G)\cong H_1(\G)\cong \Tor H_1(\G,\T)\oplus\ker D_n.
\end{equation}
The first isomorphism is canonical. The second depends on the chosen relative splitting and free generators, and its existence does not supply an effective stabilization level.
\end{cor}

\begin{cor}
\label{cor:finite-germ-full-quotient}
Assume the finite-singular-germ condition of Corollary~\ref{cor:finite-singular-H1}, and minimality of the AF core. If $\mathcal{O}_1,\ldots,\mathcal{O}_r$ are the exceptional orbits, then
\begin{equation}
\label{eq:finite-germ-full-quotient}
\F(\G)/\Sym(\G)\cong\bigoplus_{i=1}^r H_{\mathcal{O}_i}^{\ab}.
\end{equation}
\end{cor}
This applies to the cited finite-singular-germ examples of bounded type once their minimal AF core and singular data have been verified.
\begin{proof}
Theorem~\ref{thm:kernel-folner} identifies the quotient with $H_1(\G)$, and Corollary~\ref{cor:finite-singular-H1} computes that homology group.
\end{proof}

\subsection{The residual quotient under comparison}
For this subsection, we assume explicitly that $\G$ is minimal and has comparison. We consider the quotient $\ker I/\Sym(\G)$ by the subgroup generated by dynamical transpositions. These hypotheses hold in the case of minimal groupoids above. We retain the general reduction as a comparison with the general groupoid problem. For the AF-by-discrete inclusion itself, Corollary~\ref{cor:no-singletons-kernel} in the next section will make this residual quotient trivial under minimality without establishing almost finiteness first. Li's result is
\begin{equation}
\label{eq:li-kernel-reduction}
\ker I=\langle \Sym(\G),\Der(\G)\rangle,
\end{equation}
with equivalences
\begin{equation}
\label{eq:li-kernel-equivalences}
\ker I=\Sym(\G)\quad\Longleftrightarrow\quad \Der(\G)\subseteq \Sym(\G)
\quad\Longleftrightarrow\quad \Der(\G)=\Alt(\G).
\end{equation}
These are the conclusions of \cite[Corollary~6.17]{Li25}.

The residual quotient is controlled by the derived subgroup and is itself perfect.
\keepwithstatement
\begin{prop}
\label{prop:residual-kernel-quotient}
Under the preceding minimality and comparison assumptions, we have
\begin{equation}
\label{eq:actual-obstruction}
\ker I/\Sym(\G)\cong \Der(\G)/(\Der(\G)\cap \Sym(\G)),
\end{equation}
and there is an exact sequence
\begin{equation}
\label{eq:obstruction-quotients}
1\longrightarrow (\Der(\G)\cap \Sym(\G))/\Alt(\G)\longrightarrow \Der(\G)/\Alt(\G)
\longrightarrow\ker I/\Sym(\G)\longrightarrow1.
\end{equation}
The group $\ker I/\Sym(\G)$ is perfect. In particular, every homomorphism from it to an abelian group is trivial.
\end{prop}
\begin{proof}
The subgroups $\Sym(\G)$ and $\Alt(\G)$ are normal in $\F(\G)$, since conjugation preserves multisections. Also $\Alt(\G)\subseteq \Der(\G)\cap \Sym(\G)$. By \eqref{eq:li-kernel-reduction}, $\ker I=\Sym(\G)\Der(\G)$, and thus the second isomorphism theorem gives \eqref{eq:actual-obstruction}. The quotient map from $\Der(\G)/\Alt(\G)$ then has kernel $(\Der(\G)\cap \Sym(\G))/\Alt(\G)$. Finally, $\Der(\G)$ is perfect by \cite[Corollary~6.10]{Li25}. Every quotient of a perfect group is perfect.
\end{proof}

To compute the residual quotient geometrically, we must realize the local configurations and determine their relations modulo $\Sym(\G)$. Section~\ref{sec:constructive-kernel} gives this direct AF-by-discrete computation beyond minimality, while the localization methods of \cite{kua26a,kua26b} address effective representatives in prescribed subgroups.

\subsection{Abelianization and the parity term}\label{subsec:AH-terms}
Assume now one of the minimality hypotheses of Theorem~\ref{thm:kernel-folner}. By Theorem~\ref{thm:kernel-folner},
$\Der(\G)=\Alt(\G)$ and $\ker I=\Sym(\G).$
Hence the two successive quotients
\begin{equation}
\frac{\Sym(\G)}{\Alt(\G)}
\qquad\text{and}\qquad
\frac{\F(\G)}{\Sym(\G)}\cong H_1(\G)
\end{equation}
separate the parity contribution from the boundary index contribution. Under minimality and comparison, the following part of Li's low-degree AH exact sequence is available independently of almost finiteness \cite[Corollary~6.14]{Li25}:
\begin{equation}
\label{eq:AH-bounded-type}
H_2(\G)\xrightarrow{\theta_{\G}} H_0(\G;\mathbb{Z}/2\mathbb{Z})
\xrightarrow{j_{\G}} \F(\G)_{\ab}
\xrightarrow{\ I_{\ab}\ } H_1(\G)
\longrightarrow0.
\end{equation}
Here $I_{\ab}$ is induced by the index map, and $j_{\G}$ sends $[1_{\sg(V)}]$ to the class of $t_V$ in the abelianization, where $V$ is a compact open bisection with disjoint source and range. The preceding term in Li's whole exact sequence is $H_2(\Der(\G))$. Only the displayed terms will be used here. The map $\theta_{\G}$ is the secondary parity differential whose geometric computation is developed in Sections~\ref{sec:parity}--\ref{sec:surface-parity}.

\begin{defn}[AH, strong AH, and split AH]
\label{def:AH-terms}
The exact sequence~\eqref{eq:AH-bounded-type} is the \emph{AH sequence}. We say that $\G$ has the \emph{strong AH property} if $j_{\G}$ is injective. We say that the AH sequence \emph{splits} if $I_{\ab}$ admits a homomorphic section as a map of abelian groups.
\end{defn}
By exactness, the strong AH property is equivalent to $\theta_{\G}=0$.

The strong AH property and splitting are distinct assertions. See \cite[Remark~6.16]{Li25}.

Section~\ref{sec:index} determines the quotient of the abelianization by parity. Section~\ref{sec:parity} computes the first map for coherent realizations of returns: finite groups of germs satisfy this hypothesis automatically, and the sparse $\mathbb{Z}^d$ models provide infinite examples with nonzero higher homology. Under minimality and comparison, Theorem~\ref{thm:coherent-return-split-AH} then splits the AH extension.


\section{Constructive boundary cancellation and the kernel}
\label{sec:constructive-kernel}

The index formula (Theorem~\ref{thm:index-formula}) detects vanishing of the relative index class. Under the finite boundary hypotheses, the geometric index theorem (Theorem~\ref{thm:geometric-index}) expresses this as zero boundary current. We now realize the cancellation by dynamical transpositions. Throughout this section, $\G$ is any effective AF-by-discrete groupoid with Cantor unit space $\Omega(\mathsf{B})$ and fixed AF core $\T$. The possibly infinite direct-sum normal form of Remark~\ref{rmk:normal-form-scope} applies.

The construction has two stages. Vanishing of the transport coordinate first balances the finitely many exceptional sources and ranges inside each $\T$-orbit. After an AF interpolation, the exceptional germs can therefore be collected on a finite set that is invariant under the element. On every $\G$-orbit containing at least two marked points, the remaining return labels form a germ table as in Definition~\ref{def:germ-table}. Vanishing of the return coordinate puts the collected product of labels in the derived subgroup, and commutators can then be converted into swaps between two distinct marked points. The only place where this last move is unavailable is a singleton $\G$-orbit. 
Put
\begin{equation}
\label{eq:singleton-orbit-set}
Z_1=\{x\in \Omega(\mathsf{B}):\G x=\{x\}\},\qquad H_x=\G_x^x.
\end{equation}
The set $Z_1$ is allowed to be empty. In particular, it is empty when $\G$ is minimal, since a singleton orbit cannot be dense in a Cantor space. Singleton $\G$-orbits can nevertheless occur in nonminimal AF-by-discrete groupoids. Example~\ref{exmp:singleton-nonabelian-kernel} explains such a case with a nontrivial group of germs. A point in $Z_1$ is still nonisolated in $\Omega(\mathsf{B})$, since the Cantor unit space has no isolated points. 

\begin{defn}[Singleton germ evaluation]
\label{def:singleton-germ-evaluation}
Let $x\in Z_1$. The \emph{germ evaluation at $x$} is the map $\operatorname{ev}_x:\F(\G)\to H_x$ given by $\operatorname{ev}_x(g)=(g,x)$. We write $\operatorname{ev}_{Z_1}(g)=((g,x))_{x\in Z_1}$ for the family of all singleton germ evaluations.
\end{defn}

Every full-group element fixes $x$, and thus $\operatorname{ev}_x$ is a homomorphism. Moreover, $\operatorname{ev}_{Z_1}(g)$ has finite nonidentity support for every fixed $g$: a nontrivial isotropy germ at a singleton orbit cannot lie in the principal groupoid $\T$, while $U_g\setminus\T$ is finite by Lemma~\ref{lemma:finite-exceptional-support}. Thus $\operatorname{ev}_{Z_1}$ takes values in the restricted direct product of the groups $H_x$.

\begin{thm}[Kernel theorem]
\label{thm:singleton-kernel}
The restriction of the singleton germ-evaluation map to $\ker I$ induces an exact sequence of groups
\begin{equation}
\label{eq:singleton-kernel-exact}
1\longrightarrow\Sym(\G)\longrightarrow \ker I
\xrightarrow{\ \Lambda\ }\bigoplus_{x\in Z_1}[H_x,H_x]
\longrightarrow1,
\qquad \Lambda(g)=\operatorname{ev}_{Z_1}(g).
\end{equation}
Here $\bigoplus$ denotes the restricted direct product: tuples have finite nonidentity support. Hence, $\ker I=\Sym(\G)$ if and only if every group of germs at a singleton $\G$-orbit is abelian. In particular, this equality holds whenever $\G$ has no singleton orbits.
\end{thm}

The restriction to singleton orbits is essential. If a $\G$-orbit contains at least two points, Proposition~\ref{prop:finite-boundary-balance} adds a second marked point when needed, and Proposition~\ref{prop:finite-table-cancellation} factors every germ table whose total label is zero into the tables $t_{uv}(h)$ defined below. Thus a commutator concentrated at one point can be cancelled using another point of the same orbit. Nontrivial, even nonabelian, isotropy on a finite $\G$-orbit of cardinality greater than one therefore contributes no term to~\eqref{eq:singleton-kernel-exact}. Moreover, only for a singleton orbit does evaluation at a chosen point define a homomorphism $\F(\G)\to H_x$, since then every full-group element fixes that point.

Subsections~\ref{subsec:finite-AF-interpolation}--\ref{subsec:singleton-residual} prove the theorem. We then give a nonminimal example and discuss effective reconstruction.

\subsection{Finite interpolation in the AF core}\label{subsec:finite-AF-interpolation}
\begin{lemma}
\label{lemma:finite-AF-interpolation}
Let $A_0,B_0\subseteq \Omega(\mathsf{B})$ be finite, and let $\theta:A_0\to B_0$ be a bijection such that $\theta(x)\in\T x$ for every $x\in A_0$. Let $Q_0\subseteq \Omega(\mathsf{B})$ be a finite set disjoint from $A_0\cup B_0$. There is $a\in\F(\T)$ inducing $\theta$ on $A_0$ and having germ $(\Id,x)$ at every $x\in Q_0$. Its germs on $A_0$ are the unique AF arrows from $x$ to $\theta(x)$.
\end{lemma}
\begin{proof}
Set $Q=A_0\cup B_0$. For each $\T$-orbit $\mathcal{O}$ meeting $Q$, put $Q_{\mathcal{O}}=Q\cap\mathcal{O}$. The restriction of $\theta$ to $A_0\cap\mathcal{O}$ is a bijection onto $B_0\cap\mathcal{O}$, and hence extends to a permutation $\sigma_{\mathcal{O}}$ of the finite set $Q_{\mathcal{O}}$: indeed, the unused source set $Q_{\mathcal{O}}\setminus A_0$ and the unused range set $Q_{\mathcal{O}}\setminus B_0$ have the same cardinality.

Choose a basepoint $z_{\mathcal{O}}\in Q_{\mathcal{O}}$ and, for every $q\in Q_{\mathcal{O}}$, let $\alpha_q:z_{\mathcal{O}}\to q$ be the unique AF arrow. Since only finitely many arrows are involved and $\T$ is AF, they lie in one sufficiently deep elementary stage. Shrinking a common clopen neighborhood of $z_{\mathcal{O}}$ in that stage, the restrictions of the $\alpha_q$ form a finite multisection whose levels contain the points of $Q_{\mathcal{O}}$. The source can be chosen so small that all of these levels avoid $Q_0$ and, for distinct orbits $\mathcal{O}$, the corresponding multisections have disjoint supports. Realize $\sigma_{\mathcal{O}}$ in this multisection. Taking the product over the finitely many orbits and the identity off their supports gives an element $a\in\F(\T)$ inducing $\theta$ on $A_0$ and having germ $(\Id,x)$ at every $x\in Q_0$. Principality of $\T$ gives uniqueness of the specified AF germs on $A_0$.
\end{proof}

\begin{lemma}
\label{lemma:AF-two-transpositions}
Every $f\in\F(\T)$ is a product of at most two dynamical transpositions belonging to $\F(\T)$.
\end{lemma}
\begin{proof}
The full bisection of $f$ is compact and is contained in one elementary stage of $\T$. Decompose that stage into finitely many multisections and refine their clopen bases so that the permutation induced by $f$ is constant on each refined base. Every finite permutation is a product of two involutions. More explicitly, on a cycle indexed by $\mathbb{Z}/m\mathbb{Z}$, the maps $r(j)=-j$ and $s(j)=1-j$ satisfy $s r(j)=j+1$.

Perform these two reflections on every cycle in every refined multisection. This gives $f=s r$, where $s$ and $r$ are involutions whose nontrivial cycles exchange pairs of clopen levels. For either involution, choose one direction on each exchanged pair. The union of these finitely many bisections has disjoint source and range, and thus the involution is a dynamical transposition. Identity factors can be omitted.
\end{proof}

The AF hypothesis matters here: the involution in Example~\ref{exmp:AF-by-discrete-not-almost-finite} has nonzero index.

\subsection{Marked germs and clopen lifts}\label{subsec:marked-germs}
We carry out cancellation on a finite marked set, allowing AF motion elsewhere.

\begin{defn}[AF off a finite marked set]
\label{def:AF-off-P}
Let $P\subseteq \Omega(\mathsf{B})$ be finite. We say that $u\in\F(\G)$ is \emph{AF off $P$} if $\Sigma(u)\subseteq P$, where $\Sigma(u)$ is the exceptional support of Definition~\ref{def:exceptional-support}. When $P$ is fixed in a construction, we call it the \emph{marked set}.
\end{defn}

Being AF off $P$ is a condition on germs, not on pointwise support: $u$ may act nontrivially on a clopen set much larger than $P$, provided that every germ based outside $P$ belongs to $\T$. The next lemma lifts one prescribed arrow to a transposition while controlling all marked germs simultaneously.
\keepwithstatement
\begin{lemma}
\label{lemma:marked-transposition}
Let $x,y$ be distinct points of a finite set $P$, and let $\gamma:x\to y$ be an arrow of $\G$. There is a compact open bisection $V\ni\gamma$ such that $\sg(V)\cap\rg(V)=\varnothing$, the two sets meet $P$ only at $x$ and $y$, respectively, and $V\setminus\T\subseteq\{\gamma\}$. Consequently,
\begin{equation}
\label{eq:marked-transposition}
t_V=V\cup V^{-1}\cup 1_{\Omega(\mathsf{B})\setminus(\sg(V)\cup\rg(V))}
\end{equation}
preserves $P$, exchanges $x$ and $y$ with germs $\gamma$ and $\gamma^{-1}$, has germ $(\Id,z)$ at every other $z\in P$, and is AF off $P$.
\end{lemma}
\begin{proof}
Start with a compact open bisection $W\ni\gamma$. If $\gamma\in\T$, take $W\subseteq\T$ using openness of the AF core. If $\gamma\notin\T$, restrict $W$ by Lemma~\ref{lemma:finite-exceptional-support} so that $W\setminus\T=\{\gamma\}$. Since $x\neq y$ and $P$ is finite, choose disjoint clopen neighborhoods $A\ni x$ and $B\ni y$ with $A\cap P=\{x\}$ and $B\cap P=\{y\}$. Restrict $W$ once more so that its source lies in $A$ and its range lies in $B$. Call the resulting bisection $V$. Then $\sg(V)\cap\rg(V)=\varnothing$ and the only possible non-AF arrow of $V$ is $\gamma$. Since $\T$ is a subgroupoid, $V^{-1}\setminus\T\subseteq\{\gamma^{-1}\}$ as well. Formula~\eqref{eq:marked-transposition} now has all the stated properties.
\end{proof}

\begin{lemma}
\label{lemma:marked-AF-remainder}
Suppose that $u,v\in\F(\G)$ preserve a finite set $P$, are AF off $P$, and have equal germs at every point of $P$. Then $v^{-1}u\in\F(\T)$. Products and inverses of elements preserving $P$ and AF off $P$ still have these two properties.
\end{lemma}
\begin{proof}
Such elements also preserve $\Omega(\mathsf{B})\setminus P$. At a point outside $P$, every germ encountered in their composition belongs to $\T$, and thus the product germ belongs to $\T$. At a point of $P$, the equal germs of $u$ and $v$ cancel to the identity in $v^{-1}u$. Its full bisection is thus contained in $\T$. The same composition observation proves the last assertion.
\end{proof}

Thus two such products with the same marked germs differ by an element of $\F(\T)$, which Lemma~\ref{lemma:AF-two-transpositions} factors into at most two transpositions.

\subsection{Cancellation in a finite germ table}
We record the germs on an invariant marked set in a finite table, using the normalized returns of Theorem~\ref{thm:relative-normal-form} as labels.

\begin{defn}[Germ table and total label]
\label{def:germ-table}
Let $P=\{x_0,\ldots,x_{n-1}\}$ lie in one $\G$-orbit. Choose connectors $p_i:x_0\to x_i$, with $p_0=(\Id,x_0)$, and put $H=\G_{x_0}^{x_0}$. A \emph{germ table} on $P$ consists of a permutation $\pi\in\operatorname{Sym}(n)$ and arrows $\gamma_i:x_i\to x_{\pi(i)}$. Its \emph{normalized labels} are
\begin{equation}
\label{eq:germ-table-labels}
a_i=p_{\pi(i)}^{-1}\gamma_i p_i\in H.
\end{equation}
The \emph{total label} of the table is
\begin{equation}
\label{eq:table-character}
\chi\bigl(\pi,(a_i)\bigr)=\sum_i[a_i]_{\ab}\in H^{\ab}.
\end{equation}
\end{defn}

All abstract germ tables form a group isomorphic to the wreath product $H\wr\operatorname{Sym}(n)=H^n\rtimes\operatorname{Sym}(n)$, where the symmetric group permutes the factors of $H^n$. Equivalently, this is the group of full bisections of the discrete restriction $\G|_P$. Only tables whose arrows are realized by germs of clopen bisections of $\G$ will be lifted below. With the right-hand table applied first, multiplication is
\begin{equation}
\label{eq:table-product}
(\pi,(a_i))\, (\rho,(b_i))
=\bigl(\pi\rho,(a_{\rho(i)}b_i)\bigr).
\end{equation}
After abelianization, the permutation merely reorders the labels, so~\eqref{eq:table-product} shows that $\chi$ is a homomorphism.

For distinct indices $u,v$ and $h\in H$, let $t_{uv}(h)$ denote the abstract table that exchanges $u,v$, with label $h$ on the arrow from $v$ to $u$ and label $h^{-1}$ on its inverse. In terms of actual arrows, the former is $p_u h p_v^{-1}$. All other germs of this table are identities. By Lemma~\ref{lemma:marked-transposition}, this table has a lift to a dynamical transposition, AF off $P$, and the lift can fix any additional finite marked set pointwise with identity germs.

Set $D_{uv}(h)=t_{uv}(h)t_{uv}(1)$. Direct multiplication gives a table with identity permutation, label $h$ at $u$, label $h^{-1}$ at $v$, and identity labels elsewhere. Let $C_u(h)$ be the table with identity permutation, label $h$ at $u$, and identity labels elsewhere. Then
\begin{equation}
\label{eq:six-transposition-identity}
D_{uv}(a)D_{uv}(b)D_{uv}(a^{-1}b^{-1})=C_u([a,b]),
\qquad [a,b]=aba^{-1}b^{-1}.
\end{equation}
Indeed, the label at $u$ is $aba^{-1}b^{-1}$, while the label at $v$ is $a^{-1}b^{-1}ba=1$. Thus a commutator label at one point is a product of six tables of the form $t_{uv}(h)$, using one additional point. With at least two marked points, these identities yield a transposition factorization whenever the total abelianized label vanishes. 
\keepwithstatement
\begin{prop}
\label{prop:finite-table-cancellation}
Suppose $n\geq2$. The kernel of $\chi$ is generated by the tables $t_{uv}(h)$. More precisely, let $(\pi,(a_i))\in\ker\chi$ and put $r=a_{n-1}\cdots a_1a_0$ (in the order used in~\eqref{eq:diagonal-collection}). Given a commutator expression
\begin{equation}
\label{eq:return-commutator-certificate}
r=\prod_{j=1}^c[u_j,v_j]\quad\text{in }H,
\end{equation}
the table is a product of at most
\begin{equation}
\label{eq:table-factor-count}
n-\operatorname{cyc}(\pi)+2(n-1)+6c
\end{equation}
tables of the form $t_{uv}(h)$. Here $\operatorname{cyc}(\pi)$ counts all cycles, including fixed points. For $n=1$, no transposition is available and $\ker\chi=[H,H]$.
\end{prop}
\begin{proof}
Each table $t_{uv}(h)$ has zero total label. Let $P_\pi$ have permutation $\pi$ and arrows $p_{\pi(i)}p_i^{-1}$, so all its normalized labels are identities. Let $\operatorname{diag}(a_0,\ldots,a_{n-1})$ denote the table with identity permutation and labels $a_0,\ldots,a_{n-1}$. Then
\begin{equation}
\label{eq:diagonal-collection}
\begin{split}
(\pi,(a_i))&=P_\pi\operatorname{diag}(a_0,\ldots,a_{n-1}),\\
\operatorname{diag}(a_0,\ldots,a_{n-1})
&=D_{1,0}(a_1)\cdots D_{n-1,0}(a_{n-1})C_0(r).
\end{split}
\end{equation}
For the second equality, the labels at $i\neq0$ are $a_i$, while the label at $0$ is $a_1^{-1}\cdots a_{n-1}^{-1}r=a_0$. Since the total abelianized label is zero, $r\in[H,H]$ and a finite expression~\eqref{eq:return-commutator-certificate} exists by the definition of the derived subgroup.

Factor $P_\pi$ into $n-\operatorname{cyc}(\pi)$ tables of the form $t_{uv}(1)$. Each of the $n-1$ factors $D_{i,0}(a_i)$ uses two tables of the form $t_{uv}(h)$. Apply~\eqref{eq:six-transposition-identity} to every commutator in~\eqref{eq:return-commutator-certificate}, using the pair $0,1$. This gives the stated factorization and count. When $n=1$, the group of germ tables is $H$, the map $\chi$ is abelianization, and there are no pairs of distinct marked points.
\end{proof}

These identities are calculations in the wreath product $H\wr\operatorname{Sym}(n)$. They give a factorization into tables $t_{uv}(h)$ with an explicit bound on the number of factors. Each factor can be lifted to a dynamical transposition whose exceptional sources lie in the fixed finite set. We will use equality of germs, followed by a remainder in the AF core, rather than assert the identities for arbitrary lifts as homeomorphisms of $\Omega(\mathsf{B})$.

\subsection{From zero transport to an invariant finite set}
We first express zero transport without choosing a spanning tree. Put $E=\Sigma(g)$. In one exceptional $\G$-orbit, the transport coordinate of Theorem~\ref{thm:index-formula} is the sum of the vectors $e_{\T(gx)}-e_{\T(x)}$ over $x\in E$. Consequently, the coefficient of a fixed $\T$-orbit $\mathcal{T}$ is $|g(E)\cap\mathcal{T}|-|E\cap\mathcal{T}|$. If $g\in\ker I$, every such coefficient is zero, and hence
\begin{equation}
\label{eq:finite-exceptional-balance}
|\Sigma(g)\cap\mathcal{T}|=|g(\Sigma(g))\cap\mathcal{T}|.
\end{equation}
Both sets are finite. This is stronger than equality of the total numbers of exceptional sources and ranges: it balances them \emph{inside each $\T$-orbit}. AF interpolation will turn this orbitwise balance into an invariant finite marked set on which the return labels can be cancelled.
\keepwithstatement
\begin{prop}
\label{prop:finite-boundary-balance}
Let $g\in \ker I$, and suppose that $(g,x)=(\Id,x)$ for every $x\in Z_1$. There are $a,b\in\F(\T)$ and a finite set $P\supseteq \Sigma(g)$ such that $k=bag$ preserves $P$, $\Sigma(k)=\Sigma(g)$, and each $\G$-orbit meeting $P$ contains at least two points of $P$. There is at most one added point for each orbit meeting $\Sigma(g)$, and the germs of $k$ at all the added points are identities.
\end{prop}
\begin{proof}
Write $E=\Sigma(g)$. Equation~\eqref{eq:finite-exceptional-balance} gives, separately in each $\T$-orbit, equal cardinalities for the exceptional ranges $g(E)$ and exceptional sources $E$. Choose a bijection $g(E)\to E$ respecting the $\T$-orbits. Lemma~\ref{lemma:finite-AF-interpolation} extends this finite matching to $a\in\F(\T)$. Put $h=ag$. Then $h(E)=E$. Moreover, for every source $x$, the germ $(ag,x)$ belongs to $\T$ if and only if $(g,x)$ does, since the germ of $a$ at $g(x)$ lies in the subgroupoid $\T$ and can be cancelled by its inverse. Hence $\Sigma(h)=E$.

Now consider a $\G$-orbit $\mathcal{O}$ with $|E\cap\mathcal{O}|=1$. Its unique point in $E$ supports a non-AF germ of $g$. If $\mathcal{O}$ were a singleton orbit, the hypothesis on $Z_1$ would be violated. Hence $\mathcal{O}\setminus E$ is nonempty. Choose one point $y_{\mathcal{O}}$ there. Let $B$ be the finite set of these auxiliary points. Since $h$ preserves $E$, it also preserves its complement, and thus both $B$ and $h(B)$ are disjoint from $E$. For $y\in B$ we have $y\notin E=\Sigma(h)$, and hence $(h,y)\in\T$.

Use Lemma~\ref{lemma:finite-AF-interpolation} to choose $b\in\F(\T)$ with $b(h(y))=y$ for $y\in B$ and with germ $(\Id,x)$ at every $x\in E$. Set $P=E\cup B$ and $k=bh$. Then $k$ preserves $P$, permutes $E$, and fixes $B$. At $y\in B$, the germ $(k,y)$ is an isotropy arrow in $\T$ and thus is the identity by principality. Also $\Sigma(k)=E$. Every orbit meeting $P$ now has at least two marked points. If $E$ is empty, take $a=b=1$ and $P=\varnothing$.
\end{proof}

\begin{lemma}
\label{lemma:table-return-character}
Let $k\in \ker I$ preserve a finite set $P$ and be AF off $P$. For each $\G$-orbit $\mathcal{O}$ meeting $P$, the germ table of $k$ on $P\cap\mathcal{O}$ has zero total label in the abelianized group of germs, for any choice of connectors between these marked points.
\end{lemma}
\begin{proof}
Fix one orbit $\mathcal{O}$ meeting $P$, choose one marked point as the basepoint for the normal form, and choose one representative in every $\T$-orbit contained in $\mathcal{O}$. For a marked point $x$ in such a $\T$-orbit, let $p_x=a_xc_i$ be the connector of Theorem~\ref{thm:relative-normal-form}, where $a_x$ is the unique AF arrow from the chosen representative to $x$. Thus the same construction supplies a connector $p_x$ for every marked point, even when several marked points lie in one $\T$-orbit.

For $x_i\in P\cap\mathcal{O}$, the normalized label of the germ $(k,x_i):x_i\to k(x_i)$ is then
$p_{k(x_i)}^{-1}(k,x_i)p_{x_i}$. If $(k,x_i)\in\T$, the source and range lie in the same $\T$-orbit and the uniqueness of the AF arrow gives $p_{k(x_i)}=(k,x_i)p_{x_i}$. Hence, this normalized label is the identity. Since $k$ is AF off $P$, all exceptional sources of its full bisection occur among the marked points. Therefore the sum of the abelianized table labels is exactly the return coordinate of the relative class $\iota I(k)$ in the summand for $\mathcal{O}$. Since $k\in\ker I$, this relative class is zero, and so the total label vanishes.

It remains to remove this special choice of connectors. Any other connector from the same marked basepoint to $x_i$ has the form $p_i u_i$ with $u_i$ in the base group of germs. The corresponding label changes from $a_i$ to $u_{\pi(i)}^{-1}a_i u_i$. In the abelianization, the total change is $\sum_i([u_i]-[u_{\pi(i)}])=0$, since $\pi$ is a permutation. Changing the marked basepoint transports all labels through the same isomorphism of groups of germs. Thus vanishing of the total label is independent of all connector choices.
\end{proof}

\subsection{Constructive transposition factorization}
Combining the table factorization with AF interpolation gives an explicit transposition factorization, independently of the abstract kernel theorems of Section~\ref{sec:kernel}.

\keepwithstatement
\begin{thm}
\label{thm:constructive-boundary-cancellation}
Let $g\in \ker I$ have the identity germ at every singleton $\G$-orbit. Then $g$ is a finite product of dynamical transpositions. More precisely, make the choices of Proposition~\ref{prop:finite-boundary-balance}. For each orbit $\mathcal{O}$ meeting its marked set $P$, let $n_{\mathcal{O}}=|P\cap\mathcal{O}|$, let $\pi_{\mathcal{O}}$ be the induced permutation, and choose a commutator expression of length $c_{\mathcal{O}}$ for the product $r_{\mathcal{O}}$ of the normalized labels in the order specified in Proposition~\ref{prop:finite-table-cancellation}. The construction gives
\begin{equation}
\label{eq:constructive-normal-form}
g=\tau_1\cdots\tau_N f_0,\qquad f_0\in\F(\T),
\end{equation}
where all $\tau_i$ are dynamical transpositions and
\begin{equation}
\label{eq:constructive-length-bound}
N\leq\sum_{\mathcal{O}}
\bigl(n_{\mathcal{O}}-\operatorname{cyc}(\pi_{\mathcal{O}})
    +2(n_{\mathcal{O}}-1)+6c_{\mathcal{O}}\bigr).
\end{equation}
The factor $f_0\in\F(\T)$ is a product of at most two further dynamical transpositions. In particular, the resulting total length is at most $3|\Sigma(g)|+6\sum_{\mathcal{O}}c_{\mathcal{O}}+2$.
\end{thm}
\begin{proof}
Let $k=bag$ and $P$ be as in Proposition~\ref{prop:finite-boundary-balance}. The index of $k$ is zero, since $a,b\in\F(\T)$ have zero index. Only finitely many $\G$-orbits meet $P$. On each such orbit, Lemma~\ref{lemma:table-return-character} puts the germ table of $k$ in $\ker\chi$, and Proposition~\ref{prop:finite-boundary-balance} guarantees that the table has at least two marked points. Proposition~\ref{prop:finite-table-cancellation} therefore factors the germ table on each orbit into the indicated number of tables $t_{uv}(h)$.

Lift these table factors one at a time by Lemma~\ref{lemma:marked-transposition}, always using the same marked set $P$ and requiring identity germs at its other points. Let $t$ be the product in the order prescribed by the abstract table factorizations. At the marked points, multiplication of germs is exactly multiplication in the germ tables, and thus $t$ and $k$ have the same germs on $P$. Every factor preserves $P$ and is AF off $P$. Lemma~\ref{lemma:marked-AF-remainder} shows that these properties are preserved under multiplication, even if the clopen supports of different lifts overlap away from $P$. Hence $t$ is AF off $P$, and the same lemma gives $f=t^{-1}k\in\F(\T)$.

Write $p=ba$. Since $k=pg=t f$, we obtain
$g=(p^{-1}t p)(p^{-1}f)$. Conjugation by a full bisection sends a dynamical transposition to a dynamical transposition. This gives~\eqref{eq:constructive-normal-form} with $f_0=p^{-1}f\in\F(\T)$ and with the bound~\eqref{eq:constructive-length-bound}. Apply Lemma~\ref{lemma:AF-two-transpositions} to $f_0$.

For the coarser bound, if $e_{\mathcal{O}}=|\Sigma(g)\cap\mathcal{O}|$, then $n_{\mathcal{O}}-1\leq e_{\mathcal{O}}$, including the case $e_{\mathcal{O}}=1$ where one auxiliary point was added. Also $n_{\mathcal{O}}-\operatorname{cyc}(\pi_{\mathcal{O}})\leq n_{\mathcal{O}}-1$. Sum these inequalities. If $\Sigma(g)=\varnothing$, Lemma~\ref{lemma:AF-two-transpositions} alone applies.
\end{proof}

The product $r_{\mathcal{O}}$ depends on the chosen order of the marked points and on their connectors, while membership in the derived group is independent of these choices. The length bound is relative to the chosen commutator expressions. It is not a uniform bound over $g$ or over all possible groups of germs. Nor does the theorem imply a uniform bound in the generators of a prescribed acting group.

The product $D_{uv}(h)=t_{uv}(h)t_{uv}(1)$ uses two transpositions between distinct points, rather than a representative of $h$ on an invariant clopen neighborhood. Consequently, the construction applies without the finite-singular-germ hypothesis or the simultaneous invariant-neighborhood hypothesis used in some near-full completion arguments.

\subsection{The residual group at singleton orbits}\label{subsec:singleton-residual}
To prove surjectivity of $\Lambda$ in Theorem~\ref{thm:singleton-kernel}, we extend zero-defect bisections to full bisections using the AF core.

\keepwithstatement
\begin{lemma}
\label{lemma:AF-clopen-completion}
For clopen $A,B\subseteq \Omega(\mathsf{B})$, equality $[1_A]=[1_B]$ in $H_0(\T)$ is equivalent to the existence of a compact open bisection of $\T$ from $A$ to $B$. Consequently, a compact open $\G$-bisection $V$ with $[1_{\sg(V)}]=[1_{\rg(V)}]$ in $H_0(\T)$ extends to a full bisection whose added arrows all belong to $\T$.
\end{lemma}
\begin{proof}
Represent $A$ and $B$ by unions of cylinders at a common Bratteli level. Equality of their classes in the dimension group means that their cylinder-count vectors agree at some later level. At that level, pair source and range cylinders with the same terminal vertex. The corresponding prefix replacements form the required AF bisection. Conversely, the source and range of any AF bisection have equal classes in $H_0(\T)$.

For the last assertion, the complements of $\sg(V)$ and $\rg(V)$ also have equal classes. Choose an AF bisection between these complements and take its union with $V$. Sources and ranges are disjoint within the union and both cover $\Omega(\mathsf{B})$.
\end{proof}

In particular, a non-AF arrow $\gamma$ has a full representative $u$ with $U_u\setminus\T=\{\gamma\}$ if and only if $\delta[\gamma]=0$. For sufficiency, take an isolating bisection for $\gamma$ and complete it by the preceding lemma. For necessity, a full representative with only that exceptional arrow has relative index class $[\gamma]$, while exactness gives $\delta\iota I(u)=0$. Hence, $\delta[\gamma]=0$.

This also clarifies the comparison with finite germ extensions in \cite{BHM24}. Relative to the base group $\F(\T)$, a germ is singular exactly when it does not belong to $\T$. A construction that localizes one prescribed singular germ while agreeing with the base group at every other source therefore requires a full representative whose only non-AF germ is the prescribed one, and its obstruction is precisely the defect $\delta[\gamma]$. Finiteness of the exceptional set alone does not remove this obstruction. Our cancellation uses swaps between two distinct marked points instead, and thus an individual singular germ need not admit a full representative that is AF at every other source.

\begin{proof}[Proof of Theorem~\ref{thm:singleton-kernel}]
By Definition~\ref{def:singleton-germ-evaluation}, evaluation at each $x\in Z_1$ is a homomorphism. Since $\T$ is principal and the $\G$-orbit of $x$ is the singleton $\{x\}$, the corresponding $\T$-orbit is also a singleton and the transport lattice is zero. The orbit summand in the relative normal form of Theorem~\ref{thm:relative-normal-form} is therefore exactly $H_x^{\ab}$. If $g\in\ker I$, the image of $I(g)$ in the relative group is zero, and thus the class of $(g,x)$ in $H_x^{\ab}$ vanishes. Equivalently, $(g,x)\in[H_x,H_x]$. Lemma~\ref{lemma:finite-exceptional-support} shows that only finitely many of these germs are nonidentity. Hence $\Lambda$ is a well-defined homomorphism to the restricted direct product in~\eqref{eq:singleton-kernel-exact}.

Every dynamical transposition has zero index. Indeed, its exceptional arrows occur in pairs of arrows inverse to each other, and $[\gamma]+[\gamma^{-1}]=0$ in the relative group. Theorem~\ref{thm:index-formula} and the injection in~\eqref{eq:relative} thus give $\Sym(\G)\subseteq\ker I$, without using minimality or comparison. A dynamical transposition also has germ $(\Id,x)$ at a singleton orbit $\{x\}$: such a point can belong to neither of its two exchanged clopen sets, since otherwise its orbit would contain a second point. Hence $\Sym(\G)\subseteq\ker\Lambda$. Conversely, Theorem~\ref{thm:constructive-boundary-cancellation} applies to each $g\in\ker\Lambda$ and gives $g\in\Sym(\G)$. This identifies the kernel.

For surjectivity, take $x\in Z_1$ and $h\in[H_x,H_x]$ with $h\neq(\Id,x)$. Choose an isolating bisection $V\ni h$, and thus $V\setminus\T=\{h\}$. In the singleton orbit summand $H_x^{\ab}$ of Theorem~\ref{thm:relative-normal-form}, the relative class of $h$ is its abelianization. Since $h\in[H_x,H_x]$, this class is zero. Hence $[h]=0$ in $H_1(\G,\T)$, and the relative exact sequence~\eqref{eq:relative} gives $[1_{\sg(V)}]=[1_{\rg(V)}]$ in $H_0(\T)$. Complete $V$ by Lemma~\ref{lemma:AF-clopen-completion} to a full bisection $U_{u_h}$ representing $u_h\in\F(\G)$. Its only exceptional arrow is $h$ at $x$, and thus its relative index class is $[h]=0$. Injectivity of $H_1(\G)\to H_1(\G,\T)$ in~\eqref{eq:relative} gives $I(u_h)=0$. At every other singleton orbit the germ is the identity. Multiplying finitely many such lifts therefore realizes any finitely supported tuple in the target.
\end{proof}

No choice of connectors enters the map $\Lambda$. Although the proof uses the chosen AF core to balance and complete bisections, the quotient in~\eqref{eq:singleton-kernel-exact} is described intrinsically by singleton $\G$-orbits and their groups of germs. The construction chooses the individual lifts $u_h$ independently and produces no homomorphic splitting.

\begin{cor}
\label{cor:no-singletons-kernel}
If an AF-by-discrete groupoid in the standing setting has no singleton $\G$-orbits, then $\ker I=\Sym(\G)$. In particular, this holds for every minimal AF-by-discrete groupoid on a Cantor space, without compact generation or an almost-finiteness assumption. Under minimality and comparison, Li's Corollary~6.17 then also gives $\Der(\G)=\Alt(\G)$, and Corollary~6.14 gives index surjectivity.
\end{cor}
\begin{proof}
The first two assertions follow from Theorem~\ref{thm:singleton-kernel}, since a singleton orbit cannot be dense in a Cantor space. The last assertions are the indicated results of \cite{Li25}.
\end{proof}

This separates the kernel question from the approximation question left in Section~\ref{sec:kernel}. A minimal AF-by-discrete groupoid that is not compactly generated has no residual kernel quotient, even if almost finiteness has not yet been established for its chosen AF core. The corresponding question for arbitrary minimal groupoids on Cantor spaces remains open.

\subsection{A sharp nonminimal example}\label{subsec:singleton-example}
The singleton term can be nontrivial. It also distinguishes a torsion element of zero index from a dynamical transposition.

\begin{exmp}[A finite nonabelian group at a singleton orbit]
\label{exmp:singleton-nonabelian-kernel}
Let $H$ be a finite group, let $Y$ be a Cantor space, and let
\begin{equation}
\label{eq:finite-group-cone-space}
X=\{\infty\}\ \cup\ \coprod_{n\geq1}(\{n\}\times H\times Y)
\end{equation}
be the one-point compactification of the displayed topological disjoint union. Each annulus $A_n=\{n\}\times H\times Y$ is clopen, and neighborhoods of $\infty$ contain all sufficiently high annuli. The space $X$ is compact, metrizable, zero-dimensional, and has no isolated points, and thus it is a Cantor space.

We now choose a Bratteli diagram whose path space and tail groupoid realize this model. Fix a homeomorphism $Y\cong\{0,1\}^{\omega}$. The diagram $\mathsf{B}$ has an infinite spine
\begin{equation*}
v_0\longrightarrow v_1\longrightarrow v_2\longrightarrow\cdots.
\end{equation*}
For every $n\geq1$, attach $|H|$ edges, indexed by $H$, from $v_n$ to one vertex $w_{n,\varnothing}$ at the next level. From $w_{n,u}$, where $u$ is a finite binary word, there are two edges to $w_{n,u0}$ and $w_{n,u1}$. These branches do not merge with one another or with the spine. Every level contains only finitely many vertices. The unique infinite path that stays on the spine corresponds to $\infty$. Every other infinite path leaves the spine at a unique level $n$, chooses an exit edge indexed by $h\in H$, and then follows a binary ray $y\in Y$. This gives a homeomorphism
\begin{equation*}
\Omega(\mathsf{B})\cong X.
\end{equation*}
Under this identification, two paths outside the spine are tail equivalent if and only if they have the same $n$-coordinate and the same $Y$-coordinate. Their $H$-coordinates may differ. The spine path forms a singleton tail class. The tail groupoid of $\mathsf{B}$ is the equivalence relation obtained by varying the $H$-coordinate in each set $\{n\}\times H\times\{y\}$ and fixing $\infty$. Its basic bisections replace the $H$-coordinate over a fixed finite binary prefix of $y$. They are precisely the corresponding clopen restrictions of the action described below. We henceforth identify $X$ with $\Omega(\mathsf{B})$.

Let $H$ act by left multiplication on the $H$ coordinate of each annulus and fix $\infty$. This action is free on $\Omega(\mathsf{B})\setminus\{\infty\}$. Every nonidentity element acts nontrivially in every neighborhood of $\infty$, and thus the action is effective and its transformation groupoid agrees with its groupoid of germs. Put
\begin{equation}
\label{eq:cone-AF-core}
\G=H\ltimes \Omega(\mathsf{B}),\qquad
\T=\G^{(0)}\cup\{(h,x):x\neq\infty\}.
\end{equation}
The core is open and principal. An elementary exhaustion allows the full $H$ action on $A_1,\ldots,A_N$ and only units elsewhere. The complement is the finite discrete set $\{(h,\infty):h\neq1\}$. Thus, $\G$ is AF-by-discrete.

If $H\ne1$, its only exceptional orbit is $\{\infty\}$, with group of germs $H$. Thus $H_1(\G,\T)=H^{\ab}$. The same formula holds when $H=1$, since both groups are then zero. The full bisection of a global $h\in H$ has source and range $\Omega(\mathsf{B})$ and at most the one exceptional arrow $(h,\infty)$. Its defect is zero. Therefore, $H_1(\G)\cong H^{\ab}$ and the index is abelianization of the germ at $\infty$.

An element of the full group agrees with one global $h\in H$ on a neighborhood of $\infty$, since its cocycle to the finite discrete group $H$ is continuous. Its difference from that global element belongs to $\F(\T)$. Conversely, every element of $\F(\T)$ has germ $(\Id,\infty)$ at $\infty$. Evaluation and the global action thus give
\begin{equation}
\label{eq:cone-full-group}
\F(\G)=\F(\T)\rtimes H,\qquad
\Sym(\G)=\F(\T),\qquad
\ker I=\F(\T)\rtimes[H,H].
\end{equation}
For the middle equality, every dynamical transposition has germ $(\Id,\infty)$ at $\infty$ and belongs to $\F(\T)$, while Lemma~\ref{lemma:AF-two-transpositions} gives the reverse inclusion.

Take $H=\operatorname{Sym}(3)$. Then $H_1(\G)\cong C_2$ and $\ker I/\Sym(\G)\cong\operatorname{Alt}(3)\cong C_3$. A global three-cycle has order three and zero index but lies outside $\Sym(\G)$, since its germ at $\infty$ is nonidentity. Therefore, torsion of zero index may lie outside the dynamical symmetric group in the unrestricted nonminimal setting.

If $H\neq1$, the groupoid is not almost finite: every elementary subgroupoid has only the unit arrow in its source fiber at $\infty$, and the full bisection of any nonidentity $h$ has escape ratio one there. This extends the finite-orbit obstruction from Example~\ref{exmp:AF-by-discrete-not-almost-finite}. For abelian $H$, the kernel still equals $\Sym(\G)$. For nonabelian $H$, the quotient is its nontrivial derived subgroup. The singleton orbit makes the example nonminimal.
\end{exmp}

\subsection{Certificates, effectiveness, and a prescribed acting group}\label{subsec:factorization-certificates}
The construction has a finite record that can be checked independently of a uniform stopping procedure. First, list $\Sigma(g)$ and its image and pair them in each $\T$-orbit using~\eqref{eq:finite-exceptional-balance}. Record the finite AF arrows and clopen multisections implementing $a$ and $b$. Next, record the finite germ tables, their connectors, and a commutator expression~\eqref{eq:return-commutator-certificate} for the product of labels in each table. Finally, record the isolating bisections used to lift the tables $t_{uv}(h)$, the resulting remainder in $\F(\T)$, and a finite elementary stage in which that remainder is a clopen permutation. These choices produce the actual factors in~\eqref{eq:constructive-normal-form}, not merely their indices.

A computation from this record requires presentations in which the indicated clopen bisections, finite AF arrows, and germ equalities can be verified. The theorem supplies no algorithm for finding $\Sigma(g)$, deciding equality of arbitrary germs, finding a commutator expression, or recognizing a sufficiently deep stage from an arbitrary finite-state or nonstationary description. A commutator expression must be verified in the actual group of germs. An abelianization matrix whose validity has been established can verify membership in the derived subgroup, but it does not verify a specified commutator expression. The remainder in $\F(\T)$ is known to belong to a finite stage by compactness, but recognizing that stage effectively is additional input. Thus the existence proof does not by itself give a procedure for finding and verifying these data. Corollary~\ref{cor:directed-effective-factorization} below resolves them for the displayed directed model with finitely many states, with factors in $\F(\G)$.

For a specified subgroup $G\leq\F(\G)$, the factorization stays inside $G$ if $\F(\T)\leq G$ and the chosen marked transposition lifts belong to $G$. More generally, membership may be checked factor by factor. If only $\Alt(\T)$ is known to lie in $G$, an additional parity issue remains. The selector and localization constructions of \cite{kua26a,kua26b} address these membership and word-reconstruction questions.

\begin{problem}[Effective reconstruction in a specified presentation]
\label{prob:kernel}
For the concrete Bratteli or automaton models under consideration, construct from the given input the finite data recorded above and reconstruct the transposition factors as words in a prescribed acting group whenever such membership is justified. Determine quantitative bounds from the boundary data and the actual groups of germs, without assuming a uniform bound on commutator length.
\end{problem}

\section{Computations from boundary data}
\label{sec:certified}
The preceding sections reduced the computation of the first homology group to the study of finite-level defect matrices. In particular, the kernel of the defect map is eventually described by the kernels of integer matrices arising from the boundary data. This section turns that description into explicit algebraic information using Smith reduction. The resulting decomposition separates the free and torsion parts of the homological invariants and provides a finite procedure for computing the abelian quotients appearing later.

Theorem~\ref{thm:finite-level-defect} gives eventual stabilization without a stopping test. We obtain such a test for an eventually periodic AF core from a complete relative presentation and explicit defect vectors. In the finite-state examples, terminating searches also recover the exceptional arrows and the local bisections used in boundary cancellation.

A \emph{certificate} consists of input data whose validity has been established independently. Periodicity of the incidence matrices and completeness of the germ relations require separate proofs. Smith normal form will be used to keep the calculation integral. Recall that integer row and column operations by unimodular matrices reduce an integer matrix to a diagonal matrix whose nonzero diagonal entries divide the following ones. The change-of-basis matrices have integer inverses, equivalently determinant $\pm1$. See, for example, \cite{Newman72}. The examples below include a realized nonzero defect, a finite-state action with a nonabelian group of germs, and the model of Example~\ref{exmp:depth-counterexamples} with an infinite cyclic group of germs and infinitely many states.

\subsection{A periodic certificate for the degree-one computation}
Choose relative generators $z_1,\ldots,z_m$ represented by boundary arrows or integral sums of boundary arrows. The free abelian group on these generators is $\mathbb{Z}^m$. The assumption of periodicity in this subsection is used only to simplify the
finite-level computation and to make the boundary contribution explicit. The
general results above do not require the AF core to have a periodic tail.

\begin{defn}[Periodic boundary certificate]
\label{def:periodic-boundary-certificate}
A \emph{periodic boundary certificate} at level $N$ consists of the following data, whose validity must be established independently.
\begin{enumerate}
    \item There is a complete abelian presentation
\begin{equation}
\label{eq:certified-relative-presentation}
H_1(\G,\T)\cong\mathbb{Z}^m/B\mathbb{Z}^s,
\qquad B\in M_{m\times s}(\mathbb{Z}),
\end{equation}
where the columns of $B$ generate the full subgroup of relations among $z_1,\ldots,z_m$.

\item Isolating bisections and cylinder decompositions at level $N$ have been exhibited so that the $j$th column of a matrix $D\in M_{r\times m}(\mathbb{Z})$, where $r=|V_N|$, is the level-$N$ defect vector of $z_j$. Thus, for $u\in\mathbb{Z}^m$, the defect $\delta([u])$ is the direct-limit class of $Du\in\mathbb{Z}^{V_N}$. 

\item  A period $p$ from level $N$ onward has been proved, together with fixed identifications $V_{N+kp}\cong V_N$ for $k\geq0$, so that the one-period incidence map is
\begin{equation}
\label{eq:period-matrix}
P=M_{N+p-1}\cdots M_N\in M_r(\mathbb{Z})
\end{equation}
and the same map occurs at every subsequent period endpoint.
\end{enumerate}

\end{defn}

Completeness is essential in each part of Definition~\ref{def:periodic-boundary-certificate}. A finite boundary quotient supplies relative generators, but a matrix $B$ is certified only after all abelianized germ relations have been proved. For example, a complete presentation of every $H_{\mathcal{O}}^{\ab}$ together with the free transport generators gives the first part. No finite presentation of the nonabelian groups of germs is required. Similarly, agreement of finitely many displayed incidence matrices does not prove future periodicity. The rule producing the periodic tail is part of the certificate.

The relation matrix and the defect matrix need not be compatible at the initial integer level. A relation column $b$ satisfies $[Db]=0$ in the direct limit since it represents zero in $H_1(\G,\T)$, but the vector $Db\in\mathbb{Z}^{V_N}$ itself may be nonzero and become zero only after several incidence maps. 
\keepwithstatement
\begin{lemma}
\label{lemma:periodic-death}
Let $P\in M_r(\mathbb{Z})$ and let $\nu$ be the least nonnegative integer satisfying $\operatorname{rank}_{\mathbb{Q}}P^\nu=\operatorname{rank}_{\mathbb{Q}}P^{\nu+1}$. Then $\nu\leq r$, and a vector $v\in\mathbb{Z}^r$ represents zero in $\varinjlim(\mathbb{Z}^r,P)$ if and only if $P^\nu v=0$. For the periodic tail~\eqref{eq:period-matrix}, the same test decides vanishing of a class represented at level $N$.
\end{lemma}
\begin{proof}
The kernels of the powers of $P$ over $\mathbb{Q}$ increase. Equality of their dimensions at consecutive powers implies equality of those kernels. Once $\ker P^k=\ker P^{k+1}$, a vector killed by $P^{k+2}$ has image under $P$ in $\ker P^{k+1}=\ker P^k$ and thus already belongs to $\ker P^{k+1}$. Induction proves permanent stabilization. There can be at most $r$ strict increases. Intersecting with $\mathbb{Z}^r$ gives the integral assertion. A vector has zero direct-limit class exactly when some later map kills it. In a periodic system, a vector that becomes zero at an intermediate level remains zero at the next level $N+kp$. Thus the same test applies.
\end{proof}

Lemma~\ref{lemma:periodic-death} gives a finite test for vanishing: the class of $v$ is zero precisely when $P^\nu v=0$. Applying the lemma to all columns of $D$ gives an explicit stabilization level in Theorem~\ref{thm:finite-level-defect}.
\keepwithstatement
\begin{thm}
\label{thm:certified-periodic}
Assume a periodic boundary certificate as in Definition~\ref{def:periodic-boundary-certificate}, and put $A=P^\nu D$. Then $AB=0$ and
\begin{equation}
\label{eq:certified-H1}
H_1(\G)\cong\ker_{\mathbb{Z}}A/B\mathbb{Z}^s.
\end{equation}
\end{thm}
Therefore, level $N+\nu p$, with $\nu\leq r$, suffices to decide whether any combination of these defect columns is zero in the dimension group, including combinations representing torsion relations. Smith normal form computes the invariant factors and explicit coordinates of~\eqref{eq:certified-H1}.
\begin{proof}
Each column of $B$ is a genuine relative relation, and thus the corresponding column of $DB$ has zero direct-limit class. Lemma~\ref{lemma:periodic-death} gives $P^\nu DB=0$. For $u\in\mathbb{Z}^m$, the same lemma shows that $\delta([u])=0$ precisely when $Au=0$. Taking the inverse image of $\ker\delta$ in $\mathbb{Z}^m$ thus gives $\ker_{\mathbb{Z}}A$, and its relation subgroup is $B\mathbb{Z}^s$. The injection in~\eqref{eq:relative} identifies this quotient with $H_1(\G)$.

For explicit coordinates, choose unimodular matrices $U,V$ with $UAV$ in Smith normal form and with its $\rho=\operatorname{rank}A$ nonzero diagonal entries placed first. Then the last $m-\rho$ columns of $V$ form an integral basis matrix $W$ for $\ker_{\mathbb{Z}}A=\{u\in\mathbb{Z}^m:Au=0\}$. Since $AB=0$, the first $\rho$ rows of $V^{-1}B$ vanish. Let $C$ be its remaining rows. Then $B=WC$ and
\begin{equation}
\label{eq:certified-lattice-quotient}
H_1(\G)\cong\mathbb{Z}^{m-\rho}/C\mathbb{Z}^s.
\end{equation}
A second Smith reduction gives its invariant factors. For $u\in\ker_{\mathbb{Z}}A$, its coordinates before this second reduction are the last $m-\rho$ entries of $V^{-1}u$. These formulas also retain lifts of the computed homology generators to the original boundary generators.
\end{proof}

The integral basis is important. Clearing denominators independently in a rational nullspace basis can produce a proper sublattice of the integer kernel. For instance, for $A=(2\ \ 1\ \ 1)$ the two vectors $(-1,2,0)^{\mathsf{t}}$ and $(-1,0,2)^{\mathsf{t}}$ span a proper sublattice of the integer kernel: they miss $(-1,1,1)^{\mathsf{t}}$. A subgroup $K\leq\mathbb{Z}^m$ is \emph{saturated} if $nv\in K$, with $v\in\mathbb{Z}^m$ and a nonzero integer $n$, implies $v\in K$. The kernel lattice must be saturated, while the relation lattice $B\mathbb{Z}^s$ must \emph{not} be replaced by its saturation, since doing so can erase torsion.

\begin{rmk}
\label{rmk:certified-input-scope}
Theorem~\ref{thm:certified-periodic} applies to $B,D,P$ with the relative presentation, bisections, and incidence rule required by Definition~\ref{def:periodic-boundary-certificate}. Once the exceptional germs of an element of the full group have been listed, their abelianized normalized returns and transport coordinates give a vector in $\mathbb{Z}^m$. The theorem computes the corresponding class in $H_1(\G)$. Finding this list and constructing representatives of the resulting classes are separate problems, solved explicitly in the examples below. The periodic hypothesis concerns the fixed AF core, not the number of states of the added boundary maps.
\end{rmk}

The following matrix illustrates why the initial level may not suffice. The positive matrix
\begin{equation}
\label{eq:nilpotent-test-matrix}
P=\begin{pmatrix}4&4&1\\2&2&5\\3&3&3\end{pmatrix}
\end{equation}
has ranks $3,2,1,1$ for $P^0,P,P^2,P^3$. For $d=(1,0,-1)^{\mathsf{t}}$, we have $Pd=(3,-3,0)^{\mathsf{t}}$ and $P^2d=0$. Thus $\nu=2$ is necessary for this column. If an abstract presentation has the relation $2z=0$ and initial column $d$, the relation column $2d$ is nonzero initially but vanishes at the certified level. This calculation illustrates the integer procedure. No realization by boundary bisections is asserted.

\subsection{A minimal stationary example with nonzero defect}
\label{subsec:two-cut-odometer}
We construct an explicit rank-two Bratteli diagram in which defects from two different exceptional orbits cancel. Let $V_n=\{0,1\}$ for $n\geq1$, with one root edge to each vertex. For each target vertex $v\in V_{n+1}$, order the three incoming edges by their sources as $v,1-v,v$. The incidence matrix is
\begin{equation}
\label{eq:two-cut-incidence}
M=\begin{pmatrix}2&1\\1&2\end{pmatrix}.
\end{equation}
A path can be recorded by its first vertex $\epsilon\in\{0,1\}$ and the successive positions $d_j\in\{0,1,2\}$ of its later edges in the incoming orders. Its vertices satisfy $v_{j+1}=v_j+d_j\pmod2$. This identifies the path space with $\{0,1\}\times\mathbb{Z}_3$, where $\mathbb{Z}_3=\varprojlim_n\mathbb{Z}/3^n\mathbb{Z}$ is the ring of $3$-adic integers and $z=\sum_{j\geq1}d_j3^{j-1}$. The order successor, extended at both maximal paths, is the homeomorphism
\begin{equation}
\label{eq:two-cut-odometer}
\tau(\epsilon,z)=(1-\epsilon,z+1).
\end{equation}
Indeed, the first nonmaximal edge is incremented and preceding digits $2$ are reset to $0$. Their parity is even, while the increment changes parity by one. The same formula extends the maximal path $(i,-1)$ to the minimal path $(1-i,0)$.

The two exceptional-orbit defects are nonzero individually but cancel in the index of the generating homeomorphism. Write $e_0,e_1$ for the standard vertex basis of $\mathbb{Z}^{V_n}$, using the stationary identification of the vertex sets.
\keepwithstatement
\begin{prop}
\label{prop:two-cut-defect}
Let $\G$ be the groupoid of germs of the action generated by~\eqref{eq:two-cut-odometer}, and let $\T$ be the tail groupoid of~\eqref{eq:two-cut-incidence}. Then $\mathfrak{G}$ is a minimal principal AF-by-discrete groupoid. There are exactly two exceptional orbits. Each is divided into two infinite tiles and has trivial group of germs. If $r_i$ denotes the relative class of $\gamma_i:(i,-1)\to(1-i,0)$, then
\begin{equation}
\label{eq:two-cut-data}
H_1(\G,\T)=\mathbb{Z}r_0\oplus\mathbb{Z}r_1,
\qquad
D=\begin{pmatrix}1&-1\\-1&1\end{pmatrix},
\qquad MD=D.
\end{equation}
The class $[e_0-e_1]\in H_0(\T)$ is nonzero and
\begin{equation}
\label{eq:two-cut-homology}
\delta(a r_0+b r_1)=(a-b)[e_0-e_1],\qquad
H_1(\G)\cong\mathbb{Z}(r_0+r_1),\qquad
\iota I(\tau)=r_0+r_1.
\end{equation}
\end{prop}
\begin{proof}
The translation by $(1,1)$ is minimal on $\mathbb{Z}/2\mathbb{Z}\times\mathbb{Z}_3$: the Chinese remainder theorem allows an integer iterate to meet every prescribed parity and residue class modulo $3^n$. It is free, since no nonzero integer is zero in $\mathbb{Z}_3$. Tail-equivalent paths differ by a power of the successor map on a common finite tower, and thus $\T$ is an open subgroupoid of $\G$. Away from the two maximal paths, $\tau$ is a prefix replacement on a sufficiently small cylinder. Its only non-AF germs are $\gamma_0,\gamma_1$, and the corresponding assertion for powers proves that $\G\setminus\T$ is discrete.

The two maximal paths lie in different $\tau$-orbits: an integer iterate preserving the coordinate $-1\in\mathbb{Z}_3$ is the zeroth iterate. In each of these orbits, the tail relation separates the negative and nonnegative integer positions at the cut. Points on each side eventually have the same maximal or minimal tail respectively, whereas those two tails are distinct. Hence there are two infinite tiles in each exceptional orbit. There are no other exceptional orbits, and principality makes the groups of germs trivial.

Restrict $\tau$ to the cylinder of the maximal length-$n$ path ending at $i$. Its image is the cylinder of the minimal length-$n$ path ending at $1-i$. This restriction isolates $\gamma_i$, and thus its defect vector is $e_i-e_{1-i}$ at every level. This proves~\eqref{eq:two-cut-data}. The matrix $M$ is injective and fixes $e_0-e_1$, so that vector has a nonzero direct-limit class. The relative exact sequence gives~\eqref{eq:two-cut-homology}. Finally, $\tau$ crosses both cuts once in the positive direction, giving its stated index.
\end{proof}

Here the stopping certificate is immediate: $P=M$ is invertible over $\mathbb{Q}$, and thus $\nu=0$. The integer kernel of $D$ is $\mathbb{Z}(1,1)^{\mathsf{t}}$. Neither boundary generator has zero defect, and neither can occur as the sole exceptional arrow of an element of the full group, by Lemma~\ref{lemma:AF-clopen-completion}. Their sum is realized by the explicitly given $\tau$. A nonzero relative homology class supported on one exceptional orbit may fail to be realizable by itself, although it becomes realizable after adding a class supported on another exceptional orbit. The required cancellation takes place in the dimension group, not merely between the endpoint tiles in the boundary quotient.

The normalized invariant state takes both cylinder generators at vertex level $n$ to $1/(2\cdot3^{n-1})$. Thus $[e_0-e_1]$ is a nonzero infinitesimal, as required by Proposition~\ref{prop:defect-infinitesimal}. Taking the quotient by its cyclic subgroup gives $H_0(\G)\cong\mathbb{Z}[1/3]$, with the unit corresponding to $2$ under the identification given by the sum of the two coordinates at level $1$. Equivalently, normalizing the unit to $1$ identifies the group with $\frac12\mathbb{Z}[1/3]$. This computation also follows by taking direct limits of $\mathbb{Z}^2/\mathbb{Z}(e_0-e_1)$, on which $M$ induces multiplication by $3$.

The construction works with any odd integer $q\geq3$: take an alternating incoming source list of length $q$, giving diagonal entries $(q+1)/2$ and off-diagonal entries $(q-1)/2$, and use $(\epsilon,z)\mapsto(1-\epsilon,z+1)$ on $\mathbb{Z}/2\mathbb{Z}\times\mathbb{Z}_q$, where $\mathbb{Z}_q=\varprojlim_n\mathbb{Z}/q^n\mathbb{Z}$ is the group of $q$-adic integers. These models realize nonzero boundary defects and cancellation between distinct exceptional orbits. Compare the source--range incidence calculation in \cite[Theorem~7.5]{BNS21}, whose connecting map belongs to a different exact sequence.

\subsection{A finite-state model with a nonabelian group of germs}
\label{subsec:finite-state-certificate}
Let $\mathsf{X}=\{0,1,2,3\}$ and $\Omega(\mathsf{B})=\mathsf{X}^{\omega}$, with the first letter read as the least significant base-four digit. Let $a$ be the base-four odometer. Let $H=\operatorname{Sym}(\{1,2,3\})$, acting on $\mathsf{X}$ while fixing $0$. For $h\in H$, let $\sigma_h$ be the finitary rooted permutation $\sigma_h(cw)=h(c)w$ of the first letter, and define $b_h$ recursively by
\begin{equation}
\label{eq:directed-S3-recursion}
b_h(0w)=0b_h(w),\qquad b_h(1w)=1\sigma_h(w),\qquad
b_h(2w)=2w,\qquad b_h(3w)=3w.
\end{equation}
Write $s=(12)$, $t=(23)$, and $G=\langle a,b_s,b_t\rangle$. The states $a,b_s,b_t,\sigma_s,\sigma_t,1$ constitute a finite automaton. Inverses are also finite-state. Let $\G$ be the groupoid of germs of this action, with the tail groupoid of the Bratteli diagram having one vertex and four edges at each level as its AF core. For its boundary quotient, take $\mathcal{S}$ to be the full bisections of $a,a^{-1},b_s,b_t$.

The boundary calculation separates transport by the odometer from the abelianization of the nonabelian group of germs.
\keepwithstatement
\begin{prop}
\label{prop:directed-S3-index}
The AF inclusion above is minimal and AF-by-discrete. It has one exceptional orbit with two infinite tiles and group of germs isomorphic to $H\cong\operatorname{Sym}(3)$. A complete presentation of this group of germs is
\begin{equation}
\label{eq:directed-S3-germ-presentation}
\langle s,t\mid s^2=t^2=(st)^3=1\rangle.
\end{equation}
Using $a$ for the positive transport coordinate, we have
\begin{equation}
\label{eq:directed-S3-H1}
H_1(\G)\cong\mathbb{Z}\oplus C_2,\qquad
I(a)=(1,0),\qquad I(b_s)=I(b_t)=(0,1).
\end{equation}
Every class $(m,\bar\epsilon)$ is represented by $a^m b_s^\epsilon$, where $\epsilon\in\{0,1\}$. For a word in $a^{\pm1},b_s,b_t$, its index is its signed exponent sum in $a$ together with the parity of its total number of $b_s,b_t$ letters.
\end{prop}
\begin{proof}
Recursion~\eqref{eq:directed-S3-recursion} gives $b_hb_k=b_{hk}$ on every finite level and thus on $\Omega(\mathsf{B})$. For $h\neq1$, the action on each $[0^n1]$ contains the nontrivial permutation $\sigma_h$, and thus $b_h$ has a nontrivial germ at $\xi=0^\omega$. Hence, $h\mapsto(b_h,\xi)$ is injective. At every point different from $\xi$, the recursion eventually exits the $0$-spine and becomes finitary, so that germ is AF. The only non-AF germ of $a$ is its carry from $3^\omega$ to $0^\omega$. Every equal-length prefix replacement is locally a power of $a$, and thus the full tail groupoid is contained in $\G$.

The boundary quotient $Q_{\mathcal{S}}(\mathcal{O})$ of Definition~\ref{def:boundary-quotient} thus has two vertices, the tail classes of $0^\omega$ and $3^\omega$, one transport edge, and the loops $b_s,b_t$ at $0^\omega$, up to inversion. Proposition~\ref{prop:germ-fg} shows that these loops generate the entire group of germs. Their embedded copy of $H$ is consequently the whole group. Two involutions whose product has order three have at most six alternating normal forms, and thus the displayed relations give a group of order at most six. The faithful permutations $s,t$ realize all six, proving completeness of~\eqref{eq:directed-S3-germ-presentation}.

Each of the three chosen global generators $a,b_s,b_t$ has exactly one exceptional arrow, while the corresponding full bisection has source and range equal to all of $\Omega(\mathsf{B})$. Hence the defect of that unique exceptional arrow is zero by Definition~\ref{def:defect-map}. The relative normal form gives $H_1(\G,\T)\cong\mathbb{Z}\oplus H^{\ab}$, where the free factor is transport between the two infinite tiles, and $\delta=0$ on both factors. Since $H^{\ab}\cong C_2$ via permutation parity, this proves~\eqref{eq:directed-S3-H1} and the formula for words. Minimality follows from the odometer. Finally, the displayed representatives have the claimed indices by additivity.
\end{proof}

The matrices for the relative presentation and the defect, with coordinates ordered as transport, $s$, $t$, are
\begin{equation}
\label{eq:directed-S3-matrices}
P=(4),\qquad D=(0\ \ 0\ \ 0),\qquad
B=\begin{pmatrix}0&0&0\\2&0&3\\0&2&3\end{pmatrix}.
\end{equation}
The nonzero Smith factors of $B$ are $1,2$. This presentation keeps the free and torsion coordinates together while computing the defect. The certificate has $\nu=0$ and gives $\mathbb{Z}\oplus C_2$ exactly. In this example, the groups of germs are nonabelian, although their homological contribution is the one parity coordinate. For every $h\in H$, the map $b_h$ is the identity on the cylinders $[0^n2]$, which accumulate at $0^\omega$. Hence the singular point is purely non-Hausdorff. The finite-state computation applies equally to this non-Hausdorff groupoid of germs.

\begin{cor}
\label{cor:finite-permutation-returns}
Let $H$ be a finite group acting faithfully on $q\geq2$ letters, numbered from $0$ to $q-1$. Use the $q$-ary odometer. For $h\in H$, define $b_h$ by $b_h(0w)=0b_h(w)$, $b_h(1w)=1\sigma_h(w)$, and $b_h(jw)=jw$ for $2\leq j\leq q-1$, where $\sigma_h$ is the rooted permutation induced by the given action of $H$. The resulting minimal AF-by-discrete action, which has finitely many states, has one exceptional orbit with two infinite tiles and group of germs $H$. Its certified data satisfy
\begin{equation}
\label{eq:finite-permutation-homology}
H_1(\G)\cong\mathbb{Z}\oplus H^{\ab},\qquad
I(a)=(1,0),\qquad I(b_h)=(0,[h]_{\ab}).
\end{equation}
A finite multiplication table for $H$ supplies a complete matrix of the abelian relations. The incidence matrix is $(q)$ and all defect columns vanish. Every class $(m,[h]_{\ab})$ has the representative in the full group $a^m b_h$.
\end{cor}
\begin{proof}
The proof of Proposition~\ref{prop:directed-S3-index} uses only the faithful finite permutation action on the continuation after the prefix $1$. The recursion gives an embedded copy of $H$ in the group of germs, and the two-vertex boundary quotient gives no additional return generators. The full bisections corresponding to the generators give zero defects. If we use one return generator for each $h\in H$, the multiplication table gives the complete abelian relation matrix through the relations $[b_h]+[b_j]-[b_{hj}]=0$ and $[b_1]=0$. We may eliminate redundant generators afterward. The claimed representatives follow from additivity of the index. A permutation $h$ may move the letter $0$. The identity $b_h(0w)=0b_h(w)$ nevertheless ensures that $b_h$ fixes $0^\omega$.
\end{proof}

\subsection{Reconstructing exceptional arrows from words}
The preceding example also permits an exact computation of the finite exceptional set, not only of its total homology class. The exceptional orbit identifies with $\mathbb{Z}\subset\mathbb{Z}_4$: nonnegative integers have eventually zero expansions and negative integers have eventually $3$ expansions. At an integer $n\neq0$, let $k=v_4(n)=\max\{j\geq0:4^j\text{ divides }n\}$, let $d$ be the remainder of $n/4^k$ modulo $4$, and let $c$ be the remainder of $\lfloor n/4^{k+1}\rfloor$ modulo $4$. Take $d,c\in\{0,1,2,3\}$. Then
\begin{equation}
\label{eq:integer-directed-action}
b_h(n)=
\begin{cases}
n+4^{k+1}(h(c)-c),&d=1,\\
n,&d\neq1,
\end{cases}
\qquad b_h(0)=0.
\end{equation}
All operations use integers, including for negative $n$. Also, $a(n)=n+1$.

Write words in the alphabet $\{a,a^{-1},b_s,b_t\}$. Let $w=s_\ell\cdots s_1$ act by applying $s_1$ first, put $w_0=\Id$, and put $w_j=s_j\cdots s_1$ for $1\leq j\leq\ell$. For a generator $r$, let $E(r)$ denote the set of integers at which its germ is not in $\T$. In this model,
$E(a)=\{-1\}$, $E(a^{-1})=\{0\}$, and $E(b_s)=E(b_t)=\{0\}$. Therefore every exceptional source of the word must lie in the finite candidate set
\begin{equation}
\label{eq:candidate-exceptional-set}
\operatorname{Cand}(w)=\bigcup_{j=1}^\ell w_{j-1}^{-1}(E(s_j)).
\end{equation}
Each preimage is computed exactly from~\eqref{eq:integer-directed-action} by applying the inverse of the corresponding prefix word. While following the trajectory from a candidate integer $n$, attach the label $h$ exactly when a $b_h$ step is taken at $0$, and attach the identity label to every other step. Multiplying these labels in the order of germ composition gives $h_{w,n}\in H$. This gives a finite list of triples $(n,w(n),h_{w,n})$. Proposition~\ref{prop:directed-boundary-algorithm} determines which of them represent germs outside $\T$. To obtain the boundary current of Definition~\ref{def:boundary-current}, retain those triples, abelianize their return labels, and record their transport coordinates.
\keepwithstatement
\begin{prop}
\label{prop:directed-boundary-algorithm}
For the directed model, $\Sigma(w)$, defined in Definition~\ref{def:exceptional-support}, is exactly the subset of $\operatorname{Cand}(w)$ in~\eqref{eq:candidate-exceptional-set} consisting of integers $n$ such that either $n$ and $w(n)$ have different signs in the partition $\mathbb{Z}_{<0}\sqcup\mathbb{Z}_{\geq0}$, or $h_{w,n}\neq1$. Its transport entry is $1_{\mathbb{Z}_{\geq0}}(w(n))-1_{\mathbb{Z}_{\geq0}}(n)$ and its normalized return, in the sense of Subsection~\ref{subsec:boundary-returns}, is $h_{w,n}$. The procedure terminates and uses exact calculations with integers and finite permutations.
\end{prop}
\begin{proof}
A product of AF arrows is AF. Hence an exceptional arrow of a word has an exceptional generator step, and its source belongs to $\operatorname{Cand}(w)$ in~\eqref{eq:candidate-exceptional-set}. Choose base point $0$, connector $(a^{-1},0):0\to-1$, and translations in $\T$ within each of the two infinite tiles. The resulting arrow is $p_n=(a^n,0):0\to n$. Every odometer step has trivial normalized return, including the selected transport edge. Every $b_h$ step away from $0$ is AF, and the step at $0$ has normalized return $h$. These observations give the asserted product of labels by telescoping connectors along the word. An arrow belongs to the AF core exactly when its endpoints lie in the same infinite tile and its normalized return is the identity. This proves the test. There are at most $\ell$ candidates, and their trajectories use finitely many integer operations and permutations.
\end{proof}

The same procedure applies to a supplied element of the full group given by a finite clopen partition and words on its pieces. Refine the partition to cylinders and retain the candidates of each word only when they lie in its source cylinder. This gives the list of exceptional germs before abelianization. The sum of their classes in $H_1(\G,\T)$ gives the boundary current in the coordinates of Section~\ref{sec:index}. The list still retains nontrivial return germs whose abelianizations vanish. This distinction is essential: a nontrivial three-cycle germ lies in $[H,H]$ and hence has zero abelianized return coordinate, but it must not be erased from the exceptional support or from the constructive cancellation problem.

There is also a terminating reconstruction of the transposition factors of a zero-index element in this finite-state model. Integers in the same infinite tile have equal tails after a computably chosen level, giving explicit prefix matchings in $\T$. A label in $H$ is a word in $s,t$, and a label in $[H,H]\cong C_3$ is either the identity or a single commutator. To isolate the required germs, increase the cylinder depth and compute the sections of the finite-state automaton. After a common refinement, the sections under comparison are synchronous tree automorphisms. Equality of two such sections is decidable by exploring the finite product automaton: a state represents the identity exactly when every output permutation reachable from it is trivial. Once the marked germs have been canceled, the remaining homeomorphism has all germs in $\T$. For each point, there is therefore a cylinder on which the remainder is a tail prefix replacement. Compactness gives a finite cover by cylinders, and enumeration by increasing the depth must eventually reach a finite common refinement of this cover. Hence this remaining element of $\F(\T)$ is effectively recognized as a finite prefix permutation. These searches construct the certificates required by Theorem~\ref{thm:constructive-boundary-cancellation}.

\keepwithstatement
\begin{cor}
\label{cor:directed-effective-factorization}
For elements of the full group supplied as finite cylinder tables of words in $a^{\pm1},b_s,b_t$, the index, the list of exceptional germs, and a transposition factorization of every zero-index element can be computed by terminating procedures. In the zero-index case, we may take
\begin{equation}
\label{eq:directed-factor-bound}
\ell_{\Sym}(g)\leq3|\Sigma(g)|+8.
\end{equation}
Here $\ell_{\Sym}(g)$ is the least number of dynamical transpositions whose product is $g$. The factors are supplied as clopen tables in $\F(\G)$, not necessarily as words in $G$.
\end{cor}
\begin{proof}
The preceding constructions produce all the data required for the factorization. The action is minimal on an infinite Cantor space, and thus there are no singleton $\G$-orbits and Theorem~\ref{thm:constructive-boundary-cancellation} applies to every zero-index element. There is only one exceptional $\G$-orbit, and every element of the derived subgroup $[H,H]\cong C_3$ is either the identity or a single commutator. Thus $\sum c_{\mathcal{O}}\leq1$ in~\eqref{eq:constructive-length-bound}, giving $3|\Sigma(g)|+6+2$. The automaton and cylinder searches terminate for the reasons just given. No running-time bound is asserted, and the factors produced in the full group may lie outside the prescribed acting subgroup $G$.
\end{proof}

\subsection{An explicit eight-transposition identity}
\label{subsec:eight-transposition}
We can write the cancellation of a nontrivial commutator return without any search. For $h\in H$, define a dynamical transposition by
\begin{equation}
\label{eq:explicit-directed-transposition}
T_h(0w)=1b_{h^{-1}}(w),\qquad
T_h(1w)=0b_h(w),
\end{equation}
fixing $[2]\cup[3]$. Its two pieces are restrictions of the given partial homeomorphisms and prefix replacements in $\T$. It is an involution exchanging the disjoint clopen sets $[0]$ and $[1]$. Put $D_h=T_hT_1$. This acts as $b_h$ on the tail below $[0]$ and as $b_{h^{-1}}$ on the tail below $[1]$.

Let $c=[s,t]=sts^{-1}t^{-1}$ and $k=s^{-1}t^{-1}=st$, where the last equality uses that $s$ and $t$ are involutions. The exact action identities $b_hb_j=b_{hj}$ give
\begin{equation}
\label{eq:explicit-six-swap}
D_sD_tD_k=L_0(b_c),
\end{equation}
where $L_0(b_c)(0w)=0b_c(w)$ and $L_0(b_c)$ is the identity off $[0]$. At the other cylinder, the labels are $s^{-1}t^{-1}ts=1$. At $[0]$, their product is $c$. Let $P_h$ be the permutation of finite prefixes in $\F(\T)$ $P_h(1dw)=1h(d)w$, fixing the other cylinders of the first level. Recursion~\eqref{eq:directed-S3-recursion} yields $b_c=L_0(b_c)P_c$. Since $s,t$ are involutions and $(st)^3=1$, we have $c=ts$ and $P_c=P_tP_s$. Therefore,
\begin{equation}
\label{eq:explicit-eight-transposition}
[b_s,b_t]=T_sT_1T_tT_1T_{st}T_1P_tP_s.
\end{equation}
The six $T$ factors are the explicitly displayed dynamical transpositions. Since $t=(23)$ and $s=(12)$, the two prefix permutations $P_t$ and $P_s$ are themselves dynamical transpositions, swapping $[12]$ with $[13]$ and $[11]$ with $[12]$, respectively. It is an equality of homeomorphisms, not merely an equality of marked germs: the homomorphism $h\mapsto b_h$ makes all local identities exact. The left-hand side has one exceptional arrow, carrying the nontrivial three-cycle germ $c$ at $0^\omega$, but its index is zero since $c\in[H,H]$. In contrast to the singleton example of Section~\ref{sec:constructive-kernel}, this commutator germ can be cancelled using the auxiliary point $10^\omega$ in the same infinite tile.

\subsection{Stationary incidence does not force finite groups of germs}
\label{subsec:infinite-cyclic-certified}
Return to the infinite cyclic choice in Example~\ref{exmp:depth-counterexamples}, with $N_j=3^j$ and $f_j$ the binary odometer truncated at depth $j$. Let $a$ be the binary odometer and $b$ the map defined in that example. Its group of germs at $0^\omega$ has presentation $H=\langle b\mid\ \rangle\cong\mathbb{Z}$: for every $m\neq0$, the order $2^j$ eventually fails to divide $m$, and thus $b^m$ remains nontrivial arbitrarily near $0^\omega$. Together with Proposition~\ref{prop:germ-fg}, this proves that the displayed presentation is complete. The two infinite tiles again have the tail-equivalence classes of $0^\omega$ and $1^\omega$ as their vertex sets, and both global generators have one exceptional arrow and zero defect. The complete input is
\begin{equation}
\label{eq:infinite-cyclic-certified-data}
P=(2),\qquad B\in M_{2\times0}(\mathbb{Z}),\qquad D=(0\ \ 0),
\qquad H_1(\G)\cong\mathbb{Z}^2.
\end{equation}
The two coordinates are the exponent sums in $a,b$, and $a^m b^n$ realizes $(m,n)$. The stabilization certificate is again $\nu=0$. The finite geometric description and the proof using divisibility of the orders supply the germ presentation. The sections $b|_{0^{N_j}1}=f_j$ have unbounded finite orders. The displayed automorphism is infinite-state.

The finite and infinite examples both have stationary AF cores and zero defect, while their groups of germs differ. This isolates the two inputs of Definition~\ref{def:periodic-boundary-certificate}: periodic incidence determines a level at which vanishing of the defect in the dimension group can be decided, whereas the size and relations of the groups of germs determine the relative presentation. Neither input determines the other, and finite-state behavior of the boundary maps is a third, separate issue.

The germ sums and transport maps can be compared with \cite[Section~3]{BHM24}, the source--range incidence with \cite{BNS21}, and homological discretisation with \cite{LiMiller26}. Here the relative presentation and defect matrices come with  concrete compact open bisections representing the exceptional germs, so the calculation also tests realization by full bisections.

For other models, computing the exceptional set and proving complete germ relations remain separate tasks. An eventually periodic diagram and a finite automaton alone do not supply all the data of Definition~\ref{def:periodic-boundary-certificate}. The next section treats higher homology and parity.

\section{Higher homology, parity, and the abelianization}
\label{sec:parity}

The index map does not by itself determine the abelianization of the full group. We also need the secondary differential, which determines the surviving parity classes, and the resulting abelian extension of $H_1(\G)$. We first compute higher groupoid homology and then address these two questions.

Under minimality and comparison, the part of Li's ordinary AH sequence relevant here is the tail of the exact sequence
\begin{equation}
\label{eq:AH-parity-sequence}
H_2(\G)\xrightarrow{\theta_{\G}}H_0(\G;\mathbb{Z}/2\mathbb{Z})
\xrightarrow{j_{\G}}\F(\G)_{\ab}
\xrightarrow{I_{\ab}}H_1(\G)\longrightarrow0,
\end{equation}
obtained from \cite[Corollary~6.14]{Li25}. The preceding term, suppressed in \eqref{eq:AH-parity-sequence}, is $H_2(\Der(\G))$. By \cite[Remark~6.16]{Li25}, the strong AH property is equivalent to the vanishing of the displayed map $\theta_{\G}$. The map $j_{\G}$ sends $[1_{\sg(V)}]$ to the class of $t_V$ in the abelianization, for a compact open bisection $V$ with disjoint source and range. Thus $\theta_{\G}$ is exactly the obstruction to retaining all of $H_0(\G;\mathbb{Z}/2\mathbb{Z})$ as parity.

Li also proves a stable sequence. Let $\mathfrak{R}$ be the discrete pair groupoid on $\mathbb{N}$, with one arrow $(i,j)$ from $j$ to $i$. The stabilized full group $\F(\mathfrak{R}\times\G)$ is the union of the full groups on $\{1,\ldots,m\}\times\Omega(\mathsf{B})$, extended by the identity on all other copies of $\Omega(\mathsf{B})$. Li's Theorem~6.12 gives the corresponding stable low-degree exact sequence in the compact-open bisection setting with $\F(\mathfrak{R}\times\G)$. No minimality or comparison is required. Corollary~6.14 identifies the stable and ordinary sequences under minimality, comparison, and the absence of isolated units. The nonminimal examples in Section~\ref{sec:finite-return-relations} use this stable sequence without the ordinary full-group comparison.

\subsection{The discrete relative nerve in every degree}
Here we assume only the AF-by-discrete inclusion. We write $\G^{\delta}$ and $\T^{\delta}$ for the same groupoids with the discrete topology. The superscript is essential: the absolute homology of the discrete groupoid need not agree with the homology of the original topological groupoid. We call a composable $n$-tuple \emph{exceptional} when at least one of its coordinates lies outside $\T$. Equivalently, it lies in $\G^{(n)}\setminus\T^{(n)}$. For an abelian group $A$ with trivial $\G$-action, $A^{(\G^{(n)}\setminus\T^{(n)})}$ denotes the direct sum of copies of $A$ indexed by the exceptional set, equivalently the group of finitely supported $A$-valued functions on that set. Restriction to exceptional tuples identifies the relative complex with its discrete counterpart, orbit by orbit.
\keepwithstatement
\begin{lemma}
\label{lem:relative-nerve}
The set $\G^{(n)}\setminus\T^{(n)}$ is closed and discrete in $\G^{(n)}$, and restriction induces an isomorphism of relative chain complexes
\begin{equation}
\label{eq:relative-nerve-isomorphism}
\frac{C_c(\G^{(*)};A)}{C_c(\T^{(*)};A)}
\ \cong\
\frac{C_c((\G^{\delta})^{(*)};A)}{C_c((\T^{\delta})^{(*)};A)}.
\end{equation}
In degree zero, both sides are zero. In degree $n\geq1$, both sides identify with $A^{(\G^{(n)}\setminus\T^{(n)})}$.
\end{lemma}
\begin{proof}
The set $\T^{(n)}$ is open, and thus its complement is closed. Fix an exceptional tuple $(\gamma_1,\ldots,\gamma_n)$. Choose a compact open bisection $V_i$ containing each $\gamma_i$, contained in $\T$ when $\gamma_i\in\T$ and isolating $\gamma_i$ otherwise. Put $W=(V_1\times\cdots\times V_n)\cap\G^{(n)}$. For any exceptional tuple in $W$, a coordinate outside $\T$ must equal the prescribed exceptional arrow. Composability and injectivity of source and range on the neighboring bisections then determine every other coordinate. Thus $W$ contains no other exceptional tuple, proving discreteness.

A set $W$ of this form is compact, open, and Hausdorff, since its last source map identifies it with a compact open subset of $\Omega(\mathsf{B})$. Such sets form a basis for $\G^{(n)}$, and each meets the closed discrete set $\G^{(n)}\setminus\T^{(n)}$ in finitely many tuples. Restriction therefore maps the compactly supported chain group to $A^{(\G^{(n)}\setminus\T^{(n)})}$. The indicator of a set containing exactly one exceptional tuple restricts to the indicator of that tuple, so restriction is onto.

We verify the kernel without assuming that the arrow space is Hausdorff. Write a chain as a finite linear combination $\sum_i a_i1_{W_i}$, where each $W_i$ has the form just described. Fix an exceptional tuple $\gamma$ at which the restriction vanishes, and let $I_\gamma=\{i:\gamma\in W_i\}$. Choose a set $W$ of the same form containing $\gamma$, contained in every $W_i$ with $i\in I_\gamma$, and containing no other exceptional tuple in the finite union of the $W_i$. Since $W$ is compact and each $W_i$ is Hausdorff, $W$ is clopen in every $W_i$ containing it. The coefficient of $\gamma$ is zero, so $\sum_{i\in I_\gamma}a_i=0$. Replacing each $a_i1_{W_i}$ with $i\in I_\gamma$ by $a_i1_{W_i\setminus W}$ therefore leaves the total chain unchanged. Each $W_i\setminus W$ is compact and open, and the chosen exceptional tuple has been removed from every summand containing it. Repeating this over the finitely many exceptional tuples in the original presentation gives a presentation supported in $\T^{(n)}$. Hence the kernel is exactly $C_c(\T^{(n)};A)$.

Finally, every face of a tuple all of whose coordinates lie in $\T$ again lies in $\T$, since $\T$ is a subgroupoid. Hence, if $\eta$ is an exceptional face, every preimage of $\eta$ under a face map is exceptional. For a compactly supported chain only finitely many such preimages contribute, and restriction therefore commutes with every push-forward $(d_i)_*$: on an exceptional face both sides are the same finite sum over its exceptional preimages. Thus restriction is a chain map, proving \eqref{eq:relative-nerve-isomorphism}. Every composable tuple lies in one $\G$-orbit of units, and every face remains in that orbit. Since every relative chain has finite support on exceptional tuples, the resulting decomposition is the algebraic direct sum of the orbitwise subcomplexes, even when there are infinitely many exceptional orbits.
\end{proof}

\begin{thm}[Higher-homology theorem]
\label{thm:higher-germ-homology}
Let $\T\subseteq\G$ be an AF-by-discrete inclusion in the sense of Definition~\ref{def:AF-by-discrete}. For every abelian group $A$ with trivial $\G$-action and every $n\geq2$, there is an isomorphism
\begin{equation}
\label{eq:higher-germ-homology}
H_n(\G;A)\cong\bigoplus_{\mathcal{O}}H_n(H_{\mathcal{O}};A),
\end{equation}
where the sum is over exceptional $\G$-orbits. On the right, $A$ is regarded as an $H_{\mathcal{O}}$-module with trivial action. Choosing a different basepoint or connecting arrow in an orbit changes the identification only by the usual conjugacy identification in group homology.
\end{thm}
No compact generation, minimality, bounded type, or finiteness of the groups of germs is needed.
\begin{proof}
The relative long exact sequence of the pair $(\G,\T)$ contains, for $n\geq2$, the segment
\[
H_n(\T;A)\longrightarrow H_n(\G;A)\longrightarrow
H_n(\G,\T;A)\longrightarrow H_{n-1}(\T;A).
\]
Since $\T$ is an AF groupoid, its homology vanishes in positive degrees.
Therefore, both outer terms are zero, and the middle arrow induces an
isomorphism
\[
H_n(\G;A)\cong H_n(\G,\T;A).
\]
Thus, it remains to compute the relative homology.

By Lemma~\ref{lem:relative-nerve}, the relative chain complex can be identified with the discrete relative complex
\[
\frac{C_c((\G^\delta)^{(*)};A)}
{C_c((\T^\delta)^{(*)};A)}.
\]
Moreover, every composable tuple belongs to a unique $\G$-orbit of units, and the face maps preserve this orbit. Hence the relative chain complex
decomposes as the algebraic direct sum of orbitwise relative complexes:
\[
C_\bullet(\G,\T;A)\cong\bigoplus_{\mathcal O}C_\bullet((\G|_{\mathcal O})^\delta,(\T|_{\mathcal O})^\delta;A).
\]
Since every chain has finite support, this direct sum decomposition is
compatible with homology.

Fix an exceptional orbit $\mathcal O$. We first analyze the two groupoids
appearing in the corresponding relative complex. Since $\T$ is AF and
principal, the restriction $(\T|_{\mathcal O})^\delta$ is a disjoint union of pair groupoids, one for each $\T$-orbit contained in $\mathcal O$. The nerve of each pair groupoid is contractible. Choosing one object in every component therefore gives a contraction of the whole nerve, and thus
\[
H_k((\T|_{\mathcal O})^\delta;A)=0
\qquad (k\geq1).
\]

Next, consider $(\G|_{\mathcal O})^\delta$. This groupoid is transitive. Choose a basepoint $x\in\mathcal O$ and, for every $y\in\mathcal O$, choose an
arrow $p_y:x\to y$, with $p_x=(\Id,x)$. Let
$H_{\mathcal O}=\G_x^x$. Define
\[
q(\gamma:y\to z)=p_z^{-1}\gamma p_y\in H_{\mathcal O}.
\]
Then $q$ is a functor from $(\G|_{\mathcal O})^\delta$ to the one-object
groupoid associated with $H_{\mathcal O}$. Let
$i:H_{\mathcal O}\to(\G|_{\mathcal O})^\delta$ be the inclusion at the
chosen basepoint. We have $q\circ i=\Id_{H_{\mathcal O}}$. Moreover, the
chosen arrows $p_y$ define a natural transformation
\[
i\circ q\Rightarrow\Id_{(\G|_{\mathcal O})^\delta}.
\]
Hence the two groupoids have homotopy equivalent nerves, and therefore
\[
H_n((\G|_{\mathcal O})^\delta;A)
\cong
H_n(H_{\mathcal O};A).
\]

Now apply the relative long exact sequence to the orbitwise pair
\[
((\G|_{\mathcal O})^\delta,
(\T|_{\mathcal O})^\delta).
\]
Because the positive-degree homology of
$(\T|_{\mathcal O})^\delta$ vanishes, we obtain, for $n\geq2$,
\[
H_n((\G|_{\mathcal O})^\delta,
(\T|_{\mathcal O})^\delta;A)
\cong
H_n(H_{\mathcal O};A).
\]

If $\mathcal O$ is nonexceptional, then
$\G|_{\mathcal O}=\T|_{\mathcal O}$, so the corresponding relative complex is
zero. Therefore only exceptional orbits contribute. Taking the direct sum
over all exceptional orbits gives
\[
H_n(\G,\T;A)
\cong
\bigoplus_{\mathcal O}H_n(H_{\mathcal O};A).
\]
Combining this with the initial identification
$H_n(\G;A)\cong H_n(\G,\T;A)$ proves
\eqref{eq:higher-germ-homology}.

Finally, changing the basepoint or the chosen connecting arrows replaces the
identification with the composition by an inner automorphism of
$H_{\mathcal O}$. Since inner automorphisms act trivially on group homology
with trivial coefficients, the resulting identification is canonical up to the
usual conjugacy identification.
\end{proof}

The theorem is the higher-degree counterpart of the degree-one normal form, but the outcome is simpler. To see exactly where transport disappears, recall that $\mathcal{T}(\mathcal{O})$ is the set of $\T$-orbits contained in one exceptional orbit. With trivial coefficients,
$H_0((\T|_{\mathcal{O}})^\delta;A)\cong\bigoplus_{\mathcal{T}_i\in\mathcal{T}(\mathcal{O})}A$, while $H_0((\G|_{\mathcal{O}})^\delta;A)\cong A$, and the induced map is summation of the component coordinates. Its kernel is the transport term visible in relative degree one. For $n\geq2$, however, the adjacent AF term is positive-degree homology and vanishes. Thus no transport group, finite or infinite rank, survives in higher degree. Only the group homology of the exceptional isotropy remains. In particular, principal AF-by-discrete groupoids satisfy $H_n(\G;A)=0$ for every $n\geq2$. For finite $H_{\mathcal{O}}$, the group $H_2(H_{\mathcal{O}};\mathbb{Z})$ is its Schur multiplier. It is groupoid homology must not be confused with the second homology of the topological full group.

\subsection{Coherent local actions and the parity obstruction}\label{subsec:coherent-evaluation}
A coherent action of a group $L$ on a compact open set gives a map from its classifying space to Li's infinite loop space. Naturality of the Atiyah--Hirzebruch spectral sequence then gives vanishing of the secondary differential on the image of $H_2(L;\mathbb{Z})$.

Write $\mathfrak{B}_{\G}$ for Li's permutative category of compact open bisections and put $\mathscr{E}_{\G}=\mathbb{K}(\mathfrak{B}_{\G})$. We use the construction of \cite[Sections~2.6 and~3]{Li25}. For a group $L$, let $BL$ be its classifying space, let $(BL)_+$ denote the addition of a disjoint basepoint, and let $\Sigma^\infty(BL)_+$ be the corresponding suspension spectrum. We write $\mathbb{S}$ for the sphere spectrum. Li's Theorem~4.18 identifies $\widetilde H_*(\mathscr{E}_{\G};A)$ with $H_*(\G;A)$, while the map $\theta_{\G}$ is the differential $d^2_{2,0}$ in the Atiyah--Hirzebruch spectral sequence used in the proof of \cite[Theorem~6.12]{Li25}. This definition of $\theta_{\G}$ requires neither minimality nor comparison. Those hypotheses enter only when Li's stable full-group sequence is transferred to the ordinary full group.

Let $U\subseteq \Omega(\mathsf{B})$ be compact open and let $\rho:L\to\F(\G|_U)$ be a group homomorphism. Write $V_h$ for the full bisection of $\rho(h)$ in $\G|_U$, and thus $\sg(V_h)=\rg(V_h)=U$. The homomorphism identity is exactly $V_hV_k=V_{hk}$. It therefore sends a simplex in the bar complex of $L$ to the corresponding set of composable groupoid arrows. On the unnormalized bar complex recalled in Subsection~\ref{subsec:groupoid-homology}, set
\begin{equation}
\label{eq:coherent-bar-evaluation}
\begin{split}
Q_{\rho,n}[h_1|\cdots|h_n]&=1_{V_{h_1}*\cdots*V_{h_n}},\\
Q_{\rho,0}[\,]&=1_U.
\end{split}
\end{equation}
Here $*$ denotes the set of composable tuples, not their product. The maps $Q_{\rho,n}$ form a chain map $Q_{\rho,*}$. Indeed, the bar face that multiplies two adjacent labels becomes the groupoid face that composes the corresponding adjacent arrows, and the first and last faces agree since every $V_h$ has source and range $U$. Degenerate group simplices map to degenerate groupoid chains, and thus the same map is available on normalized homology.

\begin{lemma}
\label{lem:coherent-parity}
For every homomorphism $\rho$ as above,
\begin{equation}
\label{eq:coherent-theta-zero}
\theta_{\G}\circ (Q_{\rho,2})_*=0:H_2(L;\mathbb{Z})\longrightarrow H_0(\G;\mathbb{Z}/2\mathbb{Z}).
\end{equation}
\end{lemma}
\begin{proof}
The object $U$ and its automorphisms $\rho(h)$ give a map $BL\to\Omega^\infty\mathscr{E}_{\G}$, in the component represented by $U$. By the suspension-spectrum adjunction, this gives a spectrum map $f_\rho:\Sigma^\infty(BL)_+\to\mathscr{E}_{\G}$.

We verify that this map induces the chain evaluation~\eqref{eq:coherent-bar-evaluation} under Li's homology comparison. For $\nu\geq1$, let $\mathfrak{C}_\nu=L\ltimes L^\nu$, where $L$ acts by left multiplication on the first coordinate. Put $\mathfrak{C}_0=L$, regarded as a groupoid with one object. The augmented face functors multiply consecutive coordinates or delete the last coordinate. At $\nu=1$, the augmentation forgets the object and retains the acting arrow. The augmented nerves of these groupoids give the usual bar resolution of $BL$. For $\nu\geq1$, each component of $\mathfrak{C}_\nu$ is a pair groupoid, and the components are indexed by the last $\nu-1$ coordinates. Hence their classifying spaces have no positive homology, and their degree-zero homology is the free abelian group on $L^{\nu-1}$.

Use the categories $\mathfrak{B}^{(\nu)}=\mathfrak{B}_{\G\curvearrowright\G^{(\nu)}}$ from \cite[Section~4.3]{Li25}. An object $(h_0,\ldots,h_{\nu-1})$ of $\mathfrak{C}_\nu$ is assigned the compact open set of composable tuples
\begin{equation*}
W_{h_0,\ldots,h_{\nu-1}}=V_{h_0}*\cdots*V_{h_{\nu-1}}.
\end{equation*}
The arrow given by $l\in L$ is assigned the bisection that left multiplies the first coordinate by $V_l$. Its range is $W_{lh_0,h_1,\ldots,h_{\nu-1}}$. At $\nu=0$, the unique object is assigned $U$ and the arrow $l$ is assigned $V_l$. Extending these assignments to finite families with separate labels gives functors to the bisection categories. The labels keep distinct summands separate even when the corresponding sets of tuples overlap.

These functors commute with the augmented face functors. For a multiplication face, this follows from $V_hV_k=V_{hk}$. For the last face, projection onto the remaining coordinates is a bijection between the corresponding sets of composable tuples. For the augmentation, its range is $U$. Hence, after passage to classifying spaces and suspension spectra, they give a comparison of the augmented complexes used in the proof of \cite[Theorem~4.18]{Li25}, with $f_\rho$ at the augmentation.

Under \cite[Theorem~4.14]{Li25}, the image of the component indexed by $(h_1,\ldots,h_n)$ is the indicator of $V_{h_1}*\cdots*V_{h_n}$: one drops the first coordinate of $W_{h_0,h_1,\ldots,h_n}$, and this projection is a bijection. The alternating faces on the resulting degree-zero complexes are precisely the bar differentials. Naturality of the connecting maps in the augmented resolutions thus gives, for every $n>0$, the commutative diagram
\begin{equation*}
\begin{CD}
\widetilde H_n(\Sigma^\infty(BL)_+;\mathbb{Z}) @>{(f_\rho)_*}>> \widetilde H_n(\mathscr{E}_{\G};\mathbb{Z})\\
@V{\cong}VV @V{\cong}VV\\
H_n(L;\mathbb{Z}) @>{(Q_{\rho,n})_*}>> H_n(\G;\mathbb{Z}).
\end{CD}
\end{equation*}
Here the right vertical map is the isomorphism constructed in the proof of \cite[Theorem~4.18]{Li25}, obtained from the same augmented complexes $\mathfrak{B}^{(\nu)}$. Thus the lower horizontal map is the map induced by $f_\rho$ under Li's comparison, rather than through an abstract homology isomorphism. In degree zero, the comparison sends the generator to $1_U$. In degree two, the three faces are $1_{V_k}-1_{V_{hk}}+1_{V_h}$. The construction in all degrees, rather than this face identity alone, identifies the induced homology map.

Suppose now that $N\leq L$ and that the restricted action lies in $\F(\T|_U)$. The same assignments, together with the open inclusion $\T\subseteq\G$, commute with the functors in \cite[Section~3.1.1]{Li25} and give the commutative square
\begin{equation*}
\begin{CD}
\Sigma^\infty(BN)_+ @>>> \Sigma^\infty(BL)_+\\
@V{f_{\rho|_N}}VV @VV{f_\rho}V\\
\mathscr{E}_{\T} @>>> \mathscr{E}_{\G}.
\end{CD}
\end{equation*}
Passing to cofibers gives the corresponding map of cofiber sequences. Under the comparison just described, the induced map on relative homology is obtained by quotienting the bar and groupoid chain complexes by the $N$- and $\T$-subcomplexes, respectively, and is therefore the relative map induced by $Q_{\rho,*}$. This is the compatibility used in Lemma~\ref{lem:marked-surface-parity}.

It remains to identify the differential on the source spectrum. The basepoint inclusion and the augmentation $BL_+\to S^0$ split $\Sigma^\infty(BL)_+$ as $\mathbb{S}\vee\Sigma^\infty BL$. Let $p$ be the retraction onto the sphere summand. Since $BL$ is connected, $p_*$ is an isomorphism on the target group $H_0(-;\mathbb{Z}/2\mathbb{Z})$, whereas the degree-two integral homology of the sphere spectrum is zero. Therefore, for $c\in H_2(L;\mathbb{Z})$, naturality gives
$p_*d^2_{2,0}(c)=d^2_{2,0}(p_*c)=0$.
Since $p_*$ is injective in degree zero, the source differential itself is zero. Naturality for the spectrum map $f_\rho$ then yields \eqref{eq:coherent-theta-zero}.
\end{proof}

Finite subgroups satisfy the required coherent-lifting condition automatically by Lemma~\ref{lem:finite-germ-coherent-lift}, proved in Section~\ref{sec:germs}.

The secondary differential thus vanishes on degree-two classes arising from finite subgroups of groups of germs.
\keepwithstatement
\begin{thm}
\label{thm:theta-finite-germs}
Suppose that, for every exceptional orbit $\mathcal{O}$, the group $H_2(H_{\mathcal{O}};\mathbb{Z})$ is generated by the images of $H_2(L;\mathbb{Z})$ over finite subgroups $L\leq H_{\mathcal{O}}$. Then $\theta_{\G}=0$. In particular, $\theta_{\G}=0$ if all groups of germs are finite.
\end{thm}
\begin{proof}
By the higher-homology decomposition of Theorem~\ref{thm:higher-germ-homology},
the group $H_2(\G;\mathbb Z)$ is generated by the images of the groups
$H_2(H_{\mathcal O};\mathbb Z)$ over exceptional orbits $\mathcal O$.
Hence it is enough to show that $\theta_{\G}$ vanishes on each such summand.

Fix an exceptional orbit $\mathcal O$. By the hypothesis, $H_2(H_{\mathcal O};\mathbb Z)$ is generated by the images of
$H_2(L;\mathbb Z)$, where $L\leq H_{\mathcal O}$ ranges over finite
subgroups. Therefore every class in $H_2(H_{\mathcal O};\mathbb Z)$ is a
finite sum of classes induced from finite subgroup homology.

Let $L\leq H_{\mathcal O}$ be such a finite subgroup. By Lemma~\ref{lem:finite-germ-coherent-lift}, the inclusion of $L$ admits a
coherent local lift. The vanishing result of Lemma~\ref{lem:coherent-parity} then implies that $\theta_{\G}$ vanishes on the image of $H_2(L;\mathbb Z)$ in $H_2(H_{\mathcal O};\mathbb Z)$.

Since these images generate $H_2(H_{\mathcal O};\mathbb Z)$, the restriction of $\theta_{\G}$ to every exceptional-orbit summand is zero. Therefore, $\theta_{\G}$ vanishes on the whole of $H_2(\G;\mathbb Z)$.
\end{proof}

The first statement also applies to locally finite groups of germs, since group homology commutes with directed unions of finite subgroups. More generally, the same proof applies when coherent local lifts have degree-two images generating the group homology. Such coherence requires the relations to hold simultaneously on a common invariant clopen neighborhood. Lemma~\ref{lem:finite-germ-coherent-lift} supplies this for finite groups.

\subsection{Coherent realizations of returns beyond finite groups}
We now require a coherent local realization of the whole group of germs. The following definition also applies when that group is infinite.

\begin{defn}[Coherent realization of returns]
\label{def:coherent-return-realization}
Let $\mathcal{O}$ be an exceptional orbit, choose a basepoint $x\in\mathcal{O}$, and put $H_{\mathcal{O}}=\G_x^x$. A \emph{coherent realization of returns} of $H_{\mathcal{O}}$ is a clopen neighborhood $U\ni x$ and a homomorphism
\begin{equation}
\label{eq:coherent-return-realization}
\rho_{\mathcal{O}}:H_{\mathcal{O}}\longrightarrow\F(\G|_U)
\end{equation}
such that the germ of $\rho_{\mathcal{O}}(h)$ at $x$ is $h$ and $(\rho_{\mathcal{O}}(h),y)\in\T$ for every $y\in U\setminus\{x\}$.
\end{defn}
Extending by the identity outside $U$ gives a homomorphism $H_{\mathcal{O}}\to\F(\G)$, which we denote by the same symbol.

All multiplication relations must hold on the same invariant clopen neighborhood $U$. Choosing representatives independently on different neighborhoods does not give this property. Finite groups of germs admit such realizations by Lemma~\ref{lem:finite-germ-coherent-lift}. Infinite groups need not.
\keepwithstatement
\begin{thm}
\label{thm:coherent-return-split-AH}
Let $\G$ be a minimal effective AF-by-discrete groupoid with Cantor unit space and comparison. Suppose that every exceptional group of germs admits a coherent realization of returns. Then
\[\theta_{\G}=0 \addtag \label{eq:coherent-return-theta-zero}\]

and Li's AH sequence splits as a sequence of abelian groups:
\begin{equation}
\label{eq:coherent-split-AH}
\F(\G)_{\ab}\cong H_0(\G;\mathbb{Z}/2\mathbb{Z})\oplus H_1(\G).
\end{equation}
More precisely, in the connector coordinates of Theorem~\ref{thm:relative-normal-form}, every isotropy defect map vanishes. If
\begin{equation}
L=\bigoplus_{\mathcal{O}}L_{\mathcal{O}}
=\bigoplus_{\mathcal{O}}\ker\bigl(\mathbb{Z}[\mathcal{T}(\mathcal{O})]\to\mathbb{Z}\bigr),
\end{equation}
where $L_{\mathcal{O}}$ is the transport lattice of Definition~\ref{def:transport-lattice}, and $\tau:L\to H_0(\T)$ is the total transport defect, then
\begin{equation}
\label{eq:coherent-H1}
H_1(\G)\cong \bigoplus_{\mathcal{O}}H_{\mathcal{O}}^{\ab}\oplus\ker\tau
\end{equation}
and we may choose the splitting in~\eqref{eq:coherent-split-AH} so that its restriction to $\bigoplus_{\mathcal{O}}H_{\mathcal{O}}^{\ab}$ is induced by the coherent realizations of returns.
\end{thm}

\begin{proof}
Fix an exceptional orbit $\mathcal{O}$. By Theorem~\ref{thm:higher-germ-homology}, its contribution in degree two is $H_2(H_{\mathcal{O}};\mathbb{Z})$. Apply Lemma~\ref{lem:coherent-parity} to the homomorphism $\rho_{\mathcal{O}}$. Under the identification by the relative nerve, the induced map on homology in degree two is the inclusion of this whole orbit summand: every nondegenerate tuple of the coherent action is AF away from the single basepoint, while its exceptional tuple at the basepoint is exactly the corresponding tuple in $H_{\mathcal{O}}$. Hence $\theta_{\G}$ vanishes on every summand, proving~\eqref{eq:coherent-return-theta-zero}.

For $h\in H_{\mathcal{O}}$, the extended element of the full group $\rho_{\mathcal{O}}(h)$ has no exceptional source except $x$ and fixes $x$. The index formula thus gives
$I(\rho_{\mathcal{O}}(h))=([h]_{\ab},0)$
in the return--transport coordinates. Since every full-group index lies in $\ker\delta$, the isotropy defect of $[h]_{\ab}$ is zero. Equation~\eqref{eq:coherent-H1} follows. Passing $\rho_{\mathcal{O}}$ to abelianizations supplies a section on the corresponding summand of $\bigoplus_{\mathcal{O}}H_{\mathcal{O}}^{\ab}$.

The group $L$ is free abelian. Hence, its subgroup $\ker\tau$ is free abelian and has a basis. For each basis element choose a full-group element with that index. The surjectivity of the index map is the last arrow in Li's exact sequence. Sending the chosen basis to their classes in $\F(\G)_{\ab}$ extends uniquely to a homomorphism $\ker\tau\to\F(\G)_{\ab}$ and gives a section of $I_{\ab}$ on the transport summand. Together with the sections coming from the coherent actions on the isotropy summands, this produces a section of $I_{\ab}$ on all of $H_1(\G)$.

By the first part of the proof, $\theta_{\G}=0$, and thus exactness of \eqref{eq:AH-parity-sequence} makes $j_{\G}$ injective. We therefore obtain the direct-sum decomposition \eqref{eq:coherent-split-AH}.
\end{proof}

\subsection{A split abelianization theorem for finite groups of germs}
In the finite-group case, the preceding coherent-realization theorem can be written entirely in boundary data. Put
\begin{equation}
L=\bigoplus_{\mathcal{O}}\ker\bigl(\mathbb{Z}[\mathcal{T}(\mathcal{O})]\longrightarrow\mathbb{Z}\bigr),
\end{equation}
and let $\tau:L\to H_0(\T)$ be the transport defect in chosen connector coordinates. Since $\bigoplus_{\mathcal{O}}H_{\mathcal{O}}^{\ab}$ is torsion and $H_0(\T)$ is torsion-free, the isotropy defect vanishes. \begin{thm}
\label{thm:finite-germ-split-AH}
Let $\G$ be a minimal effective AF-by-discrete groupoid with Cantor unit space and comparison. Suppose that every group of germs is finite. Then
\begin{equation}
\label{eq:split-AH-finite-germs}
0\longrightarrow H_0(\G;\mathbb{Z}/2\mathbb{Z})
\xrightarrow{j_{\G}}\F(\G)_{\ab}
\xrightarrow{I_{\ab}}H_1(\G)\longrightarrow0
\end{equation}
is split exact. In particular, there is a generally noncanonical isomorphism
\begin{equation}
\label{eq:finite-germ-full-abelianization}
\F(\G)_{\ab}\cong
\bigl(H_0(\T)/(\tau(L)+2H_0(\T))\bigr)
\oplus \bigoplus_{\mathcal{O}}H_{\mathcal{O}}^{\ab}\oplus\ker\tau.
\end{equation}
\end{thm}
This applies in particular when the AF core is minimal, or when $\G$ is compactly generated and minimal, using Section~\ref{sec:kernel} to supply comparison. No assumption of bounded type is needed.
\begin{proof}
By Lemma~\ref{lem:finite-germ-coherent-lift}, every exceptional group of germs admits a coherent realization of returns. Theorem~\ref{thm:coherent-return-split-AH} therefore gives the split exact sequence and
$H_1(\G)\cong\bigoplus_{\mathcal{O}}H_{\mathcal{O}}^{\ab}\oplus\ker\tau$.
Since every $H_{\mathcal{O}}^{\ab}$ is finite while the dimension group $H_0(\T)$ is torsion-free, the isotropy part of the defect is zero. Hence, $\operatorname{im}\delta=\tau(L)$. The low-degree relative exact sequence then gives
$H_0(\G)=H_0(\T)/\tau(L)$. Since degree zero has no preceding homology term, change of coefficients yields
$H_0(\G;\mathbb{Z}/2\mathbb{Z})\cong H_0(\G)/2H_0(\G)$,
which is the quotient $H_0(\T)/(\tau(L)+2H_0(\T))$ appearing in \eqref{eq:finite-germ-full-abelianization}.
\end{proof}

The splitting is a splitting of abelian groups. The finite groups of germs are realized on clopen neighborhoods, while representatives of a free transport basis are chosen using index surjectivity.

For periodic incidence data, the parity group is computed by a stabilized matrix calculation over $\mathbb{Z}/2\mathbb{Z}$.
\keepwithstatement
\begin{prop}
\label{prop:periodic-parity-matrix}
Suppose the incidence sequence of $\mathsf{B}$ is eventually periodic from level $N$, with period product $P\in M_r(\mathbb{Z})$. Let $D$ be the matrix of actual defect columns for a finite generating set of $H_1(\G,\T)$ at that level. Put $\overline P=P\bmod2$, $\overline D=D\bmod2$, and
\begin{equation}
\nu_2=\min\{k\geq0:\operatorname{rank}_{\mathbb{Z}/2\mathbb{Z}}\overline P^k
=\operatorname{rank}_{\mathbb{Z}/2\mathbb{Z}}\overline P^{k+1}\}.
\end{equation}
Then $\nu_2\leq r$, and
\begin{equation}
\label{eq:periodic-parity-matrix}
H_0(\G;\mathbb{Z}/2\mathbb{Z})\cong
\frac{\operatorname{im}\overline P^{\nu_2}}
{\operatorname{im}(\overline P^{\nu_2}\overline D)}
\cong C_2^{\,k_2},\qquad
k_2=\operatorname{rank}_{\mathbb{Z}/2\mathbb{Z}}\overline P^{\nu_2}
-\operatorname{rank}_{\mathbb{Z}/2\mathbb{Z}}(\overline P^{\nu_2}\overline D).
\end{equation}
Under the hypotheses of Theorem~\ref{thm:finite-germ-split-AH}, this is the actual parity summand and $\F(\G)_{\ab}\cong H_1(\G)\oplus C_2^{k_2}$.
\end{prop}
\begin{proof}
After passing to the cofinal period endpoints and reducing coefficients modulo two, the dimension group associated with $\mathsf{B}$ becomes $\varinjlim((\mathbb{Z}/2\mathbb{Z})^r,\overline P)$. The sequence $\ker\overline P^k$ stabilizes at $k=\nu_2$. The induced endomorphism on $(\mathbb{Z}/2\mathbb{Z})^r/\ker\overline P^{\nu_2}$ is invertible, and thus the map from the initial period endpoint onto the direct limit is surjective with kernel $\ker\overline P^{\nu_2}$. Applying $\overline P^{\nu_2}$ identifies this quotient with the stable image $\operatorname{im}\overline P^{\nu_2}$.

It remains to justify that reducing the integral defect matrix $D$ captures the whole mod-two defect. The relative chain complex has degree-zero group equal to zero since $\G$ and $\T$ have the same unit space. Hence the universal coefficient sequence gives
$H_1(\G,\T;\mathbb{Z}/2\mathbb{Z})\cong H_1(\G,\T)\otimes\mathbb{Z}/2\mathbb{Z}$.
There is no additional $\operatorname{Tor}$ term from relative degree zero. Thus the reductions of the chosen integral generators still generate the relative group with mod-two coefficients, and their defect columns are exactly $\overline D$. Under the identification with $\operatorname{im}\overline P^{\nu_2}$, these columns become $\overline P^{\nu_2}\overline D$. The right-exact degree-zero relative sequence therefore gives \eqref{eq:periodic-parity-matrix}. The last assertion is Theorem~\ref{thm:finite-germ-split-AH}.
\end{proof}

The integers $\nu_2$ and $\nu$ of Section~\ref{sec:certified} can differ: the former is computed over $\mathbb{Z}/2\mathbb{Z}$, the latter over $\mathbb{Q}$. For example, $P=(4)$ has $\nu=0$ but $\nu_2=1$, and the parity group of its AF core is zero.

\subsection{A finite certificate for the abelian extension}
Even after $\theta_{\G}$ has been computed, the abelianization need not be the direct sum of its parity and index terms. The remaining problem is an ordinary extension problem.

\begin{defn}[Surviving parity group]
\label{def:surviving-parity}
Assume that the ordinary AH sequence \eqref{eq:AH-parity-sequence} applies. The \emph{surviving parity group} is $\mathsf{P}_{\G}=H_0(\G;\mathbb{Z}/2\mathbb{Z})/\operatorname{im}\theta_{\G}$.
\end{defn}
By exactness, $j_{\G}$ factors through an injection $\mathsf{P}_{\G}\hookrightarrow\F(\G)_{\ab}$ whose image is $\ker I_{\ab}$. Thus the AH sequence determines a short exact sequence
$0\to\mathsf{P}_{\G}\to\F(\G)_{\ab}\to H_1(\G)\to0$,
but it does not by itself determine whether this sequence splits.

For the remainder of this subsection, assume that $\G$ is minimal and has comparison. Suppose a certified computation gives $H_1(\G)=\mathbb{Z}^b\oplus\bigoplus_{j=1}^t\mathbb{Z}/n_j\mathbb{Z}$, with specified generators and index representatives $g_j$ for the torsion generators. By the constructive kernel theorem (Theorem~\ref{thm:constructive-boundary-cancellation}), $g_j^{n_j}$ is a product of dynamical transpositions. Write a chosen such factorization as
$g_j^{n_j}=t_{V_{j,1}}\cdots t_{V_{j,m_j}},\qquad \sg(V_{j,\ell})\cap\rg(V_{j,\ell})=\varnothing,$
including the transpositions used for its remainder in $\F(\T)$. The parities of these factorizations record the obstruction to splitting the abelian extension.
\keepwithstatement
\begin{prop}
\label{prop:extension-parity-certificate}
The classes
\begin{equation}
\label{eq:extension-parity-class}
\kappa_j=\sum_{\ell=1}^{m_j}[1_{\sg(V_{j,\ell})}]
\quad\text{in }\mathsf{P}_{\G}/n_j\mathsf{P}_{\G}
\end{equation}
are independent of the chosen transposition factorization and of the chosen index representatives. Together, they determine the isomorphism class of the abelian extension
$0\to\mathsf{P}_{\G}\to\F(\G)_{\ab}\to H_1(\G)\to0$.
The extension splits if and only if every $\kappa_j$ is zero. The obstructions for odd $n_j$ vanish automatically. For even $n_j$, $n_j\mathsf{P}_{\G}=0$, and thus \eqref{eq:extension-parity-class} is an actual element of $\mathsf{P}_{\G}$.
\end{prop}
\begin{proof}
Let $\widetilde t_j$ be the class of $g_j$ in the abelianization of the full group. The displayed factorization and the definition of $j_{\G}$ identify $n_j\widetilde t_j$ with the sum in \eqref{eq:extension-parity-class}, before reduction by $n_j\mathsf{P}_{\G}$. This is independent of the factorization since $\mathsf{P}_{\G}$ embeds in the abelianization. Replacing a lift adds an element $p\in \mathsf{P}_{\G}$ and changes its multiple by $n_jp$. Hence, the reduced class is independent of the lift.

Choose representatives $\widehat\kappa_j\in \mathsf{P}_{\G}$. The extension has the abelian presentation
\begin{equation}
\label{eq:extension-presentation}
\bigl(\mathsf{P}_{\G}\oplus\mathbb{Z}^{b+t}\bigr)\Big/
\left\langle n_je_{b+j}-\widehat\kappa_j:1\leq j\leq t\right\rangle.
\end{equation}
Indeed, the chosen lifts together with $\mathsf{P}_{\G}$ generate the extension. Any relation among them projects to a relation in the chosen decomposition of $H_1(\G)$, and thus after subtracting the displayed torsion relations it lies entirely in $\mathsf{P}_{\G}$. Hence, \eqref{eq:extension-presentation} is a complete presentation. The free summand always splits. For the cyclic summand of order $n_j$, a splitting exists exactly when its generator has a lift of order dividing $n_j$, which is equivalent to making its multiple zero by changing the lift. This is precisely the condition $\kappa_j=0$ in $\mathsf{P}_{\G}/n_j\mathsf{P}_{\G}$. Multiplication by an odd integer is the identity on the exponent-two group $\mathsf{P}_{\G}$, while multiplication by an even integer is zero.
\end{proof}

An order-$n$ index class leaves the order of a lift in the abelianization unresolved. For example, with $\mathsf{P}_{\G}=C_2$ and even $n$, the relation $n\widetilde t=1\in C_2$ gives the nonsplit cyclic extension $C_{2n}$. The zero relation gives $C_n\oplus C_2$. This illustrates the abelian extension calculation independently of a groupoid realization. In Theorem~\ref{thm:finite-germ-split-AH}, the coherent group actions supply torsion lifts whose abelianized orders divide the corresponding orders in $H_{\mathcal{O}}^{\ab}$, and thus these obstructions vanish.

\subsection{An explicit signature for the directed examples with finite groups of germs}
We construct a parity coordinate for the directed examples from their level permutations. This gives an independent verification of the splitting in Theorem~\ref{thm:finite-germ-split-AH}.

Consider the family from Corollary~\ref{cor:finite-permutation-returns}, on $\{0,\ldots,q-1\}^{\omega}$, with a finite faithful subgroup $H\leq\operatorname{Sym}(q)$. Recall that $b_h$ has section $b_h$ at $0$, section $\sigma_h$ at $1$, and identity sections elsewhere. There is one exceptional $\G$-orbit with two infinite tiles, group of germs $H$, and zero defect in the dimension group. Hence Theorem~\ref{thm:higher-germ-homology} gives
\begin{equation}
\label{eq:directed-all-homology}
H_0(\G)=\mathbb{Z}[1/q],\qquad
H_1(\G)=\mathbb{Z}\oplus H^{\ab},\qquad
H_n(\G)\cong H_n(H;\mathbb{Z})\quad(n\geq2).
\end{equation}

Let $\chi_H:H\to\mathbb{Z}/2\mathbb{Z}$ be the sign of the given permutation representation. Since the target is abelian, it factors through $H^{\ab}$. Composing with the return coordinate of $I(g)$ therefore gives a well-defined value $\chi(g)$. After refining a compact open table for $g$ to a common sufficiently deep level $n$, the element $g$ permutes the level-$n$ cylinders. Let $\epsilon_n(g)\in\mathbb{Z}/2\mathbb{Z}$ be the sign of this finite permutation. When $q$ is odd, the return coordinate gives a correction for which the resulting sign is constant at all sufficiently large levels.
\keepwithstatement
\begin{thm}
\label{thm:corrected-boundary-signature}
If $q$ is odd, then for all sufficiently large $n$, the value
\begin{equation}
\label{eq:corrected-boundary-signature}
\operatorname{sgn}_{\partial}(g)=\epsilon_n(g)+(n-1)\chi(g)
\end{equation}
is independent of $n$ and defines a homomorphism $\F(\G)\to\mathbb{Z}/2\mathbb{Z}$. It is zero on $a$ and every $b_h$, and has value $1$ on a transposition in $\Sym(\T)$ exchanging two equal-length cylinders. The induced map
\begin{equation}
\label{eq:directed-full-abelianization}
(I,\operatorname{sgn}_{\partial}):\F(\G)_{\ab}
\longrightarrow\mathbb{Z}\oplus H^{\ab}\oplus C_2
\end{equation}
is an isomorphism. If $q$ is even, $I_{\ab}$ alone is an isomorphism $\F(\G)_{\ab}\cong\mathbb{Z}\oplus H^{\ab}$.
\end{thm}
\begin{proof}
Let $J$ be the self-similar group generated by $a$, the $b_h$, and the rooted permutations $\sigma_h$. The rooted permutations belong to $\F(\T)$, and adjoining them leaves $\G$ unchanged. For a rooted automorphism $u$, let $r_m(u)$ be the sum modulo two of the signs of the root permutations of all its sections at level $m$. The section multiplication rule shows that $r_m$ is a homomorphism: multiplication only permutes the section positions before multiplying their root permutations.

For odd $q$ and every $m\geq1$, direct inspection of the recursions gives $r_m(a)=0$, $r_m(\sigma_h)=0$, and $r_m(b_h)=\chi_H(h)$. In the last case, the only nontrivial root contribution is the section $\sigma_h$ at $0^{m-1}1$. It follows that $r_m(u)=\chi(u)$ for every $u\in J$ and $m\geq1$. Equality on generators suffices since both sides are homomorphisms.

Refine a table for $g$ to a common level $N$, with section words $u_v\in J$. Its index is the sum of the indices of its localized sections, since the permutation of the level-$N$ cylinders belongs to $\F(\T)$. For a localized $a$, the normalized return is trivial, for a localized $b_h$, it is $h$, and for a localized $\sigma_h$, it is trivial. These assertions follow using the odometer connectors in the exceptional orbit. Therefore, $\chi(g)=\sum_{|v|=N}\chi(u_v)$. For $n\geq N+1$, the total contribution of the root signs of the sections of $g$ at level $n$ is consequently $\chi(g)$.

For these sufficiently deep levels, write $r_n(g)$ for the sum of the root signs of the sections of $g$ at level $n$, extending the notation above. The sign formula for a block permutation is $\epsilon_{n+1}(g)=q\epsilon_n(g)+r_n(g)$. Since $q$ is odd, we obtain $\epsilon_{n+1}(g)=\epsilon_n(g)+\chi(g)$ for $n\geq N+1$. This proves stabilization of \eqref{eq:corrected-boundary-signature}. For two elements, choose a common valid level. Both the finite permutation sign and $\chi$ are homomorphisms there, proving additivity.

The odometer permutes $q^n$ words in one cycle, of even sign when $q$ is odd. For $b_h$, its sign on the first level is zero and its level-$n$ sign is $(n-1)\chi_H(h)$, and thus both specified signatures vanish. A cylinder transposition in $\F(\T)$ is repeated $q^{n-N}$ times at level $n$, an odd number. Here $\chi(g)=0$, and its corrected signature is thus $1$.

Now $H_0(\G;\mathbb{Z}/2\mathbb{Z})=C_2$ for odd $q$, and the signature detects its image under $j_{\G}$. This independently proves that $j_{\G}$ is injective and $\theta_{\G}=0$ in this family. The exact sequence shows that \eqref{eq:directed-full-abelianization} is injective. It is onto since $a$ realizes the free index generator, the actual subgroup $h\mapsto b_h$ realizes $H^{\ab}$ with zero corrected signature, and a cylinder transposition in $\F(\T)$ realizes the last factor. For even $q$, $H_0(\G;\mathbb{Z}/2\mathbb{Z})=0$, and thus the assertion follows directly from Li's sequence.
\end{proof}

The parity factor and the sign character of $H$ are different invariants. When $q$ is odd, the signs $\epsilon_n(b_h)$ alternate with the level whenever $\chi_H(h)=1$. The correction term in~\eqref{eq:corrected-boundary-signature} makes the resulting value constant. For a localization of $b_h$ on a cylinder of length $N$, the corrected signature is $N\chi_H(h)$ modulo two. Another localized lift can change its parity coordinate without changing its index. When $q$ is even, there is no extra AF parity factor, although $H^{\ab}$ may still contain substantial two-torsion.

\begin{exmp}[Parity benchmarks for finite returns]
\label{exmp:parity-benchmarks}
In the original four-letter model with $H=\operatorname{Sym}(3)$, $H_2(\G)=0$ and $\F(\G)_{\ab}\cong\mathbb{Z}\oplus C_2$. For the same construction with $q=3$ and the natural action of $\operatorname{Sym}(3)$, $H_2(\G)=0$ while $\F(\G)_{\ab}\cong\mathbb{Z}\oplus C_2\oplus C_2$. The first $C_2$ is the return abelianization and the second is AF parity.

For a test with nonzero homology in degree two, take $q=5$ and let $H=C_2\times C_2$ act regularly on four letters and fix the fifth. Then $H_2(\G)\cong C_2$, but $\theta_{\G}:C_2\to C_2$ is the zero map, and $\F(\G)_{\ab}\cong\mathbb{Z}\oplus C_2^3$. Therefore, vanishing of the parity map here is not explained by vanishing of its source. All elements of this permutation representation of $H$ are even. The corrected signature reduces to the eventually constant ordinary sign. More concretely, for the two independent involutions $u,v\in H$, the bar cycle $z=[u|v]-[v|u]$ represents the nonzero element of $H_2(H;\mathbb{Z})$. Its double is the boundary of $[u|u|v]-[u|v|u]+[v|u|u]$. The corresponding groupoid cycle is the difference of the indicators of the two sets of composable arrows obtained from $b_u,b_v$ in opposite orders. Their coherent action as a finite group gives a lift of this return cycle, which explains its zero parity obstruction. The bar matrices in the following subsection give a finite procedure for checking these claims.

For a finite return with order greater than two, take $q=5$ and $H=C_4$ acting as a four-cycle and fixing the fifth letter. The result is $\F(\G)_{\ab}\cong\mathbb{Z}\oplus C_4\oplus C_2$. The $C_4$ is realized by the actual subgroup $h\mapsto b_h$, and thus no $C_8$ extension occurs.
\end{exmp}

\subsection{Other benchmarks and computations from verified finite data}
The principal rank-two system of Proposition~\ref{prop:two-cut-defect} has no higher homology by Theorem~\ref{thm:higher-germ-homology}. Its dimension group modulo $2$ is the direct limit of $(\mathbb{Z}/2\mathbb{Z})^2$ under $\left(\begin{smallmatrix}0&1\\1&0\end{smallmatrix}\right)$. Quotienting by the image of the defect identifies the two coordinates. Thus, $H_0(\G;\mathbb{Z}/2\mathbb{Z})=C_2$. Since $H_1(\G)=\mathbb{Z}$, Li's sequence and a lift of its generator give $\F(\G)_{\ab}\cong\mathbb{Z}\oplus C_2$. The nonzero infinitesimal defect matters when the dimension group is computed, even though the final abelianization is that of a minimal Cantor $\mathbb{Z}$-system. This agrees with the classical abelianization formula for minimal Cantor systems \cite{Matui06}.

For the sparse binary model with an infinite cyclic group of germs in Subsection~\ref{subsec:infinite-cyclic-certified}, $H_2(\G)=H_2(\mathbb{Z};\mathbb{Z})=0$ and $H_0(\G;\mathbb{Z}/2\mathbb{Z})=0$. Consequently, $\F(\G)_{\ab}\cong\mathbb{Z}^2$. 

For a finite group given by a complete multiplication table, the normalized bar matrices provide a finite exact computation of $H_2$. With symbols containing the identity set to zero, they are
\begin{equation}
\label{eq:bar-degree-two-three}
\begin{split}
\partial_2[h|k]&=[k]-[hk]+[h],\\
\partial_3[h|k|l]&=[k|l]-[hk|l]+[h|kl]-[h|k].
\end{split}
\end{equation}
For $|H|=d$, their matrix sizes are $(d-1)\times(d-1)^2$ and $(d-1)^2\times(d-1)^3$. The Smith normal form of $\partial_2$ supplies an integral basis of its kernel. Expressing $\partial_3$ in that basis and taking its Smith normal form computes $H_2(H;\mathbb{Z})$. This uses the exact arithmetic of Section~\ref{sec:certified}. No finite presentation is substituted for a multiplication table without a completeness proof.

\subsection{Sparse groups of germs isomorphic to \texorpdfstring{$\mathbb{Z}^d$}{Z-d}}
\label{subsec:sparse-free-abelian-returns}
Let $\mathsf{X}=\{0,1,2\}$ and $\Omega(\mathsf{B})=\mathsf{X}^{\omega}$, with the first letter read as the least significant ternary digit. Let $a$ be the ternary odometer
\begin{equation}
\label{eq:ternary-odometer}
a(0w)=1w,\qquad a(1w)=2w,\qquad a(2w)=0a(w).
\end{equation}
Put $\xi=0^\omega$ and $\eta=2^\omega$. For $j\geq1$, let $\alpha_j$ be the depth-$j$ truncated ternary odometer: on the first $j$ letters it adds one modulo $3^j$, and all sections at level $j$ are trivial. Thus $\alpha_j$ has order $3^j$.

Choose $N_j=3^j$ and the pairwise disjoint cylinders
\begin{equation}
\label{eq:sparse-ternary-cylinders}
U_j=[0^{N_j}1].
\end{equation}
Fix $d\geq1$ and partition the positive integers into the residue classes $J_1,\ldots,J_d$ modulo $d$. Define $b_i\in\operatorname{Aut}(\mathsf{X}^*)$ by
\begin{equation}
\label{eq:sparse-free-abelian-generators}
b_i(0^{N_j}1w)=0^{N_j}1\alpha_j(w)\quad(j\in J_i),
\qquad
b_i=1\quad\text{off }\bigcup_{j\in J_i}U_j.
\end{equation}
Every $b_i$ fixes $\xi$. Since $N_{j+1}>N_j+j+2$, at each level there is at most one nontrivial section of $b_i$ outside the prefix of $\xi$. Thus the $b_i$ are bounded automorphisms, although the finitary depths of their sections on the cylinders $U_j$ are unbounded.

Let
\begin{equation}
\label{eq:sparse-free-abelian-groupoid}
G_d=\langle a,b_1,\ldots,b_d\rangle,
\qquad
\G_d=\operatorname{Germ}(G_d\curvearrowright \Omega(\mathsf{B})),
\end{equation}
and let $\T$ be the ternary tail groupoid.

The maps $b_i$ realize $\mathbb{Z}^d$ as the group of germs, and the defect map is zero.
\keepwithstatement
\begin{prop}
\label{prop:sparse-Zd-return-model}
For every $d\geq1$, the groupoid $\G_d$ is minimal, effective and AF-by-discrete. It has one exceptional $\G_d$-orbit, and $\T$ has exactly two orbits in it, the tail classes of $\xi$ and $\eta$. The group of germs at $\xi$ is
\begin{equation}
\label{eq:sparse-Zd-germs}
H_\xi\cong\mathbb{Z}^d,
\end{equation}
generated by the germs of $b_1,\ldots,b_d$. With the transport generator oriented so that $I(a)$ is positive,
\begin{equation}
\label{eq:sparse-Zd-H1}
H_1(\G_d,\T)\cong\mathbb{Z}^{d+1},\qquad
\delta=0,\qquad
H_1(\G_d)\cong\mathbb{Z}^{d+1},
\end{equation}
and
\begin{equation}
\label{eq:sparse-Zd-indices}
I(a)=e_0,\qquad I(b_i)=e_i\quad(1\leq i\leq d).
\end{equation}
Moreover,
\begin{equation}
\label{eq:sparse-Zd-higher}
H_n(\G_d;\mathbb{Z})\cong
\bigwedge^n\mathbb{Z}^d\quad(n\geq2),
\end{equation}
with the right-hand side zero for $n>d$, and
\begin{equation}
\label{eq:sparse-Zd-H0}
H_0(\G_d)\cong\mathbb{Z}[1/3],\qquad
H_0(\G_d;\mathbb{Z}/2\mathbb{Z})\cong\mathbb{Z}/2\mathbb{Z}.
\end{equation}
\end{prop}

\begin{proof}
The ternary odometer is minimal. Its groupoid of germs contains $\T$, and thus $\G_d$ is minimal. The only non-AF germ of the odometer $a$ is the carry $\eta\to\xi$. Its inverse is the exceptional germ of $a^{-1}$. For $x\neq\xi$, every $b_i$ changes only finitely many letters on a neighborhood of $x$, hence its germ belongs to $\T$. At $\xi$, every neighborhood contains a cylinder $U_j$, with $j\in J_i$, on which $b_i$ is nontrivial. Hence, the germ of $b_i$ is nontrivial. Every word has only finitely many exceptional sources. They are obtained by pulling back, along its prefixes, the finite sets of exceptional sources of its letters. Restricting the word bisection around one such source isolates the corresponding non-AF germ. Hence $\T$ is open and $\G_d\setminus\T$ is discrete.

The supports used for distinct $b_i$ are disjoint, and thus the $b_i$ commute. If
\begin{equation*}
b_1^{m_1}\cdots b_d^{m_d}
\end{equation*}
has trivial germ at $\xi$, then for every $i$ and every sufficiently large $j\in J_i$ its restriction to $U_j$ is $\alpha_j^{m_i}=1$. Hence $3^j$ divides $m_i$ for arbitrarily large $j$, forcing $m_i=0$. Therefore, the displayed germs generate a copy of $\mathbb{Z}^d$. For the boundary quotient, take $\mathcal{S}$ to be the full bisections of $a,a^{-1},b_1^{\pm1},\ldots,b_d^{\pm1}$. Choosing one edge from each pair of edges inverse to each other, the boundary quotient $Q_{\mathcal{S}}(\mathcal{O})$ of Definition~\ref{def:boundary-quotient} has two vertices corresponding to the two infinite tiles, one edge represented by the exceptional odometer germ, and the $d$ loops represented by $(b_i,\xi)$. Proposition~\ref{prop:germ-fg} shows that there are no additional generators of the group of germs, proving~\eqref{eq:sparse-Zd-germs}.

Choose the edge represented by the exceptional odometer germ for the spanning tree. The relative normal form is then one transport coordinate together with the $d$ return coordinates. Each $b_i$ is a global element of the full group with exactly one exceptional arrow, at $\xi$, and $a$ has exactly the single carry germ. Their source and range are both $\Omega(\mathsf{B})$, and thus each displayed relative basis element has zero defect in the dimension group. This proves~\eqref{eq:sparse-Zd-H1} and~\eqref{eq:sparse-Zd-indices}. Theorem~\ref{thm:higher-germ-homology} and the standard homology of the free abelian group give~\eqref{eq:sparse-Zd-higher}. Finally, the ternary tail groupoid with one Bratteli vertex at each level has dimension group $\mathbb{Z}[1/3]$. Since $\delta=0$, the map to $H_0(\G_d)$ is an isomorphism. Reducing modulo two gives~\eqref{eq:sparse-Zd-H0}.
\end{proof}

The group of germs in Proposition~\ref{prop:sparse-Zd-return-model} is coherently realized on all of $\Omega(\mathsf{B})$ by
\begin{equation}
\label{eq:sparse-Zd-coherent-lift}
\rho:\mathbb{Z}^d\longrightarrow\F(\G_d),\qquad
(m_1,\ldots,m_d)\longmapsto b_1^{m_1}\cdots b_d^{m_d}.
\end{equation}
It is AF away from $\xi$ and induces the isomorphism~\eqref{eq:sparse-Zd-germs} on germs at $\xi$. This coherent realization gives a split abelianization, with the remaining parity detected by the eventual signs of the level permutations.
\keepwithstatement
\begin{thm}
\label{thm:sparse-Zd-parity}
For every $d\geq1$, the groupoid $\G_d$ satisfies
\begin{equation}
\label{eq:sparse-Zd-theta}
\theta_{\G_d}=0
\end{equation}
and
\begin{equation}
\label{eq:sparse-Zd-abelianization}
\F(\G_d)_{\ab}\cong\mathbb{Z}^{d+1}\oplus C_2.
\end{equation}
There is an explicit parity homomorphism
\begin{equation}
\label{eq:sparse-Zd-stable-sign}
\epsilon_\infty:\F(\G_d)\longrightarrow\mathbb{Z}/2\mathbb{Z}
\end{equation}
given by the eventual sign of the ternary level permutations. It vanishes on $a,b_1,\ldots,b_d$. On a transposition in $\Sym(\T)$ between two cylinders of the same level, it takes the value $1$. Consequently,
\begin{equation}
\label{eq:sparse-Zd-abelianization-coordinates}
(I,\epsilon_\infty):\F(\G_d)\longrightarrow\mathbb{Z}^{d+1}\oplus C_2
\end{equation}
induces the isomorphism~\eqref{eq:sparse-Zd-abelianization}.
\end{thm}

\begin{proof}
The coherent realization~\eqref{eq:sparse-Zd-coherent-lift} and Theorem~\ref{thm:coherent-return-split-AH} give~\eqref{eq:sparse-Zd-theta} and the abstract splitting. We give an independent parity coordinate.

Every permutation induced by the ternary odometer on level $n$ is a cycle of odd length $3^n$, hence is even. For the truncated odometer $\alpha_j$, its permutation on a level $m\leq j$ is a $3^m$-cycle, while for $m\geq j$ it is a product of $3^{m-j}$ cycles of length $3^j$. These are again even permutations. A section of $a^{\pm1}$ is again an odometer section or the identity. A section of $b_i^{\pm1}$ is the identity, a section of one $\alpha_j^{\pm1}$, or a sparse product of deeper $\alpha_j^{\pm1}$ on disjoint cylinders. Hence every such section has even permutations on every finite level. This property is preserved under products and further sections.

Take $g\in\F(\G_d)$ and refine a compact open table for $g$ to a common ternary level $N$. At every deeper level $n$, the permutation of the level-$N$ cylinders is repeated $3^{n-N}$ times, an odd number, while all permutations contributed by the sections are even. Hence, the sign of the level-$n$ permutation is independent of $n$ for all $n\geq N$. This defines~\eqref{eq:sparse-Zd-stable-sign} and makes it a homomorphism. The generators $a,b_i$ have zero sign by the preceding paragraph. A transposition in $\Sym(\T)$ of two level-$N$ cylinders becomes a product of $3^{n-N}$ transpositions at level $n$, and thus has sign one.

By~\eqref{eq:sparse-Zd-H0}, the parity term in Li's exact sequence is one-dimensional over $\mathbb{Z}/2\mathbb{Z}$. The homomorphism $\epsilon_\infty$ is nonzero on its generator represented by a dynamical transposition, and thus the parity map into the abelianization of the full group is injective. Exactness gives again $\theta_{\G_d}=0$. Finally,~\eqref{eq:sparse-Zd-indices} supplies the free index coordinates, and $\epsilon_\infty$ supplies the parity coordinate, proving~\eqref{eq:sparse-Zd-abelianization-coordinates}.
\end{proof}

\begin{cor}
\label{cor:sparse-Z2-parity-test}
For $d=2$,
$H_2(\G_2;\mathbb{Z})\cong\mathbb{Z}, \qquad H_0(\G_2;\mathbb{Z}/2\mathbb{Z})\cong C_2,$
but
\begin{equation}
\theta_{\G_2}:\mathbb{Z}\longrightarrow C_2
\end{equation}
is the zero homomorphism. Under the identification $H_\xi=\langle u,v\rangle\cong\mathbb{Z}^2$, the normalized bar cycle
\begin{equation}
\label{eq:sparse-Z2-fundamental-cycle}
[u|v]-[v|u]
\end{equation}
represents a generator of $H_2(\G_2;\mathbb{Z})$.
\end{cor}

The example has an infinite finitely generated group of germs, nonzero free second homology, and a nonzero parity target, while its coherent realization of returns forces the differential to vanish. The next section separates such global coherence from the weaker lifting property for finite relations.

\section{Finite return relations and a nonzero parity differential}
\label{sec:finite-return-relations}

A coherent realization requires all relations of a group of germs to hold on one invariant clopen neighborhood. A class in $H_2(H_x;\mathbb{Z})$, however, involves only finitely many labels and multiplication relations. We prove that lifting finite relations on neighborhoods chosen separately suffices for vanishing of the secondary differential and splitting of the AH sequence. A minimal example separates this condition from coherence for the whole group. A nonminimal Houghton--Cantor example has nonzero differential.

\subsection{Lifting a finite presentation}
\begin{defn}[Lifting finite relations]
\label{def:finite-relation-lifting}
Let $x$ be a basepoint of an exceptional orbit and put $H_x=\G_x^x$. We say that $(\G,\T)$ has \emph{the lifting property for finite relations at $x$} if, for every finitely presented group $L$ and every homomorphism $\phi:L\to H_x$, there are a clopen neighborhood $U\ni x$ and a homomorphism $\rho:L\to\F(\G|_U)$ such that every $\rho(l)$ fixes $x$, has germ $\phi(l)$ at $x$, and is AF at every point of $U\setminus\{x\}$.
\end{defn}

Neither $\phi$ nor $\rho$ is required to be injective, and the neighborhood may depend on the finite presentation and on $\phi$. A coherent realization implies this property by composition. The converse fails even for finitely generated $H_x$. See Proposition~\ref{prop:wreath-no-coherent-lift}. We impose the property at one chosen basepoint of each exceptional orbit.

\begin{lemma}
\label{lem:finite-presentation-cycle}
Let $H$ be a group and let $z\in H_2(H;\mathbb{Z})$. There are a finitely presented group $L$, a homomorphism $\phi:L\to H$, and $\widetilde z\in H_2(L;\mathbb{Z})$ such that $\phi_*(\widetilde z)=z$.
\end{lemma}
\begin{proof}
Represent $z$ by a finite normalized bar cycle $\sum_i n_i[h_i|k_i]$. Introduce one symbol $x_h$ for each distinct nonidentity element occurring among the finitely many labels $h_i$, $k_i$, and $h_ik_i$, and put $x_1=1$. Impose the finitely many relations $x_{h_i}x_{k_i}=x_{h_ik_i}$. The assignments $x_h\mapsto h$ define a homomorphism $\phi:L\to H$ from the resulting finitely presented group.

Now replace every label of the chosen bar cycle by the corresponding element $x_h\in L$. Its boundary is the same finite formal sum of the labels $x_{k_i}$, $x_{h_ik_i}$, and $x_{h_i}$ that occurred in the original boundary. Whenever two terms cancelled in the bar complex of $H$, they carried the same group element and hence now carry the same symbol in $L$. Thus the lifted chain is again a $2$-cycle. Its image under $\phi$ is the original representative of $z$, proving the claim. Relations in $L$ that identify additional labels can only create further cancellations and do not affect the argument.
\end{proof}

\begin{lemma}
\label{lem:exponent-two-splitting}
Let $0\to P\to E\xrightarrow{q}A\to0$ be an exact sequence of abelian groups, with $2P=0$. It splits if every $a\in A$ satisfying $2a=0$ has a lift $e\in E$ satisfying $2e=0$.
\end{lemma}
\begin{proof}
Suppose $p=2e\in P$. The class $a=q(e)$ is killed by two. Choose a lift $e'$ of $a$ killed by two. Then $e-e'\in P$, and thus $p=2(e-e')+2e'=0$. Hence, $P\cap2E=0$. The map $P\to E/2E$ is injective. Both groups are vector spaces over $\mathbb{Z}/2\mathbb{Z}$, and thus this inclusion has a linear retraction. Composing that retraction with $E\to E/2E$ gives a homomorphism $E\to P$ that is the identity on $P$. Its kernel supplies the required complement. No finite-generation assumption is needed.
\end{proof}

\begin{thm}[Finite-relation AH splitting]
\label{thm:finite-relations-split-AH}
Assume the AF-by-discrete inclusion and the lifting property for finite relations at a basepoint of every exceptional orbit. Then $\theta_{\G}=0$, every isotropy defect map is zero, and Li's stable AH sequence splits. More precisely, put $L=\bigoplus_{\mathcal{O}}L_{\mathcal{O}}$, with $L_{\mathcal{O}}$ as in Definition~\ref{def:transport-lattice}, and let $\tau:L\to H_0(\T)$ be the transport defect in the sense of Definition~\ref{def:defect-map}, in chosen connector coordinates. Then
\begin{equation}
\label{eq:finite-relations-stable-split}
\begin{split}
H_1(\G)&\cong \bigoplus_{\mathcal{O}}H_{\mathcal{O}}^{\ab}\oplus\ker\tau,\\
\F(\mathfrak{R}\times\G)_{\ab}
&\cong \frac{H_0(\T)}{\tau(L)+2H_0(\T)}\oplus \bigoplus_{\mathcal{O}}H_{\mathcal{O}}^{\ab}\oplus\ker\tau.
\end{split}
\end{equation}
Here $\mathfrak{R}$ is the discrete pair groupoid on a countable infinite set, and the isomorphism in the second line is noncanonical. If $\G$ is minimal and has comparison, then the same formula holds with $\F(\G)_{\ab}$ in place of $\F(\mathfrak{R}\times\G)_{\ab}$.
\end{thm}
The groups of germs and their abelianizations may be infinitely generated.
\begin{proof}
We separate the argument into the secondary differential, the index group, and the abelian extension.

First fix one orbit summand $H_2(H_x;\mathbb{Z})$ in Theorem~\ref{thm:higher-germ-homology}, and let $z$ be a class in that summand. Lemma~\ref{lem:finite-presentation-cycle} gives a finitely presented group $L'$, a homomorphism $\phi:L'\to H_x$, and $\widetilde z\in H_2(L';\mathbb{Z})$ with $\phi_*(\widetilde z)=z$. By Definition~\ref{def:finite-relation-lifting}, after shrinking to a clopen neighborhood $U$ of $x$ there is an actual action $\rho:L'\to\F(\G|_U)$ whose germ at $x$ is $\phi$ and which is AF at every other point of $U$. The chain map $Q_{\rho,*}$ from~\eqref{eq:coherent-bar-evaluation} therefore has exactly the relative behavior used in Lemma~\ref{lem:relative-nerve}: modulo $\T$, the associated set of composable arrows contributes only the tuple at $x$, whose entries are prescribed by $\phi$. Consequently $(Q_{\rho,2})_*(\widetilde z)$ is identified with $z$ in the $H_2(H_x;\mathbb{Z})$ summand. Lemma~\ref{lem:coherent-parity} gives $\theta_{\G}(z)=0$. Since $H_2(\G)$ is the direct sum of these orbit summands, $\theta_{\G}=0$.

Next fix $h\in H_x$ and apply Definition~\ref{def:finite-relation-lifting} to the homomorphism $\mathbb{Z}\to H_x$, $1\mapsto h$. Extend the resulting local full-group element by the identity outside its clopen domain. It is AF away from $x$, fixes $x$, and has germ $h$ there. Hence its relative index has no transport coordinate and has isotropy coordinate $[h]_{\ab}$. Every full-group index lies in $\ker\delta$, and thus Definition~\ref{def:defect-map} gives $\lambda_x([h]_{\ab})=0$. Thus every isotropy defect map is zero. In the relative normal form, $\delta$ is therefore just the transport map $\tau$ on $L=\bigoplus_{\mathcal{O}}L_{\mathcal{O}}$, and exactness yields
$H_1(\G)\cong\bigoplus_{\mathcal{O}}H_{\mathcal{O}}^{\ab}\oplus\ker\tau$.
Similarly $H_0(\G)\cong H_0(\T)/\tau(L)$. The degree-zero universal coefficient identification then gives
$H_0(\G;\mathbb{Z}/2\mathbb{Z})\cong H_0(\T)/(\tau(L)+2H_0(\T))$.

It remains to split Li's stable AH sequence. Since $\theta_{\G}=0$, Theorem~6.12 of \cite{Li25} gives a short exact sequence with kernel $H_0(\G;\mathbb{Z}/2\mathbb{Z})$, an exponent-two group, and quotient $H_1(\G)$. We verify the criterion of Lemma~\ref{lem:exponent-two-splitting}. Let $[h]_{\ab}\in H_x^{\ab}$ be killed by two. Then $h^2\in[H_x,H_x]$, and thus there is a finite expression $h^2=\prod_{i=1}^r[a_i,b_i]$. Use the finitely presented group
\begin{equation}
\label{eq:two-torsion-return-presentation}
L_h=\langle z,u_1,v_1,\ldots,u_r,v_r\mid
z^2=\prod_{i=1}^r[u_i,v_i]\rangle
\end{equation}
and the homomorphism sending its displayed generators to $h,a_i,b_i$. Definition~\ref{def:finite-relation-lifting} realizes this one finite relation by an actual local action. In the abelianization of the resulting full group, the commutators on the right disappear and therefore $2[\rho(z)]=0$. The index of $\rho(z)$ is $([h]_{\ab},0)$.

The transport term $\ker\tau$ is a subgroup of the free abelian group $L$, and hence is torsion-free. Therefore every element of $H_1(\G)$ killed by two is a finite sum of two-torsion classes from the groups $H_{\mathcal{O}}^{\ab}$. Applying the preceding construction to those finitely many coordinates and adding the resulting classes in the stable abelianization produces a lift killed by two. Lemma~\ref{lem:exponent-two-splitting} now splits the stable short exact sequence and gives the second line of~\eqref{eq:finite-relations-stable-split}. If $\G$ is minimal and has comparison, Li's Corollary~6.14 identifies the corresponding ordinary AH sequence, and the same finite-relation lifts give the same splitting for $\F(\G)_{\ab}$.
\end{proof}

In the relation in~\eqref{eq:two-torsion-return-presentation}, the element $h$ may have infinite order: only its abelianized class is killed by two. The neighborhood may depend on this finite relation, so the splitting does not require a coherent action of the whole group of germs.

\subsection{Invariant neighborhoods and rooted actions}
The invariant clopen bases used in Proposition~\ref{prop:invariant-neighborhood-stabilization} also allow finitely many germ relations to hold as actual maps on one common neighborhood.

\keepwithstatement
\begin{prop}
\label{prop:invariant-neighborhood-finite-relations}
Suppose that every finite collection of groups of germs at $x$ has representatives by compact open bisections with a common invariant clopen neighborhood basis at $x$, where each nonidentity germ is represented by an isolating bisection and the identity germ is represented by a unit bisection. Write $f_1,\ldots,f_d$ for the associated local homeomorphisms. Thus there are clopen sets $U_n\ni x$, decreasing to $\{x\}$, such that, for all sufficiently large $n$, every $f_i$ is defined on $U_n$ and satisfies $f_i(U_n)=U_n$. Then the lifting property for finite relations holds at $x$.
\end{prop}
\begin{proof}
Choose a finite presentation $L=\langle s_1,\ldots,s_d\mid r_1,\ldots,r_m\rangle$. For each nonidentity germ $\phi(s_i)$ choose an isolating representative, and for an identity germ choose a unit bisection. Denote the resulting local homeomorphisms by $f_i$. Since $\phi(r_j)=(\Id,x)$, the local homeomorphism represented by the word $r_j(f_1,\ldots,f_d)$ agrees with the identity on some neighborhood of $x$. There are only finitely many relators. Choose one sufficiently deep $U_n$ contained in all these neighborhoods on which the relators are identities and in the domains needed for the finitely many words, and such that every $f_i$ preserves $U_n$. Then the restrictions $f_i|_{U_n}$ are homeomorphisms of $U_n$ and the defining relators hold there as actual equalities. The universal property of the presentation gives a homomorphism $\rho:L\to\F(\G|_{U_n})$ with the prescribed germs. Since each generator is AF on $U_n\setminus\{x\}$ and this punctured set is invariant under all the restricted generators, every word is AF there as well.
\end{proof}

Stabilizers of boundary rays preserve their prefix cylinders, giving the required invariant clopen bases.
\keepwithstatement
\begin{cor}
\label{cor:rooted-finite-relation-AH}
Suppose that $\G$ is the groupoid of germs of an action by automorphisms of a locally finite rooted tree whose boundary is a Cantor set, and that $\G$ has the AF-by-discrete inclusion. Then the lifting property for finite relations holds at every group of germs. In particular, the stable AH sequence splits. The unstabilized sequence splits under minimality and comparison.
\end{cor}
\begin{proof}
Let $\xi$ be the boundary ray under consideration. Represent each nonidentity isotropy germ by an element $g$ of the acting group with $g(\xi)=\xi$. Represent the identity germ by the identity map. Since a rooted-tree automorphism fixing $\xi$ fixes every vertex on that ray, each such $g$ preserves every prefix cylinder $[\xi_n]$. For a nonidentity germ $(g,\xi)$, discreteness of $\G\setminus\T$ gives a neighborhood of that arrow in the groupoid containing no other non-AF arrow. Intersecting its source with a sufficiently deep prefix cylinder $[\xi_n]$ therefore makes the action bisection of $g$ on $[\xi_n]$ isolating, while $[\xi_n]$ remains $g$-invariant. For a finite collection of germs, choose one common depth with this property for every nonidentity representative. All still deeper prefix cylinders are then preserved by every representative, the nonidentity representatives remain isolating there, and the identity representatives remain unit bisections. Hence these cylinders form the common invariant basis required by Proposition~\ref{prop:invariant-neighborhood-finite-relations}.
\end{proof}

Unlike the bounded-depth criterion of Proposition~\ref{prop:jns-depth} and the finiteness test of Proposition~\ref{prop:tail-stabilization}, this argument requires only invariant clopen neighborhoods. It applies to infinite groups of germs with unbounded finitary depths. For general partial maps on Bratteli cylinders, a common neighborhood on which the relations hold need not be invariant.

\subsection{A minimal groupoid without a coherent return action}
We construct a concrete separation between the two realization conditions. Let
\begin{equation}
\label{eq:wreath-return-group}
H=\mathbb{Z}\wr\mathbb{Z}
=\left(\bigoplus_{n\in\mathbb{Z}}\mathbb{Z} e_n\right)\rtimes\langle t\rangle,
\qquad te_nt^{-1}=e_{n+1},
\end{equation}
and write $u=e_0$. Its presentation is $\langle u,t\mid[u,t^kut^{-k}]=1\ (k\geq1)\rangle$. All the constructions below are explicit finite permutations. The infinite presentation is used only to identify the limiting group of germs.

For $j\geq1$, put $L_j=3^j$, $r_j=2j+1$, and $M_j=3^{3j}$. On $\mathbb{Z}/M_j\mathbb{Z}$, define
\begin{equation}
\label{eq:wreath-finite-approximations}
T_j(z)=z+1,\qquad
A_j=(0\ L_j\ 2L_j\ \cdots\ (r_j-1)L_j),
\end{equation}
with $A_j$ fixing every other point. We have $r_jL_j<M_j$. Identify these $M_j$ points with ternary words of length $3j$, read as least-significant-digit-first integers. Each permutation extends to a prefix replacement in $\T$, leaving the infinite suffix unchanged. The maps $A_j$ may move the intermediate prefix levels. These finite models eventually detect each fixed relation, although no single model satisfies all the relations.
\keepwithstatement
\begin{lemma}
\label{lem:wreath-marked-limit}
For every fixed word $w$ in $u^{\pm1},t^{\pm1}$,
\begin{equation}
\label{eq:wreath-eventual-relations}
w=1\text{ in }H
\quad\Longleftrightarrow\quad
w(A_j,T_j)=1\text{ for all sufficiently large }j.
\end{equation}
A nonidentity word is nonidentity in every sufficiently large finite model. Nevertheless,
\begin{equation}
\label{eq:wreath-moving-relator}
[A_j,T_j^{L_j}A_jT_j^{-L_j}]\neq1
\end{equation}
for every $j$.
\end{lemma}
\begin{proof}
Write the word as a product of conjugates $t^nut^{-n}$ and their inverses, followed by a power $t^k$. For a fixed word, only finitely many indices $n$ occur, and the number of factors and their exponents are bounded. For large $j$, different shift indices give disjoint supports of the corresponding conjugates of $A_j$: modulo $L_j$ their indices are different, and $M_j$ is divisible by $L_j$. Hence the corresponding conjugates of $A_j$ commute. Their orders $r_j$ tend to infinity, and thus a fixed nonzero exponent gives a nonidentity power for all sufficiently large $j$. If $k=0$, this proves both assertions in~\eqref{eq:wreath-eventual-relations}. If $k\neq0$, let $m$ be the number of these conjugates of $A_j$. Their product moves at most $mr_j$ points. The nontrivial translation $T_j^k$ moves all $M_j$ points for large $j$. Thus that product cannot equal $T_j^{-k}$, and the word is nonidentity. For instance, for a word of length $\ell$, the estimates in this argument hold for $j>\ell$.

For~\eqref{eq:wreath-moving-relator}, the second cycle is $(L_j\ 2L_j\ \cdots\ r_jL_j)$. With the convention $[a,b]=aba^{-1}b^{-1}$, its commutator with $A_j$ sends $0$ to $r_jL_j$, which is not $0$ modulo $M_j$.
\end{proof}

Let $\Omega(\mathsf{B})=\{0,1,2\}^{\omega}$, $\xi=0^\omega$, and let $a$ be the ternary odometer. Put $n_j=10j^2$ and $U_j=[0^{n_j}1]$. Define $b$ and $c$ on $U_j$ by the prefix permutations $A_j$ and $T_j$ on the next $3j$ letters, respectively. Set both maps equal to the identity elsewhere, including at $\xi$. The disjoint clopen supports accumulate only at $\xi$. These maps are homeomorphisms. Write $\G_{\mathrm{w}}$ for the groupoid of germs of $\langle a,b,c\rangle$, and retain the ternary tail groupoid as $\T$.
\keepwithstatement
\begin{prop}
\label{prop:wreath-no-coherent-lift}
The groupoid $\G_{\mathrm{w}}$ is minimal, compactly generated, effective, and AF-by-discrete. It has one exceptional orbit with two infinite tiles and group of germs $H_\xi\cong\mathbb{Z}\wr\mathbb{Z}$, with $(b,\xi)=u$ and $(c,\xi)=t$. It satisfies the lifting property for finite relations, but $H_\xi$ has no coherent realization of returns in $\G_{\mathrm{w}}$ on any clopen neighborhood of $\xi$.
\end{prop}
\begin{proof}
Every germ of $b$ or $c$ away from $\xi$ is AF, whereas their germs at $\xi$ are nontrivial since the cylinders $U_j$ accumulate at $\xi$. Lemma~\ref{lem:wreath-marked-limit} identifies the subgroup of germs generated by these two returns with $H$: a word has identity germ exactly when it acts trivially on all cylinders $U_j$ with sufficiently large $j$. The ternary odometer contains every equal-length prefix replacement locally and supplies the only transport connection between the two relevant tail classes, from $2^\omega$ to $0^\omega$. Up to reversing that connection, the boundary quotient has two vertices, one transport edge, and the two return loops represented by $b$ and $c$. Proposition~\ref{prop:germ-fg} therefore shows that the full isotropy at $\xi$ is exactly the group generated by those loops, namely $H$. The exceptional support of a word is contained in finitely many pullbacks of the exceptional supports of the generators, and thus $\G_{\mathrm{w}}\setminus\T$ is discrete. The odometer acts minimally, hence so does $\G_{\mathrm{w}}$. Since a groupoid of germs is effective and the action is generated by three global homeomorphisms, $\G_{\mathrm{w}}$ is effective and compactly generated.

Set $N_j=n_j+1+3j$. The intervals of levels on which the permutations $A_j,T_j$ act are disjoint, and $n_{j+1}>N_j$. Both $b$ and $c$ preserve the partitions into length-$N_j$ cylinders: the maps on $U_i$ with $i\leq j$ are prefix permutations at that level, and each $U_i$ with $i>j$ is contained in one level-$N_j$ cylinder. The odometer also preserves each such partition. Hence, the action is by automorphisms of the locally finite rooted tree obtained by retaining the levels $N_j$. Corollary~\ref{cor:rooted-finite-relation-AH} applies. Equivalently, for the return germs, the cylinders $[0^n]$ already form a common invariant basis.

Suppose there were a coherent realization of returns $\rho:H_\xi\to\F(\G_{\mathrm{w}}|_U)$. Its values at $u,t$ agree with $b,c$ on one sufficiently small cylinder about $\xi$. Every sufficiently late $U_j$ lies in that cylinder. The restrictions of $\rho(u),\rho(t)$ to $U_j$ are consequently exactly $A_j,T_j$, and preserve $U_j$. But the relation $[u,t^{L_j}ut^{-L_j}]=1$ of $H$ would then hold on $U_j$, contradicting~\eqref{eq:wreath-moving-relator}. This argument applies to any representatives with the required germs, not only to the original pair $b,c$.
\end{proof}

The example separates $\forall\mathscr{R}\,\exists U_{\mathscr{R}}$, with $\mathscr{R}$ a finite system of relations, from $\exists U\,\forall\mathscr{R}$. Every fixed relation holds on some neighborhood of $\xi$, while on every neighborhood a longer relation fails on some cylinder $U_j$. Finite generation and residual finiteness of the abstract wreath product do not repair this failure. The permutation tables on the cylinders $U_j$ have growing size, and the retained rooted-tree presentation has unbounded valency.

Despite this failure of coherence, finite-relation lifting still gives zero differential and a split abelianization.
\keepwithstatement
\begin{prop}
\label{prop:wreath-homology-parity}
For $\G_{\mathrm{w}}$,
\begin{equation}
\label{eq:wreath-homology-parity}
\begin{split}
H_0(\G_{\mathrm{w}})&=\mathbb{Z}[1/3],\qquad H_1(\G_{\mathrm{w}})=\mathbb{Z}^3,\\
H_n(\G_{\mathrm{w}})&\cong
\bigoplus_{(d_1,\ldots,d_{n-1})\in\mathbb{N}_{>0}^{\,n-1}}\mathbb{Z}\quad(n\geq2),\\
\theta_{\G_{\mathrm{w}}}&=0,\qquad
\F(\G_{\mathrm{w}})_{\ab}\cong\mathbb{Z}^3\oplus C_2.
\end{split}
\end{equation}
The classes of $a,b,c$ give the three index coordinates. The last coordinate is the eventual sign of the permutations at the retained levels $N_j$ and is nonzero on a cylinder transposition in $\F(\T)$.
\end{prop}
\begin{proof}
Each of $a,b,c$ is a full bisection with one exceptional source, and thus its individual defect is zero. The relative group is $\mathbb{Z}\oplus H^{\ab}$, and $H^{\ab}=\mathbb{Z}^2$. The asserted formulas for $H_0,H_1$ and the index representatives follow.

Put $B=\bigoplus_{n\in\mathbb{Z}}\mathbb{Z} e_n$. For a module with an action of $\mathbb{Z}$, superscripts and subscripts $\mathbb{Z}$ denote its invariants and coinvariants, respectively. The homology sequence for $B\rtimes\mathbb{Z}$, or the Lyndon--Hochschild--Serre spectral sequence, which has two columns here, gives
\begin{equation}
0\longrightarrow (\textstyle\bigwedge^n B)_{\mathbb{Z}}
\longrightarrow H_n(H;\mathbb{Z})
\longrightarrow (\textstyle\bigwedge^{n-1}B)^{\mathbb{Z}}
\longrightarrow0;
\end{equation}
see \cite[Chapter~VII]{Brown82}. For positive exterior degree, the invariant group is zero: the simultaneous shift has infinite orbits on the basis of finite increasing index sets, whereas an element has finite support. The coinvariants of $\bigwedge^n B$ are free on translation classes of increasing $n$-tuples. Such classes are parametrized by their positive consecutive gaps. Theorem~\ref{thm:higher-germ-homology} gives the asserted groupoid homology in the higher degrees. In degree two, the gap-$d$ generator is represented by $[u|t^dut^{-d}]-[t^dut^{-d}|u]$.

The lifting property for finite relations and Theorem~\ref{thm:finite-relations-split-AH} give $\theta_{\G_{\mathrm{w}}}=0$ and the stable splitting. For the ordinary full group, Proposition~\ref{prop:wreath-no-coherent-lift} shows that $\G_{\mathrm{w}}$ is minimal, compactly generated, and AF-by-discrete. Theorem~\ref{thm:minimal-AF-by-discrete-almost-finite} therefore gives almost finiteness, and the comparison result recalled in Section~\ref{sec:kernel} gives comparison. Hence the ordinary splitting also follows from Theorem~\ref{thm:finite-relations-split-AH}.

There is an independent parity check. Both permutations in~\eqref{eq:wreath-finite-approximations} are even, since their nontrivial cycles have the odd lengths $r_j$ and $M_j$. At the retained levels, the nontrivial sections of $b,c$ consist only of the even permutations on the cylinders $U_i$ with $n_i>N_j$, all contained in $[0^{N_j}]$. The ternary odometer induces a cycle of odd length on every finite ternary level, and hence also has even sign. A section of a fixed word thus has even sign at every sufficiently late retained level.

For an element of the full group given by a table of words on finitely many cylinders, choose $N_j$ beyond the table's cylinder lengths. All its words preserve the length-$N_j$ partition. Refining to later retained levels multiplies each old finite permutation by an odd number of copies and adds only even section permutations. Its sign is consequently constant. This gives a homomorphism $\epsilon_\infty:\F(\G_{\mathrm{w}})\to C_2$. It is zero on $a,b,c$ and is one on a cylinder transposition in $\F(\T)$, since its number of descendant copies is odd. As $H_0(\G_{\mathrm{w}};\mathbb{Z}/2\mathbb{Z})=C_2$, this proves directly that the parity homomorphism in Li's sequence is injective. Finally, $(I,\epsilon_\infty)$ gives the stated abelianization coordinates.
\end{proof}

In particular, $H_2$ is free of infinite rank and the whole group of germs has no coherent realization of returns, yet the parity differential is zero. This example lies outside the criterion using homology from finite subgroups: its group of germs is torsion-free, and thus finite subgroups cannot generate its nonzero second homology. 

\subsection{The Houghton obstruction on a Cantor space}
For an infinite set $D$, write $\operatorname{Sym}_{\mathrm{fin}}(D)$ for the finitely supported permutations and $\operatorname{Alt}_{\mathrm{fin}}(D)$ for its alternating subgroup. The quotient $\operatorname{Sym}(D)/\operatorname{Sym}_{\mathrm{fin}}(D)$ is the balanced near symmetric group in the realization by global permutations. The central extension
\begin{equation*}
1\longrightarrow \operatorname{Sym}_{\mathrm{fin}}(D)/\operatorname{Alt}_{\mathrm{fin}}(D)
\longrightarrow \operatorname{Sym}(D)/\operatorname{Alt}_{\mathrm{fin}}(D)
\longrightarrow \operatorname{Sym}(D)/\operatorname{Sym}_{\mathrm{fin}}(D)
\longrightarrow1
\end{equation*}
has kernel $\mathbb{Z}/2\mathbb{Z}$, detected by the sign of a finitely supported permutation. Pulling this extension back along a balanced near action gives the Kapoudjian class. For a near action of $\mathbb{Z}^2$, Cornulier's criterion says that the class is nonzero exactly when chosen permutation lifts of the two generators have an odd finitely supported commutator. See \cite[Proposition~8.C.4]{Cornulier19}. Example~8.C.5 of \cite{Cornulier19} is the three-ray Houghton action. We realize the same obstruction on a Cantor space and identify it with Li's secondary differential. The permutation sign criterion is the cited result of Cornulier.

Let $Y$ be a Cantor set, let $\mathcal{D}=\{1,2,3\}\times\mathbb{N}_0$, and form the one-point compactification
\begin{equation}
\label{eq:Houghton-Cantor-space}
X_{\mathrm{H}}=\{\infty\}\cup(\mathcal{D}\times Y).
\end{equation}
Each block $\{p\}\times Y$ is clopen, and a neighborhood of $\infty$ contains all but finitely many entire blocks. This is a compact metrizable zero-dimensional space without isolated points, hence a Cantor space.

We choose a Bratteli diagram that realizes the AF relation used below. Fix an enumeration $\mathcal{D}=\{p_1,p_2,\ldots\}$ and a homeomorphism $Y\cong\{0,1\}^{\omega}$. The diagram $\mathsf{B}_{\mathrm{H}}$ has an infinite spine $v_0\to v_1\to\cdots$. At level $n\geq1$, it also has vertices $w_u$ indexed by binary words $u$ of length $n$. There is an exit edge from $v_{n-1}$ to each such $w_u$, and the other edges are $w_u\to w_{u0},w_{u1}$. Distinct binary prefixes never merge along these latter edges. The spine path represents $\infty$. A path taking an exit from $v_{n-1}$ chooses the prefix of a unique $y\in Y$ of length $n$ and then follows its binary branch. It represents $(p_n,y)$. Hence $\Omega(\mathsf{B}_{\mathrm{H}})\cong X_{\mathrm{H}}$. Two paths that leave the spine are tail equivalent exactly when their $Y$-coordinates agree, regardless of their $\mathcal{D}$-coordinates. Prefix replacements correspond to the maps between blocks that preserve $y$ on a clopen cylinder in $Y$. This identification also gives the topology of the AF groupoid used below. We henceforth identify $X_{\mathrm{H}}$ with $\Omega(\mathsf{B}_{\mathrm{H}})$.

Define permutations $\alpha,\beta$ of $\mathcal{D}$ as follows:
\begin{equation}
\label{eq:Houghton-generators}
\begin{aligned}
\alpha(1,n)&=(1,n+1),&
\alpha(2,0)&=(1,0),&
\alpha(2,n+1)&=(2,n),\\
\beta(1,n)&=(1,n+1),&
\beta(3,0)&=(1,0),&
\beta(3,n+1)&=(3,n).
\end{aligned}
\end{equation}
The first permutation fixes the third ray and the second fixes the second ray. They act on $\Omega(\mathsf{B}_{\mathrm{H}})$ by these block permutations, preserving the coordinate in $Y$ and fixing $\infty$. These formulas are a choice of generators for the three-ray Houghton group. They differ by conventions from those in \cite[Example~8.C.5]{Cornulier19}. Direct evaluation gives
\begin{equation}
\label{eq:Houghton-odd-commutator}
[\alpha,\beta]=((1,0)\ (1,1)).
\end{equation}
The commutator is an exchange of two entire Cantor blocks. In particular, the images of $\alpha$ and $\beta$ commute modulo finitely supported block permutations and define the balanced near $\mathbb{Z}^2$-action whose Kapoudjian class is nonzero.

Let $G_{\mathrm{H}}=\langle\alpha,\beta\rangle$ and let $\G_{\mathrm{H}}$ be its groupoid of germs on $\Omega(\mathsf{B}_{\mathrm{H}})$. Let $\mathfrak{T}_{\mathsf{B}_{\mathrm{H}}}$ contain the units and all arrows between $(p,y)$ and $(q,y)$, for $p,q\in\mathcal{D}$, with no nontrivial arrows at $\infty$. Thus the AF core records finite block permutations that preserve the $Y$-coordinate, while the exceptional germs at $\infty$ remember the eventual translations of the three rays.

The block action has a single exceptional orbit, whose $\mathbb{Z}^2$ isotropy determines the positive-degree groupoid homology.
\keepwithstatement
\begin{prop}
\label{prop:Houghton-AF-core}
The groupoid $\G_{\mathrm{H}}$ is effective, compactly generated, and AF-by-discrete relative to the open AF subgroupoid $\mathfrak{T}_{\mathsf{B}_{\mathrm{H}}}$. Its only exceptional orbit is $\{\infty\}$, and its group of germs there is $\mathbb{Z}^2$. Let $C(Y,\mathbb{Z})$ denote the group of continuous integer-valued functions on $Y$. Then
\begin{equation}
\label{eq:Houghton-homology}
\begin{split}
H_0(\G_{\mathrm{H}})&\cong\mathbb{Z}\oplus C(Y,\mathbb{Z}),\\
H_1(\G_{\mathrm{H}})&\cong\mathbb{Z}^2,\qquad H_2(\G_{\mathrm{H}})\cong\mathbb{Z},
\qquad H_n(\G_{\mathrm{H}})=0\quad(n\geq3).
\end{split}
\end{equation}
In the first line, $[1_{\Omega(\mathsf{B}_{\mathrm{H}})}]=(1,0)$ and the class of one block is $e=(0,1_Y)$.
\end{prop}
\begin{proof}
The commutator~\eqref{eq:Houghton-odd-commutator} and its conjugates give the edge transpositions of a connected graph on $\mathcal{D}$: powers of $\alpha$ give adjacent exchanges along the first and second rays, while conjugates using $\beta$ connect the third ray. Consequently $G_{\mathrm{H}}$ contains every finitely supported permutation of $\mathcal{D}$.

Every element of the Houghton group is an eventual translation on each ray. The translation vectors of $\alpha$ and $\beta$ are $(1,-1,0)$ and $(1,0,-1)$. They form a basis of the subgroup $\{(n_1,n_2,n_3)\in\mathbb{Z}^3:n_1+n_2+n_3=0\}$. An element has zero translation vector exactly when it is finitely supported. Near $\infty$, every finitely supported block permutation is the identity, and thus the germ map kills precisely this subgroup and identifies $H_\infty$ with $\mathbb{Z}^2$. At a point $(p,y)$ in a block, a germ is determined by the source block, target block, and unchanged $Y$-coordinate. If an element fixes such a point, then its block permutation fixes $p$, and the action is the identity on the whole block. Hence, the restriction away from $\infty$ is principal. Since the groupoid is a groupoid of germs it is effective, and the two global bisections $\alpha,\beta$ compactly generate it.

Enumerate $\mathcal{D}$ and, at stage $m$, allow every pair arrow among its first $m$ blocks while retaining units elsewhere. This is an elementary compact open subgroupoid, and the increasing union of these stages is $\mathfrak{T}_{\mathsf{B}_{\mathrm{H}}}$. Every nonidentity germ away from $\infty$ belongs to this AF core, whereas the nonidentity germs at $\infty$ are exactly the nonzero eventual-translation classes. A representative of such a class agrees with an AF block permutation away from $\infty$, and thus a sufficiently small bisection around its germ meets $\G_{\mathrm{H}}\setminus\mathfrak{T}_{\mathsf{B}_{\mathrm{H}}}$ only in that germ. Hence the complement is discrete. The groupoid $\G_{\mathrm{H}}$ is nevertheless non-Hausdorff: for example, $\alpha$ has identity germs on the whole third ray, and these units accumulate at its nonidentity germ at $\infty$.

If $f\in C(\Omega(\mathsf{B}_{\mathrm{H}}),\mathbb{Z})$ and $c=f(\infty)$, then $f=c$ on every sufficiently late block. The map
\begin{equation}
\label{eq:Houghton-dimension-coordinate}
f\longmapsto\left(c,\sum_{p\in\mathcal{D}}(f|_{\{p\}\times Y}-c)\right)
\end{equation}
is a finite sum and induces an isomorphism $H_0(\mathfrak{T}_{\mathsf{B}_{\mathrm{H}}})\cong\mathbb{Z}\oplus C(Y,\mathbb{Z})$. Indeed, the relations coming from $\mathfrak{T}_{\mathsf{B}_{\mathrm{H}}}$ identify copies of the same function on different blocks, and these generate the kernel of~\eqref{eq:Houghton-dimension-coordinate}. The two full bisections $\alpha,\beta$ each have one exceptional arrow and have source and range $\Omega(\mathsf{B}_{\mathrm{H}})$. Their individual defects vanish. The exceptional orbit is the singleton $\{\infty\}$, which is one $\mathfrak{T}_{\mathsf{B}_{\mathrm{H}}}$-orbit. Hence $H_1(\mathfrak{G}_{\mathrm{H}},\mathfrak{T}_{\mathsf{B}_{\mathrm{H}}})\cong\mathbb{Z}^2$ and the connecting map is zero. The relative exact sequence gives $H_0,H_1$, and Theorem~\ref{thm:higher-germ-homology} gives the remaining homology groups.
\end{proof}

The groupoid is not minimal. Its singleton orbit $\{\infty\}$ also prevents almost finiteness, exactly as in the earlier counterexamples with finite groups of germs: every elementary subgroupoid has only its unit arrow at $\infty$, while multiplication by a bisection with a nonidentity germ there gives an arrow outside that elementary subgroupoid. The ratio in Definition~\ref{def:almost-finite} is therefore at least one. Every invariant probability measure is concentrated at $\infty$. In fact, permutations of finitely many blocks force all block masses to be equal, hence zero. This last observation explains why we cannot make this example minimal by simply adding arrows while keeping these compact clopen blocks and an invariant probability measure of full support.

The odd block commutator yields a nonzero secondary differential, even though the group of germs is abelian.
\keepwithstatement
\begin{thm}
\label{thm:Houghton-nonzero-theta}
In the coordinates of~\eqref{eq:Houghton-homology}, the secondary map in Li's stable sequence is
\begin{equation}
\label{eq:Houghton-nonzero-theta}
\theta_{\G_{\mathrm{H}}}:\mathbb{Z}\longrightarrow
\mathbb{Z}/2\mathbb{Z}\oplus C(Y,\mathbb{Z}/2\mathbb{Z}),
\qquad m\longmapsto(0,(m\bmod2)1_Y).
\end{equation}
In particular, AF-by-discreteness alone does not imply $\theta_{\G}=0$, even for a compactly generated groupoid whose only nontrivial group of germs is the abelian group $\mathbb{Z}^2$. The stable abelianization is
\begin{equation}
\label{eq:Houghton-stable-abelianization}
\F(\mathfrak{R}\times\G_{\mathrm{H}})_{\ab}
\cong\mathbb{Z}^2\oplus\mathbb{Z}/2\mathbb{Z}\oplus
\frac{C(Y,\mathbb{Z}/2\mathbb{Z})}{\langle1_Y\rangle}.
\end{equation}
\end{thm}
\begin{proof}
Let $V$ be the bisection from $(1,0)\times Y$ onto $(1,1)\times Y$ induced by the block exchange. The associated dynamical transposition $t_V$ is exactly the commutator in~\eqref{eq:Houghton-odd-commutator}. The stable parity map in Li's Theorem~6.12 sends the mod-two class $e$ of the source block to the stabilized class $[t_V]$. But $t_V=[\alpha,\beta]$ is already a commutator in the ordinary full group, and thus its image in the stabilized abelianization is zero. Exactness of Li's stable sequence therefore places $e$ in $\operatorname{im}\theta_{\G_{\mathrm{H}}}$.

By Proposition~\ref{prop:Houghton-AF-core}, $H_2(\G_{\mathrm{H}})\cong\mathbb{Z}$, while $H_0(\G_{\mathrm{H}};\mathbb{Z}/2\mathbb{Z})\cong\mathbb{Z}/2\mathbb{Z}\oplus C(Y,\mathbb{Z}/2\mathbb{Z})$. The element $e=(0,1_Y)$ is nonzero. Since the target has exponent two, the image of any homomorphism from $\mathbb{Z}$ has at most two elements. We have just shown that it contains the nonzero element $e$, and thus it is exactly $\{0,e\}$. This determines $\theta_{\G_{\mathrm{H}}}$ and proves~\eqref{eq:Houghton-nonzero-theta}. The cokernel of $\theta$ is
$(\mathbb{Z}/2\mathbb{Z}\oplus C(Y,\mathbb{Z}/2\mathbb{Z}))/\langle e\rangle$, and the remaining quotient $H_1(\G_{\mathrm{H}})\cong\mathbb{Z}^2$ is free. Hence the stable AH extension splits abstractly, giving~\eqref{eq:Houghton-stable-abelianization}.
\end{proof}

The proof uses only the relative calculation in low degrees, the explicit block commutator, and Li's stable exactness. For this realization, the resulting nonzero class is the Cantor-block version of the Kapoudjian obstruction from \cite[Section~8.C]{Cornulier19}.

\begin{cor}
\label{cor:Houghton-no-finite-coherence}
The lifting property for finite relations fails at $\infty$ in $\G_{\mathrm{H}}$, already for the identity homomorphism $\mathbb{Z}^2\to H_\infty$. There is no coherent realization of returns of $H_\infty$.
\end{cor}
\begin{proof}
The identity homomorphism $\mathbb{Z}^2\to H_\infty$ has a fundamental class generating $H_2(\mathbb{Z}^2;\mathbb{Z})$. If Definition~\ref{def:finite-relation-lifting} held for this homomorphism, the resulting local action would evaluate that class to the generator of the $H_2(H_\infty)$ summand in $H_2(\G_{\mathrm{H}})$, exactly as in the first paragraph of the proof of Theorem~\ref{thm:finite-relations-split-AH}. Lemma~\ref{lem:coherent-parity} would force its image under $\theta_{\G_{\mathrm{H}}}$ to vanish, contradicting~\eqref{eq:Houghton-nonzero-theta}. A coherent realization of the whole group of germs would in particular give such a finite-relation lift, and thus it is impossible as well.

There is also a direct permutation proof. Extend any proposed local commuting lifts by the identity. They differ from $\alpha$ and $\beta$ by elements of $\F(\mathfrak{T}_{\mathsf{B}_{\mathrm{H}}})$, hence by finitely supported block permutations depending locally constantly on $y\in Y$. For each fixed $y$, changing the two Houghton lifts by finitely supported permutations does not change the sign of their finitely supported commutator: after quotienting by the even finitely supported permutations, the remaining order-two kernel is central. The commutator therefore remains odd, whereas two commuting lifts would have identity commutator. This is precisely Cornulier's criterion in \cite[Proposition~8.C.4]{Cornulier19}.
\end{proof}

\begin{rmk}
\label{rmk:Houghton-ordinary-full-group}
Let $N$ be the group of continuous maps from $Y$ to the discrete group $\operatorname{Sym}_{\mathrm{fin}}(\mathcal{D})$. Compactness of $Y$ makes the image of such a map finite, and thus the union of all block supports appearing in it is finite. These are exactly the full-group elements of the AF core, and hence $N=\F(\mathfrak{T}_{\mathsf{B}_{\mathrm{H}}})$. The germ at $\infty$ gives a surjective homomorphism $\F(\G_{\mathrm{H}})\to H_\infty\cong\mathbb{Z}^2$ with kernel $N$. Pointwise sign induces
$N_{\ab}\cong C(Y,\mathbb{Z}/2\mathbb{Z})$: on each clopen piece this is the usual abelianization of the finitary symmetric group, and compactness permits only finitely many such pieces. Conjugation by $\alpha$ and $\beta$ merely relabels blocks and therefore preserves pointwise sign. Modulo $[N,N]$, the only relation needed to commute lifts of the two quotient generators is~\eqref{eq:Houghton-odd-commutator}, whose pointwise sign is the constant function $1_Y$. Hence
\begin{equation}
\label{eq:Houghton-ordinary-abelianization}
\F(\G_{\mathrm{H}})_{\ab}\cong
\mathbb{Z}^2\oplus C(Y,\mathbb{Z}/2\mathbb{Z})/\langle1_Y\rangle.
\end{equation}
There are no further abelian relations. Indeed, modulo $[N,N]$ every element can be written as $f\alpha^m\beta^n$ with $f\in C(Y,\mathbb{Z}/2\mathbb{Z})$, and multiplication is the direct-product law with the central correction $nm'1_Y$ coming from commuting $\beta^n$ past $\alpha^{m'}$. Abelianization therefore kills exactly the subgroup generated by the constant function $1_Y$, proving~\eqref{eq:Houghton-ordinary-abelianization}.

The index detects only the germ in the singleton exceptional orbit, and thus $\ker I=N$. A dynamical transposition has identity germ at that singleton orbit, while every element of the AF full group $N$ is a product of dynamical transpositions. Hence $\ker I=N=\Sym(\G_{\mathrm{H}})$. Thus nonzero $\theta$ can coexist with equality of the index kernel and the dynamical symmetric group. The additional $\mathbb{Z}/2\mathbb{Z}$ in the stable formula~\eqref{eq:Houghton-stable-abelianization} is represented by a transposition exchanging two copies of the singleton orbit in the stabilization. Minimality fails, so the hypotheses identifying the stable and ordinary sequences are not satisfied.
\end{rmk}

\section{Surface returns and the secondary differential}\label{sec:surface-parity}

\begin{defn}[Surface relation]
\label{def:surface-relation}
A \emph{surface relation} in a group $H$ consists of labels $h_1,k_1,\ldots,h_g,k_g$ satisfying $\prod_{i=1}^g[h_i,k_i]=1$.
\end{defn}
Such a relation determines a homomorphism $\pi_1(\Sigma_g)\to H$ and hence, by the oriented fundamental class, a class in $H_2(H;\mathbb{Z})$.

Every class in $H_2(H;\mathbb{Z})=H_2(BH;\mathbb{Z})$ is represented by the image of the fundamental class of a closed oriented surface. This is the standard degree-two oriented-bordism realization: in total degree two the oriented-bordism spectral sequence has no positive-degree coefficient contribution. Thus surface relations lose no degree-two information. Their advantage here is geometric: after the labels are lifted from germs to full-group elements, the surface relation need no longer hold as an equality in $\F(\G)$. The product of commutators of the chosen lifts is AF away from the marked points and has the identity germ at each marked point, and thus lies in $\F(\T)$. Its AF parity records exactly the secondary differential.

\subsection{Parity of returns and relative transgression}
\label{subsec:return-parity}

Recall from Subsection~\ref{subsec:AFbd-homology} that $\mathsf{P}_{\mathsf{B}}=H_0(\T;\mathbb{Z}/2\mathbb{Z})$ and that the AF parity homomorphism $\varepsilon_{\mathsf{B}}:\F(\T)\to \mathsf{P}_{\mathsf{B}}$ records the signs of the finite permutations at a sufficiently deep elementary stage. Let
\begin{equation}
\label{eq:AF-parity-to-G}
q_2:\mathsf{P}_{\mathsf{B}}\longrightarrow H_0(\G;\mathbb{Z}/2\mathbb{Z})
\end{equation}
be the homomorphism induced by the inclusion $\T\subseteq\G$.

After lifting the generators of a punctured surface, the product of the lifted generators representing the boundary loop gives an element of the AF full group. The relative cofiber boundary records the stable class of this element, and the AF homology vanishing identifies that class with its parity.

For a connective spectrum $E$, write $H_i(E;A)=\pi_i(HA\wedge E)$ for its stable homology, where $HA$ is the Eilenberg--Mac Lane spectrum of $A$. For the spectra $\mathscr{E}_{\G}$ of Subsection~\ref{subsec:coherent-evaluation}, Li's Theorem~4.18 identifies this reduced stable homology with groupoid homology. We denote by $\theta_E:H_2(E;\mathbb{Z})\to H_0(E;\mathbb{Z}/2\mathbb{Z})$ the map appearing in the five-term Atiyah--Hirzebruch sequence for the sphere spectrum. For $E=\mathscr{E}_{\G}$ it is the differential $d^2_{2,0}$ used in the proof of \cite[Theorem~6.12]{Li25}, and hence it is the map denoted by $\theta_{\G}$ in this paper.
\keepwithstatement
\begin{lemma}
\label{lem:relative-spectrum-transgression}
Let $A\to B\to C$ be a cofiber sequence of connective spectra. Suppose that $H_i(A;\mathbb{Z})=0$ for $i>0$ and that $H_0(A;\mathbb{Z})\to H_0(B;\mathbb{Z})$ is onto. Then $H_2(B;\mathbb{Z})\to H_2(C;\mathbb{Z})$ is an isomorphism and $\pi_2(C)\to H_2(C;\mathbb{Z})$ is onto. If $\xi\in\pi_2(C)$ maps to the image of $c\in H_2(B;\mathbb{Z})$, then $\theta_B(c)$ is the image of the cofiber boundary $\partial\xi\in\pi_1(A)$ under
\begin{equation}
\label{eq:relative-spectrum-transgression}
\pi_1(A)\cong H_0(A;\mathbb{Z}/2\mathbb{Z})
\longrightarrow H_0(B;\mathbb{Z}/2\mathbb{Z}).
\end{equation}
\end{lemma}
\begin{proof}
Let $Q$ be the cofiber of the unit $\mathbb{S}\to H\mathbb{Z}$, where $\mathbb{S}$ is the sphere spectrum. We have $\pi_i(Q)=0$ for $i<2$ and $\pi_2(Q)=\mathbb{Z}/2\mathbb{Z}$, and thus the $2$-truncation of $Q$ is $\Sigma^2H(\mathbb{Z}/2\mathbb{Z})$. If $E$ is connective, smashing the Postnikov map with $E$ is still an isomorphism on $\pi_2$ and gives
\begin{equation*}
\pi_2(Q\wedge E)\cong H_0(E;\mathbb{Z}/2\mathbb{Z}),
\qquad
\pi_1(Q\wedge E)=0.
\end{equation*}
The cofiber sequence obtained by smashing $\mathbb{S}\to H\mathbb{Z}\to Q$ with $E$ thus gives
\begin{equation*}
\pi_2(E)\longrightarrow H_2(E;\mathbb{Z})
\xrightarrow{\theta_E}H_0(E;\mathbb{Z}/2\mathbb{Z})
\longrightarrow\pi_1(E)\longrightarrow H_1(E;\mathbb{Z})\longrightarrow0.
\end{equation*}
The connecting homomorphism displayed here is the differential $d^2_{2,0}$ in the five-term sequence of the Atiyah--Hirzebruch spectral sequence for the sphere spectrum. This is precisely the sequence in low degrees used in the proof of \cite[Theorem~6.12]{Li25}.

The homology sequence of $A\to B\to C$ gives $H_0(C;\mathbb{Z})=0$ since $H_0(A)\to H_0(B)$ is onto. Since $H_1(A)=H_2(A)=0$, the same sequence gives the asserted isomorphism $H_2(B)\cong H_2(C)$. The cofiber $C$ is connective, and thus $H_{-1}(C;\mathbb{Z})=0$ and the degree-zero stable Hurewicz map identifies $\pi_0(C)$ with $H_0(C;\mathbb{Z})$. The universal coefficient sequence therefore gives
$H_0(C;\mathbb{Z}/2\mathbb{Z})=0$. The five-term sequence for $C$ now makes $\pi_2(C)\to H_2(C;\mathbb{Z})$ surjective. For $A$, the hypotheses $H_1(A)=H_2(A)=0$ reduce the same five-term sequence to the canonical isomorphism in~\eqref{eq:relative-spectrum-transgression}.

Smash the cofiber sequence $A\to B\to C$ with the sequence $\mathbb{S}\to H\mathbb{Z}\to Q$. This gives a commuting diagram whose rows and columns are cofiber sequences:
\begin{equation*}
\begin{CD}
A @>>> B @>>> C\\
@VVV @VVV @VVV\\
H\mathbb{Z}\wedge A @>>> H\mathbb{Z}\wedge B @>>> H\mathbb{Z}\wedge C\\
@VVV @VVV @VVV\\
Q\wedge A @>>> Q\wedge B @>>> Q\wedge C.
\end{CD}
\end{equation*}
The image of $c$ in $\pi_2(H\mathbb{Z}\wedge C)$ lifts to $\xi$. Chase this lift around the diagram. Its boundary in $\pi_1(A)$ corresponds, under the left column, to a class in $\pi_2(Q\wedge A)$ whose image in $\pi_2(Q\wedge B)$ is the image of $c$ under the middle column. These two descriptions are exactly~\eqref{eq:relative-spectrum-transgression} and $\theta_B(c)$. The sign convention for cofiber boundaries has no effect in a group of exponent two. This also proves that the result is independent of the chosen lift $\xi$.
\end{proof}

Let $E\subseteq\Omega(\mathsf{B})$ be finite. Suppose that $u_1,v_1,\ldots,u_g,v_g\in\F(\G)$ fix $E$ pointwise and are AF outside $E$, and put $r=\prod_i[u_i,v_i]$. Assume that $r\in\F(\T)$. Since $r$ fixes $E$ pointwise and $\T$ is principal, the germ of $r$ at every $x\in E$ is the identity. Hence the germs $(u_i,x),(v_i,x)\in\G_x^x$ satisfy the surface relation and therefore define a class $c_x\in H_2(\G_x^x;\mathbb{Z})$. If several marked points lie in the same $\G$-orbit, connectors identify their groups of germs up to inner automorphism. Inner automorphisms act trivially on group homology. Hence the sum of the $c_x$ is a well-defined class $c\in H_2(\G)$ under Theorem~\ref{thm:higher-germ-homology}.

\begin{lemma}
\label{lem:marked-surface-parity}
With the preceding hypotheses, let $c\in H_2(\G)$ be this sum of surface classes. Then
\begin{equation}
\label{eq:marked-surface-parity}
\theta_{\G}(c)=q_2\bigl(\varepsilon_{\mathsf{B}}(r)\bigr).
\end{equation}
\end{lemma}
\begin{proof}
Let $L\leq\F(\G)$ be generated by the displayed lifts and put $N=L\cap\F(\T)$. Every displayed generator fixes $E$ pointwise, and hence preserves the complement of $E$ setwise. At every point outside $E$ its germ belongs to $\T$. The same holds for its inverse. Since $\T$ is a subgroupoid, every word in these generators is therefore AF away from $E$. Evaluation of germs consequently gives a homomorphism
\begin{equation*}
L\longrightarrow\prod_{x\in E}\G_x^x,
\qquad
u\longmapsto\bigl((u,x)\bigr)_{x\in E}.
\end{equation*}
Its kernel is exactly $N$. Indeed, if all marked germs are trivial, the preceding observation shows that the element is AF everywhere. Conversely, an element of $L\cap\F(\T)$ fixes every $x\in E$, and principality of $\T$ forces its germ there to be the identity. In particular, $N\trianglelefteq L$. The surface relator has trivial germ at every marked point and hence lies in $N$.

Remove an open disc from $\Sigma_g$. The punctured surface has free fundamental group on the standard surface generators, and thus the chosen lifts define a map to $BL$. Its oriented boundary is the loop represented by $r$, now regarded as a loop in $BN$. This gives a relative class $z\in H_2(BL,BN;\mathbb{Z})$ with $\partial z=[r]\in H_1(N;\mathbb{Z})$. Thus the ordinary topological boundary of the punctured surface is exactly the element $r\in\F(\T)$.

Apply coherent evaluation to the actual actions $N\to\F(\T)$ and $L\to\F(\G)$. The relative compatibility proved in Lemma~\ref{lem:coherent-parity} gives a compatible map from the cofiber of $\Sigma^\infty(BN)_+\to\Sigma^\infty(BL)_+$ to the cofiber of $\mathscr{E}_{\T}\to\mathscr{E}_{\G}$, and identifies the induced homology map with the map obtained by passing $Q_{\rho,*}$ to the quotient complexes. Since $BN$ and $BL$ are connected, the first cofiber has zero integral homology in degree zero, and therefore also zero degree-zero homology with $\mathbb{Z}/2\mathbb{Z}$-coefficients. Its low-degree Atiyah--Hirzebruch sequence consequently makes the stable Hurewicz map onto in degree two. Choose a stable lift of $z$.

Use the canonical splittings
\begin{equation*}
\Sigma^\infty(BN)_+\simeq\mathbb{S}\vee\Sigma^\infty BN,
\qquad
\Sigma^\infty(BL)_+\simeq\mathbb{S}\vee\Sigma^\infty BL.
\end{equation*}
The map induced by $N\hookrightarrow L$ is the identity on the sphere summands. Hence the sphere summands cancel after passage to the cofiber. Equivalently, the cofiber boundary has zero component in $\pi_1(\mathbb{S})\cong\mathbb{Z}/2\mathbb{Z}$. Its $H_1(N;\mathbb{Z})$-component is $[r]$ by naturality of the stable Hurewicz map. Thus coherent evaluation sends the boundary to the stable loop represented by $r$ in $\pi_1(\mathscr{E}_{\T})$.

Now use that $\T$ is AF. Its positive integral homology vanishes. The five-term sequence used in Lemma~\ref{lem:relative-spectrum-transgression}, together with Li's identification of spectrum homology with groupoid homology, therefore gives the canonical isomorphism
$\pi_1(\mathscr{E}_{\T})\cong H_0(\T;\mathbb{Z}/2\mathbb{Z})=\mathsf{P}_{\mathsf{B}}$.
We identify the stable loop of $r$ under this isomorphism. The full bisection of $r$ is compact and $\T=\bigcup_n\mathfrak{T}_{\mathsf{B},n}$ is an increasing open exhaustion, so it is contained in one elementary stage. At that stage, $r$ is a finite permutation on each tower. The calculation in the proof of \cite[Theorem~6.12]{Li25}, together with naturality for the inclusion of an elementary pair groupoid, sends a transposition of two levels to the mod-two class of one source level. Additivity therefore sends a permutation of a tower to its sign. Passing to the Bratteli direct limit gives exactly the parity vector of Definition~\ref{def:AF-parity}. Hence the stable class of $r$ is $\varepsilon_{\mathsf{B}}(r)$.

Finally, apply the relative compatibility from Lemma~\ref{lem:coherent-parity} to the class $z$. After passing $Q_{\rho,*}$ to the quotient by the chains coming from $\T$, the image of $z$ is supported only at the marked set $E$. At each $x\in E$ it is the bar cycle of the germ surface relation. Under Theorem~\ref{thm:higher-germ-homology}, the sum of these cycles is precisely $c\in H_2(\G)\cong H_2(\G,\T)$. We may therefore apply Lemma~\ref{lem:relative-spectrum-transgression} to the cofiber sequence $\mathscr{E}_{\T}\to\mathscr{E}_{\G}$: the AF property gives $H_i(\T;\mathbb{Z})=0$ for $i>0$, and $H_0(\T)\to H_0(\G)$ is onto since the relative chain group in degree zero is zero. The cofiber boundary is the class $\varepsilon_{\mathsf{B}}(r)$ just computed, and its image in the absolute group is obtained by $q_2$. This gives~\eqref{eq:marked-surface-parity}.
\end{proof}

\subsection{The formula for a single return}
Fix an exceptional orbit $\mathcal{O}$, a basepoint $x\in\mathcal{O}$, and write $H=H_{\mathcal{O}}=\G_x^x$. Recall from Subsection~\ref{subsec:defect-map} that the isotropy defect is
\begin{equation}
\lambda_{\mathcal{O}}:H^{\ab}\longrightarrow H_0(\T).
\end{equation}
Assume in this subsection that $\lambda_{\mathcal{O}}=0$. If $h\in H$ is nonidentity, choose an isolating bisection representing $h$. Its defect in $H_0(\T)$ is zero, since its relative class has return coordinate $[h]_{\ab}$ and no transport coordinate. Lemma~\ref{lemma:AF-clopen-completion} therefore completes it by AF arrows to an element $u_h\in\F(\G)$ that fixes $x$, has germ $h$ at $x$, and is AF at every other point. For the identity germ we take $u_{(\Id,x)}=\Id$. These are the lifts at $x$ used in this subsection.

\keepwithstatement
\begin{thm}
\label{thm:surface-return-parity}
Assume that $\lambda_{\mathcal{O}}=0$, and let
$c\in H_2(H;\mathbb{Z})\subseteq H_2(\G;\mathbb{Z})$
belong to the corresponding orbit summand under Theorem~\ref{thm:higher-germ-homology}. Represent $c$ by a map $f:\Sigma\to BH$ from a closed oriented surface. On a connected component of genus $g$, choose the standard generators
\begin{equation}
a_1,b_1,\ldots,a_g,b_g,
\qquad
\prod_{i=1}^{g}[a_i,b_i]=1
\end{equation}
in its fundamental group, and put $h_i=f_*(a_i)$ and $k_i=f_*(b_i)$. Choose lifts $u_i,v_i\in\F(\G)$ of $h_i,k_i$ at $x$ as above, with AF germs at every other point. Then
\begin{equation}
\label{eq:surface-return-relation}
r_f=\prod_{i=1}^{g}[u_i,v_i]\in\F(\T).
\end{equation}
For a disconnected surface, take the product over its components. The secondary differential is
\begin{equation}
\label{eq:surface-return-parity}
\theta_{\G}(c)=q_2\bigl(\varepsilon_{\mathsf{B}}(r_f)\bigr).
\end{equation}
In particular, the right-hand side is independent of the surface representative and of all choices of these lifts.
\end{thm}

\begin{proof}
Every integral degree-two homology class of $BH$ is represented by the image of the fundamental class of a closed oriented surface. A spherical component maps trivially to $H_2(BH;\mathbb{Z})$ and may be discarded. On each remaining component, choose the standard surface generators and their lifts at $x$.

Each lift fixes $x$ and is AF away from $x$. Hence the complement of $x$ is invariant under every lift, and a word based outside $x$ remains outside $x$ while all of its successive germs lie in $\T$. Thus $r_f$ is AF away from $x$. At $x$, its germ is the product of commutators of the labels $h_i,k_i$, which is the identity since these labels come from a homomorphism $\pi_1(\Sigma_g)\to H$. Consequently the germ of $r_f$ belongs to $\T$ at every point, and $r_f\in\F(\T)$.

Apply Lemma~\ref{lem:marked-surface-parity} with marked set $E=\{x\}$. The germ surface class appearing there is exactly $f_*[\Sigma]$, namely the prescribed class $c$ in the $\mathcal{O}$-summand of Theorem~\ref{thm:higher-germ-homology}. This gives~\eqref{eq:surface-return-parity}. For a disconnected surface, the fundamental class is the sum of the component classes and $\varepsilon_{\mathsf{B}}$ is a homomorphism, and thus the same formula holds for the product of the component relators. Since the left-hand side depends only on $c$, the right-hand side is independent of the surface representative and of every choice of lifts at $x$.
\end{proof}

\begin{defn}[Parity homomorphism for returns]
\label{def:return-parity-homomorphism}
When $\lambda_{\mathcal{O}}=0$, define
\begin{equation}
\label{eq:return-parity-homomorphism}
\kappa_{\mathcal{O}}:H_2(H_{\mathcal{O}};\mathbb{Z})
\longrightarrow H_0(\G;\mathbb{Z}/2\mathbb{Z})
\end{equation}
by the right-hand side of~\eqref{eq:surface-return-parity}. We call $\kappa_{\mathcal{O}}$ the \emph{parity homomorphism for returns} of the exceptional orbit.
\end{defn}
By Theorem~\ref{thm:surface-return-parity}, $\kappa_{\mathcal{O}}$ is the restriction of $\theta_{\G}$ to that orbit summand.

When the lifts at $x$ determine a balanced near action and AF parity reduces to the parity of finitely supported permutations, $\kappa_{\mathcal{O}}$ is the evaluation of the Kapoudjian class on $H_2(H_{\mathcal{O}};\mathbb{Z})$. Cornulier describes the corresponding cocycle by the parity of the finitely supported permutation measuring the failure of the chosen lifts to multiply exactly. See the proof of \cite[Proposition~8.C.3]{Cornulier19}.

\begin{cor}
\label{cor:commuting-return-parity}
Assume $\lambda_{\mathcal{O}}=0$. If $h,k\in H_{\mathcal{O}}$ commute and $u_h,u_k$ are lifts at the basepoint with AF germs at every other point, then
\begin{equation}
\label{eq:commuting-return-parity}
\theta_{\G}\bigl([h|k]-[k|h]\bigr)
=
q_2\bigl(\varepsilon_{\mathsf{B}}([u_h,u_k])\bigr).
\end{equation}
Here the bar cycle denotes the image of the fundamental class of the torus.
\end{cor}
\begin{proof}
Apply Theorem~\ref{thm:surface-return-parity} to a torus, using its standard pair of generators.
\end{proof}

Formula~\eqref{eq:commuting-return-parity} puts the Houghton calculation in its natural form. In Theorem~\ref{thm:Houghton-nonzero-theta}, the two return germs commute, the chosen lifts are $\alpha,\beta$, and their commutator is one transposition exchanging two Cantor blocks. Its parity is the constant class $1_Y$. Hence the parity homomorphism for returns is nonzero. This is exactly the criterion using an odd commutator for Cornulier's Kapoudjian class in the balanced near $\mathbb{Z}^2$ action \cite[Proposition~8.C.4 and Example~8.C.5]{Cornulier19}.

The lifting property for finite relations gives $\kappa_{\mathcal{O}}=0$. A surface representative uses only finitely many labels and the single surface relator. Definition~\ref{def:finite-relation-lifting} allows these finitely many relations to be realized on one common clopen neighborhood, and thus the surface relator is the identity there. After extending the local action by the identity outside this neighborhood, the corresponding product of commutators is the identity and therefore has zero AF parity. Theorem~\ref{thm:surface-return-parity} gives $\kappa_{\mathcal{O}}=0$. Thus Theorem~\ref{thm:finite-relations-split-AH} has two logically distinct parts: finite-relation lifting first gives $\theta_{\G}=0$, and the subsequent exponent-two argument constructs a splitting of the resulting AH extension. The vanishing of the homomorphisms $\kappa_{\mathcal{O}}$ alone proves only the first statement.

\subsection{Balancing at three sites and arbitrary isotropy defect}
\label{subsec:balanced-return-parity}
The condition $\lambda_{\mathcal{O}}=0$ allows each return germ to be completed at one point. With three points in the same exceptional orbit, we instead balance the abelianized return labels before constructing the corresponding full-group
element. Two surface homeomorphisms cancel the degree-one labels, while the three degrees $1,1,-1$ preserve the original degree-two class.

\keepwithstatement
\begin{lemma}
\label{lem:three-surface-balancing}
Let $\Sigma_g$ be a closed connected oriented surface of genus $g\geq1$. There exist self-homeomorphisms $\varphi_+,\varphi_-:\Sigma_g\to\Sigma_g$ such that
\begin{equation}
\label{eq:three-surface-degrees}
\deg(\varphi_+)=1,
\qquad
\deg(\varphi_-)=-1,
\end{equation}
and
\begin{equation}
\label{eq:three-surface-H1}
\operatorname{id}+ (\varphi_+)_*+(\varphi_-)_*=0
\quad\text{on }H_1(\Sigma_g;\mathbb{Z}).
\end{equation}
\end{lemma}

\begin{proof}
For one pair of vectors in a symplectic basis, use
\begin{equation}
\label{eq:surface-balance-matrices}
A_0=
\begin{pmatrix}
-2&-1\\
-1&-1
\end{pmatrix},
\qquad
B_0=
\begin{pmatrix}
1&1\\
1&0
\end{pmatrix}.
\end{equation}
Then $A_0+B_0=-I_2$, $\det A_0=1$ and $\det B_0=-1$. If
$J_0=\left(\begin{smallmatrix}0&1\\-1&0\end{smallmatrix}\right)$, direct multiplication gives
$A_0^{\mathsf{t}}J_0A_0=J_0$ and $B_0^{\mathsf{t}}J_0B_0=-J_0$. Take the block diagonal matrices $A=A_0^{\oplus g}$ and $B=B_0^{\oplus g}$ in a symplectic basis of $H_1(\Sigma_g;\mathbb{Z})$. It follows that $A$ is symplectic, $B$ is anti-symplectic and $I+A+B=0$. The symplectic representation of the orientation-preserving mapping class group is onto $\operatorname{Sp}(2g,\mathbb{Z})$. See, for example, \cite[Chapter~6]{FM12}. Hence $A$ is realized by an orientation-preserving homeomorphism. Fix any orientation-reversing homeomorphism $R$ of $\Sigma_g$. Its action $R_*$ is anti-symplectic, and $R_*^{-1}B$ is symplectic. Realize $R_*^{-1}B$ by an orientation-preserving homeomorphism and compose with $R$. The result is orientation reversing and acts by $B$. These two homeomorphisms have the required degrees and homological actions.
\end{proof}

Fix an exceptional orbit $\mathcal{O}$ containing three distinct points $x_0,x_+,x_-$. Choose a basepoint $x\in\mathcal{O}$ and write $H=\G_x^x$. Use the connector system of Theorem~\ref{thm:relative-normal-form}: if $x_\sigma$ lies in the $\T$-orbit $\mathcal{T}_i$ with representative $y_i$ and connector $c_i:x\to y_i$, let $a_\sigma:y_i\to x_\sigma$ be the unique tail arrow and put $p_\sigma=a_\sigma c_i$. We identify an isotropy arrow $\gamma$ at $x_\sigma$ with $p_\sigma^{-1}\gamma p_\sigma\in H$. This construction also covers the case in which two or all three marked points lie in the same $\T$-orbit. They then use the same $c_i$ and different tail arrows. If $f:\Sigma_g\to BH$ represents a class $c\in H_2(H;\mathbb{Z})$, put
\begin{equation}
\label{eq:three-surface-maps}
f_0=f,
\qquad
f_+=f\circ\varphi_+,
\qquad
f_-=f\circ\varphi_-.
\end{equation}
For every oriented $1$-cell $s$ in the standard surface presentation, the three normalized return labels satisfy
\begin{equation}
\label{eq:three-return-balance}
[f_{0*}(s)]_{\ab}+[f_{+*}(s)]_{\ab}+[f_{-*}(s)]_{\ab}=0
\quad\text{in }H^{\ab},
\end{equation}
since the induced homomorphism $H_1(\Sigma_g;\mathbb{Z})\to H^{\ab}$ is additive and~\eqref{eq:three-surface-H1} holds. Equation~\eqref{eq:three-return-balance} is exactly the degree-one condition needed below: the three prescribed isotropy germs have zero return coordinate in the relative normal form, while their transport coordinate is automatically zero.

\begin{thm}[Parity formula for balanced returns]
\label{thm:balanced-surface-return-parity}
Let $\mathcal{O}$ be an exceptional $\G$-orbit containing at least three points, with group of germs $H=H_{\mathcal{O}}$. No assumption is made on the isotropy defect
$\lambda_{\mathcal{O}}:H^{\ab}\to H_0(\T)$. Let
$c\in H_2(H;\mathbb{Z})\subseteq H_2(\G;\mathbb{Z})$.

Represent $c$ by a map $f:\Sigma\to BH$ from a closed oriented surface. Discard spherical components, which represent zero in $H_2(BH;\mathbb{Z})$. For each remaining connected component, choose $\varphi_+,\varphi_-$ as in Lemma~\ref{lem:three-surface-balancing} and define $f_0,f_+,f_-$ by~\eqref{eq:three-surface-maps}. Choose also three distinct points $x_0,x_+,x_-\in\mathcal{O}$. For every standard surface generator $s$ and each of the three prescribed isotropy germs, choose an isolating bisection when that germ is nonidentity and a sufficiently small unit bisection when it is the identity. The three source neighborhoods may be chosen pairwise disjoint, and likewise the three range neighborhoods. These representative bisections have normalized return labels, as in Subsection~\ref{subsec:relative-normal-form},
\begin{equation}
\label{eq:balanced-generator-labels}
f_{0*}(s),\qquad f_{+*}(s),\qquad f_{-*}(s).
\end{equation}
Their union has zero class in $H_1(\G,\T)$ and thus, by Lemma~\ref{lemma:AF-clopen-completion}, extends by AF arrows to an element of the full group $u_s\in\F(\G)$. In particular,
\begin{equation}
\label{eq:balanced-generator-index-zero}
I(u_s)=0.
\end{equation}
If
\begin{equation}
\label{eq:balanced-surface-relator}
r_f^{\mathrm{bal}}
=\prod_{i=1}^{g}[u_{a_i},u_{b_i}],
\end{equation}
with the products over connected components multiplied together, then
$r_f^{\mathrm{bal}}\in\F(\T)$ and
\begin{equation}
\label{eq:balanced-surface-parity}
\theta_{\G}(c)
=q_2\bigl(\varepsilon_{\mathsf{B}}(r_f^{\mathrm{bal}})\bigr).
\end{equation}
The right-hand side is independent of the three points, the connectors, the representative bisections, the completions by the AF core and the chosen balancing homeomorphisms.
\end{thm}

\begin{proof}
Fix one standard surface generator $s$. The three prescribed germs at $x_0,x_+,x_-$ are isotropy germs, and thus their transport coordinates in Theorem~\ref{thm:relative-normal-form} vanish. Their return coordinate is the sum in~\eqref{eq:three-return-balance}, which is zero in $H^{\ab}$. Hence the union of the three representative bisections represents zero in $H_1(\G,\T)$. In particular, its source-minus-range defect in $H_0(\T)$ is zero. Lemma~\ref{lemma:AF-clopen-completion} therefore adjoins AF arrows and produces a full bisection $u_s$. Adding AF arrows does not change the relative class, and thus the index formula gives $I(u_s)=0$.

Do this independently for every surface generator. At the three marked points, the germ of the attaching word~\eqref{eq:balanced-surface-relator} is respectively the surface relator associated with $f_0$, $f_+$ and $f_-$. Each is the identity. Every $u_s$ fixes the three marked points, so their complement is invariant under every generator and its inverse. A word based outside the marked set therefore remains outside it, where all successive germs lie in $\T$. Since $\T$ is a subgroupoid, the germ of the attaching word also lies in $\T$. Therefore $r_f^{\mathrm{bal}}\in\F(\T)$.

It remains to identify which absolute degree-two class is represented by these three marked surface relations. On one connected component, precomposition with a self-map of degree $d$ multiplies the image of the oriented fundamental class by $d$. Hence the three marked contributions in the $\mathcal{O}$-summand of Theorem~\ref{thm:higher-germ-homology} are
$(f_0)_*[\Sigma_g]+(f_+)_*[\Sigma_g]+(f_-)_*[\Sigma_g]=c_0+c_0-c_0=c_0$,
since $\deg(\varphi_+)=1$ and $\deg(\varphi_-)=-1$, where $c_0$ is the class of that component. Summing over the connected components gives the original class $c$. Lemma~\ref{lem:marked-surface-parity}, applied componentwise to the same three marked points, now identifies the AF parity of the product of the component relators with $\theta_{\G}(c)$ and gives~\eqref{eq:balanced-surface-parity}. Since the left-hand side is intrinsic, the right-hand side is independent of the marked points, connectors, representative bisections, AF completions and balancing homeomorphisms.
\end{proof}

The intrinsic quantity in Theorem~\ref{thm:balanced-surface-return-parity} is the image under $q_2$. The class $\varepsilon_{\mathsf{B}}(r_f^{\mathrm{bal}})$ in $\mathsf{P}_{\mathsf{B}}$ may change when the marked points, connectors, or AF completions are changed. Any such change lies in $\ker q_2$. Section~\ref{sec:boundary-extensions} studies this parity before applying $q_2$.

\begin{defn}[Parity homomorphism from balanced returns]
\label{def:balanced-return-parity}
For an exceptional orbit $\mathcal{O}$ with at least three points, define
\begin{equation}
\label{eq:balanced-return-parity-map}
\kappa^{\mathrm{bal}}_{\mathcal{O}}:
H_2(H_{\mathcal{O}};\mathbb{Z})
\longrightarrow H_0(\G;\mathbb{Z}/2\mathbb{Z})
\end{equation}
by~\eqref{eq:balanced-surface-parity}.
\end{defn}
Theorem~\ref{thm:balanced-surface-return-parity} identifies $\kappa^{\mathrm{bal}}_{\mathcal{O}}$ with the restriction of $\theta_{\G}$ to the $\mathcal{O}$-summand of Theorem~\ref{thm:higher-germ-homology}. If $\lambda_{\mathcal{O}}=0$, it agrees with the homomorphism for a single return $\kappa_{\mathcal{O}}$ of Definition~\ref{def:return-parity-homomorphism}.

For a torus, the construction at three points gives explicit lifts of two commuting return germs, each with index zero.
\keepwithstatement
\begin{cor}
\label{cor:balanced-commuting-return}
Let $h,k\in H_{\mathcal{O}}$ commute and assume that $\mathcal{O}$ contains three points. Choose three marked points $x_0,x_+,x_-$. There are zero-index elements of the full group $U,V$ whose normalized exceptional labels are
\begin{equation}
\label{eq:balanced-torus-labels}
\begin{array}{c|ccc}
&x_0&x_+&x_-\\ \hline
U&h&h^{-2}k^{-1}&hk\\
V&k&h^{-1}k^{-1}&h
\end{array}
\end{equation}
with AF germs elsewhere. Their commutator belongs to $\F(\T)$ and
\begin{equation}
\label{eq:balanced-torus-parity}
\theta_{\G}([h|k]-[k|h])
=q_2\bigl(\varepsilon_{\mathsf{B}}([U,V])\bigr).
\end{equation}
\end{cor}

\begin{proof}
Use the matrices in~\eqref{eq:surface-balance-matrices} for the torus. Their first columns give the three labels of $U$ and their second columns give the three labels of $V$. Since $h$ and $k$ commute, the pair in each column of~\eqref{eq:balanced-torus-labels} defines a homomorphism from the torus group. The abelianized labels in each row sum to zero, and thus the two prescriptions of three germs have zero relative class and admit completions by the AF core. Theorem~\ref{thm:balanced-surface-return-parity} gives~\eqref{eq:balanced-torus-parity}.
\end{proof}

\begin{cor}
\label{cor:minimal-balanced-parity}
If $\G$ is minimal and AF-by-discrete on a Cantor space, then every exceptional $\G$-orbit is infinite. Hence, under the decomposition
\begin{equation}
H_2(\G;\mathbb{Z})
\cong\bigoplus_{\mathcal{O}}H_2(H_{\mathcal{O}};\mathbb{Z}),
\end{equation}
the secondary differential is the sum of the parity homomorphisms from balanced returns
\begin{equation}
\label{eq:minimal-theta-balanced-sum}
\theta_{\G}
=\sum_{\mathcal{O}}\kappa^{\mathrm{bal}}_{\mathcal{O}}.
\end{equation}
In particular, nonzero isotropy defect does not constitute a separate unresolved case in the minimal setting. For minimal AF-by-discrete groupoids with comparison, the strong AH property of Definition~\ref{def:AH-terms} is equivalent to the vanishing of all the homomorphisms $\kappa^{\mathrm{bal}}_{\mathcal{O}}$.
\end{cor}
\begin{proof}
A finite orbit in a Cantor space is closed. Minimality would therefore make it the whole unit space, which is impossible. Hence, every $\G$-orbit is infinite and in particular contains three distinct points. Theorem~\ref{thm:higher-germ-homology} decomposes $H_2(\G;\mathbb{Z})$ as the direct sum of the exceptional orbit summands, and Definition~\ref{def:balanced-return-parity} identifies the restriction of $\theta_{\G}$ to each summand with $\kappa^{\mathrm{bal}}_{\mathcal{O}}$. This proves~\eqref{eq:minimal-theta-balanced-sum}. Under comparison, Definition~\ref{def:AH-terms} and Li's low-degree exact sequence identify the strong AH property with $\theta_{\G}=0$, which, for a direct-sum domain, is equivalent to the vanishing of every summand map.
\end{proof}

The corollary reduces strong AH in the minimal case to the parity of balanced surface relations. A counterexample would require a balanced surface relation with nonzero image under~\eqref{eq:balanced-surface-parity}. The Houghton--Cantor example lies outside this test since its exceptional orbit is a singleton, even though the group of germs at that point is $\mathbb{Z}^2$. The lifting property for finite relations annihilates these parity homomorphisms. Whether minimality alone supplies enough lifting remains open.

Under comparison, the kernel theorem identifies the surviving parity contribution as $\Sym(\G)/\Alt(\G)$.
\keepwithstatement
\begin{prop}
\label{prop:minimal-comparison-parity-quotient}
Let $\G$ be minimal, effective, and AF-by-discrete on a Cantor space, and assume comparison. Then
\begin{equation}
\label{eq:minimal-comparison-KS-DA}
\ker I=\Sym(\G),
\qquad
\Der(\G)=\Alt(\G).
\end{equation}
Li's AH exact sequence consequently identifies the parity quotient as
\begin{equation}
\label{eq:minimal-comparison-parity-quotient}
\frac{\Sym(\G)}{\Alt(\G)}
\cong
\frac{H_0(\G;\mathbb{Z}/2\mathbb{Z})}{\operatorname{im}\theta_{\G}}.
\end{equation}
Equivalently, there is a short exact sequence
\begin{equation}
\label{eq:minimal-comparison-abelianization-extension}
0\longrightarrow \Sym(\G)/\Alt(\G)
\longrightarrow \F(\G)_{\ab}
\xrightarrow{\ I_{\ab}\ }H_1(\G)
\longrightarrow0.
\end{equation}
In particular, the equality $\ker I=\Sym(\G)$ does \emph{not} by itself imply the strong AH property. The latter is equivalent to the canonical parity map
\begin{equation}
H_0(\G;\mathbb{Z}/2\mathbb{Z})\longrightarrow \Sym(\G)/\Alt(\G)
\end{equation}
being injective, or equivalently to $\theta_{\G}=0$.
\end{prop}
\begin{proof}
Corollary~\ref{cor:no-singletons-kernel} gives $\ker I=\Sym(\G)$, since a minimal Cantor groupoid has no singleton orbit. Under minimality and comparison, \cite[Corollary~6.17]{Li25} then gives $\Der(\G)=\Alt(\G)$. In Li's exact sequence \cite[Corollary~6.14]{Li25}, the image of the parity map is the kernel of $I_{\ab}$. Since
$\ker I_{\ab}=\ker I/\Der(\G)=\Sym(\G)/\Alt(\G),$
and the kernel of the parity map is $\operatorname{im}\theta_{\G}$, the first isomorphism theorem gives~\eqref{eq:minimal-comparison-parity-quotient}. Equation~\eqref{eq:minimal-comparison-abelianization-extension} follows at once. The last assertion is exactly the injectivity criterion in \cite[Remark~6.16]{Li25}.
\end{proof}

\subsection{Boundary signatures and the closest known obstructions to strong AH}
\label{subsec:boundary-signature}

The balanced formula reduces the minimal problem to the homomorphisms
$\kappa_{\mathcal{O}}^{\mathrm{bal}}$, but it is useful to record an equivalent formulation in group-theoretic terms. It also makes precise the analogy with the signature of a finite permutation and with the Kapoudjian class of a near action.

\begin{defn}[Boundary signature]
\label{def:boundary-signature}
Assume that every $\G$-orbit has at least three points. A \emph{boundary signature} is a homomorphism
\begin{equation*}
\varepsilon_{\partial}:\Sym(\G)\longrightarrow H_0(\G;\mathbb{Z}/2\mathbb{Z})
\end{equation*}
such that, for every compact open bisection $F$ with $\sg(F)\cap\rg(F)=\varnothing$,
$\varepsilon_{\partial}(t_F)=[1_{\sg(F)}].$
Here $t_F$ is the dynamical transposition associated with $F$.
\end{defn}

Definition~\ref{def:boundary-signature} is intrinsic to $\G$. For an AF-by-discrete inclusion, the problem is to extend AF parity by corrections at the exceptional boundary.

\begin{prop}
\label{prop:boundary-signature-criterion}
Let $\G$ be a minimal effective AF-by-discrete groupoid with Cantor unit space and comparison. Then the following are equivalent.
\begin{enumerate}
\item The secondary homomorphism $\theta_{\G}:H_2(\G;\mathbb{Z})\to H_0(\G;\mathbb{Z}/2\mathbb{Z})$ is zero.
\item The canonical epimorphism $H_0(\G;\mathbb{Z}/2\mathbb{Z})\to \Sym(\G)/\Alt(\G)$ is an isomorphism.
\item A boundary signature exists.
\end{enumerate}
If $\T\subseteq\G$ is an AF core, then every boundary signature restricts on $\Sym(\T)$ to the AF parity map followed by the natural quotient
$\mathsf{P}_{\mathsf{B}}\to H_0(\G;\mathbb{Z}/2\mathbb{Z})$.
\end{prop}
\begin{proof}
By \cite[Corollary~6.14 and Remark~6.16]{Li25}, (1) is equivalent to injectivity of the canonical parity map into the abelianization. Proposition~\ref{prop:minimal-comparison-parity-quotient} identifies its image with $\Sym(\G)/\Alt(\G)$. Thus the relevant corestriction
\begin{equation*}
j_{\G}:H_0(\G;\mathbb{Z}/2\mathbb{Z})\longrightarrow \Sym(\G)/\Alt(\G)
\end{equation*}
is onto, and (1) is equivalent to (2).

Assume (2). Composing the quotient $\Sym(\G)\to\Sym(\G)/\Alt(\G)$ with $j_{\G}^{-1}$ gives a homomorphism $\varepsilon_{\partial}$, and the defining formula for the canonical parity map gives $\varepsilon_{\partial}(t_F)=[1_{\sg(F)}]$. Hence a boundary signature exists.

Conversely, suppose that $\varepsilon_{\partial}$ is a boundary signature. Every multisection $3$-cycle has order three, while the target of $\varepsilon_{\partial}$ has exponent two, and thus $\varepsilon_{\partial}$ vanishes on the generators of $\Alt(\G)$. It therefore descends to a homomorphism $\overline\varepsilon_{\partial}:\Sym(\G)/\Alt(\G)\to H_0(\G;\mathbb{Z}/2\mathbb{Z})$.

We claim that the classes $[1_{\sg(F)}]$ appearing in dynamical transpositions generate $H_0(\G;\mathbb{Z}/2\mathbb{Z})$. Indeed, let $U$ be clopen. For every $x\in U$, the orbit of $x$ is infinite, and thus choose an arrow from $x$ to a point different from $x$. Shrinking to a compact open bisection gives a clopen neighborhood $U_x\subseteq U$ whose image is disjoint from $U_x$. Compactness and zero-dimensionality refine these neighborhoods to a finite clopen partition $U=U_1\sqcup\cdots\sqcup U_m$, with a dynamical transposition having source $U_i$ for every $i$. Hence $[1_U]=\sum_i[1_{U_i}]$ is generated by source classes of transpositions. On each such generator,
$\overline\varepsilon_{\partial}\circ j_{\G}([1_{U_i}])=[1_{U_i}]$.
Thus $\overline\varepsilon_{\partial}\circ j_{\G}$ is the identity on all of $H_0(\G;\mathbb{Z}/2\mathbb{Z})$, and $j_{\G}$ is injective. This proves (2).

Finally, on $\Sym(\T)$ both a boundary signature and $q_2\circ\varepsilon_{\mathsf{B}}$ are homomorphisms, and they agree on every AF dynamical transposition. Since such transpositions generate $\Sym(\T)$, the two restrictions coincide.
\end{proof}

Matui already constructs a signature with values in the dimension group modulo two on the index kernel for a minimal Cantor homeomorphism in \cite[Section~4]{Matui06}. The present question asks for such a signature in the AF-by-discrete setting, where return isotropy may contribute an obstruction. Therefore, the remaining obstruction can be stated without spectra: one asks whether the signature on the AF core can always be corrected at the discrete boundary so as to become a homomorphism on the dynamical symmetric group. This is close in spirit to Lacourte's construction of a signature extending the parity of finitely supported permutations for piecewise continuous transformations \cite{Lac22}. It is also the splitting formulation of the Kapoudjian obstruction for near actions. See \cite[Section~8.C]{Cornulier19}. The point here is that the target is the generally nontrivial group $H_0(\G;\mathbb{Z}/2\mathbb{Z})$ rather than a single copy of $\mathbb{Z}/2\mathbb{Z}$.

Combining Proposition~\ref{prop:boundary-signature-criterion} with Corollary~\ref{cor:minimal-balanced-parity} gives the following equivalent form of the obstruction.

\begin{cor}
\label{cor:balanced-signature-equivalence}
Under the hypotheses of Proposition~\ref{prop:boundary-signature-criterion}, a boundary signature exists if and only if
\begin{equation*}
\sum_{\mathcal{O}}\kappa_{\mathcal{O}}^{\mathrm{bal}}=0.
\end{equation*}
In particular, since the decomposition of higher homology is a direct sum over exceptional orbits, it is enough that every parity homomorphism from balanced returns vanish.
\end{cor}

The criterion clarifies what a counterexample must contain. It is not enough to have an infinite group of germs, failure of coherence for the whole group, or even a nonrealizable near action in an abstract sense. One must produce a relation among finitely many returns whose residual permutation in $\F(\T)$ has nonzero class after passing to $H_0(\G;\mathbb{Z}/2\mathbb{Z})$.

The preceding criterion isolates the remaining obstruction entirely inside the AF-by-discrete setting. Once the contribution of the AF core has been encoded by the boundary signature, failure of strong AH can only come from a nontrivial compatibility obstruction between the exceptional germ data and the parity of the finite permutations in the AF core. The question is therefore whether the recurrence and comparison properties of a minimal AF-by-discrete groupoid force this boundary obstruction to vanish.

\begin{problem}[Minimal AF-by-discrete strong AH]
\label{prob:minimal-AFbd-strong-AH}
Let $\G$ be a minimal effective AF-by-discrete groupoid with Cantor unit space. Under comparison, must the boundary signature of Definition~\ref{def:boundary-signature} exist? Equivalently, must every parity homomorphism from balanced returns
\begin{equation*}
\kappa_{\mathcal{O}}^{\mathrm{bal}}:H_2(H_{\mathcal{O}};\mathbb{Z})
\longrightarrow H_0(\G;\mathbb{Z}/2\mathbb{Z})
\end{equation*}
vanish?
\end{problem}

The question is not resolved by the known permutation-type obstructions.
Proposition~\ref{prop:no-clopen-translation-rays} excludes a minimal groupoid realizing the Houghton construction by these clopen blocks whenever an invariant probability measure exists. The formulation in terms of a boundary signature gives a concrete route to a positive answer: construct a stabilized parity of the finite permutations in $\F(\T)$ produced at sufficiently large levels, corrected only by the finitely many persistent return configurations, and prove that the correction is independent of the chosen level. A negative route must instead produce a surface relation whose balanced relator has nonzero image under $q_2\circ\varepsilon_{\mathsf{B}}$ while respecting recurrence and, when an invariant probability measure is available, the obstruction of Proposition~\ref{prop:no-clopen-translation-rays}.

\subsubsection*{Clopen translation rays and invariant measures}
In the Houghton--Cantor example, the obstruction from returns is realized by permuting infinitely many disjoint Cantor blocks that accumulate at a singleton orbit. This mechanism is incompatible with an invariant probability measure of full support.

\keepwithstatement
\begin{prop}
\label{prop:no-clopen-translation-rays}
Let $\G$ be a minimal groupoid on a Cantor space and suppose that $\G$ admits an invariant probability measure $\mu$. Let $g\in\F(\G)$. There do not exist pairwise disjoint nonempty clopen sets $U_0,U_1,\ldots$ such that
$g(U_n)=U_{n+1}\qquad(n\geq0).$
More generally, let $\mathcal{X}$ be a connected locally finite graph with infinitely many vertices. There is no family of pairwise disjoint nonempty clopen sets $(U_v)_{v\in V\mathcal{X}}$ and compact open bisections that, along every edge $v--w$ outside a finite set of edges, map $U_v$ onto $U_w$, provided that deleting those finitely many edges leaves an infinite connected component.
\end{prop}
\begin{proof}
Minimality implies that every invariant probability measure has full support, hence every nonempty clopen set has positive measure. In the first assertion, invariance gives
$\mu(U_n)=\mu(U_0)>0$ for every $n$, contradicting finite total mass.

For the second assertion, pass to an infinite connected component after deleting the exceptional edges. Invariance along its edges makes $\mu(U_v)$ constant on that component. The common value is positive, while the sets are pairwise disjoint and there are infinitely many of them, again a contradiction.
\end{proof}

Almost finite Cantor groupoids admit invariant probability measures, and minimality makes every such measure have full support. See the construction preceding and the conclusion of \cite[Lemma~6.8]{Matui12}. Proposition~\ref{prop:no-clopen-translation-rays} therefore excludes clopen-block realizations of the Houghton near action in every minimal almost finite groupoid, including the minimal AF-by-discrete groupoids covered by Theorem~\ref{thm:minimal-AF-by-discrete-almost-finite}. The same measure argument applies to the one-ended near $\mathbb{Z}^2$-sets $X_{m,s}$ from \cite[Sections~5.E and 9.H]{Cornulier19}. This excludes the translation-ray mechanism itself. A realization not built from pairwise disjoint clopen blocks remains possible.

\section{Boundary extensions from balanced returns}\label{sec:boundary-extensions}
\label{subsec:balanced-boundary-extension}

We now construct an extension class before choosing a surface representative. Balanced triples of return germs admit zero-index lifts. For two such triples, the product of their chosen lifts multiplied by the inverse of the chosen lift of their product lies in the AF full group. Its parity gives a group-cohomology class.

The boundary extension controls the secondary differential. The AH extension
$0\to\mathsf{P}_{\G}\to\F(\G)_{\ab}\to H_1(\G)\to0$ discussed in Section~\ref{sec:parity} instead describes the subsequent abelian extension of $H_1(\G)$ by the surviving parity group.

Fix an exceptional $\G$-orbit $\mathcal{O}$ containing three distinct points $x_0,x_1,x_2$. Choose connectors $p_j:x_0\to x_j$, with $p_0=(\Id,x_0)$, and put $H=H_{x_0}$. We use the connectors to identify $H$ with the isotropy group at $x_j$ by
$h\mapsto p_jh p_j^{-1}$. Equivalently, an isotropy germ $g\in\G_{x_j}^{x_j}$ is normalized to $p_j^{-1}gp_j\in H$.

\begin{defn}[Balanced triple group]
\label{def:balanced-triple-group}
The \emph{balance homomorphism} is
\begin{equation*}
\operatorname{bal}:H^3\longrightarrow H^{\ab},\qquad
(h_0,h_1,h_2)\longmapsto
[h_0]_{\ab}+[h_1]_{\ab}+[h_2]_{\ab},
\end{equation*}
where $H^3$ has componentwise multiplication. The \emph{balanced triple group} is
\begin{equation}
\label{eq:balanced-triple-group}
B_3(H)=\ker(\operatorname{bal}).
\end{equation}
\end{defn}
Thus $B_3(H)$ consists exactly of triples of isotropy germs at $x_0,x_1,x_2$ whose return coordinate in the relative normal form is zero. Since all three prescribed germs are isotropy germs, their transport coordinate is already zero.

\begin{lemma}
\label{lemma:three-site-lifts}
For every $\mathbf{h}=(h_0,h_1,h_2)\in B_3(H)$, there is an element
$u_{\mathbf{h}}\in\ker I$ that fixes $x_0,x_1,x_2$, satisfies
$(u_{\mathbf{h}},x_j)=p_jh_jp_j^{-1}$ for $j=0,1,2$, and is AF at every other point. We take the identity homeomorphism as the lift of the identity element of $B_3(H)$.

If $u_{\mathbf{h}}$ and $u'_{\mathbf{h}}$ are two such lifts, then
$u'_{\mathbf{h}}u_{\mathbf{h}}^{-1}\in\F(\T)$ and this AF element has germ $(\Id,x_j)$ at each marked point.
\end{lemma}
\begin{proof}
For each $j$, if $p_jh_jp_j^{-1}$ is nonidentity, choose an isolating bisection $V_j$ containing that isotropy germ. If it is the identity germ, choose instead a sufficiently small unit bisection at $x_j$. Since the marked points are distinct, the three source neighborhoods may be chosen pairwise disjoint, and likewise the three range neighborhoods. To read the relative coordinate, compare $p_j$ with the connector supplied by Theorem~\ref{thm:relative-normal-form}. The two arrows differ on the right by an element of $H$, so the normalized return label is conjugate to $h_j$ and hence has the same class $[h_j]_{\ab}$ in $H^{\ab}$. The transport coordinate is zero since the arrow is isotropy. Thus the union $V_0\sqcup V_1\sqcup V_2$ represents the relative class
$([h_0]_{\ab}+[h_1]_{\ab}+[h_2]_{\ab},0)=0$.
Its source--range defect in $H_0(\T)$ therefore vanishes. Lemma~\ref{lemma:AF-clopen-completion} adjoins an AF bisection on the complementary clopen sets and produces a full bisection representing an element $u_{\mathbf{h}}\in\F(\G)$. The added arrows do not change the relative class, and thus Theorem~\ref{thm:index-formula} gives zero image of $I(u_{\mathbf{h}})$ in $H_1(\G,\T)$. The injection $H_1(\G)\hookrightarrow H_1(\G,\T)$ then gives $I(u_{\mathbf{h}})=0$.

For the second assertion, the two lifts have the same germ at each marked point, and thus their quotient has germ $(\Id,x_j)$ there. Both lifts fix the marked set pointwise, and therefore preserve its complement. If $y$ is not marked, then $u_{\mathbf{h}}^{-1}(y)$ is not marked either. The two arrows entering the germ of $u'_{\mathbf{h}}u_{\mathbf{h}}^{-1}$ at $y$ consequently lie in $\T$. Since $\T$ is a subgroupoid, the quotient germ lies in $\T$. Hence every germ of $u'_{\mathbf{h}}u_{\mathbf{h}}^{-1}$ belongs to $\T$, and thus the quotient lies in $\F(\T)$.
\end{proof}

The next lemma explains why the parity of the elements $u_{\mathbf{h}}u_{\mathbf{k}}u_{\mathbf{h}\mathbf{k}}^{-1}$ defined below takes values in an ordinary group-cohomology module with trivial action. Only AF elements that are the identity on a neighborhood of the marked points occur in the argument.

\keepwithstatement
\begin{lemma}
\label{lemma:admissible-conjugation-parity}
Let $u\in\F(\G)$ be AF at every point outside a finite set $E$ and fix every point of $E$. Let $a\in\F(\T)$ have the identity germ on a neighborhood of $E$. Then
$uau^{-1}\in\F(\T)$ and
\begin{equation}
\label{eq:admissible-conjugation-parity}
\varepsilon_{\mathsf{B}}(uau^{-1})=
\varepsilon_{\mathsf{B}}(a)
\quad\text{in }\mathsf{P}_{\mathsf{B}}.
\end{equation}
\end{lemma}
\begin{proof}
Choose a clopen neighborhood $N$ of $E$ on which $a$ is the identity. Choose an elementary stage of $\T$ containing the full bisection of $a$, and refine its tower decomposition so that $N$ is a union of clopen levels and $a$ permutes the levels in each tower. The levels contained in $N$ are fixed pointwise. On the remaining levels, $a$ factors into dynamical transpositions whose source and range sets are disjoint from $N$, hence from $E$. It is therefore enough to treat one such transposition $t_F$.

On $\sg(F)\cup\rg(F)$ the full-group element $u$ is AF. Hence conjugation sends $t_F$ to the dynamical transposition associated with the AF bisection $uFu^{-1}$, and thus $ut_Fu^{-1}\in\F(\T)$. By Definition~\ref{def:AF-parity},
$\varepsilon_{\mathsf{B}}(t_F)=[1_{\sg(F)}]$ and
$\varepsilon_{\mathsf{B}}(ut_Fu^{-1})=[1_{u(\sg(F))}]$.
The restriction of $u$ to $\sg(F)$ is an AF bisection, and therefore these two clopen sets define the same class in $H_0(\T;\mathbb{Z}/2\mathbb{Z})=\mathsf{P}_{\mathsf{B}}$. Summing over the transposition factorization proves both assertions.
\end{proof}

Choose once and for all a lift $u_{\mathbf{h}}$ as in Lemma~\ref{lemma:three-site-lifts} for every $\mathbf{h}\in B_3(H)$, with $u_{\mathbf{1}}=\Id$. Since every lift fixes the marked set pointwise, its complement is invariant. Hence $u_{\mathbf{h}}u_{\mathbf{k}}$ is AF away from the marked set and has at $x_j$ the germ $p_jh_jk_jp_j^{-1}$. It is therefore another lift of $\mathbf{h}\mathbf{k}$ in the sense of Lemma~\ref{lemma:three-site-lifts}. Consequently, the element
\begin{equation*}
d(\mathbf{h},\mathbf{k})=
 u_{\mathbf{h}}u_{\mathbf{k}}u_{\mathbf{h}\mathbf{k}}^{-1}
\end{equation*}
lies in $\F(\T)$ and has the identity germ at each marked point. Since the marked set is finite, this identity-germ condition means that $d(\mathbf{h},\mathbf{k})$ is the identity on a neighborhood of the marked set. Define
\begin{equation}
\label{eq:balanced-extension-cocycle}
\omega_{\mathcal{O}}(\mathbf{h},\mathbf{k})
=
\varepsilon_{\mathsf{B}}\!\left(d(\mathbf{h},\mathbf{k})\right)
\in\mathsf{P}_{\mathsf{B}}.
\end{equation}
We regard $\mathsf{P}_{\mathsf{B}}$ as a trivial $B_3(H)$-module. For this action, a normalized $2$-cocycle satisfies
$\omega(\mathbf{1},\mathbf{h})=\omega(\mathbf{h},\mathbf{1})=0$ and
$\omega(\mathbf{h},\mathbf{k})+\omega(\mathbf{h}\mathbf{k},\mathbf{l})
=\omega(\mathbf{k},\mathbf{l})+\omega(\mathbf{h},\mathbf{k}\mathbf{l})$.
Two such cocycles differ by a coboundary if their difference has the form
$b(\mathbf{h})+b(\mathbf{k})-b(\mathbf{h}\mathbf{k})$.

\keepwithstatement
\begin{prop}
\label{prop:balanced-boundary-extension-class}
The function~\eqref{eq:balanced-extension-cocycle} is a normalized $2$-cocycle with values in the trivial $B_3(H)$-module $\mathsf{P}_{\mathsf{B}}$. Replacing the chosen lifts changes it by a coboundary. Hence the marked triple and the connector identifications determine a cohomology class
\begin{equation}
\label{eq:balanced-extension-class}
[\omega_{\mathcal{O}}]
\in H^2\!\left(B_3(H);\mathsf{P}_{\mathsf{B}}\right).
\end{equation}
\end{prop}
\begin{proof}
Normalization follows from $u_{\mathbf{1}}=\Id$. Associativity gives
\begin{equation*}
d(\mathbf{h},\mathbf{k})d(\mathbf{h}\mathbf{k},\mathbf{l})
=u_{\mathbf{h}}d(\mathbf{k},\mathbf{l})u_{\mathbf{h}}^{-1}
 d(\mathbf{h},\mathbf{k}\mathbf{l}).
\end{equation*}
The conjugated factor $d(\mathbf{k},\mathbf{l})$ is AF and has identity germ at the three marked points. Lemma~\ref{lemma:admissible-conjugation-parity} therefore shows that its parity is unchanged by conjugation with $u_{\mathbf{h}}$. Applying $\varepsilon_{\mathsf{B}}$ to the displayed identity gives the cocycle equation.

Now let $u'_{\mathbf{h}}=a_{\mathbf{h}}u_{\mathbf{h}}$ be another normalized system of lifts, so that $u'_{\mathbf{1}}=\Id$. Lemma~\ref{lemma:three-site-lifts} gives $a_{\mathbf{h}}\in\F(\T)$ with identity germ at all three marked points. Expanding the new defect yields
\begin{equation*}
d'(\mathbf{h},\mathbf{k})
=a_{\mathbf{h}}
 \bigl(u_{\mathbf{h}}a_{\mathbf{k}}u_{\mathbf{h}}^{-1}\bigr)
 d(\mathbf{h},\mathbf{k})
 a_{\mathbf{h}\mathbf{k}}^{-1}.
\end{equation*}
Lemma~\ref{lemma:admissible-conjugation-parity} and additivity of AF parity therefore give
\begin{equation*}
\omega'_{\mathcal{O}}(\mathbf{h},\mathbf{k})-\omega_{\mathcal{O}}(\mathbf{h},\mathbf{k})
=
\varepsilon_{\mathsf{B}}(a_{\mathbf{h}})+
\varepsilon_{\mathsf{B}}(a_{\mathbf{k}})-
\varepsilon_{\mathsf{B}}(a_{\mathbf{h}\mathbf{k}}),
\end{equation*}
which is the coboundary of $\mathbf{h}\mapsto\varepsilon_{\mathsf{B}}(a_{\mathbf{h}})$.
\end{proof}

\begin{defn}[Balanced boundary extension class]
\label{def:balanced-boundary-extension}
For the chosen marked triple and connector identifications, the class
$[\omega_{\mathcal{O}}]$ of Proposition~\ref{prop:balanced-boundary-extension-class} is the \emph{balanced boundary extension class} of the orbit $\mathcal{O}$.
\end{defn}

We construct the corresponding central extension from the full-group lifts. Put $E=\{x_0,x_1,x_2\}$ and let $G_E$ be the subgroup of $\F(\G)$ consisting of elements that fix $E$ pointwise, are AF on $\Omega(\mathsf{B})\setminus E$, and whose normalized germ triple belongs to $B_3(H)$. Evaluation at the three marked germs gives a surjective homomorphism
\begin{equation}
\label{eq:concrete-balanced-extension}
G_E\longrightarrow B_3(H).
\end{equation}
Its kernel $K_E$ consists precisely of the elements of $\F(\T)$ with identity germ at the three marked points. Lemma~\ref{lemma:admissible-conjugation-parity} shows that $\varepsilon_{\mathsf{B}}|_{K_E}$ is invariant under conjugation by $G_E$. Therefore
$N_E=\ker(\varepsilon_{\mathsf{B}}|_{K_E})$ is normal in $G_E$, the quotient $K_E/N_E$ is central in $G_E/N_E$, and
\begin{equation*}
K_E/N_E\cong\varepsilon_{\mathsf{B}}(K_E)\leq\mathsf{P}_{\mathsf{B}}.
\end{equation*}
Pushing this central extension out along the inclusion
$\varepsilon_{\mathsf{B}}(K_E)\hookrightarrow\mathsf{P}_{\mathsf{B}}$ gives a central extension of $B_3(H)$ by $\mathsf{P}_{\mathsf{B}}$. The section determined by the chosen $u_{\mathbf{h}}$ has multiplication cocycle exactly~\eqref{eq:balanced-extension-cocycle}.

Equivalently, the pushed-out extension has underlying set
$\mathsf{P}_{\mathsf{B}}\times B_3(H)$ and multiplication
$(a,\mathbf{h})(b,\mathbf{k})=(a+b+\omega_{\mathcal{O}}(\mathbf{h},\mathbf{k}),\mathbf{h}\mathbf{k})$.
Its extension class is exactly $[\omega_{\mathcal{O}}]$, and it splits precisely when this class vanishes. This is the standard classification of central extensions by degree-two group cohomology \cite[Chapter~IV]{Brown82}.

The class in Definition~\ref{def:balanced-boundary-extension} is not asserted to be intrinsic before coordinate choices. For example, replacing $p_j$ by $p_ja_j$ with $a_j\in H$ conjugates the $j$th coordinate by $a_j^{-1}$. This preserves $B_3(H)$ since inner automorphisms act trivially on $H^{\ab}$, and it changes the displayed cocycle by the corresponding pullback together with a change of section. Changing the marked points gives a possibly different concrete extension. We do not identify these extensions canonically before applying $q_2$. What becomes intrinsic is the evaluation on balanced surface classes after applying the natural quotient
$q_2:\mathsf{P}_{\mathsf{B}}\to H_0(\G;\mathbb{Z}/2\mathbb{Z})$ from~\eqref{eq:AF-parity-to-G}. We write
$\overline\omega_{\mathcal{O}}=q_2\circ\omega_{\mathcal{O}}$ and
$[\overline\omega_{\mathcal{O}}]=q_{2*}[\omega_{\mathcal{O}}]$.

\begin{thm}[Boundary-extension formula]
\label{thm:cohomological-balanced-parity}
Let $c\in H_2(H;\mathbb{Z})$ and represent it by a map
$f:\Sigma_g\to BH$ from a closed connected oriented surface of genus $g\geq1$. Choose the balancing homeomorphisms
$\varphi_+,\varphi_-$ of Lemma~\ref{lem:three-surface-balancing}, and write
$\rho=f_*:\pi_1(\Sigma_g)\to H$, $\rho_+=\rho\circ(\varphi_+)_*$, and
$\rho_-=\rho\circ(\varphi_-)_*$.
Then
\begin{equation}
\label{eq:balanced-surface-homomorphism}
\widehat\rho:\pi_1(\Sigma_g)\longrightarrow B_3(H),\qquad
\gamma\longmapsto
(\rho(\gamma),\rho_+(\gamma),\rho_-(\gamma))
\end{equation}
is a homomorphism, and
\begin{equation}
\label{eq:theta-boundary-extension-pairing}
\kappa_{\mathcal{O}}^{\mathrm{bal}}(c)
=
q_2\!\left(
\left\langle
\widehat\rho^{\,*}[\omega_{\mathcal{O}}],[\Sigma_g]
\right\rangle
\right).
\end{equation}
Here we identify $[\Sigma_g]$ with the fundamental class in the group homology of $\pi_1(\Sigma_g)$. The bracket is the usual Kronecker pairing between group cohomology with trivial coefficients and group homology. Equivalently,
$\langle\widehat\rho^{\,*}[\omega_{\mathcal{O}}],[\Sigma_g]\rangle
=\langle[\omega_{\mathcal{O}}],\widehat\rho_*[\Sigma_g]\rangle$.
In particular, after applying $q_2$, the value is independent of the surface representative, the balancing homeomorphisms, the marked points, the connectors, the representative bisections, and the chosen lifts.
\end{thm}
\begin{proof}
The identity
$\operatorname{id}+(\varphi_+)_*+(\varphi_-)_*=0$ on
$H_1(\Sigma_g;\mathbb{Z})$ implies
\begin{equation*}
[\rho(\gamma)]_{\ab}+[\rho_+(\gamma)]_{\ab}+[\rho_-(\gamma)]_{\ab}=0
\quad\text{in }H^{\ab}
\end{equation*}
for every $\gamma\in\pi_1(\Sigma_g)$. Hence~\eqref{eq:balanced-surface-homomorphism} lands in $B_3(H)$.

To identify the pairing, use the standard presentation
$\pi_1(\Sigma_g)=\langle a_1,b_1,\ldots,a_g,b_g\mid\prod_i[a_i,b_i]=1\rangle$ and pull the central extension preceding the theorem back along $\widehat\rho$. Give $\Sigma_g$ its standard CW structure with one vertex, the displayed $1$-cells and one $2$-cell attached by the surface word. Choose the section lifts of the $1$-cell labels. The obstruction to extending these lifts across the $2$-cell is the element of the central kernel obtained by evaluating the lifted attaching word, namely the product of the commutators of the chosen lifts. By the cellular description of the extension class, this obstruction is
$\langle\widehat\rho^{\,*}[\omega_{\mathcal{O}}],[\Sigma_g]\rangle$. Depending on the convention for the surface fundamental class and commutator word, the element may be written with the opposite sign. The distinction is immaterial since $\mathsf{P}_{\mathsf{B}}$ has exponent $2$.

For the concrete extension~\eqref{eq:concrete-balanced-extension}, take the AF completions in Theorem~\ref{thm:balanced-surface-return-parity} to be the chosen lifts $u_{\widehat\rho(s)}$ for the triples
$(\rho(s),\rho_+(s),\rho_-(s))$. Their lifted attaching word is then precisely the balanced relator $r_f^{\mathrm{bal}}$. It belongs to $K_E$, and its image in the central kernel of the pushed-out extension is
$\varepsilon_{\mathsf{B}}(r_f^{\mathrm{bal}})$. Therefore
\begin{equation*}
\left\langle
\widehat\rho^{\,*}[\omega_{\mathcal{O}}],[\Sigma_g]
\right\rangle
=
\varepsilon_{\mathsf{B}}(r_f^{\mathrm{bal}})
\quad\text{in }\mathsf{P}_{\mathsf{B}}.
\end{equation*}
Theorem~\ref{thm:balanced-surface-return-parity} then gives
$q_2(\varepsilon_{\mathsf{B}}(r_f^{\mathrm{bal}}))
=\kappa_{\mathcal{O}}^{\mathrm{bal}}(c)=\theta_{\G}(c)$,
which proves~\eqref{eq:theta-boundary-extension-pairing}. Since the last expression is intrinsic, all choices disappear after applying $q_2$.
\end{proof}

\begin{defn}[Balanced surface subgroup]
\label{def:balanced-surface-subgroup}
Let $\mathcal{B}_2(H)\subseteq H_2(B_3(H);\mathbb{Z})$ be the subgroup generated by the classes
$\widehat\rho_*[\Sigma_g]$ arising from all surface representatives of classes in $H_2(H;\mathbb{Z})$ and all choices of balancing homeomorphisms allowed by Lemma~\ref{lem:three-surface-balancing}, through the construction of Theorem~\ref{thm:cohomological-balanced-parity}. We call $\mathcal{B}_2(H)$ the \emph{balanced surface subgroup}.
\end{defn}

The following criterion uses evaluation on $\mathcal{B}_2(H)$ rather than vanishing of the whole cohomology class.
\keepwithstatement
\begin{cor}
\label{cor:boundary-extension-criterion}
Assume that $\G$ is minimal, effective, and AF-by-discrete with Cantor unit space and comparison. For every exceptional orbit, choose three marked points and form the balanced boundary extension class. The conditions below are independent of these choices, and the following are equivalent.
\begin{enumerate}
\item $\G$ has the strong AH property, equivalently $\theta_{\G}=0$.
\item For every exceptional orbit $\mathcal{O}$, the class
$[\overline\omega_{\mathcal{O}}]=q_{2*}[\omega_{\mathcal{O}}]$ pairs trivially with
$\mathcal{B}_2(H_{\mathcal{O}})$.
\item Every class obtained from the balanced lift of a surface relation has zero residual AF parity in $H_0(\G;\mathbb{Z}/2\mathbb{Z})$.
\end{enumerate}
A sufficient, generally stronger, condition is that, for every exceptional orbit,
\begin{equation*}
[\overline\omega_{\mathcal{O}}]=0
\quad\text{in}\quad
H^2\!\left(B_3(H_{\mathcal{O}});H_0(\G;\mathbb{Z}/2\mathbb{Z})\right).
\end{equation*}
\end{cor}
\begin{proof}
By Theorem~\ref{thm:higher-germ-homology},
$H_2(\G;\mathbb{Z})$ is the direct sum of the groups
$H_2(H_{\mathcal{O}};\mathbb{Z})$. Corollary~\ref{cor:minimal-balanced-parity} identifies the restriction of $\theta_{\G}$ to the $\mathcal{O}$-summand with
$\kappa_{\mathcal{O}}^{\mathrm{bal}}$, and Theorem~\ref{thm:cohomological-balanced-parity} identifies this homomorphism with evaluation of
$[\overline\omega_{\mathcal{O}}]$ on the balanced surface classes. This proves the equivalence of (1)--(3). If the entire pushed-forward cohomology class vanishes, then all such pairings vanish, proving the final assertion.
\end{proof}

Minimality removes the obstruction from singleton orbits to the kernel of the index map and ensures that every exceptional orbit contains three distinct points. The remaining strong AH question is whether the pairings in Corollary~\ref{cor:boundary-extension-criterion} vanish.

The lifting property for finite relations supplies a local realization of every finite system of relations occurring in a surface class, and Theorem~\ref{thm:finite-relations-split-AH} therefore makes all of these pairings zero. The Houghton--Cantor example exhibits the analogous parity of an element measuring the failure of chosen lifts to multiply exactly in a nonminimal setting. The construction of the balanced boundary extension class does not apply there, since its exceptional orbit is a singleton. Whether recurrence forces this vanishing in every minimal AF-by-discrete groupoid remains the question of Problem~\ref{prob:minimal-AFbd-strong-AH}.

The Kapoudjian class of a balanced near action is likewise obtained by lifting near permutations and measuring the parity of the finitely supported permutation produced when the chosen lifts fail to multiply exactly. See \cite[Section~8.C]{Cornulier19}. Here that finitely supported permutation is replaced by an element of $\F(\T)$, and its parity takes values in the dimension group of $\mathsf{B}$ modulo two. The analogy is structural. The present class is defined from the fixed AF inclusion, and we identify it with Cornulier's class only when an actual near action realizes the same lifts.

\section{Applications and remaining questions}\label{sec:applications}

We apply the boundary formulas to near full groups and to minimal Cantor systems, and then record the remaining questions.

\subsection{Near full groups}
\label{sec:nearfull}
A subgroup $G\leq\F(\G)$ is \emph{near full} when it has finite index (index of a subgroup; not to confuse with the index map) in the topological full group of its groupoid of germs. We use the finite-singular-germ setting of \cite{kua26a,kua26b}. The completion identity determines the image of the index map. Finite index alone gives the subgroup consequence below.

\keepwithstatement
\begin{cor}
\label{cor:near-full-index}
Let $G\leq\F(\G)$ and suppose that
\begin{equation*}
\F(\G)=\Sym(\T)G.
\end{equation*}
Equivalently, since $\T$ is AF, $\F(\G)=\F(\T)G$. Then $I(G)=I(\F(\G))$. In particular, if the AF core $\T$ is minimal, then the restriction
$I|_G:G\to H_1(\G)$ is surjective.
\end{cor}
\begin{proof}
Since $\T$ is AF, $\Sym(\T)=\F(\T)$ by Subsection~\ref{subsec:full-groups-generation}. Moreover, $H_1(\T)=0$, and thus the index vanishes on $\F(\T)$. Hence every $g\in\F(\G)$ can be written as $g=ah$ with $a\in\F(\T)$ and $h\in G$, and thus $I(g)=I(h)$. This proves $I(\F(\G))\subseteq I(G)$, and the reverse inclusion is immediate. If $\T$ is minimal, Corollary~\ref{cor:index-surjective} gives $I(\F(\G))=H_1(\G)$.
\end{proof}

Finite index by itself gives only that $I(G)$ has finite index in $I(\F(\G))$. Equality in Corollary~\ref{cor:near-full-index} uses the completion identity. In \cite[Theorem~4.10]{kua26a}, localization of the finite singular germs gives generation of the full group after adjoining representatives of the missing AF parity classes. The product identity used here additionally requires the parity condition on AF truncations in that theorem. Thus the present index formula can be applied to a near full group without first replacing its generators by arbitrary full-group elements: once completion has been verified, the subgroup already realizes every index class realized by the full group.

There is a separate finite-index consequence for lamplighter subgroups. Let $U$ be a nonempty clopen subset of the unit space and let $t:U\to U$ be a minimal Cantor homeomorphism. We write $[[t]]$ for its topological full group, extended by the identity on the complement of $U$ when $U$ is proper.

\keepwithstatement
\begin{cor}
\label{cor:near-full-lamplighter}
Let $G\leq\F(\G)$ have finite index. Suppose that $[[t]]\leq\F(\G)$ for a minimal non-odometer Cantor homeomorphism $t$ as above. Then $G$ contains a subgroup isomorphic to $C_2\wr\mathbb{Z}$.
\end{cor}
\begin{proof}
Let $N$ be the normal core of $G$ in $\F(\G)$. Then $N$ is a finite-index normal subgroup of $\F(\G)$. Put $D=D([[t]])$. Matui proved that $D$ is infinite and simple for a minimal Cantor system \cite{Matui06}, and that $D$ contains $C_2\wr\mathbb{Z}$ exactly when $t$ is not an odometer \cite[Theorem~2.4]{Matui13}. Since $D\cap N$ is a finite-index normal subgroup of the infinite simple group $D$, it is nontrivial and therefore equals $D$. Hence $D\leq N\leq G$, and the lamplighter subgroup is contained in $G$.
\end{proof}

The inclusion $[[t]]\leq\F(\G)$, on the whole unit space or after extension from a clopen reduction, is a separate dynamical input. Once that input is established, finite index alone is enough. The completion identity is not needed.

We use the near full groups of \cite{kua26a,kua26b} only in the finite-singular-germ setting treated there. No extension of near fullness to infinite groups of germs, and no implication from the lifting property for finite relations to near fullness, is used in this paper.

\subsection{Minimal Cantor systems and other boundary models}
\label{sec:examples}
Let $\tau$ be a minimal Cantor homeomorphism and cut its orbit at one point as in Example~\ref{exmp:one-border-index}. The exceptional orbit then contains exactly two $\T$-orbits, namely, the two half-orbits of the AF core, all groups of germs are trivial, and the only degree-one boundary information is transport across the cut. The geometric index is therefore the signed migration number of~\eqref{eq:one-border-flux}, recovering the classical index. Theorem~\ref{thm:higher-germ-homology} also gives $H_n(\G;\mathbb{Z})=0$ for $n\geq2$, since the exceptional groups of germs are trivial. Hence the secondary differential vanishes automatically, and Theorem~\ref{thm:finite-germ-split-AH} recovers the split AH description of the abelianization in this case.

In the standard binary adding-machine presentation, the same picture is visible directly in the Bratteli diagram: the cut produces two infinite tiles with trivial groups of germs, while the carry generator supplies the transport across the unique boundary cut. In the odometer example there is one cut and no return coordinate from isotropy. In contrast, Corollary~\ref{cor:near-full-lamplighter} does not apply to it, exactly in agreement with Matui's odometer exception.

The finite-singular-germ examples motivating \cite{kua26a,kua26b} fit a different part of the general theory. The modified Grigorchuk construction of \cite[Section~5.3.5]{nek22}, the fragmentation group of the modified LMS-group associated with the Penrose tiling in \cite[Section~5]{kua26a}, and the fragmented Fabrykowski--Gupta construction of \cite[Section~6]{kua26b} are boundary models in which the groups of germs at the relevant exceptional orbits are finite. Once minimality, effectiveness, and comparison are verified for the groupoid under consideration, Theorem~\ref{thm:finite-germ-split-AH} applies. The structural calculation then reduces to the following data: the groups $H_{\mathcal{O}}^{\ab}$, the transport lattices $L_{\mathcal{O}}$, and the transport defect $\tau$. In particular, the secondary differential vanishes in this finite-germ regime and the full-group abelianization is given by~\eqref{eq:finite-germ-full-abelianization}.

For each model, we must identify the persistent boundary data, compute the connector coordinates and $\tau$, and, for statements about a prescribed acting subgroup rather than $\F(\G)$, verify that the required AF permutations and the elements realizing the prescribed singular germs belong to that subgroup. The selector and localization mechanisms of \cite{kua26a,kua26b} address precisely this last membership problem. 

\subsection{Remaining questions}
\label{sec:remaining}
\paragraph{Strong AH.}
The principal open problem remains Problem~\ref{prob:minimal-AFbd-strong-AH}. Sections~\ref{sec:surface-parity} and~\ref{sec:boundary-extensions} reduce it to an exact boundary statement: for every exceptional orbit $\mathcal{O}$, the pushed-forward balanced boundary extension class must annihilate the balanced surface subgroup $\mathcal{B}_2(H_{\mathcal{O}})$. Full vanishing of the cohomology class is stronger than necessary. A positive solution therefore requires a recurrence mechanism forcing these particular pairings to vanish. A negative solution requires a surface relation for which the product of commutators obtained after balancing and lifting has nonzero image under $q_2\circ\varepsilon_{\mathsf{B}}$ while remaining compatible with the recurrence and invariant-measure constraints of a minimal AF-by-discrete groupoid.

\paragraph{Nonsplit abelianization after vanishing of the secondary differential.}
Even if $\theta_{\G}=0$, the resulting short exact sequence
$0\to H_0(\G;\mathbb{Z}/2\mathbb{Z})\to\F(\G)_{\ab}\to H_1(\G)\to0$
need not split for formal homological reasons alone. The present results, however, exclude nonsplitting in two large regimes: finite groups of germs by Theorem~\ref{thm:finite-germ-split-AH}, and, more generally, the lifting property for finite relations by Theorem~\ref{thm:finite-relations-split-AH}. Consequently, under minimality and comparison, any nonsplit example must have at least one infinite group of germs and must fail the finite-relation lifting property at some exceptional orbit. Whether such an AF-by-discrete example exists is open.

\paragraph{Almost finiteness without compact generation.}
Theorem~\ref{thm:AF-exhaustion-criterion} completely settles approximation by the fixed AF core: it works exactly when no exceptional arrow starts on a finite $\T$-orbit. Theorem~\ref{thm:minimal-AF-by-discrete-almost-finite} shows that, under minimality and compact generation, failure of that fixed-core criterion can be repaired by replacing the approximating elementary subgroupoids. The remaining minimal case is when $\G$ is not compactly generated and the chosen AF core has finite orbits carrying exceptional arrows. If the AF core has no such finite orbit, the fixed-core criterion already gives almost finiteness without compact generation.

\paragraph{Effective reconstruction.}
The constructive kernel theorem expresses an element of $\ker I$ using finite AF matchings, exact germ relations, dynamical transpositions realizing the prescribed germs, and a final factor in $\F(\T)$. These data form a certificate, but the theorem gives no algorithm for finding them. Section~\ref{sec:certified} gives terminating procedures for the displayed finite-state model and a stabilization bound for verified periodic incidence data. For more general presentations, we still need to determine complete germ relations, verify the future incidence rule, and find a sufficiently deep AF stage. For a prescribed subgroup $G\leq\F(\G)$, we must also express the transpositions and the required elements of $\F(\T)$ as words in its generators. Neither task follows from the relative normal form alone.



\end{document}